\documentclass[10pt,a4paper]{article}

\usepackage[tt=false]{libertine}

\usepackage[utf8]{inputenc}
\usepackage{answers}
\usepackage{setspace}

\usepackage[nottoc]{tocbibind}
\usepackage{graphicx}
\usepackage{enumitem}
\usepackage{multicol}
\usepackage{csquotes}
\usepackage[margin=2cm]{geometry}
\usepackage{tikz}
\usetikzlibrary{arrows.meta}
\usepackage{tikz-cd}
\usepackage{amsmath,amsthm}
\usepackage{newtxmath}

\usepackage{pst-node}
\usepackage{mathtools}
\usepackage{diagbox}
\usepackage{quiver}
\usepackage{hyperref}
\usepackage{mathpartir}
\usepackage{xcolor}
\usepackage{tcolorbox}
\usepackage{abraces}
\usepackage{relsize}

\usepackage{xcolor}
\definecolor{newpurple}{RGB}{120, 0, 40}

\usepackage{hyperref}
\hypersetup{
	citecolor=newpurple,
	linkcolor=newpurple,
	colorlinks=true,
}

\usepackage[backend=biber,style=alphabetic]{biblatex}
\setlist[itemize]{itemsep=-1pt, topsep=3pt}
\setlist[enumerate]{itemsep=-1pt, topsep=3pt}

\theoremstyle{definition}
\newtheorem{theorem}{Theorem}[section]
\newtheorem{lemma}[theorem]{Lemma}
\newtheorem{definition}[theorem]{Definition}
\newtheorem{example}[theorem]{Example}
\newtheorem*{definition*}{Definition}
\newtheorem*{proposition*}{Proposition}
\newtheorem{proposition}[theorem]{Proposition}

\newtheorem{remark}[theorem]{Remark}
\newtheorem{notation}[theorem]{Notation}
\newtheorem*{notation*}{Notation}

\newtheorem*{theorem*}{Theorem}
\newtheorem*{lemma*}{Lemma}
\newtheorem*{convention*}{Convention}
\newtheorem*{example*}{Example}

\newtheorem{corollary}[theorem]{Corollary}
\newtheorem{construction}[theorem]{Construction}

\numberwithin{equation}{section}

\newcommand{\Id}{\text{Id}}

\newcommand{\Ar}{\text{Ar}}
\newcommand{\iiAr}{\textbf{Ar}}

\newcommand{\Gpd}{\text{Gpd}}

\newcommand{\op}{\text{op}}

\newcommand{\Set}{\text{Set}}

\newcommand{\Ho}{\text{Ho}}
\newcommand{\iiHo}{\textbf{Ho}}
\newcommand{\Ob}{\text{Ob}}
\newcommand{\iiOb}{\textbf{Ob}}

\newcommand{\Cat}{\text{Cat}}
\newcommand{\iiCat}{\textbf{Cat}}

\newcommand{\dom}{\text{dom}}

\newcommand{\SSet}{\text{SSet}}

\newcommand{\GAT}{\text{GAT}}
\newcommand{\iiLFP}{\textbf{LFP}}

\newcommand{\Cont}{\text{Cont}}
\newcommand{\iiCont}{\textbf{Cont}}

\newcommand{\iiDMC}{\textbf{DMC}}
\newcommand{\dmc}{\text{dmc}}

\newcommand{\pre}{\text{pre}}

\newcommand{\Precont}{\text{Precont}}
\newcommand{\iiPrecont}{\textbf{Precont}}

\newcommand{\cont}{\text{cont}}

\newcommand{\Mod}{\text{Mod}}
\newcommand{\iiMod}{\textbf{Mod}}

\newcommand{\gpd}{\text{gpd}}

\newcommand{\Att}{\text{Att}}
\newcommand{\att}{\text{att}}

\newcommand{\hatotimes}{\mathrel{\widehat{\otimes}}}

\renewcommand{\lim}{\underleftarrow{\text{lim }}}
\newcommand{\colim}{\underrightarrow{\text{lim }}}

\newcommand{\tp}{\;\mathsf{sort}}
\newcommand{\tm}{\;\mathsf{term}}

\newcommand{\ctx}{\;\textsf{ctx}}

\newcommand{\Fam}{\text{Fam}}

\newcommand{\bbA}{\mathbb A}
\newcommand{\bbB}{\mathbb B}

\newcommand{\bbE}{\mathbb E}

\newcommand{\bbK}{\mathbb K}
\newcommand{\bbN}{\mathbb N}
\newcommand{\bbO}{\mathbb O}

\newcommand{\bbS}{\mathbb S}
\newcommand{\bbT}{\mathbb T}

\newcommand{\frakL}{\mathfrak L}
\newcommand{\frakR}{\mathfrak R}

\newcommand{\btle}{\blacktriangleleft}

\newcommand{\llp}{\mathfrak L}
\newcommand{\rlp}{\mathfrak R}

\newcommand{\cof}{\text{cof}}
\newcommand{\fib}{\text{fib}}
\newcommand{\acof}{\text{acof}}
\newcommand{\afib}{\text{afib}}
\newcommand{\bcof}{\text{cof}^*}
\newcommand{\bacof}{\text{acof}^*}
\newcommand{\cell}{\text{cell}}

\def\circrightarrow{\mathrel{{
			\setbox0\hbox{$\longrightarrow$}
			\rlap{\hbox to \wd0{\hss$\circ$\hss}}\box0
}}}

\def\bulletrightarrow{\mathrel{{
			\setbox0\hbox{$\longrightarrow$}
			\rlap{\hbox to \wd0{\hss$\bullet$\hss}}\box0
}}}

\newcommand{\squa}[8]{\begin{tikzcd}[ampersand replacement=\&]
		{#1} \& {#2} \\
		{#3} \& {#4}
		\arrow["{#7}"', from=1-1, to=2-1]
		\arrow["{#6}"', from=2-1, to=2-2]
		\arrow["{#8}", from=1-2, to=2-2]
		\arrow["{#5}", from=1-1, to=1-2]
	\end{tikzcd}
}

\newcommand{\widesqua}[8]{\begin{tikzcd}[ampersand replacement=\&,column sep=2.5cm]
		#1 \arrow[swap]{d}{#7} \arrow[]{r}{#5} \& #2 \arrow[]{d}{#8}\\
		#3 \arrow[swap]{r}{#6} \& #4
	\end{tikzcd}
}

\newcommand{\dsqua}[8]{\begin{tikzcd}[ampersand replacement=\&]
		#1 \arrow[swap,->>]{d}{#7} \arrow[]{r}{#5} \& #2 \arrow[->>]{d}{#8}\\
		#3 \arrow[swap]{r}{#6} \& #4
	\end{tikzcd}
}

\makeindex

\begin{document}

\title{A categorical model structure\\
for generalized algebraic theories}

\author{Daniel Almeida\thanks{email: ddeal056@uottawa.ca}}

\date{}

\maketitle

\begin{abstract}
We describe a (combinatorial, monoidal, $\Cat$-enriched) Quillen model structure on the category of Cartmell's generalized algebraic theories (gats); its homotopy bicategory consists essentially of Taylor's rooted display map categories. This allows us to compare two kinds of morphisms of gats: one where sort dependency and substitution are preserved strictly, thus directly matching the syntax, and one where the given structure is preserved up to isomorphism.

We prove a strictification result for morphisms out of cofibrant theories, which are the retracts of theories without sort equality axioms. Along the way, we give a structural characterization of when a contextual category can be presented without sort equality axioms. Our results also imply that when restricted to cofibrant objects, the tensor product of gats has the expected semantic behaviour, namely, it corresponds to the tensor product of locally finitely presentable categories equipped with a cofibrantly generated weak factorization system.

Strict and weak morphisms specialize, respectively, to two familiar concepts of model of a gat $\mathbb A$: ones valued in iterated families of sets, with substitution interpreted as reindexing, and set-valued models of the contextual category $\mathcal C(\bbA)$ viewed as a finite-limit sketch. We characterize strictifiability of a model of the latter kind via a loop freeness condition on a certain map of functors out of the category of context projections of $\mathbb A$.
\end{abstract}

\setcounter{tocdepth}{2}
\tableofcontents

\newpage

\section{Introduction}
\label{sec: introduction}

Generalized algebraic theories (gats for short), introduced by J. Cartmell in the late 1970s (\cite{Car78}, \cite{Car86}), provide a way of extending the framework of equational theories (or, in categorical language, Lawvere theories), in which classical universal algebra had been developed, by allowing the presence of dependent sorts -- that is, sorts (or types) whose specification depends, in a recursive manner, on a set of variables of previously introduced sorts.

For example, we can formulate the concept of a category via a gat having
\begin{itemize}[noitemsep]
	\item a sort $O$, to be interpreted as a set (of ``objects of a category");
	
	\item a sort $A(x,y)$, where $x$, $y$ are variables of sort $O$, to be interpreted as an assignment of a set (of ``arrows from $x$ to $y$") to each pair of elements of the set assigned to $O$.
\end{itemize}

The operations present in a category are encoded by terms that have a specified sort and depend on variables of previously constructed sorts. Composition is introduced, in sequent notation, by
$$
x, y, z:O,\; f:A(x,y),\; g:A(y,z) \vdash g \circ f: A(x,z),
$$
and the assignment of identity arrows to objects by $x:O \vdash i(x):O$, where $\circ$ and $i$ are newly introduced symbols.\footnote{We are making an abuse of notation, which is harmless in this example, by making the dependency of $g \circ f$ on $x$, $y$, $z$ implicit. This is what Cartmell refers to as the ``informal syntax" of gats; see \cite{Car86} for a discussion of when this is admissible.} Then we add axioms that encode the expected behaviour of these operations, such as
$$
x, y, z, w:O,\; f:A(x,y),\; g:A(y,z),\; h:A(z,w) \vdash (h \circ g) \circ f \equiv h \circ (g \circ f): A(x,w)
$$

for associativity. Generally, we can also consider gats that include axioms stating equalities between sorts.

As the above example suggests, we can consider a family-based semantics for gats. Fixing a universe ``set of small sets" $\mathscr U$, a ($\mathscr U$-)model of a given generalized algebraic theory sends, for instance,
\begin{itemize}[noitemsep]
	\item a sort $X$ with no free variables (a ``closed" sort) to a set $\mathbf{X} \in \mathscr U$,
	
	\item a sort $x:X \vdash Y(x) \tp$ to a family $\mathbf{Y}:\mathbf{X} \rightarrow \mathscr U$,
	
	\item a sort $x:X, y:Y(x) \vdash Z(x,y)$ to a family $\mathbf{Z}:\coprod_{\mathbf{x} \in \mathbf{X}} \mathbf{Y}(\mathbf{x}) \rightarrow \mathscr U$,
	
	\item a term $\vdash k:X$ to an element of $\mathbf{X}$,
	
	\item a term $x:X \vdash f(x):Y(x)$ to a family $\varphi \in \prod_{\mathbf{x} \in \mathbf{X}}\mathbf{Y}(\mathbf{x})$,
\end{itemize}
and similarly for further dependent sorts and terms, compatibly with the equality axioms of the theory. Following this approach, Cartmell defined (\cite{Car86}, \S11) the \emph{category of (``family-valued") models} of a given gat, and sketched a proof that the category of set-valued models of any essentially algebraic theory (as in \cite{Fre72}) can be obtained, up to equivalence, as the category of models of some gat. This implies that categories of models of gats include, up to equivalence, all categories of models of finite-limit theories, thus all locally finitely presentable categories (\cite{GabUlm71}).

Generalized algebraic theories can be assembled into a category $\GAT$ where morphisms $\bbA \rightarrow \bbB$ are equivalence classes, with respect to equality derivable from the axioms of $\bbB$, of interpretations of $\bbA$ in $\bbB$. Interpretations, in turn, are (roughly) certain assignments of expressions in $\bbB$ to symbols in the alphabet of $\bbA$ in such a way that the recursively induced map sending judgments in $\bbA$ to judgments in $\bbB$ preserves derivability. See \cite[\S12]{Car86}.

Contextual categories, also introduced by Cartmell in \cite{Car78}, are certain category-like objects whose structure matches what is preserved by morphisms of generalized algebraic theories. For a gat $\bbA$, we have a contextual category $\mathcal C(\bbA)$ -- its syntactic category -- intended to retain information on the equivalence classes of derivable judgments and the interactions between them\footnote{This perspective appears more explicitly in Voevodsky's theory of B-systems (\cite{Voe14}) and Garner's description of gats as algebras for a certain monad on the category of type-and-term structures (\cite{Gar15}). See \cite{AhrEmmNorRij25} for further discussion of these different approaches.}, but not on the raw syntax. Let us briefly (and informally) recall the structure present in a contextual category, say $\mathcal A$. It has an underlying category $|\mathcal A|$, whose objects are thought of as contexts taken up to derivable equality. It comes with a \emph{length} function $\ell:\Ob(|\mathcal A|) \rightarrow \bbN$ such that there exists a unique object of length $0$, corresponding to the empty context. There is a special set of arrows, the \emph{display maps}\footnote{This terminology is due to \cite{Tay86}.}, that encode type dependency: a display map $\Gamma' \twoheadrightarrow \Gamma$ is viewed as the standard projection out of an extension of $\Gamma$ by a free variable of some sort derivable in $\Gamma$; this map is only available when $\ell(\Gamma') \ge 1$ and, in that case, it is determined up to equality, rather than up to isomorphism, by $\Gamma'$. Substitution of a dependent sort $\Gamma' \twoheadrightarrow \Gamma$ along a context morphism $f:\Delta \rightarrow \Gamma$ is encoded by a cartesian square
\[
\tag{\texttt{*}}
\dsqua{\Delta'}{\Gamma'}{\Delta}{\Gamma,}{f'}{f}{}{}
\]
which is specified up to equality by $\Delta \overset{f}{\rightarrow} \Gamma \twoheadleftarrow \Gamma'$. Moreover, the set of such diagrams -- which we call \emph{distinguished pullback squares} -- is required to be \emph{strictly} closed under horizontal composition; in particular, writing $f^*(\Gamma')$ for $\Delta'$ in the above diagram, we have
\[
\tag{$\pitchfork$}
g^*f^*(\Gamma') = (fg)^*(\Gamma')
\]
for any $g:\Omega \rightarrow \Delta$. The corresponding syntactic property is the (strict) functoriality of substitution of expressions along context morphisms.

Defining morphisms of contextual categories as functors that strictly preserve all available structure -- the \emph{contextual functors} --, we obtain a category $\Cont$, and the assignment $\bbA \mapsto \mathcal C(\bbA)$ lifts to an equivalence of categories $\GAT \simeq \Cont$. This correspondence can be used, for example, to (re)define a family-valued model of a gat $\bbA$ as a morphism $\mathcal C(\bbA) \rightarrow \Fam$ where $\Fam$ is the contextual category of iterated families of (small) sets from \cite{Car86}:
\begin{itemize}[noitemsep]
	\item Its objects are defined recursively\footnote{We work inside the (large) set of all finite sequences of elements of $\{(\mathbf{E},\mathbf{K}) \mid \mathbf{E} \in \mathscr U,\; \mathbf{K}:\mathbf{E} \rightarrow \mathscr U\}$} alongside their ``total sets" as follows:
	\begin{itemize}[noitemsep]
		\item[-] the unique length-$0$ object is the empty sequence $()$, with total set $\Sigma() = \{*\}$;
		
		\item[-] a length-$(n+1)$ object is a tuple $(\textbf{X}_i)_{i \le n+1}$ where $(\textbf{X}_i)_{i \le n}$ is a length-$n$ object and $\textbf{X}_{n+1}:\Sigma(\textbf{X}_i)_{i \le n} \rightarrow \mathscr U$; the total set is $\Sigma(\mathbf{X}_i)_{i \le n+1} = \coprod_{\overline{\mathbf{x}} \in \Sigma(\mathbf{X}_i)_{i \le n}}\mathbf{X}_{n+1}(\overline{\mathbf{x}})$.
	\end{itemize}
	
	\item A morphism between two objects of $\Fam$ is defined as a map between their total sets. Display maps are the projections $\Sigma(\mathbf{X}_i)_{i \le n+1} \rightarrow \Sigma(\mathbf{X}_i)_{i \le n}$, and distinguished pullbacks are given by reindexing.
\end{itemize}

Lifting $\Cont$ to a strict $2$-category $\iiCont$ whose $2$-cells are the natural transformations between contextual functors, the category of models of $\bbA$ can be described (working, if needed, in a universe with respect to which $\Fam$ is small) as $\iiCont(\mathcal C(\bbA),\Fam)$; thus for a contextual category $\mathcal A$, we define the \emph{category of family-valued models} of $\mathcal A$ as $\iiCont(\mathcal A,\Fam)$.

Another possible approach is to consider functors $|\mathcal A| \rightarrow \Set$ that send distinguished squares (as in (\texttt{*}) above) to pullback squares, and the length-$0$ terminal object to a singleton. We call them \emph{set-valued models} of $\mathcal A$. These are the models of $\mathcal A$ viewed as a finite-limit sketch (specifying a set of pullbacks and a terminal object) or as a clan (closing the set of display maps under composition and isomorphism) in the sense \cite{Joy17}; see, e.g., \cite{Fre25}.

The underlying category of the contextual category $\Fam$ is canonically equivalent to $\Set$ via the total set functor $\Sigma:|\Fam| \rightarrow \Set$. Composing with $\Sigma$ defines a full-and-faithful functor from the category of family-valued models of $\mathcal A$ to that of set-valued models. The latter concept of model can thus be seen as a weakening of the former one. We will denote these categories of models by $\iiMod_s(\mathcal A)$ and $\iiMod_w(\mathcal A)$ ($s$ and $w$ stand for ``strict" and ``weak"). As it turns out, there are cases where
$$
\iiMod_s(\mathcal A) \xrightarrow{\Sigma \circ -} \iiMod_w(\mathcal A)
$$
is not an equivalence of categories: there exists a set-model $|\mathcal A| \rightarrow \Set$ that is not strictifiable in that it is not isomorphic to $\Sigma \circ F$ for any contextual functor $F:\mathcal A \rightarrow \Fam$. An example is given in \cite[Rem. B.54]{BarHen25}: consider the gat $\bbK$ specified by
$$
\vdash X \tp \qquad\quad x:X \vdash f(x):X \qquad\quad x:X \vdash Y(x) \tp \qquad\quad x:X \vdash Y(f(x)) \equiv Y(x) \tp
$$
A family-valued model of $\bbK$ consists of a set $\mathbf{X}$, a function $\mathbf{f}:\mathbf{X} \rightarrow \mathbf{X}$, and a family $\mathbf{Y}:\mathbf{X} \rightarrow \mathscr U$ such that $\mathbf{Y}(\mathbf{f}(\mathbf{x})) = \mathbf{Y}(\mathbf{x})$ for all $\mathbf{x} \in \mathbf{X}$. Up to isomorphism, these data correspond to a set $\textbf{X}$, a map $\textbf{f}:\textbf{X} \rightarrow \textbf{X}$, and a set $\textbf{E}$ equipped with a map $\pi:\textbf{E} \rightarrow \textbf{X}/\sim$ to the set of orbits of the action of $\textbf{f}$.

On the other hand, a set-valued model $\mathcal C(\bbK) \rightarrow \Set$ corresponds\footnote{It is not immediate that the data below extend (essentially uniquely) to a functor $\mathcal C(\bbK) \rightarrow \Set$ preserving the required finite limits. We prove this in Example \ref{example: gat from a category} as an instance of a more general result.} to a pullback square as in
% https://q.uiver.app/#q=WzAsMTAsWzAsMCwiW3g6WCx5OlgoeCldIl0sWzAsMiwiW3g6WF0iXSxbMiwyLCJbeDpYXSJdLFsyLDAsIlt4OlgseTpYKHgpXSJdLFsyLDFdLFs1LDFdLFs1LDAsIlxcdGV4dGJme0V9Il0sWzcsMCwiXFx0ZXh0YmZ7RX0iXSxbNSwyLCJcXHRleHRiZntYfSJdLFs3LDIsIlxcdGV4dGJme1h9Il0sWzEsMiwiW2YoeCldIiwyXSxbMCwzLCJbZih4KSx5XSJdLFswLDEsIiIsMSx7InN0eWxlIjp7ImhlYWQiOnsibmFtZSI6ImVwaSJ9fX1dLFszLDIsIiIsMSx7InN0eWxlIjp7ImhlYWQiOnsibmFtZSI6ImVwaSJ9fX1dLFs0LDUsIiIsMSx7InNob3J0ZW4iOnsic291cmNlIjoyMCwidGFyZ2V0IjoyMH0sInN0eWxlIjp7InRhaWwiOnsibmFtZSI6Im1hcHMgdG8ifSwiYm9keSI6eyJuYW1lIjoiZGFzaGVkIn19fV0sWzcsOSwiXFxwaSJdLFs2LDgsIlxccGkiLDJdLFs4LDksIlxcbWF0aGJme2Z9IiwyXSxbNiw3LCJcXG1hdGhiZntnfSJdLFs2LDksIiIsMSx7InN0eWxlIjp7Im5hbWUiOiJjb3JuZXIifX1dXQ==
\[\begin{tikzcd}[ampersand replacement=\&,cramped]
	{[x:X,\; y:X(x)]} \&\& {[x:X,\; y:X(x)]} \&\&\& {\textbf{E}} \&\& {\textbf{E}} \\
	\&\& {} \&\&\& {} \\
	{[x:X]} \&\& {[x:X]} \&\&\& {\textbf{X}} \&\& {\textbf{X}}
	\arrow["{[f(x),y]}", from=1-1, to=1-3]
	\arrow[two heads, from=1-1, to=3-1]
	\arrow[two heads, from=1-3, to=3-3]
	\arrow["{\mathbf{g}}", from=1-6, to=1-8]
	\arrow["\pi"', from=1-6, to=3-6]
	\arrow["\lrcorner"{anchor=center, pos=0.125}, draw=none, from=1-6, to=3-8]
	\arrow["\pi", from=1-8, to=3-8]
	\arrow[dotted, maps to, from=2-3, to=2-6, shorten <=2em, shorten >=2em]
	\arrow["{[f(x)]}"', from=3-1, to=3-3]
	\arrow["{\mathbf{f}}"', from=3-6, to=3-8]
\end{tikzcd}\]
In other words, we have sets $\textbf{X}$, $\textbf{E}$ equipped with endomorphisms $\textbf{f}$, $\textbf{g}$, respectively, and a map $\pi:\textbf{E} \rightarrow \textbf{X}$ such that $\pi\textbf{g} = \textbf{f}\pi$ and for all $\textbf{x} \in \textbf{X}$, the map
\[
\tag{\texttt{**}}
\pi^{-1}(\textbf{x}) \longrightarrow \pi^{-1}(\textbf{f}(\textbf{x})), \qquad \textbf{y} \longmapsto \textbf{g}(\textbf{y})
\]
is bijective. The distinction between family-valued and set-valued models of $\bbK$ lies in the fact that in the former case, we have a family of sets indexed by the set-quotient of $\textbf{X}$ by $\textbf{f}$, while in the latter, we have a family indexed by the corresponding groupoid-quotient. More precisely, in the latter case we can encode the actions of $\textbf{f}$, $\textbf{g}$ on $\textbf{X}$, $\textbf{E}$ via functors
$$
\alpha_\textbf{f},\; \alpha_\textbf{g}:B\bbN \longrightarrow \Set,
$$
where $B\bbN$ is the monoid $(\bbN,+)$ viewed as a one-object category; $\pi$ defines a natural transformation $\pi_*:\alpha_\textbf{g} \Rightarrow \alpha_\textbf{f}$ which, as the maps (\texttt{**}) above are bijective, is cartesian in the sense that every naturality square (in this case, the only one) is a pullback. Taking categories of elements, we obtain a discrete opfibration
$$
\smallint \pi_*:\smallint \alpha_\textbf{g} \longrightarrow \smallint \alpha_\textbf{f}
$$
which, as $\pi_*$ is cartesian, is such that its fiber functor, say
$$
\Phi:\smallint \alpha_\textbf{f} \longrightarrow \Set,
$$
sends every arrow to an isomorphism. As a consequence, $\Phi$ factors uniquely as
$$
\smallint \alpha_\textbf{f} \longrightarrow \gpd(\smallint \alpha_\textbf{f}) \overset{\Phi}{\longrightarrow} \Set
$$
where $\gpd:\Cat \rightarrow \Gpd$ is the functor (left adjoint to the inclusion) that sends a category to the groupoid obtained by localizing it at the set of all arrows. The groupoid $\gpd(\smallint \textbf{f})$ is the pseudo-quotient of the composite $2$-functor $B\bbN \overset{\alpha_\textbf{f}}{\rightarrow} \Set \hookrightarrow \textbf{Gpd}$, and the given set-valued model is strictifiable (as we will see in \S\ref{sec: strictifying set-valued models}) if and only if $\Phi'$ sends every automorphism to an identity arrow. For instance, the quotient map $\pi:\mathbb Z \longrightarrow \mathbb Z/n\mathbb Z$ where each set is equipped with the endomorphism $n \mapsto n+1$ gives a non-strictifiable model of $\bbK$ since there is a loop in $\mathbb Z/n$ that lifts to a non-trivial automorphism of $\mathbb Z$:
% https://q.uiver.app/#q=WzAsMTAsWzAsMCwiMCJdLFswLDEsIjAiXSxbMSwwLCIxIl0sWzEsMSwiMSJdLFsyLDEsIlxcY2RvdHMiXSxbMiwwLCJcXGNkb3RzIl0sWzMsMCwibi0xIl0sWzMsMSwibi0xIl0sWzQsMSwibiA9IDAiXSxbNCwwLCJuIFxcbmVxIDAiXSxbMCwxLCJcXHBpIiwyLHsic3R5bGUiOnsidGFpbCI6eyJuYW1lIjoibWFwcyB0byJ9fX1dLFswLDIsIisxIiwwLHsic3R5bGUiOnsidGFpbCI6eyJuYW1lIjoibWFwcyB0byJ9fX1dLFsyLDUsIisxIiwwLHsic3R5bGUiOnsidGFpbCI6eyJuYW1lIjoibWFwcyB0byJ9fX1dLFs1LDYsIisxIiwwLHsic3R5bGUiOnsidGFpbCI6eyJuYW1lIjoibWFwcyB0byJ9fX1dLFs2LDksIisxIiwwLHsic3R5bGUiOnsidGFpbCI6eyJuYW1lIjoibWFwcyB0byJ9fX1dLFsxLDMsIisxIiwyLHsic3R5bGUiOnsidGFpbCI6eyJuYW1lIjoibWFwcyB0byJ9fX1dLFszLDQsIisxIiwyLHsic3R5bGUiOnsidGFpbCI6eyJuYW1lIjoibWFwcyB0byJ9fX1dLFs0LDcsIisxIiwyLHsic3R5bGUiOnsidGFpbCI6eyJuYW1lIjoibWFwcyB0byJ9fX1dLFs3LDgsIisxIiwyLHsic3R5bGUiOnsidGFpbCI6eyJuYW1lIjoibWFwcyB0byJ9fX1dLFs5LDgsIlxccGkiLDIseyJzdHlsZSI6eyJ0YWlsIjp7Im5hbWUiOiJtYXBzIHRvIn19fV0sWzIsMywiXFxwaSIsMix7InN0eWxlIjp7InRhaWwiOnsibmFtZSI6Im1hcHMgdG8ifX19XSxbNSw0LCJcXHBpIiwyLHsic3R5bGUiOnsidGFpbCI6eyJuYW1lIjoibWFwcyB0byJ9fX1dLFs2LDcsIlxccGkiLDIseyJzdHlsZSI6eyJ0YWlsIjp7Im5hbWUiOiJtYXBzIHRvIn19fV1d
\[
\tag{\texttt{***}}
\begin{tikzcd}[ampersand replacement=\&,cramped]
	0 \& 1 \& \cdots \& {n-1} \& {n \neq 0} \\
	0 \& 1 \& \cdots \& {n-1} \& {n = 0}
	\arrow["{+1}", maps to, from=1-1, to=1-2]
	\arrow["\pi"', maps to, from=1-1, to=2-1]
	\arrow["{+1}", maps to, from=1-2, to=1-3]
	\arrow["\pi"', maps to, from=1-2, to=2-2]
	\arrow["{+1}", maps to, from=1-3, to=1-4]
	\arrow["\pi"', maps to, from=1-3, to=2-3]
	\arrow["{+1}", maps to, from=1-4, to=1-5]
	\arrow["\pi"', maps to, from=1-4, to=2-4]
	\arrow["\pi"', maps to, from=1-5, to=2-5]
	\arrow["{+1}"', maps to, from=2-1, to=2-2]
	\arrow["{+1}"', maps to, from=2-2, to=2-3]
	\arrow["{+1}"', maps to, from=2-3, to=2-4]
	\arrow["{+1}"', maps to, from=2-4, to=2-5]
\end{tikzcd}\]
For any contextual category $\mathcal A$, we can characterize its strictifiable set-valued models by adapting the above discussion. Let $\textbf{D}(\mathcal A)$ be the category of display maps and distinguished pullback squares, composed horizontally. Working as in the example yields for a model $M:\mathcal A \rightarrow \Set$ a functor
\[
\tag{\texttt{****}}
\gpd(K) \longrightarrow \Set
\]
where $K$ is the category of elements of the composite $\textbf{D}(\mathcal A) \xrightarrow{\text{cod}} \mathcal A \xrightarrow{M} \Set$. We will prove that $M$ is strictifiable precisely (\texttt{****}) sends every arrow to an identity map; we say that a model with the latter property is \emph{loop-free}.

The underlying idea is that a family-valued model of a gat is required to interpret the fiberwise isomorphisms in distinguished pullback squares as fiberwise identities, which relies on the non-categorical structure of the universe $\mathscr U$: we use it not just a category (namely, $\Set$), but as a category equipped with an essentially surjective functor from an essentially discrete category (namely, $\text{disc}(\mathscr U) = \text{disc}(\Ob(\Set)) \hookrightarrow \Set$). The loop freeness condition in our characterization of strictifiable set-valued models of $\mathcal A$ comes from the fact that for a category $K$, the following conditions on a functor $F:K \rightarrow \text{core}(\Set)$\footnote{We write $\text{core}(A)$ for the core groupoid of a category $A$, that is, its subcategory consisting of all isomorphisms.} are equivalent:
\begin{itemize}[noitemsep]
	\item if factors up to isomorphism through the subcategory $\text{disc}(\mathscr U) \hookrightarrow \text{core}(\Set)$;
	
	\item it factors strictly through a functor $V \rightarrow \text{core}(\Set)$ where $V$ is a coproduct of codiscrete categories (which is the case iff $V$ is equivalent to a set);
	
	\item the induced functor $\gpd(K) \rightarrow \text{core}(\Set)$ sends every automorphism to an identity map.
\end{itemize}

This phenomenon reflects the observation, which appears e.g. in \cite{Tay99}, \cite{Voe15}\footnote{In particular, it motivated Voevodsky to refer to contextual categories as \emph{C-systems} in order to avoid a categorical connotation.}, \cite{Voe16}, \cite{AhrNorShuTse25}, that contextual categories are more appropriately viewed as set-level rather than categorical structures since their definition uses the equality predicate on the set of objects. For instance, recall the equality $g^*f^*(\Gamma') = (fg)^*(\Gamma')$ from the statement that distinguished pullback squares are strictly closed under composition (see $\pitchfork$ above), or the fact that for display maps
$$
\Gamma' \twoheadrightarrow \Gamma, \qquad \Delta' \twoheadrightarrow \Delta
$$
(which are uniquely determined by their domains), $\Gamma' \cong \Delta'$ does not imply $\Gamma \cong \Delta$. Also, we generally cannot transfer a contextual category structure on $A \in \Cat$ along an equivalence $A \simeq A'$. As suggested by our previous discussion, the non-categorical content of a contextual category $\mathcal A$ that affects its family-valued semantics in an essential way is encoded by the category of display maps $\textbf{D}(\mathcal A)$.

This motivates looking for a categorical theory of contextual categories, that is, a setting where the above issues of non-invariance under isomorphism disappear and where the passage from family-valued to set-valued models becomes more transparent. One such framework is that of categories equipped with a ``class of displays", or display map categories, as studied by Taylor in \cite[Chap. VIII]{Tay99}.

\begin{definition}[Taylor]
\label{def: display map category}
A \emph{display map category}, dmc for short, is a pair $\mathcal A = (A,P)$ consisting af a category $A$ and a set of arrows $P$ (\emph{display maps}) such that arrows in $P$ have pullbacks along arbitrary arrows, and all such pullbacks belong to $P$ (in particular, $P$ is closed under isomorphism in the arrow category $\iiAr(A)$).

$\mathcal A$ is \emph{rooted} if its underlying category has a terminal object (not chosen), and every object $\Gamma$ fits into a finite sequence of display maps
$$
\Gamma = \Gamma_n \twoheadrightarrow \Gamma_{n-1} \twoheadrightarrow \cdots \twoheadrightarrow \Gamma_1 \twoheadrightarrow \Gamma_0
$$
where $\Gamma_0$ is terminal.\footnote{Note that the (possible rooted) dmc structure on $A$ and the corresponding properties are invariant under isomorphism in a suitable sense, and any equivalence of categories $A \simeq A'$ can be used to induce from $(A,P)$ a weak contextual category structure on $A'$.}

\vspace{0.5em}

For dmcs $\mathcal A$, $\mathcal A'$, a \emph{(weak) morphism} $\mathcal A \rightarrow \mathcal A'$ is a functor that preserves display maps and their pullbacks; restricting to small dmcs, this gives a $2$-category $\iiDMC$ with natural transformations as $2$-cells. We write $\iiDMC_r$ for the locally full sub-$2$-category spanned by the rooted dmcs and terminal-object-preserving morphisms between them.
\end{definition}

\begin{definition}
\label{def: weakening functor}
Let $\mathcal A$ be a contextual category. Its \emph{weakening} is the rooted dmc $\mathcal A_w$ that has the same underlying category as $\mathcal A$, and where an arrow is a display map if and only it is isomorphic to a strict display map of $\mathcal A$.

Note that this defines a locally full-and-faithful $2$-functor $\iiCont \rightarrow \iiDMC_r$.
\end{definition}

Rephrased in our language, one of Taylor's contributions is a proof (see \cite[\S8.4]{Tay99}) that every rooted dmc is equivalent in $\iiDMC$ to $\mathcal C(\bbA)_w$ for some gat $\bbA$ without sort equality axioms. In fact, he only ever considers gats of this kind, a choice that hints at the potential problems (as for the gat $\bbA$ described above) caused by the presence of sort equality axioms; see \cite[Rem. 8.1.9]{Tay99} and the footnote on page 451. He justifies this via two ideas: (1) that interchangeability of objects should be mediated by isomorphisms, and (2) that interpreting a gat without sort equality axioms is made simpler by the fact that every sort is of the form $A(f_1, ..., f_n)$ for a unique sort symbol $A$ and an essentially unique tuple of terms $(f_1, ..., f_n)$.\footnote{This latter idea led us, as we will explain later, to introduce the concept of a \emph{basis} of a contextual category.}

\vspace{0.5em}

For a dmc $\mathcal A = (A,P)$, we let $\textbf{D}_w(\mathcal A)$ be the subcategory of $\iiAr(A)$ formed by the elements of $P$ and the pullback squares between them. Observe that a morphism $F:\mathcal A \rightarrow \mathcal B$ in $\textbf{DMC}$ is an equivalence if and only if its underlying functor is an equivalence of categories and $\textbf{D}_w(F):\textbf{D}_w(\mathcal A) \rightarrow \textbf{D}_w(\mathcal B)$ is essentially surjective (thus also an equivalence). It follows that, viewing $\Set$ as rooted dmc with every arrow as a display map, the total set functor defines an equivalence $\Fam_w \simeq \Set$ in $\iiDMC_r$ (possibly passing to a larger universe). For a contextual category $\mathcal A$, this gives
$$
\iiMod_w(\mathcal A) \simeq \iiDMC_r(\mathcal A_w,\Set) \simeq \iiDMC_r(\mathcal A_w,\Fam_w).
$$
In particular, the comparison functor $\iiMod_s(\mathcal A) \rightarrow \iiMod_w(\mathcal A)$ can be described as $\iiCont(\mathcal A,\Fam) \rightarrow \iiDMC_r(\mathcal A_w,\Fam_w)$. Moreover, if a contextual functor $\mathcal A \rightarrow \mathcal B$ is an equivalence as a morphism $\mathcal A_w \rightarrow \mathcal B_w$, then $\iiMod_w(\mathcal A) \simeq \iiMod_w(\mathcal B)$.

\begin{definition}
\label{def: weak equivalence}
A contextual functor $F \in \Cont(\mathcal A,\mathcal B)$ is a \emph{weak equivalence} if it is an equivalence $\mathcal A_w \rightarrow \mathcal B_w$ in $\iiDMC$.
\end{definition}

Summarizing some of the main results of the paper, we will connect contextual categories and dmcs as follows:
\begin{theorem}
\label{th: summary model structure}
There exists a $\omega$-combinatorial Quillen model structure on $\Cont$ (equivalently, on $\GAT$) whose weak equivalences are as above, and with generating cofibrations whose pushouts encode adding sorts, terms, and term equalities (no sort equalities); every object is fibrant. Moreover:
\begin{itemize}
	\item it is a monoidal and $\Cat$-enriched model category (with respect to the monoidal structure from \cite{Alm26}, the $\Cat$-enrichment $\iiCont$, and the categorical model structure on $\Cat$);
	
	\item its homotopy $(\infty,1)$-category is the underlying $(2,1)$-category of $\iiDMC_r$;
\end{itemize}
\end{theorem}

A cell complex with respect to the above generating cofibrations is a contextual category isomorphic to $\mathcal C(\bbA)$ for some gat $\bbA$ that doesn't have any sort equality axioms. If $\mathcal A$ is cofibrant (a retract of a cell complex), then
$$
\iiCont(\mathcal A,\mathcal S) \simeq \iiDMC_r(\mathcal A_w,\mathcal S_w) \text{ for any } \mathcal S \in \iiCont,\footnote{However, we will prove the theorem using this equivalence, not conversely.} \qquad\quad \iiMod_s(\mathcal A) \simeq \iiMod_w(\mathcal A).
$$

The key property of cell complexes -- which we will refer to as \emph{cellular} theories/contextual categories -- needed to prove the first equivalence above is the existence of what we will call a \emph{basis}: a set $\textbf{P}$ of display maps such that every display map can be obtained in a unique way as a distinguished pullback of an element of $\textbf{P}$. A contextual category $\mathcal A$ has a basis if and only if every connected component of $\textbf{D}(\mathcal A)$ has a terminal object, in which case a choice of such terminal objects gives a basis.

For a gat $\bbA$ with without sort equality axioms, the display maps associated with the sort introduction axioms form a basis of $\mathcal C(\bbA)$, as already observed (in other terms) in \cite{Tay99}. This is all we need to know about bases to characterize the homotopy category of $\Cont$, but we also give a structural characterization, which may be of independent interest, of cellular contextual categories: they are the ones that have a basis satisfying a certain well-foundedness condition (Theorem \ref{th: cellular iff basis}).

\vspace{0.5em}

Other than those already mentioned, many ideas that we use have appeared in some form in the literature. For instance, our study relies on the fact that sorts over each context form a category rather than just a set: for an object $\Gamma$ of a contextual category $\mathcal A$, we can consider the full subcategory $\mathcal A_\Gamma$ of $|\mathcal A|/\Gamma$ spanned by the display maps with codomain $\Gamma$. Taking substitution along context morphisms, the assignment $\Gamma \mapsto \mathcal A_\Gamma$ extends to a functor $\mathcal A_\bullet:|\mathcal A|^{\op} \rightarrow \Cat$ that refines the presheaf of sorts of $\mathcal A$. This is the construction of the context-indexed category of types from \cite[Prop. 2.7]{ClaDyb14} applied to the category with families associated with $\mathcal A$.

Weak equivalences in $\Cont$, as defined above, are precisely the contextual functors $F:\mathcal A \rightarrow \mathcal B$ that induce for each display map $p$ in $\mathcal A$ a bijection $\{\text{sections of } p\} \cong \{\text{sections of } Fp\}$ (syntactically, this means that $F$ acts bijectively on equivalence classes of terms) and, for each $\Gamma \in |\mathcal A|$, an essentially surjective functor $\mathcal A_\Gamma \rightarrow \mathcal B_{F\Gamma}$. Hence, our weak equivalences are analogous to the type-theoretic equivalences from \cite{KapLum18}; but the latter are introduced in a much more structured homotopical setting (with intensional identity types, etc.) that leads to the $\infty$-categorical world by viewing such a structured contextual category as a category of fibrant objects in the sense of \cite{Bro73}. In contrast, our framework is intended to express the role of judgmental equality (which is the only notion of equality between terms and sorts available to us by default) in the problem of realizing contextual categories as $1$-categorical rather than set-level structures.

We also note that it might be possible to exhibit the passage from $\iiCont$ to $\iiDMC_r$ as an instance of Lack's presentation of the $2$-category of pseudoalgebras for a strict $2$-monad via a $\Cat$-enriched model structure on the $2$-category of strict algebras, but we don't pursue that direction (or even construct the required $2$-monad, which we expect to be on a $2$-category of suitably structured categories).

\subsection*{Organization of the text}

\begin{itemize}
	\item In Section \ref{sec: contextual categories}, we fix notation and introduce a few constructions related to (pre)contextual categories and display map categories that will be used throughout the paper.
	
	\item In Section \ref{sec: model structure}, we start by discussing full-and-faithful contextual functors and weak equivalences in more detail. Then we construct the ``categorical" Quillen model structure on the category of contextual categories, recalling, for that, the construction of $\Cat$-tensors/powers of contextual categories. We connect the model structure on $\Cont$ with the monoidal structure from \cite{Alm25, Alm26} by proving that we actually have a $\Cat$-enriched monoidal model category, where $\Cat$ is equipped with the categorical model structure.
	
	\item In Section \ref{sec: cellular theories}, we give a structural characterization of those contextual categories that are cellular with respect to the generating cofibrations from \S\ref{sec: model structure}. Except for the fact that it relies on the choice of generating cofibrations, this section is independent from the previous one.
	
	\item Section \ref{sec: strictification} is devoted to the strictification problem for weak morphisms of contextual categories. In \S\ref{subsec: what is needed to strictify} we provide a general criterion for strictifiability based on what we call \emph{local strictifications}; in \S\ref{subsec: strictification basis or cofibrant} we use this criterion to prove that weak morphisms whose domain is cofibrant (or has a basis) are strictifiable, which allows us to conclude that the model structure on $\Cont$ presents the $(2,1)$-category of rooted display map categories.
	
	\item In Section \ref{sec: strictifying set-valued models}, we specialize the discussion from \S\ref{sec: strictification} to the question of whether a set-valued model of a gat is isomorphic to a family-valued one. In this setting, we obtain a further characterization of strictifiability in terms of an explicit loop freeness condition. We also provide examples to illustrate certain aspects of the strictification problem for set-valued models. Some generalities on cartesian natural transformations needed in this section are reviewed in the \hyperref[sec: appendix]{appendix}.
	
	\item Finally, in Section \ref{sec: remarks closed monoidal structure and set-valued semantics} we discuss the closed monoidal structure on $\iiDMC_r$ presented by the one on $\Cont \simeq \GAT$, as well as the compatibility of this structure with the passage, via set-valued models, from gats to locally finitely presentable categories equipped with a suitable class of arrows induced from the display maps.
\end{itemize}

\subsection*{Conventions}

\begin{itemize}[noitemsep]
	\item We work within ZFC set theory with, additionally, two uncountable Grothendieck universes $\mathscr U$, $\mathscr U^+$ such that $\mathscr U \in \mathscr U^+$. Elements of $\mathscr U$ will be called \emph{small sets}.
	
	\item We assume familiarity with category theory (as in \cite{MLa98}, \cite{Rie16}), the theory of locally presentable categories (\cite{AdaRos94}), and basics of Quillen model categories (see \cite{Hov99}, \cite[A.2]{Lur09})\footnote{What we'll need is far from the whole content of these. We need to know about weak factorization systems (and the small object argument), the definition of a (cofibrantly generated) model category and its homotopy category, Quillen adjunctions and Quillen bifunctors, monoidal model categories and enriched model categories. We sometimes talk about combinatorial model categories (a reference for which is \cite[A.2]{Lur09}), but we only use local presentability to ensure that the sets of arrows $I$, $J$ as in \cite[Th. 2.1.19]{Hov99} permit the small object argument.}.
	
	\item Categories will usually be denoted by letters such as $A$, $B$, etc. For us, a category $A$ has sets $\Ob(A)$ and $\Ar(A)$ of objects and of arrows, respectively, neither of which is required to belong to $\mathscr U$ or $\mathscr U^+$. We often abbreviate $a \in \Ob(A)$ as $a \in A$, and we write $A(a,b)$ for the set of arrows from $a$ to $b$. The category of arrows and commutative squares is denoted by $\iiAr(A)$.
	
	\item For categories $A$ and $B$, we write $B^A$ for the category of functors and natural transformations between them.
	
	\item We say that a category $A$ is \emph{small} if $\Ob(A)$ and $\Ar(A)$ are small; we say that it is \emph{locally small} if $\Ob(A) \in \mathscr U^+$ and for all $a$, $b \in \Ob(A)$, the set $A(a,b)$ is small.
	
	\item We denote by $\Set$ (resp. $\Set^+$) the category of small sets (resp. of elements of $\mathscr U^+$) and functions. We let $\Cat$ (resp. $\iiCat$) be the category (resp. $2$-category) of small categories and functors (resp. and natural transformations). Their counterparts with locally small categories as objects are denoted by $\Cat^+$, $\iiCat^+$.
\end{itemize}

\subsection*{Acknowledgement}

I want to thank Simon Henry for introducing me to this topic and for his invaluable guidance throughout the process of learning the ideas that made their way into this paper.
\section{Contextual categories}
\label{sec: contextual categories}

In this section we review the definition of a contextual category, due to Cartmell, and present some constructions and results involving contextual categories and display map categories. For more detail on contextual categories, generalized algebraic theories, and the relationship between them, we refer the reader to the original references on the subject, \cite{Car78} and \cite{Car86}.

\vspace{0.5em}

We will consider two kinds of morphisms between contextual categories: strict and weak ones. What we call \emph{strict morphisms} are the morphisms originally considered by Cartmell, known as contextual functors; \emph{weak morphisms} will be the morphisms between the associated display map categories.

\begin{definition}
\label{def: (pre)contextual category}
A \emph{precontextual category} is a quadruple $\mathcal A = (A,\ell,\partial,\textbf{p},\textbf{Q})$ consisting of:
\begin{enumerate}[label=(\roman*)]
	\item A category $A$.
	
	\item A function $\ell:\Ob(A) \rightarrow \bbN$, assigning to each object its \emph{length}.
	
	\item A function $\partial:\Ob(A) \rightarrow \Ob(A)$ such that $\ell(\partial a) = \ell(a) - 1$ if $\ell(a) \ge 0$, and $\ell(\partial a) = 0$ if $\ell(a) = 0$.
	
	\item A family $\textbf{p}$ assigning to each $a \in \Ob(A)$ such that $\ell(a) \ge 1$ a morphism
	$$
	\textbf{p}_a: a \longrightarrow \partial a.
	$$
	Morphisms of this form will be called \emph{strict display maps}, and we denote them using arrows such as $\twoheadrightarrow$.
	
	\item A set $\textbf{Q}$ of commutative squares in $A$ of the form
	\[
	\dsqua{a}{b}{\partial a}{\partial b.}{f'}{f}{\textbf{p}_a}{\textbf{p}_b}
	\]
	Elements of $\textbf{Q}$ will be called \emph{distinguished squares}.
\end{enumerate}

We say that $\mathcal A$ is a \emph{contextual category} (\cite{Car78}; called a \emph{C-system} in \cite{Voe16}) if
\begin{enumerate}[label=(\roman*)]
\setcounter{enumi}{5}
	\item It has unique length-$0$ object, which is terminal. We denote it by $1_\mathcal A$ and call it the \emph{distinguished terminal object}.
	
	\item Any diagram
	\[
	\begin{tikzcd}
		& b \arrow[two heads]{d}{\textbf{p}_b} \\
		a \arrow[swap]{r}{f} & \partial b
	\end{tikzcd}
	\]
	can be uniquely extended into a square in $\textbf{Q}$.

	\item $Q$ is closed under horizontal composition and contains
	\[
	\dsqua{a}{a}{\partial a}{\partial a}{id_a}{id_{\partial a}}{\textbf{p}_a}{\textbf{p}_a}
	\]
	for all $a \in \Ob(A)$ of length $\ge 1$.
\end{enumerate}
\end{definition}

\begin{notation}
We write the quadruple $\mathcal A = (A,\ell,\partial,\textbf{p},\textbf{Q})$ as $(|\mathcal A|,\ell_\mathcal A,\textbf{p}_\mathcal A,\textbf{Q}_\mathcal A)$. However, we will drop the subscript $\mathcal A$ whenever this causes no ambiguity.

For $a \in \Ob(|\mathcal A|)$, say of length $n$, we have a unique chain of ($n$ many) display maps joining $a$ and $1_\mathcal A$. We denote the corresponding objects by
$$
a = \partial_n a \twoheadrightarrow \cdots \twoheadrightarrow \partial_i a \twoheadrightarrow \cdots \twoheadrightarrow \partial_1 a \twoheadrightarrow \partial_0 a = 1_\mathcal A.
$$

We say that $\mathcal A$ is (\emph{locally}) \emph{small} when its underlying category $|\mathcal A|$ is (locally) small.
\end{notation}

\begin{definition}[The category of contextual categories and strict morphisms]
\label{def: category of contextual categories}
Let $\mathcal A$ and $\mathcal B$ be contextual categories. A \emph{strict morphism} (contextual functor in the terminology of \cite{Car78}) from $\mathcal A$ to $\mathcal B$ is a functor $F:|\mathcal A| \rightarrow |\mathcal B|$ that strictly preserves all additional structure:
$$
\ell(F(a)) = \ell(a), \qquad \partial(F(a)) = F(\partial a), \qquad \textbf{p}_{F(a)} = F(\textbf{p}_a), \qquad F(\textbf{Q}_\mathcal A) \subset \textbf{Q}_\mathcal B.
$$
Composing functors as usual, we obtain a category $\Cont$ (resp. $\Cont^+$) of small (resp. locally small) contextual categories and strict morphisms. Regarding all natural transformations between strict morphisms as $2$-cells, we have $2$-categories $\iiCont$ and $\iiCont^+$ (which can also be viewed as $\Cat$- and $\Cat^+$-enriched categories, respectively).
\end{definition}

\begin{remark}
\label{rem: use of precontextual categories}
Let $\Precont$ be the category whose objects are the small precontextual categories and whose morphisms are the functors as in Definition \ref{def: category of contextual categories}; taking all natural transformations between those as $2$-cells, we have a $2$-category $\iiPrecont$. Our motivation for talking about precontextual categories is the fact, proved in \cite[Cor. 5.33]{Alm26}, that the inclusion $2$-functor
$$
\iiCont \hookrightarrow \iiPrecont
$$
has a strict left $2$-adjoint (which implies that $\Cont \hookrightarrow \Precont$ has a left adjoint\footnote{The $1$-categorical adjunction, however, is obtained directly from the orthogonal reflection construction for locally presentable categories; see \cite[Chap. 1]{AdaRos94}}). This gives a way of diagrammatically presenting contextual categories with a certain kind of $\Cat$-enriched universal property: letting $L:\iiPrecont \rightarrow \iiCont$ be a left adjoint of the inclusion, for $\mathcal A \in \iiPrecont$ and $\mathcal S \in \iiCont$ we have
$$
\iiPrecont(\mathcal A,\mathcal S) \cong \iiCont(L\mathcal A,\mathcal S).
$$
The vast majority of this paper can be read without any knowledge of precontextual categories (and, in particular, of the above result). The main exceptions are Proposition \ref{prop: cont is cat-tensored}, where we describe $\Cat$-tensors of contextual categories, and several constructions and results in \S\ref{sec: strictification}. The latter results allow us to recognize strictifiable morphisms out of precontextual categories, and that will be useful for studying some examples in \S\ref{subsec: examples}. Other than that, precontextual categories are useful throughout the text for giving simple diagrammatic descriptions of certain contextual categories (especially in \S\ref{subsec: sorts terms and axioms via contextual categories}), but it is also straightforward to present those using gats.\footnote{We could also use gats instead of precontextual categories in Proposition \ref{prop: cont is cat-tensored}, although the description of tensors would be lengthier and more indirect, and the desired universal property would not follow as clearly.}
\end{remark}

\begin{notation}
We refer to \cite{Car78, Car86} for the definition of a generalized algebraic theory (gat for short). For us, unless stated otherwise, the sets of sort/term symbols of a gat are small. We will denote by $\GAT$ the category of generalized algebraic theories and equivalence classes of interpretations.

Cartmell's thesis establishes an equivalence $\GAT \simeq \Cont$ associating with a gat $\bbA$ its syntactic category, denoted by $\mathcal C(\bbA)$, which consists of equivalence classes (with respect to derivable equality) of contexts in $\bbA$.
\end{notation}

\begin{notation}
Recall the concept of a (rooted) display map category from Definition \ref{def: display map category} (as in \cite{Tay99}), as well as the $2$-categories $\iiDMC$ and $\iiDMC_r$ of (possibly rooted) display map categories. Taking as objects, instead, the locally small (possibly rooted) dmcs, we obtain $2$-categories $\iiDMC^+$ and $\iiDMC_r^+$.

Moreover, recall from Definition \ref{def: weakening functor} the weakening $2$-functor $(-)_w:\iiCont \rightarrow \iiDMC_r$ given by closing the set of strict display maps of a contextual category under isomorphism in its arrow category, and forgetting the length function and the choice of pullbacks of strict display maps. Similarly, we have a $2$-functor $(-)_w:\iiCont^+ \rightarrow \iiDMC_r^+$.

For us, a \emph{display map} in a contextual category $\mathcal A$ is a display map in $\mathcal A_w$, that is, an arrow isomorphic to a strict display map.
\end{notation}

\begin{example}
Some natural examples of display map categories do not arise directly as $\mathcal A_w$ for some contextual category $\mathcal A$. If a category has finite limits (equivalently, pullbacks and a terminal object), then it has a structure of (rooted) dmc in which every arrow is a display map. Any category has a minimal and a maximal (not necessarily rooted) dmc structure: in the first one, the display maps are precisely the isomorphisms, and in the latter, the display maps are the arrows $a' \rightarrow a$ that admit pullbacks along all $b \rightarrow a$. A clan (\cite{Joy17}) is a rooted dmc whose set of display maps contains all isomorphisms and is closed under composition.
\end{example}

\begin{remark}[Equivalences of contextual categories]
\label{rem: equivalences of contextual categories}
Recall from Definition \ref{def: weak equivalence} that a strict morphism of contextual categories $F:\mathcal A \rightarrow \mathcal B$ is a \emph{weak equivalence} if $F_w$ is an equivalence in $\iiDMC$ in the usual sense, that is, if it has a two-sided inverse up to invertible $2$-cells. It follows that $F$ is a weak equivalence if and only if
\begin{enumerate}[label=(\roman*)]
	\item it defines an equivalence between the underlying categories, and
	
	\item it (preserves and) reflects display maps: if $F(p) \in P_{\mathcal B_w}$, then $p \in P_{\mathcal A_w}$.
\end{enumerate}
Indeed, if $F:\mathcal A \rightarrow \mathcal B$ satisfies (i), then it has a weak inverse $G:|\mathcal B| \rightarrow |\mathcal A|$ in $\iiCat$, and in that case, $F$ satisfies (ii) if and only if $G$ preserves display maps (this relies on display maps being closed under isomorphism). Note that if we assume (i), then (ii) is equivalent to
\begin{itemize}
	\item[(ii')] for every strict display map $p$ in $\mathcal B$, there exists a strict display map $p'$ in $\mathcal A$ with $F(p') \cong p$.
\end{itemize}
On the other hand, we can consider equivalences in the $2$-category $\iiCont$. It turns out, as Theorem \ref{th: summary model structure} suggests, that not every weak equivalence of contextual categories is an equivalence in $\iiCont$. For example, let $\mathcal A$ be the syntactic category of the gat
$$
\vdash A \tp \qquad \vdash A' \tp \qquad x:A \vdash B_1(x) \tp \qquad x:A \vdash B_2(x) \tp \qquad x:A' \vdash f(x):A
$$
\begin{align*}
	x:A', \; y:B_1(f(x)) & \vdash g(x,y): B_2(f(x)) & x:A', \; y:B_1(f(x)) & g'(x,g(x,y)) \equiv y:B_1(f(x))\\
	x:A', \; y:B_2(f(x)) & \vdash g'(x,y):B_1(f(x)) & x:A', \; y:B_2(f(x)) & g(x,g'(x,y)) \equiv y:B_2(f(x)),
\end{align*}
and let $\mathcal B$ be the syntactic category of
$$
\vdash A \tp \qquad \vdash A' \tp \qquad x:A \vdash B_1(x) \tp \qquad x:A \vdash B_2(x) \tp \qquad x:A' \vdash f(x):A
$$
$$
x:A' \vdash B_1(f(x)) \equiv B_2(f(x)) \tp.
$$
Diagrammatically, $\mathcal A$ is freely generated by (applying $L:\Precont \rightarrow \Cont$ to the precontextual category consisting of) a single length-$0$ object, length-$1$ objects $a$, $a'$, a morphism $f:a' \rightarrow a$, display maps $b_1 \twoheadrightarrow a$ and $b_2 \twoheadrightarrow a$, and an isomorphism $g:f^*(b_1) \rightarrow f^*(b_2)$ compatible with the respective display maps to $a$. We obtain $\mathcal B$ similarly, but instead of adding the isomorphism $g$, we impose $f^*(b_1) = f^*(b_2)$. The generating data of $\mathcal A$ and $\mathcal B$ can be depicted (omitting the length-$0$ object), respectively, as
% https://q.uiver.app/#q=WzAsMTEsWzEsMiwiYSciXSxbNCwzLCJhIl0sWzMsMSwiYl8xIl0sWzUsMSwiYl8yIl0sWzAsMCwiZl4qKGJfMSkiXSxbMiwwLCJmXiooYl8yKSJdLFs2LDIsImEnIl0sWzksMywiYSJdLFs4LDEsImJfMSJdLFsxMCwxLCJiXzIiXSxbNiwwLCJmXiooYl8xKSA9IGZeKihiXzIpIl0sWzAsMSwiZiIsMl0sWzIsMSwiIiwwLHsic3R5bGUiOnsiaGVhZCI6eyJuYW1lIjoiZXBpIn19fV0sWzMsMSwiIiwwLHsic3R5bGUiOnsiaGVhZCI6eyJuYW1lIjoiZXBpIn19fV0sWzQsMCwiIiwyLHsic3R5bGUiOnsiaGVhZCI6eyJuYW1lIjoiZXBpIn19fV0sWzUsMCwiIiwyLHsic3R5bGUiOnsiaGVhZCI6eyJuYW1lIjoiZXBpIn19fV0sWzQsMl0sWzUsM10sWzQsNSwiXFxjb25nIiwxLHsibGFiZWxfcG9zaXRpb24iOjYwfV0sWzYsNywiZiIsMl0sWzgsNywiIiwwLHsic3R5bGUiOnsiaGVhZCI6eyJuYW1lIjoiZXBpIn19fV0sWzksNywiIiwwLHsic3R5bGUiOnsiaGVhZCI6eyJuYW1lIjoiZXBpIn19fV0sWzEwLDYsIiIsMix7InN0eWxlIjp7ImhlYWQiOnsibmFtZSI6ImVwaSJ9fX1dLFsxMCw4XSxbMTAsOV0sWzQsNSwiZyIsMCx7ImxhYmVsX3Bvc2l0aW9uIjo2MCwic3R5bGUiOnsiYm9keSI6eyJuYW1lIjoibm9uZSJ9LCJoZWFkIjp7Im5hbWUiOiJub25lIn19fV1d
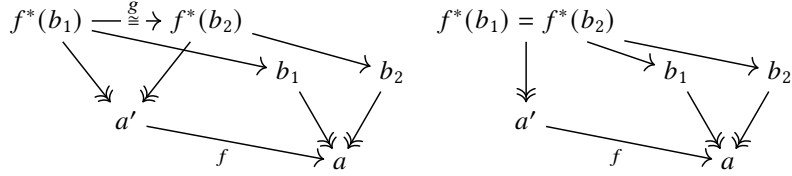
\[\begin{tikzcd}[ampersand replacement=\&,cramped,row sep=tiny, column sep=tiny]
	{f^*(b_1)} \&\& {f^*(b_2)} \&\&\&\& {f^*(b_1) = f^*(b_2)} \&\&\&\& \\
	\&\&\& {b_1} \&\& {b_2} \&\&\& {b_1} \&\& {b_2} \\
	\& {a'} \&\&\&\&\& {a'} \\
	\&\&\&\& a \&\&\&\&\& a
	\arrow["\cong"{description, pos=0.6}, from=1-1, to=1-3]
	\arrow["g"{pos=0.6}, draw=none, from=1-1, to=1-3]
	\arrow[from=1-1, to=2-4]
	\arrow[two heads, from=1-1, to=3-2]
	\arrow[from=1-3, to=2-6]
	\arrow[two heads, from=1-3, to=3-2]
	\arrow[from=1-7, to=2-9]
	\arrow[from=1-7, to=2-11]
	\arrow[two heads, from=1-7, to=3-7]
	\arrow[two heads, from=2-4, to=4-5]
	\arrow[two heads, from=2-6, to=4-5]
	\arrow[two heads, from=2-9, to=4-10]
	\arrow[two heads, from=2-11, to=4-10]
	\arrow["f"', from=3-2, to=4-5]
	\arrow["f"', from=3-7, to=4-10]
\end{tikzcd}\]
We have an evident strict morphism $F:\mathcal A \rightarrow \mathcal B$ that sends $g$ to $id_{f^*(\Delta_1)}$; using the gats that present $\mathcal A$, $\mathcal B$ and Prop. \ref{prop: characterization full-and-faithful}, \ref{prop: characterization acyclic fibrations}, it can be proved that $F$ is an acyclic fibration (Def. \ref{def: cofibrations, acyclic fibrations}), thus a weak equivalence. Now, if $G \in \iiCont(\mathcal B,\mathcal A)$ were a quasi-inverse of $F$, it would send $a$ to $a$, and so on for $a'$, $f$, $b_1$, $b_2$. But as $G$ preserves distinguished squares, we must have
$$
f^*(b_1) = G(f^*(b_1)) = G(f^*(b_2)) = f^*(b_2),
$$
which is not possible.\footnote{In fact, it follows from $F$ being an acyclic fibration for the model structure on $\Cont$ that $\mathcal B$ is not cofibrant.}
\end{remark}

\begin{construction}[Categories of display maps]
\label{constr: categories of display maps}
One of our main techniques for studying contextual categories and display map categories will be assembling display maps into suitable categories.

\vspace{0.5em}

For a contextual category $\mathcal A$, we let $\textbf{D}(\mathcal A)$ be the category whose objects are the strict display maps and whose morphisms are the distinguished squares (with composition performed horizontally).

Note that $\textbf{D}(\mathcal A)$ is canonically isomorphic to the category of elements of the functor $\tp:|\mathcal A|^{\op} \rightarrow \Set$ sending $a \in \Ob(|\mathcal A|)$ to the set of all strict display maps $a' \twoheadrightarrow a$, and acting on morphisms by pullback according to Definition \ref{def: (pre)contextual category}(vii). This is fundamental for connecting contextual categories with other categorical models of dependent type theory. For example, the diagram
\[\begin{tikzcd}[ampersand replacement=\&,row sep=small]
	{\smallint \tp \cong \textbf{D}(\mathcal A)} \&\& {\iiAr(|\mathcal A|)} \\
	\& {|\mathcal A|}
	\arrow["\iota", hook, from=1-1, to=1-3]
	\arrow["\pi"', from=1-1, to=2-2]
	\arrow["{\textbf{t}}", from=1-3, to=2-2]
\end{tikzcd}\]
(with $\iota$ the inclusion, $\textbf{t}$ the codomain functor, and $\pi$ the projection) realizes $\mathcal A$ as a comprehension category (\cite{Jac93}), or, more specifically (as $\pi$ is a discrete fibration), as a category with attributes (\cite{Car78}\footnote{The concept we refer to is a commonly used variant of the one introduced by Cartmell. See e.g. \cite[Def. 4.1]{KapLum18}.}).

\vspace{0.5em}

For a display map category $\mathcal A$, we let $\textbf{D}_w(\mathcal A)$ be the category whose objects are the display maps and whose morphisms are the cartesian squares between them (again, with horizontal composition). The composite
$$
\textbf{D}_w(\mathcal A) \hookrightarrow \iiAr(|\mathcal A|) \overset{\textbf{t}}{\longrightarrow} |\mathcal A|
$$
is a Grothendieck fibration whose fiber over $a \in \Ob(|\mathcal A|)$ is the subcategory of $|\mathcal A|/a$ consisting of all display maps $a' \rightarrow a$ and isomorphisms between them. This realizes $\mathcal A$ as a (non-full, non-split) comprehension category.
\end{construction}

As a first step towards understanding the functor $(-)_w:\iiCont \rightarrow \iiDMC$, we observe that it is $2$-essentially surjective. This follows from \cite[Th. 8.4.10]{Tay99}, but we believe it can be helpful to give a self-contained proof using Giraud-Bénabou's classical strictification construction for Grothendieck fibrations, which allows us to obtain a contextual category without passing through the equivalence $\GAT \simeq \Cont$.\footnote{We don't claim originality of the proof we present, but we did not find a reference for it. We also observe that Taylor's approach gives a stronger result, namely, that every rooted dmc is weakly equivalent to the syntactic category of a gat without sort equality axioms. In our framework, this is obtained by combining Proposition \ref{prop: strictification for a rooted dmc} and the existence of cellular replacements of contextual categories; still, the theory so obtained (say via the small object argument) is much less explicit than Taylor's.}

\begin{proposition}
\label{prop: strictification for a rooted dmc}
For every rooted dmc $\mathcal A$, there exists a contextual category $\mathcal B$ such that $\mathcal B_w \simeq \mathcal A$ in $\iiDMC$.
\end{proposition}

\begin{proof}
Consider the comprehension category
	% https://q.uiver.app/#q=WzAsMyxbMCwwLCJcXGlpd2RpcyhcXG1hdGhjYWwgQSkiXSxbMiwwLCJcXGlpQXIoXFxtYXRoY2FsIEEpIl0sWzEsMSwifFxcbWF0aGNhbCBBfCJdLFswLDIsIlxcdGV4dGJme3R9IiwyXSxbMCwxLCJcXGlvdGEiLDAseyJzdHlsZSI6eyJ0YWlsIjp7Im5hbWUiOiJob29rIiwic2lkZSI6InRvcCJ9fX1dLFsxLDIsIlxcdGV4dGJme3R9Il1d
	\[\begin{tikzcd}[ampersand replacement=\&,row sep=small]
		{\textbf{D}_w(\mathcal A)} \&\& {\iiAr(|\mathcal A|)} \\
		\& {|\mathcal A|}
		\arrow["\iota", hook, from=1-1, to=1-3]
		\arrow["{\textbf{t}}"', from=1-1, to=2-2]
		\arrow["{\textbf{t}}", from=1-3, to=2-2]
	\end{tikzcd}\]
	where $\iota$ is the inclusion functor and $\textbf{t}$ is the codomain projection. It is known -- see the discussion in \cite[\S2.2]{LumWar15} -- that the above (as any comprehension category) is equivalent to some \emph{split} comprehension category
	% https://q.uiver.app/#q=WzAsMyxbMCwwLCJFIl0sWzIsMCwiXFxpaUFyKEMpIl0sWzEsMSwiQyJdLFswLDIsIlxccGkiLDJdLFswLDEsIlxcY2hpIl0sWzEsMiwiXFx0ZXh0YmZ7dH0iXV0=
	\[\begin{tikzcd}[ampersand replacement=\&,row sep=small]
		E \&\& {\iiAr(|\mathcal A|)} \\
		\& {|\mathcal A|.}
		\arrow["\chi", from=1-1, to=1-3]
		\arrow["\pi"', from=1-1, to=2-2]
		\arrow["{\textbf{t}}", from=1-3, to=2-2]
	\end{tikzcd}\]
	Now, replacing $E$ by its subcategory having the same objects but as arrows only the cartesian ones specified by the splitting of $\pi$, we obtain a category with attributes (or discrete comprehension category) $\mathcal A'$ with underlying category $|\mathcal A|$ and pointed by any chosen terminal object of $|\mathcal A|$.
	
	Denote by $\Att$ the category of categories with attributes, and let $\cont:\Att \rightarrow \Cont$ be the right adjoint of the forgetful functor $\att:\Cont \rightarrow \Att$ (see \cite[Prop. 4.4]{KapLum18}). Letting $\mathcal B = \cont(\mathcal A')$, the adjunction counit component
	$$
	|\mathcal B| = |\att(\mathcal B)| = |\att(\cont(\mathcal A'))| \longrightarrow |\mathcal A'| = |\mathcal A|
	$$
	is full-and-faithful, as is the case for any cwa. Moreover, $|\mathcal B| \rightarrow |\mathcal A|$ is essentially surjective: this is obtained by combining the assumption that $\mathcal A$ is rooted with the fact that $\iota$ and $\chi$ have the same essential image in $\iiAr(|\mathcal A|)$. This also implies that an arrow in $|\mathcal B|$ is isomorphic to a strict display map precisely when its image under $|\mathcal B| \rightarrow |\mathcal A|$ is in the essential image of $\iota$. It follows that $|\mathcal B| \rightarrow |\mathcal A|$ lifts to an equivalence of rooted dmcs $\mathcal B_w \rightarrow \mathcal A$.
\end{proof}

\subsection{Sorts, terms and axioms from the perspective of contextual categories}
\label{subsec: sorts terms and axioms via contextual categories}

Since the syntactic category construction defines an equivalence $\GAT \simeq \Cont$ and, on the other hand, every gat can be built recursively from the empty theory by adding sort, term, sort equality and term equality axioms, we expect that these four kinds of axioms can be translated via the above equivalence of categories into four classes of ``building blocks" for all contextual categories.

\begin{notation}
\label{not: set of terms}
For a contextual category $\mathcal A$ and an object $a \in \mathcal A$ with $\ell(a) \ge 1$, we write $\tm_\mathcal A(a)$, or just $\tm(a)$ when $\mathcal A$ is implicit, for the set of sections of the strict display map $\textbf{p}_a:a \rightarrow \partial a$, i.e. of arrows $t:\partial a \rightarrow a$ such that $\textbf{p}_a \circ t = id_{\partial a}$.
\end{notation}

\begin{construction}
\label{constr: O_n etc}
The following functors are representable:
\begin{itemize}
	\item $\Ob_n:\Cont \rightarrow \Set$, where $n \ge 0$, sending a contextual category to its set of length-$n$ objects.
	
	\item $\Ob_n^+:\Cont \rightarrow \Set$, where $n \ge 1$, sending $\mathcal A$ to the set of all pairs $(a,t)$ where $a \in \Ob_n(\mathcal A)$ and $t \in \tm_\mathcal A(a)$.
\end{itemize}
Indeed, $\Ob_n$ is represented by the syntactic category of the gat $\bbO_n$ with axioms
\begin{align*}
	& \vdash O_1 \tp\\
	x_1:O_1 & \vdash O_2(x_1) \tp\\
	& \cdots\\
	x_1:O_1, \;..., \; x_{n-1}:O_{n-1}(x_1, ..., x_{n-2}) & \vdash O_n(x_1, ..., x_{n-1}) \tp,
\end{align*}
and $\Ob_n^+$ is represented by the syntactic category of the gat $\bbO_n^+$ obtained from $\bbO_n$ by adding a term axiom
$$
x_1:O_1, ..., x_{n-1}:O_{n-1}(x_1, ..., x_{n-2}) \vdash f(x_1, ..., x_{n-1}) : O_n(x_1, ..., x_{n-1}).
$$
In other words, the contextual category $\mathcal C(\bbO_n)$ is freely generated by a sequence of strict display maps
% https://q.uiver.app/#q=WzAsNSxbMCwwLCJvX24iXSxbMSwwLCJvX3tuLTF9Il0sWzIsMCwiXFxjZG90cyJdLFszLDAsIm9fMSJdLFs0LDAsIm9fMCJdLFswLDEsIiIsMSx7InN0eWxlIjp7ImhlYWQiOnsibmFtZSI6ImVwaSJ9fX1dLFsxLDIsIiIsMSx7InN0eWxlIjp7ImhlYWQiOnsibmFtZSI6ImVwaSJ9fX1dLFsyLDMsIiIsMSx7InN0eWxlIjp7ImhlYWQiOnsibmFtZSI6ImVwaSJ9fX1dLFszLDQsIiIsMSx7InN0eWxlIjp7ImhlYWQiOnsibmFtZSI6ImVwaSJ9fX1dXQ==
\[
\tag{\texttt{*}}
\begin{tikzcd}[ampersand replacement=\&]
	{o_n} \& {o_{n-1}} \& \cdots \& {o_1} \& {o_0}
	\arrow[two heads, from=1-1, to=1-2]
	\arrow[two heads, from=1-2, to=1-3]
	\arrow[two heads, from=1-3, to=1-4]
	\arrow[two heads, from=1-4, to=1-5]
\end{tikzcd}\]
where $o_i$ has length $i$. This can be expressed using precontextual categories (see Definition \ref{def: (pre)contextual category} and Remark \ref{rem: use of precontextual categories}): the above sequence specifies a precontextual category $\mathcal O_n^\pre$, and applying $L:\Precont \rightarrow \Cont$ to it yields a contextual category isomorphic to $\mathcal C(\bbO_n)$. The object $o_i$ corresponds to the equivalence class $[x_1:O_1, ..., x_i:O_i(x_1, ..., x_{i-1})]$. Similarly, $\mathcal C(\bbO_n^+)$ is freely generated by (the precontextual category $\mathcal O_n^{+\; pre}$ given by) a diagram
% https://q.uiver.app/#q=WzAsNSxbMCwwLCJvX24iXSxbMSwwLCJvX3tuLTF9Il0sWzIsMCwiXFxjZG90cyJdLFszLDAsIm9fMSJdLFs0LDAsIm9fMCJdLFswLDEsIiIsMSx7InN0eWxlIjp7ImhlYWQiOnsibmFtZSI6ImVwaSJ9fX1dLFsxLDIsIiIsMSx7InN0eWxlIjp7ImhlYWQiOnsibmFtZSI6ImVwaSJ9fX1dLFsyLDMsIiIsMSx7InN0eWxlIjp7ImhlYWQiOnsibmFtZSI6ImVwaSJ9fX1dLFszLDQsIiIsMSx7InN0eWxlIjp7ImhlYWQiOnsibmFtZSI6ImVwaSJ9fX1dLFsxLDAsInMiLDAseyJvZmZzZXQiOi0zLCJzdHlsZSI6eyJ0YWlsIjp7Im5hbWUiOiJob29rIiwic2lkZSI6InRvcCJ9fX1dXQ==
\[
\tag{\texttt{**}}
\begin{tikzcd}[ampersand replacement=\&]
	{o_n} \& {o_{n-1}} \& \cdots \& {o_1} \& {o_0}
	\arrow[two heads, from=1-1, to=1-2]
	\arrow["s", shift left=3, hook, from=1-2, to=1-1]
	\arrow[two heads, from=1-2, to=1-3]
	\arrow[two heads, from=1-3, to=1-4]
	\arrow[two heads, from=1-4, to=1-5]
\end{tikzcd}\]
obtained from the previous one by adjoining a section $s:o_{n-1} \rightarrow o_n$ of $o_n \twoheadrightarrow o_{n-1}$. Here, $s$ corresponds to the morphism
$$
[x_1, ..., x_{n-1}, f(x_1, ..., x_{n-1})] : [x_1:O_1, ..., x_{n-1}:O_{n-1}(x_1, ..., x_{n-2})] \longrightarrow [x_1:O_1, ..., x_n:O_n(x_1, ..., x_{n-1})].
$$
\end{construction}

\begin{notation}
We let
$$
\mathcal O_n = \mathcal C(\bbO_n), \qquad \mathcal O_n^+ = \mathcal C(\bbO_n^+).
$$
\end{notation}

\begin{remark}
In particular, $\mathcal O_0$ is an initial object of $\Cont$: the distinguished terminal object is its only object.
\end{remark}

\begin{construction}
For $n \ge 1$, we have morphisms of gats
$$
\bbO_{n-1} \longrightarrow \bbO_n, \qquad \bbO_n \longrightarrow \bbO_n^+
$$
given by the respective interpretations that send each derivable judgment to itself. We let
$$
\iota_n^S:\mathcal O_{n-1} \longrightarrow \mathcal O_n, \qquad \iota_n^T:\mathcal O_n \longrightarrow \mathcal O_n^+
$$
be their respective images under $\GAT \simeq \Cont$. These morphisms can be used to encode introduction of sorts and of terms, respectively. Consider a contextual category $\mathcal A$ endowed with a length-$(n-1)$ object $a$, and let $\overline{a}:\mathcal O_{n-1} \rightarrow \mathcal A$ be the unique strict morphism that sends the generating object $o_{n-1}$ to $a$. Then the contextual category $\mathcal A'$ as in the pushout
\[
\squa{\mathcal O_{n-1}}{\mathcal A}{\mathcal O_n}{\mathcal A'}{\overline{a}}{}{\iota_n^S}{}
\]
has the following universal property, derived from that of $\mathcal O_n$: strict morphisms $\mathcal A' \rightarrow \mathcal B$ are in natural bijection with pairs consisting of a morphism $F:\mathcal A \rightarrow \mathcal B$ and a strict display map $b \twoheadrightarrow F(a)$. But now, suppose that $\mathcal A = \mathcal C(\bbA)$ for a gat $\bbA$, and write $a = [x_1:X_1, ..., x_{n-1}:X_{n-1}]$ for a context $x_1:X_1, ..., x_{n-1}:X_{n-1}$. Then the universal property of $\mathcal A'$ is precisely that of $\mathcal C(\bbA')$ for a gat $\bbA'$ obtained from $\bbA$ by adding a sort axiom
$$
x_1:X_1, ..., x_{n-1}:X_{n-1} \vdash S(x_1, ..., x_{n-1}) \tp,
$$
so we actually have a pushout square
\[
\squa{\mathcal O_{n-1}}{\mathcal C(\bbA)}{\mathcal O_n}{\mathcal C(\bbA').}{\overline{a}}{}{\iota_n^S}{}
\]
A similar argument shows that if $\bbA$ is a gat endowed with a context $x_1:X_1, ..., x_n:X_n$ with $n \ge 1$, then we have a pushout square
\[
\squa{\mathcal O_n}{\mathcal C(\bbA)}{\mathcal O_n^+}{\mathcal C(\bbA').}{\overline{a}}{}{\iota_n^T}{}
\]
in $\Cont$ where $a = [x_1:X_1, ..., x_n:X_n]$ and $\bbA'$ is obtained from $\bbA$ by adding a term axiom
$$
x_1:X_1, ..., x_{n-1}:X_{n-1} \vdash f(x_1, ..., x_{n-1}): X_n.
$$
\end{construction}

\begin{construction}
For $n \ge 1$, we define contextual categories $\mathcal O_n^\vee$ and $\mathcal O_n^{++}$ via the pushout diagrams
% https://q.uiver.app/#q=WzAsOCxbMCwwLCJcXG1hdGhjYWwgT197bi0xfSJdLFswLDEsIlxcbWF0aGNhbCBPX24iXSxbMSwwLCJcXG1hdGhjYWwgT19uIl0sWzEsMSwiXFxtYXRoY2FsIE9fbl5cXHZlZSwiXSxbMywwLCJcXG1hdGhjYWwgT19uIl0sWzMsMSwiXFxtYXRoY2FsIE9fbl4rIl0sWzQsMCwiXFxtYXRoY2FsIE9fbl4rIl0sWzQsMSwiXFxtYXRoY2FsIE9fbl57Kyt9LiJdLFswLDEsIlxcaW90YV9uXlMiLDJdLFswLDIsIlxcaW90YV9uXlMiXSxbMSwzXSxbMiwzXSxbNCw1LCJcXGlvdGFfbl5UIiwyXSxbNCw2LCJcXGlvdGFfbl5UIl0sWzUsN10sWzYsN11d
\[\begin{tikzcd}[ampersand replacement=\&]
	{\mathcal O_{n-1}} \& {\mathcal O_n} \&\& {\mathcal O_n} \& {\mathcal O_n^+} \\
	{\mathcal O_n} \& {\mathcal O_n^\vee,} \&\& {\mathcal O_n^+} \& {\mathcal O_n^{++}.}
	\arrow["{\iota_n^S}", from=1-1, to=1-2]
	\arrow["{\iota_n^S}"', from=1-1, to=2-1]
	\arrow[from=1-2, to=2-2]
	\arrow["{\iota_n^T}", from=1-4, to=1-5]
	\arrow["{\iota_n^T}"', from=1-4, to=2-4]
	\arrow[from=1-5, to=2-5]
	\arrow[from=2-1, to=2-2]
	\arrow[from=2-4, to=2-5]
\end{tikzcd}\]
The universal properties of $\mathcal O_n^\vee$ and $\mathcal O_n^{++}$ can then be stated as follows: $\mathcal O_n^\vee$ is freely generated by a diagram of display maps
% https://q.uiver.app/#q=WzAsNixbMCwwLCJvX24iXSxbMCwyLCJvX24nIl0sWzEsMSwib197bi0xfSJdLFsyLDEsIlxcY2RvdHMiXSxbMywxLCJvXzEiXSxbNCwxLCJvXzAiXSxbMiwzLCIiLDAseyJzdHlsZSI6eyJoZWFkIjp7Im5hbWUiOiJlcGkifX19XSxbMyw0LCIiLDAseyJzdHlsZSI6eyJoZWFkIjp7Im5hbWUiOiJlcGkifX19XSxbNCw1LCIiLDAseyJzdHlsZSI6eyJoZWFkIjp7Im5hbWUiOiJlcGkifX19XSxbMCwyLCIiLDAseyJzdHlsZSI6eyJoZWFkIjp7Im5hbWUiOiJlcGkifX19XSxbMSwyLCIiLDAseyJzdHlsZSI6eyJoZWFkIjp7Im5hbWUiOiJlcGkifX19XV0=
\[\begin{tikzcd}[ampersand replacement=\&, row sep=tiny]
	{o_n} \&\&\&\& \\
	\& {o_{n-1}} \& \cdots \& {o_1} \& {o_0,} \\
	{o_n'}
	\arrow[two heads, from=1-1, to=2-2]
	\arrow[two heads, from=2-2, to=2-3]
	\arrow[two heads, from=2-3, to=2-4]
	\arrow[two heads, from=2-4, to=2-5]
	\arrow[two heads, from=3-1, to=2-2]
\end{tikzcd}\]
while $\mathcal O_n^{++}$ is freely generated by a diagram
% https://q.uiver.app/#q=WzAsNSxbMCwwLCJvX24iXSxbMSwwLCJvX3tuLTF9Il0sWzIsMCwiXFxjZG90cyJdLFszLDAsIm9fMSJdLFs0LDAsIm9fMCJdLFsxLDIsIiIsMCx7InN0eWxlIjp7ImhlYWQiOnsibmFtZSI6ImVwaSJ9fX1dLFsyLDMsIiIsMCx7InN0eWxlIjp7ImhlYWQiOnsibmFtZSI6ImVwaSJ9fX1dLFszLDQsIiIsMCx7InN0eWxlIjp7ImhlYWQiOnsibmFtZSI6ImVwaSJ9fX1dLFswLDEsIiIsMCx7InN0eWxlIjp7ImhlYWQiOnsibmFtZSI6ImVwaSJ9fX1dLFsxLDAsInMnIiwwLHsib2Zmc2V0IjotMywic3R5bGUiOnsidGFpbCI6eyJuYW1lIjoiaG9vayIsInNpZGUiOiJ0b3AifX19XSxbMSwwLCJzIiwyLHsib2Zmc2V0IjozLCJzdHlsZSI6eyJ0YWlsIjp7Im5hbWUiOiJob29rIiwic2lkZSI6ImJvdHRvbSJ9fX1dXQ==
\[\begin{tikzcd}[ampersand replacement=\&]
	{o_n} \& {o_{n-1}} \& \cdots \& {o_1} \& {o_0}
	\arrow[two heads, from=1-1, to=1-2]
	\arrow["{s'}", shift left=3, hook, from=1-2, to=1-1]
	\arrow["s"', shift right=3, hook', from=1-2, to=1-1]
	\arrow[two heads, from=1-2, to=1-3]
	\arrow[two heads, from=1-3, to=1-4]
	\arrow[two heads, from=1-4, to=1-5]
\end{tikzcd}\]
where $s$, $s'$ are two sections of $o_n \twoheadrightarrow o_{n-1}$. Since $\mathcal C(-):\GAT \rightarrow \Cont$ preserves pushouts (being an equivalence), we can use $\mathcal O_n \cong \mathcal C(\bbO_n)$ to obtain an isomorphism
$$
\mathcal O_n^\vee \cong \mathcal C(\bbO_n^\vee)
$$
where $\bbO_n^\vee$ is defined by adding to $\bbO_{n-1}$ two sort axioms
\begin{align*}
	x_1:O_1, ..., x_{n-1}:O_{n-1}(x_1, ..., x_{n-2}) & \vdash O_n(x_1, ..., x_{n-1}) \tp\\
	x_1:O_1, ..., x_{n-1}:O_{n-1}(x_1, ..., x_{n-2}) & \vdash O'_n(x_1, ..., x_{n-1}) \tp.
\end{align*}
Similarly, using $\mathcal O_n^+ \cong \mathcal C(\bbO_n^+)$ we obtain an isomorphism
$$
\mathcal O_n^{++} \cong \mathcal C(\bbO_n^{++})
$$
where $\bbO_n^{++}$ is defined by adding to $\bbO_n$ term axioms
\begin{align*}
	x_1:O_1, ..., x_{n-1}:O_{n-1}(x_1, ..., x_{n-2}) & \vdash f(x_1, ..., x_{n-1}) : O_n(x_1, ..., x_{n-1})\\
	x_1:O_1, ..., x_{n-1}:O_{n-1}(x_1, ..., x_{n-2}) & \vdash f'(x_1, ..., x_{n-1}) : O_n(x_1, ..., x_{n-1}).
\end{align*}
\end{construction}

\begin{construction}
We let
$$
\pi_n^S:\mathcal O_n^\vee \longrightarrow \mathcal O_n, \qquad \pi_n^T:\mathcal O_n^{++} \longrightarrow \mathcal O_n^+
$$
be the fold maps of $\iota_n^S$ and $\iota_n^T$, respectively, that is, the unique dashed arrows making the diagrams
% https://q.uiver.app/#q=WzAsMTAsWzAsMCwiXFxtYXRoY2FsIE9fe24tMX0iXSxbMCwxLCJcXG1hdGhjYWwgT19uIl0sWzEsMCwiXFxtYXRoY2FsIE9fbiJdLFsxLDEsIlxcbWF0aGNhbCBPX25eXFx2ZWUiXSxbMywwLCJcXG1hdGhjYWwgT19uIl0sWzMsMSwiXFxtYXRoY2FsIE9fbl4rIl0sWzQsMCwiXFxtYXRoY2FsIE9fbl4rIl0sWzQsMSwiXFxtYXRoY2FsIE9fbl57Kyt9Il0sWzIsMiwiXFxtYXRoY2FsIE9fbiwiXSxbNSwyLCJcXG1hdGhjYWwgT19uXisiXSxbMCwxLCJcXGlvdGFfbl5TIiwyXSxbMCwyLCJcXGlvdGFfbl5TIl0sWzEsM10sWzIsM10sWzQsNSwiXFxpb3RhX25eVCIsMl0sWzQsNiwiXFxpb3RhX25eVCJdLFs1LDddLFs2LDddLFsyLDgsIlxcSWRfe1xcbWF0aGNhbCBPX259IiwwLHsiY3VydmUiOi0zfV0sWzUsOSwiXFxJZF97XFxtYXRoY2FsIE9fbl4rfSIsMix7ImN1cnZlIjozfV0sWzYsOSwiXFxJZF97XFxtYXRoY2FsIE9fbl4rfSIsMCx7ImN1cnZlIjotM31dLFszLDgsIiIsMix7InN0eWxlIjp7ImJvZHkiOnsibmFtZSI6ImRhc2hlZCJ9fX1dLFs3LDksIiIsMCx7InN0eWxlIjp7ImJvZHkiOnsibmFtZSI6ImRhc2hlZCJ9fX1dLFsxLDgsIlxcSWRfe1xcbWF0aGNhbCBPX259IiwyLHsiY3VydmUiOjN9XV0=
\[\begin{tikzcd}[ampersand replacement=\&]
	{\mathcal O_{n-1}} \& {\mathcal O_n} \&\& {\mathcal O_n} \& {\mathcal O_n^+} \& \\
	{\mathcal O_n} \& {\mathcal O_n^\vee} \&\& {\mathcal O_n^+} \& {\mathcal O_n^{++}} \\
	\&\& {\mathcal O_n,} \&\&\& {\mathcal O_n^+}
	\arrow["{\iota_n^S}", from=1-1, to=1-2]
	\arrow["{\iota_n^S}"', from=1-1, to=2-1]
	\arrow[from=1-2, to=2-2]
	\arrow["{\Id_{\mathcal O_n}}", curve={height=-18pt}, from=1-2, to=3-3]
	\arrow["{\iota_n^T}", from=1-4, to=1-5]
	\arrow["{\iota_n^T}"', from=1-4, to=2-4]
	\arrow[from=1-5, to=2-5]
	\arrow["{\Id_{\mathcal O_n^+}}", curve={height=-18pt}, from=1-5, to=3-6]
	\arrow[from=2-1, to=2-2]
	\arrow["{\Id_{\mathcal O_n}}"', curve={height=18pt}, from=2-1, to=3-3]
	\arrow[dashed, from=2-2, to=3-3]
	\arrow[from=2-4, to=2-5]
	\arrow["{\Id_{\mathcal O_n^+}}"', curve={height=18pt}, from=2-4, to=3-6]
	\arrow[dashed, from=2-5, to=3-6]
\end{tikzcd}\]
commute. For a contextual category $\mathcal A$, a pair $(a,b)$ of length-$n$ objects such that $\partial a = \partial b$ is classified, using the universal property of $\mathcal O_n^\vee$, by a strict morphism $(\overline{a},\overline{b}):\mathcal O_n^\vee \rightarrow \mathcal A$. This morphism extends (and, if so, uniquely) along $\pi_n^S$ if and only if $a = b$. Syntactically, this means that for a gat $\bbA$ and derivable judgments
\begin{align*}
	x_1:X_1, ..., x_{n-1}:X_{n-1} & \vdash U \tp\\
	x_1:X_1, ..., x_{n-1}:X_{n-1} & \vdash V \tp,
\end{align*}
we have a pushout square
% https://q.uiver.app/#q=WzAsNCxbMSwwLCJcXG1hdGhjYWwgQyhcXGJiQSkiXSxbMSwxLCJcXG1hdGhjYWwgQyhcXGJiQScpIl0sWzAsMCwiXFxtYXRoY2FsIE9fbl5cXHZlZSJdLFswLDEsIlxcbWF0aGNhbCBPX24iXSxbMiwzLCJcXHBpX25eUyIsMl0sWzAsMV0sWzMsMV0sWzIsMCwiKFxcb3ZlcmxpbmV7YX0sXFxvdmVybGluZXtifSkiXV0=
\[\begin{tikzcd}[ampersand replacement=\&]
	{\mathcal O_n^\vee} \& {\mathcal C(\bbA)} \\
	{\mathcal O_n} \& {\mathcal C(\bbA')}
	\arrow["{(\overline{a},\overline{b})}", from=1-1, to=1-2]
	\arrow["{\pi_n^S}"', from=1-1, to=2-1]
	\arrow[from=1-2, to=2-2]
	\arrow[from=2-1, to=2-2]
\end{tikzcd}\]
where $a = [x_1:X_1, ..., x_{n-1}:X_{n-1}, x_n:U]$, $b = [x_1:X_1, ..., x_{n-1}:X_{n-1}, x_n:V]$, and $\bbA'$ is obtained from $\bbA$ by adding the axiom
$$
x_1:X_1, ..., x_{n-1}:X_{n-1} \vdash U \equiv V \tp.
$$
On the other hand, for $\mathcal A \in \Cont$, a pair $(u,v)$ of sections $u$, $v$ of a display map $\textbf{p}_a$, where $\ell(a) = n$, is classified by a strict morphism $(\overline{u},\overline{v}):\mathcal O_n^{++} \rightarrow \mathcal A$. This map extends (uniquely) along $\pi_n^T$ if and only if $u = v$. Syntactically, for a gat $\bbA$ and derivable judgments
\begin{align*}
	x_1:X_1, ..., x_{n-1}:X_{n-1} & \vdash f:X_n\\
	x_1:X_1, ..., x_{n-1}:X_{n-1} & \vdash g: X_n,
\end{align*}
we have a pushout square
% https://q.uiver.app/#q=WzAsNCxbMSwwLCJcXG1hdGhjYWwgQyhcXGJiQSkiXSxbMSwxLCJcXG1hdGhjYWwgQyhcXGJiQScpIl0sWzAsMCwiXFxtYXRoY2FsIE9fbl5cXHZlZSJdLFswLDEsIlxcbWF0aGNhbCBPX24iXSxbMiwzLCJcXHBpX25eUyIsMl0sWzAsMV0sWzMsMV0sWzIsMCwiKFxcb3ZlcmxpbmV7YX0sXFxvdmVybGluZXtifSkiXV0=
\[\begin{tikzcd}[ampersand replacement=\&]
	{\mathcal O_n^{++}} \& {\mathcal C(\bbA)} \\
	{\mathcal O_n^+} \& {\mathcal C(\bbA')}
	\arrow["{(\overline{u},\overline{v})}", from=1-1, to=1-2]
	\arrow["{\pi_n^T}"', from=1-1, to=2-1]
	\arrow[from=1-2, to=2-2]
	\arrow[from=2-1, to=2-2]
\end{tikzcd}\]
where $u = [x_1, ..., x_{n-1},f]$, $v = [x_1, ..., x_{n-1},g]$, and $\bbA'$ is obtained from $\bbA$ by adding the axiom
$$
x_1:X_1, ..., x_{n-1}:X_{n-1} \vdash f \equiv g: X_n.
$$
\end{construction}

In summary, the above constructions allow us to express the addition of sort, term, sort equality and term equality axioms by taking pushouts of, respectively,
$$
\iota_n^S:\mathcal O_{n-1} \rightarrow \mathcal O_n, \qquad \iota_n^T:\mathcal O_n \rightarrow \mathcal O_n^+, \qquad \pi_n^S:\mathcal O_n^\vee \rightarrow \mathcal O_n, \qquad \pi_n^T:\mathcal O_n^{++} \rightarrow \mathcal \mathcal O_n^+.
$$

\begin{proposition}
For $\mathcal A \in \Cont$, the initial morphism $\mathcal O_0 \rightarrow \mathcal A$ can be expressed as the colimit of a diagram
$$
\mathcal A_*:(\bbN,\le) \rightarrow \Cont
$$
where $\mathcal A_0 = \mathcal O_0$ and for each $k$, the map $\mathcal A_k \rightarrow \mathcal A_{k+1}$ fits into a pushout square of the form
	\[
	\squa{\coprod_{i \in I} \mathcal B_i}{\mathcal A_k}{\coprod_{i \in I} \mathcal C_i}{\mathcal A_{k+1}}{}{}{\coprod_{i \in I} F_i}{}
	\]
	for a small family $(f_i)_{i \in I}$ of morphisms among $\iota_n^S$, $\iota_n^T$, $\pi_n^S$, $\pi_n^T$ for $n \ge 0$.
\end{proposition}

\begin{proof}
One way of obtaining this result is choosing an isomorphism $\mathcal A \cong \mathcal C(\bbA)$ for a gat $\bbA$, and defining a sequence $(\bbA_k)_{k \ge 0}$ of subtheories of $\bbA$ where $\bbA_0  = \varnothing$, and $\bbA_{k+1}$ is obtained by adding to $\bbA_k$ all axioms of $\bbA$ that are well-formed in $\bbA_k$. Then we can take $\mathcal A_* = \mathcal C(\bbA_*)$.

Alternatively, since $\Cont$ is locally finitely presentable and the domains and codomains of the arrows $\iota_n^S$, $\iota_n^T$, $\pi_n^S$, $\pi_n^T$ are finitely presentable, we can use the small object argument indexed by $\omega$ to obtain a factorization
$$
\mathcal O_0 \longrightarrow \mathcal A' \overset{I}{\longrightarrow} \mathcal A
$$
where $\mathcal O_0 \rightarrow \mathcal A'$ is, by construction, a colimit of a sequence of the desired form, and $I$ has the right lifting property with respect to $\iota_n^S$, $\iota_n^T$, $\pi_n^S$, $\pi_n^T$ for all $n \ge 0$. But since $\pi_n^S$ and $\pi_n^T$ are the fold maps of $\iota_n^S$ and $\iota_n^T$, this means that $I$ is right-orthogonal to $\iota_n^S$ and $\iota_n^T$ for $n \ge 0$. Being right-orthogonal to $\iota_n^S$ for $n \ge 0$ is equivalent to being bijective on objects; on the other hand, being right-orthogonal to $\iota_n^T$ for $n \ge 0$ is equivalent, as we will see in Proposition \ref{prop: characterization full-and-faithful}, to its underlying functor being full-and-faithful. It follows that the underlying functor of $I$ is an isomorphism of categories. It can be checked in a straightforward way that the inverse functor of $I$ preserves the contextual structure, from which we conclude that $I$ is an isomorphism $\mathcal A' \cong \mathcal A$ in $\Cont$.
\end{proof}

\section{The model structure}
\label{sec: model structure}

We will now construct what we will call the \emph{categorical} (Quillen) model structure on the category of contextual categories. Weak equivalences will be as in Definition \ref{def: weak equivalence}. After describing our candidate classes of cofibrations and fibrations, we will use a classical recognition result, due to Kan, to prove that these define a ($\omega$-combinatorial) model structure on $\Cont$.

In \S\ref{subsec: Cont is a monoidal and Cat-enriched model category}, we verify that it is a monoidal model category (as in \cite{Hov98}) with respect to the monoidal structure on $\Cont$ defined in \cite{Alm25, Alm26}. We also prove that $\iiCont$ (see Definition \ref{def: category of contextual categories}) is a $\Cat$-enriched model category where $\Cat$ is equipped with the categorical model structure (see, e.g., \cite{Rez96}).

\begin{notation}
\label{not: wfs}
For a subset $I$ of the set of arrows of a category $C$, we write $\frakL(I)$ (resp. $\frakR(I)$) for the set of all arrows in $C$ that have the left (resp. right) lifting property with respect to every element of $I$.

For $I \subset \Ar(C)$ where $C$ is a cocomplete category, we write $\cell(I)$ for the set of all relative $I$-cell complexes, that is, of all arrows in $C$ that can be expressed as a transfinite composite of pushouts of elements of $I$.
\end{notation}

We will use the theorem below, due to Kan, to obtain the Quillen model structure on $\Cont$.

\begin{theorem}[Kan]
	\label{th: recognition cof gen model categories}
	Let $C$ be a complete and cocomplete locally small category endowed with sets of morphisms $I$, $J$ and $W$ such that
	\begin{enumerate}[label=(\arabic*), noitemsep]
		\item $W$ has the 2-out-of-3 property and is closed under retracts in $\iiAr(C)$.
		
		\item $I$, $J$ are small and permit the small object argument.
		
		\item $\cell(J) \subset \mathfrak L\mathfrak R(I) \cap W$.
		
		\item $\mathfrak R(I) \subset \mathfrak R(J) \cap W$.
		
		\item $\mathfrak L\mathfrak R(I) \cap W \subset \mathfrak L\mathfrak R(J)\;\;$ or $\;\; \mathfrak R(J) \cap W \subset \mathfrak R(I)$.
	\end{enumerate}
	Then $C$ admits a Quillen model structure whose set of weak equivalences is $W$, and having $I$, $J$ as sets of generating cofibrations and of generating acyclic cofibrations, respectively.
\end{theorem}

A proof can be found, for example, in \cite{Hov99}, Theorem 2.1.19, or \cite{Hir03}, Theorem 11.3.1.

\subsection{Full-and-faithful morphisms and weak equivalences}

A strict morphism of contextual categories $F:\mathcal A \rightarrow \mathcal B$ is said to be \emph{full-and-faithful} if its underlying functor is so.

\begin{proposition}
	\label{prop: characterization full-and-faithful}
	A strict morphism of contextual categories $F:\mathcal A \rightarrow \mathcal B$ is full-and-faithful if and only if for every object $a \in \mathcal A$ with $\ell(a) \ge 1$, the induced map $\tm_\mathcal A(a) \rightarrow \tm_\mathcal B(Fa)$ is bijective. Note that the latter condition is equivalent to $F$ having the right lifting property with respect to
	$$
	\iota_n^S:\mathcal O_{n-1} \rightarrow \mathcal O_n, \qquad \pi_n^S:\mathcal O_n^\vee \rightarrow \mathcal O_n
	$$
	for all $n \ge 1$.
\end{proposition}

\begin{proof}
	It is immediate that if $F$ is full-and-faithful, then $\tm_\mathcal A(a) \rightarrow \tm_\mathcal B(Fa)$ is bijective for all $a \in \mathcal A$ with $\ell(a) \ge 1$. Conversely, suppose that the latter condition holds. Given $a \in \mathcal A$, we will prove by induction on $n \ge 0$ that $F_{a,b}:\mathcal A(a,b) \rightarrow \mathcal B(Fa,Fb)$ is bijective whenever $\ell(b) = n$.
	
	This is the case for $n = 0$ as the only length-$0$ object of $\mathcal A$ and its image under $F$ are terminal. Given $a \in \mathcal A$ with $\ell(a) = n \ge 1$, assume that the claim holds for $0$, ..., $n-1$. Note that we can decompose the hom-sets under consideration as
	\begin{align*}
		\mathcal A(a,b) & = \bigcup_{f:a \rightarrow \partial b}\{f':a \rightarrow b \mid \textbf{p}_b \circ f' = f\}, \\
		\mathcal B(Fa,Fb) & = \bigcup_{g:F(a) \rightarrow \partial(Fb)}\{g':Fa \rightarrow Fb \mid F\textbf{p}_b \circ g' = g\}.
	\end{align*}
	
	By the induction hypothesis, $F_{a,\partial b}:\mathcal A(a,\partial b) \rightarrow \mathcal B(Fa,F(\partial b))$ is bijective, so it suffices to prove that for each $f:a \rightarrow \partial b$, the function
	\[
	\tag{\texttt{*}}
	\{f':a \rightarrow b \mid \textbf{p}_b \circ f' = f\} \longrightarrow \{g':Fa \rightarrow Fb \mid F\textbf{p}_b \circ g' = Ff\}
	\]
	obtained by co/restricting $F_{a,\partial b}$ is bijective. Now, consider the distinguished square
	\[
	\tag{\texttt{**}}
	\dsqua{a'}{b}{a}{\partial b}{}{f}{\textbf{p}_{a'}}{\textbf{p}_b}
	\]
	in $\mathcal A$; it is mapped by $F$ to a distinguished square
	\[
	\dsqua{Fa'}{Fb}{Fa}{\partial(Fb).}{}{Ff}{F\textbf{p}_{a'}}{F\textbf{p}_b}
	\]
	As (\texttt{**}) is cartesian, we have a canonical bijection between the domain of (\texttt{*}) and the set of sections of $\textbf{p}_{a'}$, i.e. $\tm_\mathcal A(a')$. Similarly, we have a canonical bijection between the codomain of (\texttt{*}) and $\tm_\mathcal B(Fa')$. Since, by assumption, $\tm_\mathcal A(a') \rightarrow \tm_\mathcal B(Fa')$ is bijective, we conclude that so is (\texttt{*}), as required.
\end{proof}

\begin{definition}
\label{def: slice length 1}
For a contextual category $\mathcal A$ and an object $a \in \mathcal A$, we write $\mathcal A_a$ for the category whose set of objects is $\{b \in \Ob(\mathcal A) \mid \partial b = a\}$, whose morphisms $b \rightarrow c$ are the arrows $f:b \rightarrow c$ in $\mathcal A$ such that $\textbf{p}_c \circ f = \textbf{p}_b$, and with composition inherited from $\mathcal A$.
	
For a strict morphism $F:\mathcal A \rightarrow \mathcal B$ and $a \in \mathcal A$, we let
$$
F_a:\mathcal A_a \longrightarrow \mathcal B_{Fa}
$$
be the induced functor.
\end{definition}

Weak equivalences of contextual categories were introduced in Definition \ref{def: weak equivalence} and described more explicitly in Remark \ref{rem: equivalences of contextual categories}. We now consider a variant concept:

\begin{definition}[Analogous to \cite{KapLum18}, Def. 3.1]
\label{def: type-theoretic equivalence}
A morphism of contextual categories $F:\mathcal A \rightarrow \mathcal B$ is a \emph{type-theoretic equivalence} if for every object $a \in \mathcal A$, the map
	$$
	\tm_\mathcal A(a) \longrightarrow \tm_\mathcal B(b)
	$$
	is bijective, and the functor
	$$
	F_a:\mathcal A_a \longrightarrow \mathcal B_{F(a)}
	$$
	is essentially surjective.
\end{definition}

We will prove in Proposition \ref{prop: characterization weak equivalences} that a strict morphism is a weak equivalence if and only if it is a type-theoretic equivalence.

\begin{remark}
\label{rem: weak equivalences remark 1}
Note that by Proposition \ref{prop: characterization full-and-faithful}, if $F$ is a type-theoretic equivalence, then $F_a$ is an equivalence of categories for every $a \in \mathcal A$. On the other hand, requiring that the functors $F_a$ be equivalences is not enough to ensure that $F$ is a type-theoretic equivalence. Consider, for example, the following gats (which are actually Lawvere theories):
$$
	\bbA = \begin{pmatrix*}[l]
		& \vdash X \tp\\
		x:X & \vdash e(x):X\\
		x, y:X & \vdash e(x) \equiv e(y):X
	\end{pmatrix*},
\qquad\quad
	\bbB= \begin{pmatrix*}[l]
		& \vdash X \tp\\
		 & \vdash a:X
	\end{pmatrix*}.
$$
In words, $\bbA$ is the theory of sets equipped with an endomorphism that sends any two elements to the same element, and $\bbB$ is the theory of pointed sets. For a pointed set $(X,a)$, the endomorphism $e:X \rightarrow X$ given by $x \mapsto a$ defines an $\bbA$-model; conversely, any non-empty $\bbA$-model $(X,e)$ defines a $\bbB$-model $(X,a)$ where $a = e(x)$ for any $x \in X$. It can be checked in a straightforward way that this defines a full-and-faithful functor $\Mod(\bbB) \rightarrow \Mod(\bbA)$ whose essential image consists of every $\bbA$-model except for the empty one. This functor is encoded by an interpretation $I$ of $\bbA$ in $\bbB$ that sends the judgment $x:X \vdash e(x):X$ to $x:X \vdash a:X$.

Now, recall that the contextual category $\mathcal C(\bbA)$ has a single length-$n$ object $\underline{n}$ for each $n \ge 0$ -- namely, the equivalence class of the context $x_1:X, ..., x_n:X$ --, and arrows $\underline{m} \rightarrow \underline{n}$ correspond to morphisms $F_\bbA(n) \rightarrow F_\bbA(m)$ where $F_\bbA(k)$ is the free $\bbA$-model on a set with $k$ elements. Similarly for $\bbB$. For $n \ge 1$, both $F_\bbA(n)$ and $F_\bbB(n)$ are sets with $n+1$ elements (the $n$ generators plus, for $\bbA$, the image point of the endomorphism, or, for $\bbB$, the distinguished point). However, $F_\bbA(0) = \varnothing$, while $F_\bbB(0)$ is a singleton.

Under this translation, the strict morphism
$$
\mathcal C(I):\mathcal C(\bbA) \longrightarrow \mathcal C(\bbB)
$$
sends $F_{\mathcal C(\bbA)}(n)$ to $F_{\mathcal C(\bbB)}(n)$ and, for $m$, $n \neq 0$, it sends a morphism $F_{\mathcal C(\bbA)}(m) \rightarrow F_{\mathcal C(\bbA)}(n)$ to itself viewed as a morphism of $\bbB$-models. This implies that $\mathcal C(I)$ induces an isomorphism between the full subcategories of $\mathcal C(\bbA)$ and $\mathcal C(\bbB)$ spanned by the objects of length $\ge 1$. As a consequence, $F_{\underline{n}}:\mathcal C(\bbA)_{\underline{n}} \rightarrow \mathcal C(\bbB)_{\underline{n}}$ is an equivalence for all $n \ge 1$.

However, $F$ is not a type-theoretic equivalence. For example,
\begin{align*}
	\tm_{\mathcal C(\bbA)}(\underline{1}) \cong \Mod(\bbA)(F_\bbA(1),F_\bbA(0)) & \;\; \text{is empty, but} \\[0.5em]
	\tm_{\mathcal C(\bbB)}(\underline{1}) \cong \Mod(\bbB)(F_\bbB(1),F_\bbB(0)) & \;\; \text{is a singleton.}
\end{align*}
\end{remark}

\begin{definition}
\label{def: A[n]}
Let $\mathcal A$ be a contextual category. In what follows, for $n \ge 0$ we regard $\{0, 1, ..., n\}$ as an ordered set in the usual way, thus as category with a unique arrow $i \rightarrow j$ for $i \le j$.

For an object $a \in \mathcal A$, let $\overline{a}:\{0,1, ..., \ell(a)\}^{\op} \longrightarrow \mathcal A$ be the functor encoding the sequence of display maps
$$
a \twoheadrightarrow \partial_{n-1} a \twoheadrightarrow \cdots \twoheadrightarrow \partial_1 a \twoheadrightarrow \partial_0 a = 1_\mathcal A.
$$
	
For $n \ge 0$, we let $\mathcal A[n]$ be the following category\footnote{By construction, $\mathcal A[n]$ is canonically isomorphic to the category of (strict) morphisms $\mathcal O_n^\pre \rightarrow \mathcal A$ from \cite{Alm26}; by Cor. 5.33 from that paper, it is also isomorphic to $\iiCont(\mathcal O_n,\mathcal A)$.}:
\begin{itemize}
	\item its objects are the length-$n$ objects of $\mathcal A$;
	
	\item morphisms from $a$ to $b$ are the natural transformations $\overline{a} \Rightarrow \overline{b}$.
\end{itemize}
Such a morphism is given by a commutative diagram
% https://q.uiver.app/#q=WzAsOCxbMCwwLCJhIl0sWzAsMSwiYiJdLFsxLDAsIlxccGFydGlhbF97bi0xfWEiXSxbMSwxLCJcXHBhcnRpYWxfe24tMX0gYiJdLFsyLDAsIlxcY2RvdHMiXSxbMiwxLCJcXGNkb3RzIl0sWzMsMCwiXFxwYXJ0aWFsXzAgYSA9IDFfXFxtYXRoY2FsIEEiXSxbMywxLCJcXHBhcnRpYWxfMCBiID0xX1xcbWF0aGNhbCBBIl0sWzAsMiwiIiwwLHsic3R5bGUiOnsiaGVhZCI6eyJuYW1lIjoiZXBpIn19fV0sWzEsMywiIiwwLHsic3R5bGUiOnsiaGVhZCI6eyJuYW1lIjoiZXBpIn19fV0sWzIsNCwiIiwwLHsic3R5bGUiOnsiaGVhZCI6eyJuYW1lIjoiZXBpIn19fV0sWzMsNSwiIiwwLHsic3R5bGUiOnsiaGVhZCI6eyJuYW1lIjoiZXBpIn19fV0sWzYsNywiXFxldGFfMCA9IGlkIl0sWzQsNV0sWzQsNiwiIiwxLHsic3R5bGUiOnsiaGVhZCI6eyJuYW1lIjoiZXBpIn19fV0sWzUsN10sWzIsMywiXFxldGFfe24tMX0iXSxbMCwxLCJcXGV0YV9uIiwyXV0=
	\[\begin{tikzcd}[ampersand replacement=\&]
		a \& {\partial_{n-1}a} \& \cdots \& {\partial_0 a = 1_\mathcal A} \\
		b \& {\partial_{n-1} b} \& \cdots \& {\partial_0 b =1_\mathcal A.}
		\arrow[two heads, from=1-1, to=1-2]
		\arrow["{\eta_n}"', from=1-1, to=2-1]
		\arrow[two heads, from=1-2, to=1-3]
		\arrow["{\eta_{n-1}}", from=1-2, to=2-2]
		\arrow[two heads, from=1-3, to=1-4]
		\arrow[from=1-3, to=2-3]
		\arrow["{\eta_0 = id}", from=1-4, to=2-4]
		\arrow[two heads, from=2-1, to=2-2]
		\arrow[two heads, from=2-2, to=2-3]
		\arrow[from=2-3, to=2-4]
	\end{tikzcd}\]
For a strict morphism $F:\mathcal A \rightarrow \mathcal B$, we let $F[n]:\mathcal A[n] \rightarrow \mathcal B[n]$ be the functor given by composing with $F$. Also, for $\phi:F \Rightarrow G$ in $\iiCont(\mathcal A,\mathcal B)$ we have a natural transformation
\begin{align*}
	\phi[n]:F[n] & \Longrightarrow G[n]\\
	\phi[n]_a & = (\phi_{\partial i a}:\partial_i F(a) \rightarrow \partial_i G(a))_{0 \le i \le n}
\end{align*}
This defines a strict $2$-functor $\iiCont \rightarrow \iiCat$.
\end{definition}

\begin{proposition}
\label{prop: characterization weak equivalences}
The following conditions on a strict morphism $F:\mathcal A \rightarrow \mathcal B$ are equivalent:
\begin{enumerate}[label=(\alph*)]
	\item $F$ is a weak equivalence.
	
	\item $F$ is a type-theoretic equivalence.
	
	\item $F$ is full-and-faithful, and $F[n]:\mathcal A[n] \rightarrow \mathcal B[n]$ is essentially surjective for all $n \ge 0$.
	
	\item The functors $F:|\mathcal A| \rightarrow |\mathcal B|$ and $F[n]:\mathcal A[n] \rightarrow \mathcal B[n]$ (for all $n \ge 0$) are equivalences of categories.
\end{enumerate}
\end{proposition}

\begin{proof}
\textbf{(a) $\Rightarrow$ (b).} If $F$ is a weak equivalence, then it is full-and-faithful; it remains to check that $F_a:\mathcal A_a \rightarrow \mathcal B_{Fa}$ is essentially surjective for all $a \in \mathcal A$. Consider a strict display map $p:b \twoheadrightarrow Fa$. By (ii') from Remark \ref{rem: equivalences of contextual categories}, there exists $c \in \mathcal A$ of length $\ge 1$ and a commutative square
\[
\dsqua{b}{F(c)}{F(a)}{F(\partial c)}{f'}{f}{\textbf{p}_b}{F(\textbf{p}_c)}
\]
where $f$, $f'$ are isomorphisms. Taking $g \in \mathcal A(a,\partial c)$ such that $F(g) = f$, let $a' \in \mathcal A$ be as in the distinguished square
\[
\dsqua{a'}{c}{a}{\partial c.}{g'}{g}{\textbf{p}_{a'}}{\textbf{p}_c}
\]
Then the dashed arrow making
% https://q.uiver.app/#q=WzAsNSxbMiwxLCJGKGMpIl0sWzAsMCwiYiJdLFsxLDIsIkYoYSkiXSxbMiwyLCJGKFxccGFydGlhbCBjKSJdLFsxLDEsIkYoYScpIl0sWzEsMiwiIiwwLHsiY3VydmUiOjMsInN0eWxlIjp7ImhlYWQiOnsibmFtZSI6ImVwaSJ9fX1dLFsyLDMsImYiLDJdLFswLDMsIiIsMix7InN0eWxlIjp7ImhlYWQiOnsibmFtZSI6ImVwaSJ9fX1dLFs0LDIsIiIsMCx7InN0eWxlIjp7ImhlYWQiOnsibmFtZSI6ImVwaSJ9fX1dLFsxLDQsIiIsMSx7InN0eWxlIjp7ImJvZHkiOnsibmFtZSI6ImRhc2hlZCJ9fX1dLFs0LDAsIkYoZycpIl0sWzEsMCwiZiciLDEseyJjdXJ2ZSI6LTN9XV0=
\[\begin{tikzcd}[ampersand replacement=\&,cramped]
	b \&\& \\
	\& {F(a')} \& {F(c)} \\
	\& {F(a)} \& {F(\partial c)}
	\arrow[dashed, from=1-1, to=2-2]
	\arrow["{f'}"{description}, curve={height=-18pt}, from=1-1, to=2-3]
	\arrow[curve={height=18pt}, two heads, from=1-1, to=3-2]
	\arrow["{F(g')}", from=2-2, to=2-3]
	\arrow[two heads, from=2-2, to=3-2]
	\arrow[two heads, from=2-3, to=3-3]
	\arrow["f"', from=3-2, to=3-3]
\end{tikzcd}\]
commute is an isomorphism $\textbf{p}_b \cong F(\textbf{p}_{a'})$ in $\mathcal B_{Fa}$.

\vspace{0.5em}

\textbf{(b) $\Rightarrow$ (c).} Suppose that $F$ is a type-theoretic equivalence. By Proposition \ref{prop: characterization full-and-faithful}, it is full-and-faithful. Let us check by induction that $F[n]:\mathcal A[n] \rightarrow \mathcal B[n]$ is essentially surjective for all $n \ge 0$. For $n = 0$, this holds trivially as both $\mathcal A[n]$ and $\mathcal B[0]$ are terminal categories. For $n \ge 1$, assume that the claim holds for $0$, ..., $n-1$ and consider a length-$n$ object $b \in \mathcal B$. By the induction hypothesis, there exist a length-$n$ object $a \in \mathcal A$ and an isomorphism $\gamma:F \circ \overline{a} \Rightarrow \overline{\partial b}$. Taking the distinguished square
\[
\squa{b'}{b}{F(a)}{\partial b,}{f}{\gamma_{n-1}}{\textbf{p}_{a'}}{\textbf{p}_b}
\]
$F_a:\mathcal A_a \rightarrow \mathcal B_{F(a)}$ being essentially surjective implies that there exist a strict display map $a' \twoheadrightarrow a$ and an isomorphism $g:F(a') \rightarrow b'$ such that $\textbf{p}_{F(a')} = \textbf{p}_{b'} \circ g$. Now, we can extend $\gamma$ into a natural isomorphism $F \circ \overline{a'} \cong \overline{b}$ by taking the $n$-component as $f \circ g$.

Suppose that $F$ is a weak equivalence. Let us first check that it is full-and-faithful. By Proposition \ref{prop: characterization full-and-faithful}, it suffices to prove that for all $a \in \mathcal A$ with $\ell(a) \ge 1$, the induced map $\tm_\mathcal A(a) \rightarrow \tm_\mathcal B(b)$ is bijective.

\vspace{0.5em}

\textbf{(c) $\Leftrightarrow$ (d)} is immediate, and (i), (ii') from Remark \ref{rem: equivalences of contextual categories} yield \textbf{(d) $\Rightarrow$ (a)}.
\end{proof}

\begin{remark}
\label{rem: weak equivalences remark 2}
While Proposition \ref{prop: characterization weak equivalences}(d) states that $F$ is a weak equivalence if and only both
\begin{align*}
	F:|\mathcal A| \rightarrow |\mathcal B| & \text{ is an equivalence of categories,}\\
	F[n]:\mathcal A[n] \rightarrow \mathcal B[n] & \text{ is an equivalence of categories for all } n \ge 0
\end{align*}
hold, neither of these is separately sufficient to ensure that $F$ is a weak equivalence.

For example, $\mathcal O_0$ -- the syntactic category of the empty gat -- has a single object, which is terminal. On the other hand, consider the contextual category $\mathcal T$ having a single length-$n$ object $\underline{n}$ for each $n \ge 0$, and a unique arrow $\underline{m} \rightarrow \underline{n}$ for each $m$, $n \ge 0$ -- this is the (Lawvere) theory whose $\Set$-models are the singletons. Note that $\mathcal O_0$ and $\mathcal T$ are initial and terminal objects, respectively, of $\Cont$. The strict morphism $!:\mathcal O_0 \rightarrow \mathcal T$ is not a weak equivalence as $(\mathcal O_0)_{1_{\mathcal O_0}} \rightarrow \mathcal T_{1_\mathcal T}$ is the inclusion of the empty category into a terminal one; but the underlying functor of $!$ is an equivalence of categories.

In the second case, a counterexample is the strict morphism $\mathcal C(\bbA) \rightarrow \mathcal C(\bbB)$ from Remark \ref{rem: weak equivalences remark 1}.
\end{remark}

\begin{proposition}
\label{prop: weak equivalences - transfinite composition, retracts, 2 of 3}
The (large) set of weak equivalences of (small) contextual categories contains all isomorphisms, has the 2-out-of-3 property, is closed under transfinite composition, and is closed under formation of retracts in $\iiAr(\Cont)$.
\end{proposition}

\begin{proof}[Sketch of proof]
A lengthy but routine calculation shows that the functors $\Cont \rightarrow \Cat$ given by
$$
\mathcal A \longmapsto |\mathcal A|, \qquad\qquad \mathcal A \longmapsto \mathcal A[n] \;\; \text{(for } n \ge 0\text{)}
$$
preserve directed colimits. In particular, each of these preserves transfinite composition. The claims in the statement then follow from Proposition \ref{prop: characterization weak equivalences}(c) and the fact that the class of equivalences of categories contains all isomorphisms, has the 2-out-of-3 property, is closed under transfinite composition in $\Cat$, and is closed under formation of retracts in $\iiAr(\Cat)$.
\end{proof}

\begin{proposition}
Consider strict morphisms $F$, $G:\mathcal A \rightarrow B$. If there exists a natural isomorphism $F \cong G$ and $F$ is a weak equivalence, then so is $G$.
\end{proposition}

\begin{proof}
Consider strict morphisms $F$, $G:\mathcal A \rightarrow \mathcal B$ such that $F \cong G$. If $F$ is a weak equivalence, by Proposition \ref{prop: characterization weak equivalences} the functors $F$ and $F[n]$ ($n \ge 0$) are equivalences of categories. Applying the $2$-functors $\iiCont \rightarrow \iiCont$ from Definition \ref{def: A[n]}, an isomorphism $F \cong G$ yields $F[n] \cong G[n]$ for each $n$. It follows that $G$ and $G[n]$ ($n \ge 0$) are equivalences of categories, and we conclude from Proposition \ref{prop: characterization weak equivalences} that $G$ is a weak equivalence in $\Cont$.
\end{proof}

\begin{proposition}
\label{prop: equivalence in iiCont is weak equivalence}
Suppose that a strict morphism $F:\mathcal A \rightarrow \mathcal B$ is an equivalence in $\iiCont$, that is, there exist a strict morphism $G:\mathcal B \rightarrow \mathcal A$ and isomorphisms $GF \cong \Id_\mathcal A$ and $FG \cong \Id_\mathcal B$. Then $F$ is a weak equivalence in $\Cont$.
\end{proposition}

\begin{proof}
Since strict $2$-functors preserve equivalences, $F[n]$ is an equivalence in $\iiCat$ for all $n \ge 0$. By Proposition \ref{prop: characterization weak equivalences}, $F$ is a weak equivalence.
\end{proof}

\subsection{The model structure}
\label{subsec: the model structure}

\begin{definition}
\label{def: cofibrations, acyclic fibrations}
In $\Cont$, we define a \emph{basic cofibration} as a morphism of one of the forms (see \S\ref{subsec: sorts terms and axioms via contextual categories})
$$
\iota_n^S:\mathcal O_{n-1} \rightarrow \mathcal O_n,\qquad \iota_n^T:\mathcal O_n \rightarrow \mathcal O_n^+,\qquad \pi_n^T:\mathcal O_n^{++} \rightarrow \mathcal O_n^+
$$
for $n \ge 1$. The set of basic cofibrations will be denoted by $\bcof$. As $\Cont$ is locally (finitely) presentable, by the small object argument it admits a weak factorization system cofibrantly generated by $\bcof$, i.e.
\begin{itemize}
	\item the right class, whose elements are the \emph{acyclic fibrations}, is $\afib = \rlp(\bcof)$;
	
	\item the left class, whose elements are the \emph{cofibrations}, is $\cof = \llp\rlp(\bcof) = \llp(\fib)$.
\end{itemize}
\end{definition}

\begin{remark}
\label{rem: describing acyclic fibrations}
We saw in Proposition \ref{prop: characterization full-and-faithful} that a strict morphism $F:\mathcal A \rightarrow \mathcal B$ is full-and-faithful precisely when it has the right lifting property with respect to $\iota_n^T$ and $\pi_n^T$ for all $n \ge 1$ -- equivalently, when it is right-orthogonal to $\iota_n^T$.

On the other hand, $F$ has the right lifting property with respect to $\iota_n^S$ if and only if the following holds: if $a \in \mathcal A$ has length $n-1$ and $p:b \twoheadrightarrow a$ is a strict display map in $\mathcal B$, then there exists a strict display map $q:a' \twoheadrightarrow a$ in $\mathcal A$ such that $Fq = p$. In other words, $F_a:\mathcal A_a \rightarrow \mathcal B_{Fa}$ is surjective on objects whenever $\ell(a) = n-1$.
\end{remark}

\begin{proposition}
\label{prop: characterization acyclic fibrations}
For a strict morphism $F:\mathcal A \rightarrow \mathcal \mathcal B$, the following are equivalent:
\begin{enumerate}[label=(\alph*)]
	\item $F$ is an acyclic fibration.
	
	\item $F$ is full-and-faithful, and $F_a:\mathcal A_a \rightarrow \mathcal B_{Fa}$ is surjective on objects for all $a \in \mathcal A$.
	
	\item $F$ is full-and-faithful and surjective on objects (equivalently, its underlying functor is an acyclic fibration with respect to the categorical model structure on $\Cat$).
\end{enumerate}
\end{proposition}

\begin{proof}
(a) $\Leftrightarrow$ (b) follows from Remark \ref{rem: describing acyclic fibrations}, and arguing by induction yields (b) $\Rightarrow$ (c). To prove (c) $\Rightarrow$ (b), suppose that $F$ is surjective on objects, and consider $a \in \mathcal A$ and $b \twoheadrightarrow F(a)$. Taking $a' \in \mathcal A$ such that $F(a') = b$, we have $F(\partial (a')) = \partial(F(a')) = \partial b = F(a)$, so as $F$ is full, there exists $i:a \rightarrow \partial(a')$ such that $F(i) = id_{F(a)}$. Then the distinguished square
\[
\dsqua{c}{a'}{a}{\partial(a')}{i'}{i}{}{}
\]
satisfies $F(c) = F(a') = b$.
\end{proof}

\begin{definition}
\label{def: basic acyclic cofibrations}
For $n \ge 0$, we let $\bbO_n^\triangledown$ be the gat
\begin{align*}
	& \vdash O_1 \tp\\
	x_1:O_1 & \vdash O_2(x_1) \tp\\
	& \cdots\\
	x_1:O_1, \;..., \; x_{n-1}:O_{n-1}(x_1, ..., x_{n-2}) & \vdash O_n(x_1, ..., x_{n-1}) \tp\\
	x_1:O_1, \;..., \; x_{n-1}:O_{n-1}(x_1, ..., x_{n-2}) & \vdash O'_n(x_1, ..., x_{n-1}) \tp\\[0.5em]
	x_1:O_1, \;..., \; x_{n-1}:O_{n-1}(x_1, ..., x_{n-2}), \; x_n:O_n(x_1, ..., x_{n-1}) & \vdash f(x_1, ..., x_n):O'_n(x_1, ..., x_{n-1})\\
	x_1:O_1, \;..., \; x_{n-1}:O_{n-1}(x_1, ..., x_{n-2}), \; x'_n:O'_n(x_1, ..., x_{n-1}) & \vdash g(x_1, ..., x'_n):O_n(x_1, ..., x_{n-1})
\end{align*}
\vspace{-2em}
\begin{align*}
	g(x_1, ..., x_{n-1}, f(x_1, ..., x_{n-1}, x_n)) & \equiv x_n : O_n(x_1, ..., x_{n-1})\\
	f(x_1, ..., x_{n-1}, g(x_1, ..., x_{n-1}, x'_n)) & \equiv x'_n: O'_n(x_1, ..., x_{n-1})
\end{align*}
We will denote by $\mathcal O_n^\triangledown$ the contextual category $\mathcal C(\bbO_n^\triangledown)$. It is freely generated by a diagram
% https://q.uiver.app/#q=WzAsNixbMCwwLCJvX24iXSxbMCwyLCJvX24nIl0sWzEsMSwib197bi0xfSJdLFsyLDEsIlxcY2RvdHMiXSxbMywxLCJvXzEiXSxbNCwxLCJvXzAiXSxbMiwzLCJcXHRleHRiZntwfV97b197bi0xfX0iLDAseyJzdHlsZSI6eyJoZWFkIjp7Im5hbWUiOiJlcGkifX19XSxbMyw0LCJcXHRleHRiZntwfV97b18yfSIsMCx7InN0eWxlIjp7ImhlYWQiOnsibmFtZSI6ImVwaSJ9fX1dLFs0LDUsIlxcdGV4dGJme3B9X3tvXzF9IiwwLHsic3R5bGUiOnsiaGVhZCI6eyJuYW1lIjoiZXBpIn19fV0sWzAsMiwiXFx0ZXh0YmZ7cH1fe29fbn0iLDAseyJzdHlsZSI6eyJoZWFkIjp7Im5hbWUiOiJlcGkifX19XSxbMSwyLCJcXHRleHRiZntwfV97bydfbn0iLDIseyJzdHlsZSI6eyJoZWFkIjp7Im5hbWUiOiJlcGkifX19XSxbMCwxLCJmIiwwLHsib2Zmc2V0IjotMX1dLFsxLDAsImciLDAseyJvZmZzZXQiOi0xfV1d
\[\begin{tikzcd}[row sep=tiny]
	{o_n} &&&& \\
	& {o_{n-1}} & \cdots & {o_1} & {o_0} \\
	{o_n'}
	\arrow["{\textbf{p}_{o_n}}", two heads, from=1-1, to=2-2]
	\arrow["f", shift left, from=1-1, to=3-1]
	\arrow["{\textbf{p}_{o_{n-1}}}", two heads, from=2-2, to=2-3]
	\arrow["{\textbf{p}_{o_2}}", two heads, from=2-3, to=2-4]
	\arrow["{\textbf{p}_{o_1}}", two heads, from=2-4, to=2-5]
	\arrow["g", shift left, from=3-1, to=1-1]
	\arrow["{\textbf{p}_{o'_n}}"', two heads, from=3-1, to=2-2]
\end{tikzcd}\]
where the subscripts indicate the length of each object, such that
$$
\textbf{p}_{o'_n} \circ f = \textbf{p}_{o_n}, \qquad \textbf{p}_{o_n} \circ g = \textbf{p}_{o'_n}, \qquad g \circ f = id_{o_n}, \qquad f \circ g = id_{o'_n}.
$$
We let $\iota_n^\triangledown:\mathcal O_n \rightarrow \mathcal O_n^\triangledown$ be the unique strict morphism acting on generating objects (recall the presentation of $\mathcal O_n$ given in Construction \ref{constr: O_n etc}) by $o_n \mapsto o_n$.
\end{definition}

\begin{definition}
\label{def: acyclic cofibrations, fibrations}
In $\Cont$, we define a \emph{basic acyclic cofibration} as a morphism of the form
$$
\iota_n^\triangledown:\mathcal O_n \longrightarrow \mathcal O_n^\triangledown
$$
for $n \ge 1$. The set of basic acyclic cofibrations will be denoted by $\bacof$. By the small object argument, $\bacof$ generates a weak factorization system on $\Cont$:
\begin{itemize}
	\item the right class, whose elements are the \emph{fibrations}, is $\fib = \rlp(\bacof)$;
	
	\item the left class, whose elements are the \emph{acyclic cofibrations}, is $\acof = \llp\rlp(\bacof) = \llp(\fib)$.
\end{itemize}
\end{definition}

\begin{remark}
\label{rem: description fibrations in Cont}
It is straightforward to verify that a strict morphism $F:\mathcal A \rightarrow \mathcal B$ is a fibration (in the sense of the above definition) if and only if for every $a \in \mathcal A$, the functor $F_a:\mathcal A_a \rightarrow \mathcal B_{Fa}$ is an isofibration, or equivalently, a fibration with respect to the categorical model structure on $\Cat$.
\end{remark}

\begin{lemma}
\label{lem: in Cont, acyclic fibration iff fibration and weq}
A morphism in $\Cont$ is an acyclic fibration (Definition \ref{def: cofibrations, acyclic fibrations}) if and only if it is a fibration (Definition \ref{def: acyclic cofibrations, fibrations}) and a weak equivalence (Definition \ref{def: weak equivalence}).
\end{lemma}

\begin{proof}
Since acyclic fibrations and weak equivalences are full-and-faithful, it suffices to prove that the following are equivalent for all $a \in \mathcal A$:
\begin{itemize}
	\item $F_a:\mathcal A_a \rightarrow \mathcal B_{Fa}$ is an equivalence of categories and surjective on objects.
	
	\item $F_a$ is an isofibration and an equivalence of categories.
\end{itemize}
We conclude from the fact that an equivalence of categories is surjective on objects precisely when it is an isofibration.
\end{proof}

\subsubsection{Interlude: tensors and powers of contextual categories by categories}

An important ingredient for concluding the construction of the model structure on $\Cont$ will be the fact that $\iiCont$ -- the strict $2$-category of contextual categories, strict morphisms, and natural transformations between them -- is tensored and powered as a $\Cat$-enriched category (we refer to \cite{Kel82}, \S3.7 for the relevant definitions\footnote{We call ``powers" what Kelly calls ``cotensor products".}). This will be used in Proposition \ref{prop: cellular acyclic cof is cof and weak equivalence}; more precisely, in the proof that pushouts of the arrows $\iota_n^\triangledown$ from Definition \ref{def: basic acyclic cofibrations} are weak equivalences. For $\mathcal A \in \iiCont$, letting $\textbf{I}_\cong \in \Cat$ be the free-standing isomorphism, that is, $\textbf{I}_\cong = \{0 \overset{\cong}{\leftrightarrow} 1\}$, the tensor $\textbf{I}_\cong \otimes \mathcal A$ and the power $\mathcal A^{\textbf{I}_\cong}$ will be, respectively, a cylinder and a path object for $\mathcal A$.

Such tensors and powers are studied in \cite{Alm26} as particular examples of tensor products and exponentials between contextual categories. However, since the general construction is far more involved than we will need at this point (that is, for proving Theorem \ref{th: model structure on Cont}), we will give a self-contained presentation of the objects needed and give a path for reading the required results from \cite{Alm26}.\footnote{Later, when proving that $\Cont$ is a monoidal model category, we will assume the content of \cite{Alm26} in a more substantial way.}

\begin{construction}[$\Cat$-powers of contextual categories]
Let $\mathcal A$ a small contextual category. For each $n \ge 0$, we let
$$
(-)_\varheartsuit:\mathcal A[n] \longrightarrow |\mathcal A|,
$$
where $\mathcal A[n]$ is as in Definition \ref{def: A[n]}, be the (``top level") functor given on objects by $a \mapsto a$, and on arrows by sending a natural transformation $\varphi:\overline{a} \Rightarrow \overline{b}$ to $\varphi_n:a \rightarrow b$.

Given a small category $K$, we let $\mathcal A^K$ be the following contextual category:
\begin{itemize}
	\item For each $n \ge 0$, its length-$n$ objects are the functors $K \rightarrow \mathcal A[n]$. For $F:K \rightarrow \mathcal A[m]$ and $G:K \rightarrow \mathcal A[n]$, morphisms from $F$ to $G$ in $|\mathcal A_K|$ are the natural transformations $F_\varheartsuit \Rightarrow G_\varheartsuit$ where the latter are, by abuse of notation, $(-)_\varheartsuit \circ F$ and $(-)_\varheartsuit \circ G$.
	
	\item For an object $F:K \rightarrow \mathcal A[n]$ with $n \ge 1$, we let $\partial F:K \rightarrow \mathcal A[n-1]$ be $\overline{\partial} \circ F$ where
	$$
	\overline{\partial}: \mathcal A[n] \longrightarrow \mathcal A[n-1]
	$$
	sends $a$ to $\partial a$, and $\varphi:\overline{a} \Rightarrow \overline{b}$ to $\varphi_{n-1}:\partial a \rightarrow \partial b$.
	
	\vspace{0.5em}
	
	The strict display map $\textbf{p}_F:F \Rightarrow \overline{\partial} \circ F$ is the natural transformation whose $k$-component is $\textbf{p}_{F(k)}$ for $k \in K$.
	
	\item For a diagram
	\[
	\begin{tikzcd}
		 & F_\varheartsuit \arrow[Rightarrow]{d}{\textbf{p}_F} \\
		G_\varheartsuit \arrow[Rightarrow,swap]{r}{\varphi} & \partial F_\varheartsuit,
	\end{tikzcd}
	\]
	the corresponding distinguished square is
	\[
	\begin{tikzcd}
		G'_\varheartsuit \arrow[Rightarrow]{r}{\varphi'} \arrow[Rightarrow,swap]{d}{\textbf{p}_{G'}} & F_\varheartsuit \arrow[Rightarrow]{d}{\textbf{p}_F} \\
		G_\varheartsuit \arrow[Rightarrow,swap]{r}{\varphi} & \partial F_\varheartsuit
	\end{tikzcd}
	\]
	where for each $k \in K$, $G'(k)$ and $\varphi'_k$ are as in the distinguished square
		\[
		\dsqua{G'(k)}{F(k)}{G(k)}{\partial (F(k)) = (\partial F)(k).}{\varphi'_k}{\varphi_k}{\textbf{p}_{G'(k)}}{\textbf{p}_{F(k)}}
		\]
\end{itemize}
\end{construction}

\begin{remark}
It is not difficult, using that $\mathcal A[n]$ is functorial in $\mathcal A$ for each $n$, to extend the assignment $(K,\mathcal A) \mapsto \mathcal A^K$ into a functor $\Cat^{\op} \times \Cont \rightarrow \Cont$.
\end{remark}

\begin{construction}[$\iiCont$ is $\Cat$-powered]
By construction, we have a full-and-faithful functor $U:|\mathcal A^K| \rightarrow \iiCat(K,|\mathcal A|)$ given on objects by $F \mapsto F_\varheartsuit$. Now, suppose we are also given $\mathcal B \in \iiCont$. Then we obtain full-and-faithful and injective-on-objects functors
$$
\iiCont(\mathcal B,\mathcal A^K) \hookrightarrow \iiCat(|\mathcal B|,|\mathcal A^K|) \overset{U \circ -}{\longrightarrow} \iiCat(|\mathcal B|,\iiCat(K,|\mathcal A|)) \cong \iiCat(K,\iiCat(|\mathcal B|,|\mathcal A|)).
$$
In fact, it can be proved that:
\begin{itemize}
	\item The composite of this chain of functors is injective on objects (despite $U$ not being injective on objects). This relies on the fact that preservation of strict display maps allows us to reconstruct a strict morphism $\mathcal B \rightarrow \mathcal A^K$ from its associated functor $|\mathcal B| \rightarrow \iiCat(K,|\mathcal A|)$.
	
	\item The (strict) image of the composite functor is the full subcategory $\iiCat(K,\iiCont(\mathcal B,\mathcal A))$. This uses that strict display maps (resp. distinguished squares) in $\mathcal A^K$ are defined as $\Ob(K)$-indexed families of strict display maps (resp. of distinguished squares) in $\mathcal A$.
\end{itemize}
Verifying these can be a useful exercise to familiarize oneself with the construction of $\mathcal A^K$; a proof is given in \cite[\S3.2]{Alm26}. As a consequence, we have an isomorphism $\iiCont(\mathcal B,\mathcal A^K) \cong \iiCat(K,\iiCont(\mathcal B,\mathcal A))$ natural in $\mathcal A$, $\mathcal B$ and $K$.
\end{construction}

\begin{proposition}[$\iiCont$ is $\Cat$-tensored]
\label{prop: cont is cat-tensored}
For $\mathcal A \in \iiCont$ and $K \in \Cat$, the $\Cat$-enriched functor $\iiCont^{\op} \rightarrow \iiCat$ given by
$$
\mathcal B \longmapsto \iiCat(K,\iiCont(\mathcal A,\mathcal B))
$$
is representable. By the enriched Yoneda lemma (\cite{Kel82}, \S2.4), such an object is determined up to canonical isomorphism. We will denote by $K \otimes \mathcal A$ an arbitrary contextual category that represents this functor (a \emph{tensor} between $K$ and $\mathcal A$); that is, we have $\iiCont(K \otimes \mathcal A,\mathcal B) \cong \iiCat(K,\iiCont(\mathcal A,\mathcal B))$.
\end{proposition}

\begin{proof}[Idea of proof]
Although a proof is given in \cite[Rem. 10.4]{Alm26}, let us outline a more direct description of such tensors. The idea is to present $K \otimes |\mathcal A|$ as a contextual category freely generated by the category $K \times |\mathcal A|$ subject to the following constraints for each $k \in \Ob(K)$:
\begin{itemize}
	\item for each $a \in \Ob(|\mathcal A|)$, the length of (the object corresponding to) $(k,a)$ equals that of $a$;
	
	\item if $\ell(a) \ge 1$, then $\partial(k,a) = (k,\partial a)$ and the corresponding strict display map is $(id_k,\textbf{p}_a):(k,a) \rightarrow (k,\partial a)$;
	
	\item for a distinguished square
	\[
	\dsqua{a}{b}{\partial a}{\partial b}{f'}{f}{\textbf{p}_a}{\textbf{p}_b}
	\]
	in $\mathcal A$, the diagram
	\[
	\dsqua{(k,a)}{(k,a)}{(k,\partial a)}{(k,\partial b)}{f'}{f}{\textbf{p}_a}{\textbf{p}_b}
	\]
	is a distinguished square.
\end{itemize}
More precisely, we want to construct a contextual category $\mathcal C$ equipped with a natural bijective correspondence between strict morphisms $\mathcal C \rightarrow \mathcal B$ and functors $H:K \times |\mathcal A| \rightarrow \mathcal B$ such that for each $k \in \Ob(K)$, letting $\iota_K:|\mathcal A| \rightarrow K \times |\mathcal A|$ be the functor $a \mapsto (k,a)$, the composite $H \circ \iota_K$ is a strict morphism from $\mathcal A$ to $\mathcal B$. Note that for such a $\mathcal C$ we have a canonical isomorphism
\[
\tag{\texttt{*}}
\Cont(\mathcal C,\mathcal B) \cong \Cat(K,\iiCont(\mathcal A,\mathcal B)).
\]
This can be done as a direct application of the formalism of precontextual categories from \cite[\S2]{Alm26}: the above generating data are a description a precontextual category, so it suffices to take its reflection onto contextual categories.

On the other hand, lifting this to an isomorphism $\iiCont(\mathcal C,\mathcal B) \cong \iiCat(K,\iiCont(\mathcal A,\mathcal B))$ is more difficult: it requires knowing that the universal property of $\mathcal C$ as a free object also applies to natural transformations, or, more precisely, that the adjunction between precontextual categories and contextual categories is $\Cat$-enriched. This is done in \cite[Cor. 5.33]{Alm26}.
\end{proof}

\begin{remark}
Since for $K \in \Cat$ we have $\Cont(K \otimes \mathcal A,\mathcal B) \cong \Cont(\mathcal A, \mathcal B^K)$ naturally in $\mathcal A$ and $\mathcal B$, the functors $K \otimes -$ and $(-)^K$ preserve colimits and limits, respectively.
\end{remark}

\begin{example}
\label{ex: cylinder and path object of strict cont cat}
As remarked earlier, the free-standing isomorphism $\textbf{I}_\cong = \{0 \overset{\cong}{\leftrightarrow} 1\} \in \Cat$ will be such that for $\mathcal A \in \Cont$, the contextual categories $\textbf{I}_\cong \otimes \mathcal A$ and $\mathcal A^{\textbf{I}_\cong}$ will function, respectively, as a cylinder and as a path object for $\mathcal A$. Note that strict morphisms $\mathcal A \rightarrow \mathcal B^{\textbf{I}_\cong}$ and $\textbf{I}_\cong \otimes \mathcal A \rightarrow \mathcal B$ correspond bijectively to triples $(F,G,\eta)$ consisting of strict morphisms $F$, $G:\mathcal A \rightarrow \mathcal B$ and an isomorphism $\eta:F \Rightarrow G$.
\end{example}

\subsubsection{Strong deformation retracts}

\begin{definition}
We say that a morphism $I:\mathcal A \rightarrow \mathcal B$ in $\Cont$ realizes $\mathcal A$ as a \emph{strong deformation retract} of $\mathcal B$ if there exist a strict morphism $R:\mathcal B \rightarrow \mathcal A$ such that $RI = \Id_\mathcal A$, and a natural isomorphism $\rho:IR \Rightarrow \Id_\mathcal B$ such that (noting that $IRI = I$) the horizontal composite $\rho I:IRI \Rightarrow I$ equals $\Id_I$.
\end{definition}

\begin{remark}
\label{rem: strong deformation retract -> weak equivalence}
If $I:\mathcal A \rightarrow \mathcal B$ and $R:\mathcal B \rightarrow \mathcal A$ realize $\mathcal A$ as a strong deformation retract of $\mathcal B$, then $I$ and $R$ are equivalences quasi-inverse to each other in $\iiCont$. By Proposition \ref{prop: equivalence in iiCont is weak equivalence}, $I$ and $R$ are weak equivalences.
\end{remark}

In what follows, we let $\textbf{1}$ be the terminal category, $j_0$, $j_1:\textbf{1} \rightarrow \textbf{I}_\cong$ the inclusions of the two objects, and $t:\textbf{I}_\cong \rightarrow \textbf{1}$ the terminal functor.

\begin{remark}
\label{rem: strong deformation retract 3 versions}
As in Example \ref{ex: cylinder and path object of strict cont cat}, for strict morphisms $I:\mathcal A \rightarrow \mathcal B$ and $R:\mathcal B \rightarrow \mathcal A$ we have a bijective correspondence between:
\begin{enumerate}[label=(\roman*)]
	\item Isomorphisms $\rho:IR \Rightarrow \Id_\mathcal A$ such that $\rho I = \Id_I$.
	
	\item Strict morphisms $H:\textbf{I}_\cong \otimes \mathcal B \rightarrow \mathcal B$ such that the following diagrams commute (strictly):
	% https://q.uiver.app/#q=WzAsOCxbMiwxLCJJX1xcY29uZyBcXG90aW1lcyBcXG1hdGhjYWwgQiJdLFs0LDEsIlxcbWF0aGNhbCBCIl0sWzAsMCwiXFx0ZXh0YmZ7MX0gXFxvdGltZXMgXFxtYXRoY2FsIEIgXFxjb25nIFxcbWF0aGNhbCBCIl0sWzAsMiwiXFx0ZXh0YmZ7MX0gXFxvdGltZXMgXFxtYXRoY2FsIEIgXFxjb25nIFxcbWF0aGNhbCBCIl0sWzYsMCwiXFx0ZXh0YmZ7SX1fXFxjb25nIFxcb3RpbWVzIFxcbWF0aGNhbCBBIl0sWzgsMCwiXFx0ZXh0YmZ7SX1fXFxjb25nIFxcb3RpbWVzIFxcbWF0aGNhbCBCIl0sWzYsMiwiXFx0ZXh0YmZ7MX0gXFxvdGltZXMgXFxtYXRoY2FsIEEgXFxjb25nIFxcbWF0aGNhbCBBIl0sWzgsMiwiXFxtYXRoY2FsIEIiXSxbMywwLCJqXzEgXFxvdGltZXMgXFxJZF9cXG1hdGhjYWwgQiIsMl0sWzIsMCwial8wIFxcb3RpbWVzIFxcSWRfXFxtYXRoY2FsIEIiXSxbMiwxLCJJUiIsMCx7ImN1cnZlIjotM31dLFszLDEsIlxcSWRfXFxtYXRoY2FsIEIiLDIseyJjdXJ2ZSI6M31dLFswLDEsIkgiLDFdLFs0LDUsIlxcSWQgXFxvdGltZXMgSSJdLFs1LDcsIkgiXSxbNCw2LCJ0IFxcb3RpbWVzIFxcSWQiLDJdLFs2LDcsIkkiLDJdXQ==
	\[\begin{tikzcd}[row sep=tiny]
		{\textbf{1} \otimes \mathcal B \cong \mathcal B} &&&&&& {\textbf{I}_\cong \otimes \mathcal A} && {\textbf{I}_\cong \otimes \mathcal B} \\
		&& {\textbf{I}_\cong \otimes \mathcal B} && {\mathcal B} \\
		{\textbf{1} \otimes \mathcal B \cong \mathcal B} &&&&&& {\textbf{1} \otimes \mathcal A \cong \mathcal A} && {\mathcal B}
		\arrow["{j_0 \otimes \Id_\mathcal B}", from=1-1, to=2-3]
		\arrow["IR", curve={height=-18pt}, from=1-1, to=2-5]
		\arrow["{\Id \otimes I}", from=1-7, to=1-9]
		\arrow["{t \otimes \Id}"', from=1-7, to=3-7]
		\arrow["H", from=1-9, to=3-9]
		\arrow["H"{description}, from=2-3, to=2-5]
		\arrow["{j_1 \otimes \Id_\mathcal B}"', from=3-1, to=2-3]
		\arrow["{\Id_\mathcal B}"', curve={height=18pt}, from=3-1, to=2-5]
		\arrow["I"', from=3-7, to=3-9]
	\end{tikzcd}\]
	
	\item Strict morphisms $H:\mathcal B \rightarrow \mathcal B^{\textbf{I}_\cong}$ such that the following diagrams commute (strictly):
	% https://q.uiver.app/#q=WzAsOCxbMiwxLCJcXG1hdGhjYWwgQl57XFx0ZXh0YmZ7SX1fXFxjb25nfSJdLFs0LDEsIlxcbWF0aGNhbCBCIl0sWzAsMCwiXFxtYXRoY2FsIEJeXFx0ZXh0YmZ7MX0gXFxjb25nIFxcbWF0aGNhbCBCIl0sWzAsMiwiXFxtYXRoY2FsIEJeXFx0ZXh0YmZ7MX0gXFxjb25nIFxcbWF0aGNhbCBCIl0sWzYsMCwiXFxtYXRoY2FsIEJee1xcdGV4dGJme0l9X1xcY29uZ30iXSxbOCwwLCJcXG1hdGhjYWwgQiJdLFs2LDIsIlxcbWF0aGNhbCBBXntcXHRleHRiZntJfV9cXGNvbmd9Il0sWzgsMiwiXFxtYXRoY2FsIEFee1xcdGV4dGJmIDF9IFxcY29uZyBcXG1hdGhjYWwgQSJdLFswLDMsIlxcbWF0aGNhbCBCXntqXzF9Il0sWzAsMiwiXFxtYXRoY2FsIEJee2pfMH0iLDJdLFsxLDIsIklSIiwyLHsiY3VydmUiOjN9XSxbMSwzLCJcXElkX1xcbWF0aGNhbCBCIiwwLHsiY3VydmUiOi0zfV0sWzEsMCwiSCIsMV0sWzUsNCwiSCIsMl0sWzcsNSwiSSIsMl0sWzYsNCwiSV57XFx0ZXh0YmZ7SX1fXFxjb25nfSJdLFs3LDYsIlxcbWF0aGNhbCBBXnQiXV0=
	\[\begin{tikzcd}[row sep=tiny]
		{\mathcal B^\textbf{1} \cong \mathcal B} &&&&&& {\mathcal B^{\textbf{I}_\cong}} && {\mathcal B} \\
		&& {\mathcal B^{\textbf{I}_\cong}} && {\mathcal B} \\
		{\mathcal B^\textbf{1} \cong \mathcal B} &&&&&& {\mathcal A^{\textbf{I}_\cong}} && {\mathcal A^{\textbf 1} \cong \mathcal A}
		\arrow["H"', from=1-9, to=1-7]
		\arrow["{\mathcal B^{j_0}}"', from=2-3, to=1-1]
		\arrow["{\mathcal B^{j_1}}", from=2-3, to=3-1]
		\arrow["IR"', curve={height=18pt}, from=2-5, to=1-1]
		\arrow["H"{description}, from=2-5, to=2-3]
		\arrow["{\Id_\mathcal B}", curve={height=-18pt}, from=2-5, to=3-1]
		\arrow["{I^{\textbf{I}_\cong}}", from=3-7, to=1-7]
		\arrow["I"', from=3-9, to=1-9]
		\arrow["{\mathcal A^t}", from=3-9, to=3-7]
	\end{tikzcd}\]
\end{enumerate}
\end{remark}

The next result (which uses a standard argument; see, e.g., \cite[Prop. 2.4.9]{Hov99}) will play in important role in the remainder of this section. It will be used for proving that pushouts of the basic cofibrations $\iota_n^\triangledown$ are weak equivalences (see Proposition \ref{prop: cellular acyclic cof is cof and weak equivalence}) and, later, for checking that the monoidal model category $(\Cont,\otimes, \cdots)$ satisfies the monoid axiom (Proposition \ref{prop: monoid axiom for Cont}).

\begin{proposition}
\label{prop: strong deformation retract pushout}
Consider a pushout square
\[
\squa{\mathcal A}{\mathcal B}{\mathcal A'}{\mathcal B'}{I}{I'}{F}{F'}
\]
in $\Cont$. If $I$ realizes $\mathcal A$ as a strong deformation retract of $\mathcal B$, then $I'$ realizes $\mathcal A'$ as a strong deformation retract of $\mathcal B'$.
\end{proposition}

\begin{proof}
Let $R:\mathcal B \rightarrow \mathcal A$ and $H:\textbf{I}_\cong \otimes \mathcal A \rightarrow \mathcal B$ be strict morphisms such that $RI = \Id_\mathcal A$ and the two diagrams from Remark \ref{rem: strong deformation retract 3 versions}(ii) commute. Note that the pasting lemma for pushouts yields a commutative diagram
\[\begin{tikzcd}
	{\mathcal A} & {\mathcal B} & {\mathcal A} \\
	{\mathcal A'} & {\mathcal B'} & {\mathcal A'}
	\arrow["I", from=1-1, to=1-2]
	\arrow["{\Id}", curve={height=-24pt}, from=1-1, to=1-3]
	\arrow["F"', from=1-1, to=2-1]
	\arrow["R", from=1-2, to=1-3]
	\arrow["{F'}", from=1-2, to=2-2]
	\arrow["F", from=1-3, to=2-3]
	\arrow["{I'}"', from=2-1, to=2-2]
	\arrow["{\Id}"', curve={height=24pt}, from=2-1, to=2-3]
	\arrow["{R'}"', from=2-2, to=2-3]
\end{tikzcd}\]
in which the right square is cocartesian. To exhibit $\mathcal A'$ as a strong deformation retract of $\mathcal B'$ via $I'$, we must construct an isomorphism $I'R' \cong \Id_{\mathcal B'}$. Since $\textbf{I}_\cong \otimes -: \Cont \rightarrow \Cont$ is left adjoint to $(-)^{\textbf{I}_\cong}$, the former preserves colimits. Hence we have a pushout square
\[
\squa{\textbf{I}_\cong \otimes \mathcal \mathcal A}{\textbf{I}_\cong \otimes \mathcal B}{\textbf{I}_\cong \otimes \mathcal A'}{\textbf{I}_\cong \otimes \mathcal B',}{\Id \otimes I}{\Id \otimes I'}{\Id \otimes F}{\Id \otimes F'}
\]
which we then extend to
% https://q.uiver.app/#q=WzAsNyxbMCwxLCJJX1xcY29uZyBcXG90aW1lcyBcXG1hdGhjYWwgQSJdLFsxLDEsIklfXFxjb25nIFxcb3RpbWVzIFxcbWF0aGNhbCBCIl0sWzAsMCwiSV9cXGNvbmcgXFxvdGltZXMgXFxtYXRoY2FsIE9fbiJdLFsxLDAsIklfXFxjb25nIFxcb3RpbWVzXFxtYXRoY2FsIE9fbl5cXHRyaWFuZ2xlZG93biJdLFsyLDAsIlxcbWF0aGNhbCBPX25eXFx0cmlhbmdsZWRvd24iXSxbMiwyLCJcXG1hdGhjYWwgQiJdLFswLDIsIlxcbWF0aGNhbCBBIl0sWzIsMCwiSWQgXFxvdGltZXMgRyIsMl0sWzMsMSwiSWQgXFxvdGltZXMgRyciXSxbMiwzLCJJZCBcXG90aW1lcyBcXGlvdGFfbl5cXHRyaWFuZ2xlZG93biJdLFswLDEsIklkIFxcb3RpbWVzIEYiLDJdLFszLDQsIkgiXSxbNCw1LCJHJyJdLFswLDYsIktfXFxtYXRoY2FsIEEiLDJdLFs2LDUsIkYiLDJdXQ==
\[
\tag{\texttt{*}}
\begin{tikzcd}
	{\textbf{I}_\cong \otimes \mathcal A} & {\textbf{I}_\cong \otimes\mathcal B} & {\mathcal B} \\
	{\textbf{I}_\cong \otimes \mathcal A'} & {\textbf{I}_\cong \otimes \mathcal B'} \\
	{\textbf{1} \otimes \mathcal A' \cong \mathcal A'} && {\mathcal B'}
	\arrow["{\Id \otimes I}", from=1-1, to=1-2]
	\arrow["{\Id \otimes F}"', from=1-1, to=2-1]
	\arrow["H", from=1-2, to=1-3]
	\arrow["{\Id \otimes F'}", from=1-2, to=2-2]
	\arrow["{F'}", from=1-3, to=3-3]
	\arrow["{\Id \otimes I'}"', from=2-1, to=2-2]
	\arrow["{t \otimes \Id}"', from=2-1, to=3-1]
	\arrow["{I'}"', from=3-1, to=3-3]
\end{tikzcd}\]

The latter commutes since
\begin{align*}
	F' \circ H \circ (\Id_{\textbf{I}_\cong} \otimes I) & = F' \circ (t \otimes I)\\
	& = F' \circ I \circ (t \otimes \Id_\mathcal A)\\
	& = I' \circ F \circ (t \otimes \Id_\mathcal A)\\
	& = I' \circ (t \otimes \Id_{\mathcal A'}) \circ (\Id_{\textbf{I}_\cong} \otimes F).
\end{align*}
This yields a morphism $H':\textbf{I}_\cong \otimes \mathcal B' \rightarrow \mathcal B'$ such that the further extended diagram commutes.

Finally, let us describe the domain and the codomain of the natural isomorphism determined by $H'$. A straightforward calculation using (\texttt{*}) and the relations from Remark \ref{rem: strong deformation retract 3 versions}(ii) shows that we have commutative diagrams
% https://q.uiver.app/#q=WzAsMTQsWzAsMSwiXFxtYXRoY2FsIEEnIl0sWzEsMSwiXFxtYXRoY2FsIEInIl0sWzAsMCwiXFxtYXRoY2FsIEEiXSxbMSwwLCJcXG1hdGhjYWwgQiJdLFsyLDAsIlxcbWF0aGNhbCBCIl0sWzIsMiwiXFxtYXRoY2FsIEInLCJdLFswLDIsIlxcbWF0aGNhbCBBJyJdLFs0LDEsIlxcbWF0aGNhbCBBJyJdLFs1LDEsIlxcbWF0aGNhbCBCJyJdLFs0LDAsIlxcbWF0aGNhbCBBIl0sWzUsMCwiXFxtYXRoY2FsIEIiXSxbNiwwLCJcXG1hdGhjYWwgQiJdLFs2LDIsIlxcbWF0aGNhbCBCJy4iXSxbNCwyLCJcXG1hdGhjYWwgQSciXSxbMiwwLCJGIiwyXSxbMywxLCJGJyJdLFsyLDMsIkkiXSxbMCwxLCJJJyIsMl0sWzMsNCwiSVIiXSxbNCw1LCJGJyJdLFs2LDUsIkYiLDJdLFsxLDUsIkgnIFxcY2lyYyAoal8wIFxcb3RpbWVzIElkX3tcXG1hdGhjYWwgQid9KSIsMV0sWzksNywiRiIsMl0sWzEwLDgsIkYnIl0sWzksMTAsIkkiXSxbNyw4LCJJJyIsMl0sWzEwLDExLCJcXElkIl0sWzExLDEyLCJGJyJdLFs3LDEzLCJcXElkIiwyXSxbMTMsMTIsIkknIiwyXSxbOCwxMiwiSCcgXFxjaXJjIChqXzEgXFxvdGltZXMgSWRfXFxtYXRoY2FsIEIpIiwxXSxbMCw2LCJcXElkIiwyXV0=
\[\begin{tikzcd}[row sep=large, column sep=large]
	{\mathcal A} & {\mathcal B} & {\mathcal B} && {\mathcal A} & {\mathcal B} & {\mathcal B} \\
	{\mathcal A'} & {\mathcal B'} &&& {\mathcal A'} & {\mathcal B'} \\
	{\mathcal A'} && {\mathcal B',} && {\mathcal A'} && {\mathcal B'.}
	\arrow["I", from=1-1, to=1-2]
	\arrow["F"', from=1-1, to=2-1]
	\arrow["IR", from=1-2, to=1-3]
	\arrow["{F'}", from=1-2, to=2-2]
	\arrow["{F'}", from=1-3, to=3-3]
	\arrow["I", from=1-5, to=1-6]
	\arrow["F"', from=1-5, to=2-5]
	\arrow["\Id", from=1-6, to=1-7]
	\arrow["{F'}", from=1-6, to=2-6]
	\arrow["{F'}", from=1-7, to=3-7]
	\arrow["{I'}"', from=2-1, to=2-2]
	\arrow["\Id"', from=2-1, to=3-1]
	\arrow["{H' \circ (j_0 \otimes \Id_{\mathcal B'})}"{description}, from=2-2, to=3-3]
	\arrow["{I'}"', from=2-5, to=2-6]
	\arrow["\Id"', from=2-5, to=3-5]
	\arrow["{H' \circ (j_1 \otimes \Id_{\mathcal B'})}"{description}, from=2-6, to=3-7]
	\arrow["F"', from=3-1, to=3-3]
	\arrow["{I'}"', from=3-5, to=3-7]
\end{tikzcd}\]
Since
\begin{align*}
	(H' \circ (j_0 \otimes \Id_{\mathcal B'})) \circ F' & = F' I R = I' F R = I' R' F',\\
	(H' \circ (j_0 \otimes \Id_{\mathcal B'})) \circ I' & = I' = I' R' I',
\end{align*}
we have $H' \circ (j_0 \otimes \Id_{\mathcal B'}) = I'R'$, and as
\begin{align*}
	(H' \circ (j_1 \otimes \Id_{\mathcal B'})) \circ F' & = F', \\
	(H' \circ (j_1 \otimes \Id_{\mathcal B'})) \circ I' & = I',
\end{align*}
we have $H' \circ (j_1 \otimes \Id_{\mathcal B'}) = \Id_{\mathcal B'}$. It follows that $H'$ defines an isomorphism $I'R' \cong \Id_{\mathcal B'}$, and this concludes the proof that $I'$ realizes $\mathcal A'$ as a strong deformation retract of $\mathcal B'$.
\end{proof}

\subsubsection{Conclusion of the construction of the model structure}

In what follows, we let $W$ be the set of weak equivalences in $\Cont$.

\begin{proposition}
\label{prop: cellular acyclic cof is cof and weak equivalence}
$\cell(\bacof) \subset \cof \cap W$.
\end{proposition}

\begin{proof}
Recall that every element of $\cell(\bacof)$ is a transfinite composite of pushouts of elements of $\bacof$, that is, of arrows of the form $\iota_n^\triangledown:\mathcal O_n \rightarrow \mathcal O_n^\triangledown$ for $n \ge 1$. Since both $\cof$ (being the left class of a weak factorization system) and $W$ (by Proposition \ref{prop: weak equivalences - transfinite composition, retracts, 2 of 3}) are closed under transfinite composition, it suffices to prove that for all $n \ge 1$, every pushout of $\iota_n^\triangledown$ belongs to $\cof \cap W$.

Firstly, it follows from the syntactic presentation of $\iota_n^\triangledown$ that it belongs to $\cof$. Indeed, it is obtained by adding one sort axiom, two term axioms, and two term equality axioms, so it is a composite of five pushouts of elements of $\bcof$. As $\cof$ is closed under pushouts, we conclude that every pushout of $\iota_n^\triangledown$ belongs to $\cof$.

For the second part, by Proposition \ref{prop: strong deformation retract pushout} and Remark \ref{rem: strong deformation retract -> weak equivalence} it suffices to prove that $\iota_n^\triangledown$ realizes $\mathcal O_n$ as a strong deformation retract of $\mathcal O_n^\triangledown$. Let
$$
\pi_n^\triangledown:\mathcal O_n^\triangledown \longrightarrow \mathcal O_n
$$
be the unique strict morphism that, in terms of the presentations from Construction \ref{constr: O_n etc} and Definition \ref{def: basic acyclic cofibrations}, sends $o_i$ to $o_i$ for $1 \le i \le n$, $o'_n$ to $o_n$, and the isomorphism $o_n \cong o'_n$ to $id_{o_n}$. Note that $\pi_n^\triangledown \circ \iota_n^\triangledown = \Id_{\mathcal O_n}$. Also, by the universal property of $\mathcal O_n^\triangledown$ we have a unique strict morphism
$$
H:\mathcal O_n^\triangledown \longrightarrow (\mathcal O_n^\triangledown)^{\textbf{I}_\cong}
$$
that maps $o_i$ to the length-$i$ object $\textbf{I}_\cong \rightarrow \mathcal O_n^\triangledown[i]$ corresponding to the identity natural transformation $\overline{o_i} \Longrightarrow \overline{o_i}$, and $o'_n$ to the length-$n$ object $\textbf{I}_\cong \rightarrow \mathcal O_n^\triangledown[n]$ corresponding to the natural transformation
\begin{align*}
	\overline{o_n} & \Longrightarrow \overline{o'_n}\\
	i & \longmapsto id_{o_i} \; \text{ for } 0 \le i \le n-1\\
	n & \longmapsto (f:o_n \rightarrow o'_n).
\end{align*}
It is immediate that the natural transformation $\rho$ classified by $H$ is of the form $\iota_n^\triangledown \Rightarrow \Id_{\mathcal O_n^\triangledown}$ and satisfies $\rho \iota_n^\triangledown = \Id_{\iota_n^\triangledown}$ (equivalently, $H$ is such that, with $I = \iota_n^\triangledown$ and $R = \pi_n^\triangledown$, the diagrams from Remark \ref{rem: strong deformation retract 3 versions} commute). This concludes the proof that $\mathcal O_n$ is a strong deformation retract of $\mathcal O_n^\triangledown$ via $\iota_n^\triangledown$.
\end{proof}

\begin{theorem}
\label{th: model structure on Cont}
There exists a combinatorial Quillen model structure on $\Cont$ with the above classes of weak equivalences, (co)fibrations, and acyclic (co)fibrations.
\end{theorem}

\begin{proof}
Let us check the conditions from Theorem \ref{th: recognition cof gen model categories} with $C = \Cont$, $I = \bcof$, and $J = \bacof$:
\begin{enumerate}[label=(\arabic*)]
	\item follows from Proposition \ref{prop: weak equivalences - transfinite composition, retracts, 2 of 3}.
	
	\item follows from $\Cont$ being locally presentable (so every small set of morphisms permits the small object argument) and $\bcof$, $\bacof$ being small (by definition).
	
	\item is the content of Proposition \ref{prop: cellular acyclic cof is cof and weak equivalence}.
	
	\item and the second condition from (5) are the content of Lemma \ref{lem: in Cont, acyclic fibration iff fibration and weq}.
\end{enumerate}
\end{proof}

\begin{remark}
This model structure can be transferred via the equivalence $\GAT \simeq \Cont$ from \cite{Car78} to a model structure on the category $\GAT$ of generalized algebraic theories and equivalence classes of interpretations.
\end{remark}

\begin{remark}[Fibrant and cofibrant objects]
\label{rem: fibrant and cofibrant objects}
Recall the contextual categories $\mathcal O_0$ and $\mathcal T$ discussed in Remark \ref{rem: weak equivalences remark 2}, which are initial and terminal, respectively, in $\Cont$. It is immediate from Remark \ref{rem: description fibrations in Cont} that every $\mathcal A \in \Cont$ is fibrant, i.e. the unique morphism $\mathcal A \rightarrow \mathcal T$ is a fibration. On the other hand, not every $\mathcal A$ is cofibrant, that is, there are cases where $\mathcal O_0 \rightarrow \mathcal A$ is not a cofibration. Syntactically, recall that pushouts of the three basic cofibrations are theory extensions obtained by adding a sort axiom, a term axiom, or a term equality axiom. By the retract lemma (\cite[Lem. 1.1.9]{Hov99}), $\mathcal A$ is cofibrant precisely when $\mathcal A$ is a retract of $\mathcal C(\bbB)$ for a gat $\bbB$ that does not contain any sort equality axioms. Such $\bbB$ will be called \emph{syntactically cellular theories} in \S\ref{sec: cellular theories}.\footnote{The property of being syntactically cellular is not invariant under isomorphism: a gat is cellular with respect to the set $\cof^*$ of basic cofibrations if and only if it is isomorphic to a syntactically cellular theory.}

For a theory to be non-cofibrant, it is necessary that it not be cellular, but this is not sufficient as the set of cellular (with respect to $\cof^*$) objects of $\Cont$ is not closed under formation of retracts; see Example \ref{ex: cofibrant, not cellular}. Still, Proposition \ref{prop: cellular gat -> basis} on cellular theories yields, by taking retracts, a necessary condition for $\mathcal A \in \Cont$ being cofibrant in terms of the category of strict display maps $\textbf{D}(\mathcal A)$ (as in Construction \ref{constr: categories of display maps}); see Remark \ref{rem: necessary condition cofibrant - connected components of D(A)}. Later, we will see (Prop. \ref{prop: strictification - cofibrant domain}) that if a gat $\bbA$ is cofibrant, then its categories of family-valued models (strict morphisms to $\Fam$) and of set-valued models are equivalent -- more concretely, every set-valued model is strictifiable to a family-valued one.

We already came across two non-cofibrant theories: $\bbK$ from \S\ref{sec: introduction} (which is an instance of Example \ref{ex: E(K)}), and $\mathcal B$ from Remark \ref{rem: equivalences of contextual categories}. Another example will appear in Remark \ref{rem: not left-proper}, where we check that the model structure on $\Cont$ is not left-proper (which is only possible in the presence of non-cofibrant objects).

\end{remark}

\begin{example}[A non-cellular cofibrant gat]
\label{ex: cofibrant, not cellular}
Let $\bbA$ be the gat given by
\begin{align*}
	 & \vdash X \tp & x:X' & \vdash e(x):X'\\
	x:X & \vdash Y(x) \tp & x:X' & \vdash fe(x) \equiv f(x):X\\
	 & \vdash X' \tp & x:X' & \vdash e^2(x) \equiv e(x):X'\\
	 x:X' & \vdash f(x):X & & 
\end{align*}
On the other hand, let $\mathbb R$ be given by
\begin{align*}
	& \vdash X' \tp & x:X'& \vdash e^2(x) \equiv e(x):X' \\
	x:X' & \vdash Y'(x) \tp & x:X' & \vdash Y'(e(x)) \equiv Y'(x) \tp \\
	x:X' & \vdash e(x):X' & & 
\end{align*}
We have an interpretation $I$ of $\bbB$ in $\bbA$ that sends $X'$, $e'(x)$ to themselves and $Y'(x)$ to $Y(f(x))$. In terms of the generating data of the respective syntactic categories, we can depict $\mathcal C(I):\mathcal C(\bbB) \rightarrow \mathcal C(\bbA)$ as
% https://q.uiver.app/#q=WzAsOCxbNSwyLCJbeDpYXSJdLFs1LDAsIlt4OlgseTpZKHgpXSJdLFszLDIsIlt4OlgnXSIsWzI0MCw2MCw2MCwxXV0sWzMsMCwiW3g6WCcseTpZKGYoeCkpXSIsWzI0MCw2MCw2MCwxXV0sWzAsMiwiW3g6WCddIixbMjQwLDYwLDYwLDFdXSxbMCwwLCJbeDpYJyx5OlknKHgpXSIsWzI0MCw2MCw2MCwxXV0sWzAsMV0sWzMsMV0sWzEsMCwiIiwwLHsic3R5bGUiOnsiaGVhZCI6eyJuYW1lIjoiZXBpIn19fV0sWzIsMCwiW2YoeCldIiwyXSxbMiwyLCJbZSh4KV0iLDIseyJyYWRpdXMiOi0zLCJhbmdsZSI6OTAsImNvbG91ciI6WzI0MCw2MCw2MF19LFsyNDAsNjAsNjAsMV1dLFszLDIsIiIsMix7ImNvbG91ciI6WzI0MCw2MCw2MF0sInN0eWxlIjp7ImJvZHkiOnsibmFtZSI6ImRhc2hlZCJ9LCJoZWFkIjp7Im5hbWUiOiJlcGkifX19XSxbMywxLCIiLDAseyJzdHlsZSI6eyJib2R5Ijp7Im5hbWUiOiJkYXNoZWQifX19XSxbNSw0LCIiLDAseyJjb2xvdXIiOlsyNDAsNjAsNjBdLCJzdHlsZSI6eyJoZWFkIjp7Im5hbWUiOiJlcGkifX19XSxbNCw0LCJbZSh4KV0iLDIseyJyYWRpdXMiOi0zLCJhbmdsZSI6OTAsImNvbG91ciI6WzI0MCw2MCw2MF19LFsyNDAsNjAsNjAsMV1dLFs2LDcsIlxcbWF0aGNhbCBDKEkpIiwwLHsic3R5bGUiOnsiYm9keSI6eyJuYW1lIjoiZG90dGVkIn19fV1d
\[\begin{tikzcd}[ampersand replacement=\&,cramped]
	\textcolor{rgb,255:red,92;green,92;blue,214}{{[x:X',y:Y'(x)]}} \&\&\& \textcolor{rgb,255:red,92;green,92;blue,214}{{[x:X',y:Y(f(x))]}} \&\& {[x:X,y:Y(x)]} \\
	{} \&\&\& {} \\
	\textcolor{rgb,255:red,92;green,92;blue,214}{{[x:X']}} \&\&\& \textcolor{rgb,255:red,92;green,92;blue,214}{{[x:X']}} \&\& {[x:X]}
	\arrow[color={rgb,255:red,92;green,92;blue,214}, two heads, from=1-1, to=3-1]
	\arrow[dashed, from=1-4, to=1-6]
	\arrow[color={rgb,255:red,92;green,92;blue,214}, dashed, two heads, from=1-4, to=3-4]
	\arrow[two heads, from=1-6, to=3-6]
	\arrow["{\mathcal C(I)}", dotted, from=2-1, to=2-4, shorten=10pt]
	\arrow["{[e(x)]}"', color={rgb,255:red,92;green,92;blue,214}, from=3-1, to=3-1, loop, in=215, out=145, distance=10mm]
	\arrow["{[e(x)]}"', color={rgb,255:red,92;green,92;blue,214}, from=3-4, to=3-4, loop, in=215, out=145, distance=10mm]
	\arrow["{[f(x)]}"', from=3-4, to=3-6]
\end{tikzcd}\]
$\mathcal C(I)$ has a retraction $\mathcal C(R)$ where $R$ is the interpretation
$$
X' \longmapsto X' \qquad e'(x) \longmapsto e'(x) \qquad X \longmapsto X' \qquad Y(x) \longmapsto Y'(x) \qquad f(x) \longmapsto e(x)
$$
(note that the induced map on judgments sends the axiom $fe(x) \equiv f(x)$ to $e^2(x) \equiv e(x)$, which is derivable). Being a retract of a gat without sort equality axioms, $\bbB$ is cofibrant. But it is not cellular: if it were, then, as we will see in \S\ref{sec: cellular theories}, every connected component of the category of display maps $\textbf{D}(\mathcal C(\bbB))$ would have a terminal object, which is not the case. Indeed, it is not difficult to verify that every display map in the connected component of $[x:X', y:Y'(x)] \twoheadrightarrow [x:X]$ has a non-identity endomorphism obtained by applying $e$ to a variable of sort $X'$.
\end{example}

\begin{remark}[Failure of left-properness]
\label{rem: not left-proper}
A model category is \emph{left-proper} if weak equivalences are stable under pushout along cofibrations. We will now give an example showing that $\Cont$ does not have this property.

Let $\bbA$ be the gat given by
\begin{align*}
	& \vdash X \tp & & \vdash Y(a) \equiv Y(b) \tp\\
	x:X & \vdash Y(x) \tp & &\\
	& \vdash a:X & & \\
	& \vdash b:X & &
\end{align*}
and let $\bbB$ be given by
\begin{align*}
	 & \vdash X \tp & y:Y(a) & \vdash f(y):Y(b)\\
	 x:X & \vdash Y(x) \tp & y:Y(b) & \vdash g(y) : Y(a)\\
	 & \vdash a:X & y:Y(a) & \vdash gf(y) \equiv y: Y(a)\\
	 & \vdash b:X & y:Y(b) & \vdash fg(y) \equiv y: Y(b)
\end{align*}
In both cases, a family-valued model specifies a set $\textbf{X}$, a family of sets $\textbf{Y}:\textbf{X} \rightarrow \mathscr U$, and elements $\textbf{a}$, $\textbf{b} \in \textbf{X}$; for $\bbA$, we must have $\textbf{Y}(\textbf{a}) = \textbf{Y}(\textbf{b})$, while for $\bbB$ we choose an isomorphism $\textbf{Y}(\textbf{a}) \cong \textbf{Y}(\textbf{b})$. Let $F$ be the interpretation of $\bbB$ in $\bbA$ induced by
$$
X \longmapsto X \qquad Y(x) \longmapsto Y(x) \qquad a \longmapsto a \qquad b \longmapsto b \qquad f(y) \longmapsto y \qquad g(y) \longmapsto y
$$
It is not difficult to explicitly describe all contexts and the corresponding sorts/terms in both $\bbA$ and $\bbB$, and this allows us to conclude that $F$ is an acyclic fibration. Now, let $\bbA'$, $\bbB'$ be obtained from $\bbA$, $\bbB$, respectively, by adding the axiom $\vdash a \equiv b:X$. This yields a pushout square
\[
\squa{\bbB}{\bbA}{\bbB'}{\bbA'}{[F]}{[F']}{[I]}{[I']}
\]
where $I$, $I'$ are the identity (at the level of expressions) interpretations. Note that $[I]$, $[I']$ are cofibrations. However, $[F']$ is not a weak equivalence:
\begin{itemize}
	\item in $\bbB'$, there is a single term of sort $Y(a)$ in context $y:Y(a)$, namely, $y$;
	
	\item in $\bbA'$ there are, up to derivable equality, countably many terms of sort $Y(a)$ in context $y:Y(a)$, namely,
	$$
	f^n(y) \text{ for } n \in \mathbb Z
	$$
	where $f^n$ denotes the $n$-fold iterate of $f$ (resp. $g$) for positive (resp. negative) $n$.
\end{itemize}
\end{remark}

\subsection{$\Cont$ as a monoidal and $\Cat$-enriched model category}
\label{subsec: Cont is a monoidal and Cat-enriched model category}

We will now prove that the model structure constructed above and the monoidal structure from \cite{Alm25, Alm26} are compatible in the sense that they realize $\Cont$ as a monoidal model category (\cite{Hov98}; see Definition \ref{def: monoidal model category}). After that, we check that the embedding $\Cat \rightarrow \Cont$ is a left Quillen functor, where $\Cat$ is equipped with the categorical model structure, and conclude from this that $\iiCont$ is a $\Cat$-enriched model category (see, e.g., \cite[Def. A.3.1.5]{Lur09}).

\subsubsection{Compatibility with the monoidal structure}

\begin{definition}
\label{def: pushout-product map}
Let $(C, \otimes, \cdots)$ be a monoidal category that has pushouts. For $f:a \rightarrow b$ and $g:a' \rightarrow b'$ in $C$, the \emph{pushout-product} morphism $f \hatotimes g$ is defined as the unique dashed arrow making the following diagram commute:
% https://q.uiver.app/#q=WzAsNSxbMCwwLCJhIFxcb3RpbWVzIGIiXSxbMiwwLCJhJyBcXG90aW1lcyBiIl0sWzAsMiwiYSBcXG90aW1lcyBiJyJdLFsyLDIsImEgXFxvdGltZXMgYicgXFxzcWN1cF97YSBcXG90aW1lcyBifSBhJyBcXG90aW1lcyBiIl0sWzMsMywiYScgXFxvdGltZXMgYiciXSxbMSwzXSxbMiwzXSxbMCwyLCJpZCBcXG90aW1lcyBnIiwyXSxbMCwxLCJmIFxcb3RpbWVzIGlkIl0sWzIsNCwiZiBcXG90aW1lcyBpZCIsMix7ImN1cnZlIjozfV0sWzEsNCwiaWQgXFxvdGltZXMgZyIsMCx7ImN1cnZlIjotM31dLFszLDQsIiIsMCx7InN0eWxlIjp7ImJvZHkiOnsibmFtZSI6ImRhc2hlZCJ9fX1dLFszLDAsIiIsMCx7InN0eWxlIjp7Im5hbWUiOiJjb3JuZXIifX1dXQ==
\[\begin{tikzcd}[ampersand replacement=\&,row sep=small, column sep=small]
	{a \otimes b} \&\& {a' \otimes b} \& \\
	\\
	{a \otimes b'} \&\& {a \otimes b' \sqcup_{a \otimes b} a' \otimes b} \\
	\&\&\& {a' \otimes b'}
	\arrow["{f \otimes id}", from=1-1, to=1-3]
	\arrow["{id \otimes g}"', from=1-1, to=3-1]
	\arrow[from=1-3, to=3-3]
	\arrow["{id \otimes g}", curve={height=-18pt}, from=1-3, to=4-4]
	\arrow[from=3-1, to=3-3]
	\arrow["{f \otimes id}"', curve={height=18pt}, from=3-1, to=4-4]
	\arrow["\lrcorner"{anchor=center, pos=0.125, rotate=180}, draw=none, from=3-3, to=1-1]
	\arrow[dashed, from=3-3, to=4-4]
\end{tikzcd}\]
\end{definition}

\begin{definition}[\cite{Hov98}, Def. 1.2]
\label{def: monoidal model category}
A \emph{(symmetric) monoidal model category} is a model category $C$ equipped with a closed (symmetric) monoidal structure $(\otimes,\textbf{1},\cdots)$ on its underlying category such that
\begin{enumerate}[label=(\roman*)]
	\item For cofibrations $f:a \rightarrow a'$ and $g:b \rightarrow b'$, the pushout-product map
	$$
	f \hatotimes g: a \otimes b' \sqcup_{a \otimes b} a' \otimes b \longrightarrow a' \otimes b'
	$$
	is a cofibration.
	
	\item For cofibrations $f$, $g$, if $f$ or $g$ is a weak equivalence, then so is $f \hatotimes g$.
	
	\item For every cofibrant object $a$ and every weak equivalence $f:b \rightarrow \textbf{1}$ with $b$ cofibrant, $id \otimes f:a \otimes b \rightarrow a \otimes \textbf{1}$ and $f \otimes id:b \otimes a \rightarrow \textbf{1} \otimes a$ are weak equivalences.
\end{enumerate}
\end{definition}

\begin{remark}
\label{rem: conditions for monoidal model cat}
\leavevmode
\begin{enumerate}[label=(\alph*)]
	\item Let $C$ be a combinatorial model category equipped with a closed symmetric monoidal structure on its underlying category. Let $I$ and $J$ be generating sets of cofibrations and of generating cofibrations, respectively. To verify condition (i) above, it suffices to check that $f \hatotimes g$ is a cofibration when $f$, $g \in I$. Similarly, it suffices to check (ii) when one of $f$, $g$ is in $I$ and the other is in $J$. See \cite[Lem. 4.2.4]{Hov99}.
	
	\item Assuming (i) and (ii), condition (iii) automatically holds if $\textbf{1}$ is cofibrant. Indeed, (i), (ii) imply that if $a$ is cofibrant, then $a \otimes -$ and $- \otimes a$ are left Quillen functors; by Ken Brown's lemma (\cite[Lem. 1.1.12]{Hov99}), they preserve weak equivalences between cofibrant objects.
\end{enumerate}
\end{remark}

\begin{remark}
\label{rem: Quillen bifunctor}
For Quillen model categories $C$, $D$ and $E$, a functor $F:C \times D \rightarrow E$ is said to be a (left) Quillen bifunctor if (1) it preserves colimits in each variable, and (2) for cofibrations $f:a \rightarrow a'$ in $C$ and $g:b \rightarrow b'$ in $D$, the induced map
$$
F(a,b') \sqcup_{F(a,b)} F(a',b) \longrightarrow F(a',b')
$$
is a cofibration in $E$ which, moreover, it is a weak equivalence if $f$ or $g$ is a weak equivalence.

Thus conditions (i) and (ii) from Definition \ref{def: monoidal model category} (plus the fact that $\otimes:C \times C \rightarrow C$ preserves colimits in each variable, which follows from the monoidal structure being closed) state that $\otimes$ is a Quillen bifunctor.
\end{remark}

\begin{lemma}
\label{lem: tensoring acyclic cofibration with an object}
In $\Cont$, consider an acyclic cofibration $F:\mathcal A \rightarrow \mathcal A'$ and an object $\mathcal B$. Then $F \otimes \Id_\mathcal B:\mathcal A \otimes \mathcal B \rightarrow \mathcal A' \otimes \mathcal B$ is an inclusion of a strong deformation retract.
\end{lemma}

\begin{proof}
As $\mathcal A$ is fibrant, a solution to the lifting problem (where $\mathcal T$ is the terminal object)
% https://q.uiver.app/#q=WzAsNCxbMCwwLCJcXG1hdGhjYWwgQSJdLFswLDEsIlxcbWF0aGNhbCBBJyJdLFsxLDAsIlxcbWF0aGNhbCBBIl0sWzEsMSwiXFxtYXRoY2FsIFQiXSxbMCwyLCJJZF9cXG1hdGhjYWwgQSJdLFsyLDMsIiEiXSxbMCwxLCJGIiwyXSxbMSwzLCIhIiwyXSxbMSwyLCIiLDEseyJzdHlsZSI6eyJib2R5Ijp7Im5hbWUiOiJkYXNoZWQifX19XV0=
\[\begin{tikzcd}[ampersand replacement=\&]
	{\mathcal A} \& {\mathcal A} \\
	{\mathcal A'} \& {\mathcal T}
	\arrow["{\Id_\mathcal A}", from=1-1, to=1-2]
	\arrow["F"', from=1-1, to=2-1]
	\arrow["{!}", from=1-2, to=2-2]
	\arrow[dashed, from=2-1, to=1-2]
	\arrow["{!}"', from=2-1, to=2-2]
\end{tikzcd}\]
yields a retraction $G$ of $F$. Since the underlying functor of $F$ is an equivalence of categories, the same is the case for $G$, and, moreover, $F$ is an equivalence in the strict $2$-category $\iiCont$ with $G$ as a quasi-inverse. But by \cite[\S10]{Alm26}, the monoidal structure on $\Cont$ lifts to a $\Cat$-enriched monoidal structure on $\iiCont$. It follows that $F \otimes \Id_\mathcal B:\mathcal A \otimes \mathcal B \rightarrow \mathcal A' \otimes \mathcal B$ is an equivalence in $\iiCont$, thus, by Proposition \ref{prop: equivalence in iiCont is weak equivalence}, a weak equivalence in $\Cont$.
\end{proof}

\begin{proposition}
\label{prop: Cont is a monoidal model cat}
The closed symmetric monoidal structure $(\otimes,\mathcal O_1,\cdots)$ from \cite{Alm25, Alm26} realizes the model category $\Cont$ from $\S\ref{sec: model structure}$ as a monoidal model category.
\end{proposition}

\begin{proof}
By Remark \ref{rem: conditions for monoidal model cat}, it suffices to prove that
\begin{itemize}
	\item if $F$ and $G$ are basic cofibrations, then so is $F \hatotimes G$;
	
	\item if $F$, $G$ are cofibrations and one of them is acyclic, then $F \hatotimes G$ is a weak equivalence.
\end{itemize}
Condition (iii) from Definition \ref{def: monoidal model category} will follow from the unit $\mathcal O_1$ being cofibrant.

\vspace{0.5em}

If $F$, $G$ are basic cofibrations, $F \hatotimes G$ being a cofibration follows from \cite[\S8]{Alm26}. More precisely, for $m$, $n \ge 1$:
\begin{itemize}
	\item[-] $\iota_m^S \hatotimes \iota_n^S$ is a pushout of $\iota_{mn - 1}^S$.
	
	\item[-] $\iota_m^S \hatotimes \iota_n^T$ and $\iota_n^T \hatotimes \iota_m^S$ are pushouts of $\iota_{m(n-1)}^T$.
	
	\item[-] $\iota_m^S \hatotimes \pi_n^T$ and $\pi_n^T \hatotimes \iota_m^S$ are pushouts of $\pi_{m(n-1)}^T$.
	
	\item[-] $\iota_m^T \hatotimes \iota_n^T$ is a pushout of $\pi_{(m-1)(n-1)}^T$.
	
	\item[-] $\iota_m^T \hatotimes \pi_n^T$ and $\pi_n^T \hatotimes \iota_m^T$ are isomorphisms.
	
	\item[-] $\pi_m^T \hatotimes \pi_n^T$ is an isomorphism.
\end{itemize}

Now, let $F:\mathcal A \rightarrow \mathcal A'$ be an acyclic cofibration and $G:\mathcal B \rightarrow \mathcal B'$ a cofibration, and let us verify that $F \hatotimes G$ is a weak equivalence. It will then follow that $G \hatotimes F$, being isomorphic to $F \hatotimes G$, is also a weak equivalence. Consider the diagram
% https://q.uiver.app/#q=WzAsNSxbMCwwLCJcXG1hdGhjYWwgQSBcXG90aW1lcyBcXG1hdGhjYWwgQiJdLFsyLDAsIlxcbWF0aGNhbCBBJyBcXG90aW1lcyBcXG1hdGhjYWwgQiJdLFswLDIsIlxcbWF0aGNhbCBBIFxcb3RpbWVzIFxcbWF0aGNhbCBCJyJdLFsyLDIsIlxcbWF0aGNhbCBBIFxcb3RpbWVzIFxcbWF0aGNhbCBCJyBcXHNxY3VwX3tcXG1hdGhjYWwgQSBcXG90aW1lcyBcXG1hdGhjYWwgQn0gXFxtYXRoY2FsIEEnIFxcb3RpbWVzIFxcbWF0aGNhbCBCIl0sWzMsMywiXFxtYXRoY2FsIEEnIFxcb3RpbWVzIFxcbWF0aGNhbCBCJyJdLFsxLDNdLFsyLDNdLFswLDIsIlxcSWRfXFxtYXRoY2FsIEEgXFxvdGltZXMgRyIsMl0sWzAsMSwiRiBcXG90aW1lcyBcXElkX1xcbWF0aGNhbCBCIl0sWzIsNCwiRiBcXG90aW1lcyBcXElkX3tcXG1hdGhjYWwgQid9IiwyLHsiY3VydmUiOjN9XSxbMSw0LCJcXElkX3tcXG1hdGhjYWwgQSd9IFxcb3RpbWVzIEciLDAseyJjdXJ2ZSI6LTN9XSxbMyw0LCJGIFxcd2lkZWhhdHtcXG90aW1lc30gRyIsMV0sWzMsMCwiIiwwLHsic3R5bGUiOnsibmFtZSI6ImNvcm5lciJ9fV1d
\[\begin{tikzcd}[row sep=small, column sep=small]
	{\mathcal A \otimes \mathcal B} && {\mathcal A' \otimes \mathcal B} & \\
	\\
	{\mathcal A \otimes \mathcal B'} && {\mathcal A \otimes \mathcal B' \sqcup_{\mathcal A \otimes \mathcal B} \mathcal A' \otimes \mathcal B} \\
	&&& {\mathcal A' \otimes \mathcal B'}
	\arrow["{F \otimes \Id_\mathcal B}", from=1-1, to=1-3]
	\arrow["{\Id_\mathcal A \otimes G}"', from=1-1, to=3-1]
	\arrow[from=1-3, to=3-3]
	\arrow["{\Id_{\mathcal A'} \otimes G}", curve={height=-18pt}, from=1-3, to=4-4]
	\arrow["I",from=3-1, to=3-3]
	\arrow["{F \otimes \Id_{\mathcal B'}}"', curve={height=18pt}, from=3-1, to=4-4]
	\arrow["\lrcorner"{anchor=center, pos=0.125, rotate=180}, draw=none, from=3-3, to=1-1]
	\arrow["{F \hatotimes G}"{description}, from=3-3, to=4-4]
\end{tikzcd}\]
By Lemma \ref{lem: tensoring acyclic cofibration with an object}, $F \otimes \Id_\mathcal B$ and $F \otimes \Id_{\mathcal B'}$ are inclusions of strong deformation retracts. By Proposition \ref{prop: strong deformation retract pushout}, so is $I$ since it is a pushout of an inclusion of a strong deformation retract. By Remark \ref{rem: strong deformation retract -> weak equivalence}, $F \otimes \Id_{\mathcal B'}$ and $I$ are weak equivalences. As weak equivalences have the $2$-out-of-$3$ property, so is $F \hatotimes G$.
\end{proof}

We also observe that the monoidal model category $\Cont$ satisfies the \emph{monoid axiom} (\cite[Def. 3.3]{SchShi00}), a condition that implies that the model structure can be lifted to categories of modules or algebras for a monoid (\cite[Th. 4.1]{SchShi00}).

\begin{proposition}
\label{prop: monoid axiom for Cont}
In $\Cont$, if an arrow can be expressed as a transfinite composite of pushouts of elements of
$$
\{F \otimes \Id_\mathcal B: \mathcal A \otimes \mathcal B \rightarrow \mathcal A' \otimes \mathcal B \mid F:\mathcal A \rightarrow \mathcal A' \text{ is an acyclic cofibration,}\; \mathcal B \in \Cont\},
$$
then it is a weak equivalence.
\end{proposition}

\begin{proof}
By Lemma \ref{lem: tensoring acyclic cofibration with an object} and Remark \ref{rem: strong deformation retract -> weak equivalence}, $F \otimes \Id_\mathcal B$ is a weak equivalence whenever $F$ is an acyclic cofibration. Also, by Proposition \ref{prop: weak equivalences - transfinite composition, retracts, 2 of 3}, weak equivalences are closed under transfinite composition.
\end{proof}

\subsubsection{Compatibility with the $\Cat$-enrichment}

In \cite[\S10.1]{Alm26}, we studied the functor $\iota:\Cat \rightarrow \Cont$ that sends a small category $A$ to the contextual category freely generated by $A$ where each object of the latter is regarded as a length-$1$ object (a ``constant sort"). Equivalently, $\iota$ is left adjoint to the functor $\Ob_1:\Cont \rightarrow \Cat$ that sends $\mathcal C$ to the full subcategory of $|\mathcal C|$ spanned by the length-$1$ objects. Syntactically, we can describe $\iota(A)$ as $\mathcal C(A_{gat})$ where $A_{gat}$ is the theory given by
\begin{itemize}
	\item for each $a \in \Ob(A)$, a sort symbol $\underline{a}$ introduced by $\vdash \underline{a} \tp$
	
	\item for each arrow $f:a \rightarrow b$ in $A$, a term symbol $\underline{f}$ introduced by $x:\underline{a} \vdash \underline{f}(x): \underline{b}$
	
	\item for composable arrows $a \xrightarrow{f} b \xrightarrow{g} c$, the axiom $x:\underline{a} \vdash \underline{g}(\underline{f}(x)) \equiv \underline{g \circ f}(x): \underline{c}$
	
	\item for each $a \in \Ob(A)$, the axiom $x:\underline{a} \vdash \underline{id_a}(x) \equiv x : \underline{a}$
\end{itemize}

\begin{proposition}
$\iota:\Cat \rightarrow \Cont$ is a left Quillen functor, where $\Cat$ is equipped with the categorical model structure (see, e.g., \cite{Rez96}).
\end{proposition}

\begin{proof}
It suffices to prove that that its right adjoint $\iiOb_1$ is a right Quillen functor, i.e. that it preserves acyclic fibrations and fibrations. Note that $\iiOb_1(\mathcal A)$ is canonically isomorphic to the category $\mathcal A_{1_\mathcal A}$ as in Definition \ref{def: slice length 1}.

Consider a fibration $F:\mathcal A \rightarrow \mathcal B$ in $\Cont$. By Remark \ref{rem: description fibrations in Cont}, we have, in particular that $F_{1_\mathcal A}:\mathcal A_{1_\mathcal A} \rightarrow \mathcal B_{F(1_\mathcal A)} = \mathcal B_{1_\mathcal B}$ is an isofibration; as $F_{1_\mathcal A}$ is isomorphic to $\iiOb_1(F)$, the latter is a fibration in $\Cat$. If, additionally, $F$ is a weak equivalence, then $F_{1_\mathcal A}$ is surjective on objects, hence an acyclic fibration in $\Cat$.
\end{proof}

\begin{definition}
\label{def: enriched model category}
Let $V$ be a monoidal model category (\ref{def: monoidal model category}) and $C$ a tensored and powered $V$-enriched category whose underlying category is (complete, cocomplete, and) equipped with a Quillen model structure. We say that $C$ is a \emph{$V$-model category} if the $V$-tensor functor $V \times C \rightarrow C$ is a left Quillen functor.
\end{definition}

By \cite[Rem. 10.4]{Alm26}, for $A \in \Cat$, $\mathcal C \in \Cont$, the $\Cat$-tensor $A \otimes \mathcal C$ is canonically isomorphic to $\iota(A) \otimes \mathcal C$. As a consequence, the tensor action
$$
\Cat \times \Cont \longrightarrow \Cont
$$
is isomorphic to the composite
$$
\Cat \times \Cont \xrightarrow{\iota \times \Id_{\Cont}} \Cont \times \Cont \overset{\otimes}{\longrightarrow} \Cont.
$$
Since $\iota:\Cat \rightarrow \Cont$ is a left Quillen functor and $\otimes:\Cont \times \Cont \rightarrow \Cont$ is a left Quillen bifunctor (Proposition \ref{prop: Cont is a monoidal model cat}, Remark \ref{rem: Quillen bifunctor}), $\Cat \times \Cont \rightarrow \Cont$ is also a Quillen bifunctor, and we conclude that:

\begin{proposition}
\label{prop: Cat-enrichment of model structure on Cont}
$\Cont$ is a $\Cat$-model category. \qed
\end{proposition}

\begin{remark}
Writing $\Ho(-)$ for the homotopy category of a model category, the above result allows us to equip the homotopy category $\Ho(\Cont)$ with a structure of $\Ho(\Cat)$-enriched category; see \cite[Th. 3.10]{LewMan07}. Explicitly, since every object of $\Cat$ is cofibrant and fibrant, $\Ho(\Cat)$ can be described as the category whose objects are the small categories, and where morphisms $A \rightarrow B$ are isomorphism classes of functors from $A$ to $B$. For $\mathcal A$, $\mathcal B \in \Cont$, the object of morphisms from $\mathcal A$ to $\mathcal B$ in the $\Ho(\Cat)$-enrichment of $\Ho(\Cont)$ is $\iiCont(\mathcal A',\mathcal B)$ where $\mathcal A'$ is a cofibrant replacement of $\mathcal A$ (here we use that $\mathcal B$ is fibrant).
\end{remark}

\subsubsection{Spaces of maps, and the homotopy $(2,1)$-category of $\Cont$}
\label{subsubsec: spaces of maps}

We can use the $\Cat$-enrichment of $\Cont$ to lift it to a $\SSet$-model structure, where $\SSet$ is the category of simplicial sets equipped with the Kan-Quillen model structure (\cite[II, \S3, Th. 3]{Qui67}). This allows us to compute the derived spaces of maps between objects of $\Cont$, or, equivalently, the spaces of maps in the associated $(\infty,1)$-category.\footnote{An argument very similar to ours is used in \cite[Rem. 4.12]{Sha20}.} However, the fact that $\Cat$ presents a $(2,1)$-category will imply that the same holds for $\Cont$.

\begin{remark}[The monoidal Quillen adjunction between $\SSet$ and $\Cat$]
By \cite[Th. 6.1]{Rez96}, we have a Quillen adjunction
% https://q.uiver.app/#q=WzAsMixbMCwwLCJcXFNTZXQiXSxbMiwwLCJcXENhdCJdLFswLDEsIlxcdGV4dHtncGR9IiwyLHsib2Zmc2V0IjoyfV0sWzEsMCwiTl97Y29yZX0iLDIseyJvZmZzZXQiOjJ9XSxbMiwzLCIiLDIseyJsZXZlbCI6MSwic3R5bGUiOnsibmFtZSI6ImFkanVuY3Rpb24ifX1dXQ==
\[\begin{tikzcd}[ampersand replacement=\&]
	\SSet \&\& \Cat
	\arrow[""{name=0, anchor=center, inner sep=0}, "{\Pi_1}"', shift right=2, from=1-1, to=1-3]
	\arrow[""{name=1, anchor=center, inner sep=0}, "{N_{core}}"', shift right=2, from=1-3, to=1-1]
	\arrow["\dashv"{anchor=center, rotate=90}, draw=none, from=0, to=1]
\end{tikzcd}\]
where $\Pi_1(X)$ is the fundamental groupoid of $X$, and $N_{core}(A)$ is the nerve of the core groupoid of $A$. We recall that $\SSet$ and $\Cat$ are cartesian monoidal model categories and, as proved in \cite[II, \S7.5]{GabZis67},\footnote{$\Pi_1$ is denoted there by $\Pi:\Delta^o(\mathscr E) \rightarrow \Cat$.} $\Pi_1$ preserves finite products.
\end{remark}

As explained in \cite[Lem. 1.31]{Bar10}, if $C$ and $D$ are symmetric monoidal model categories such that the unit of $D$ is cofibrant, equipping $D$ with a structure of $C$-model category (Definition \ref{def: enriched model category}) is equivalent to giving a Quillen adjunction
% https://q.uiver.app/#q=WzAsMixbMCwwLCJDIl0sWzIsMCwiRCJdLFswLDEsIkwiLDIseyJvZmZzZXQiOjJ9XSxbMSwwLCJSIiwyLHsib2Zmc2V0IjoyfV0sWzIsMywiIiwyLHsibGV2ZWwiOjEsInN0eWxlIjp7Im5hbWUiOiJhZGp1bmN0aW9uIn19XV0=
\[\begin{tikzcd}[ampersand replacement=\&]
	C \&\& D
	\arrow[""{name=0, anchor=center, inner sep=0}, "L"', shift right=2, from=1-1, to=1-3]
	\arrow[""{name=1, anchor=center, inner sep=0}, "R"', shift right=2, from=1-3, to=1-1]
	\arrow["\dashv"{anchor=center, rotate=90}, draw=none, from=0, to=1]
\end{tikzcd}\]
such that $L$ is strong symmetric monoidal. For such an adjunction, the induced $C$-object of morphisms between $d$, $d' \in D$ is $R(D(d,d'))$, and the tensor between $c \in C$, $d \in D$ is $L(c) \otimes_D D$. Hence, the composite (strong symmetric monoidal Quillen) adjunction
% https://q.uiver.app/#q=WzAsMyxbMCwwLCJcXFNTZXQiXSxbMiwwLCJcXENhdCJdLFs0LDAsIlxcQ29udF97c3RyfSJdLFswLDEsIlxccGkiLDIseyJvZmZzZXQiOjJ9XSxbMSwwLCJOX3tjb3JlfSIsMix7Im9mZnNldCI6Mn1dLFsxLDIsIlxcaW90YSIsMix7Im9mZnNldCI6Mn1dLFsyLDEsIlxcaWlPYl8xIiwyLHsib2Zmc2V0IjoyfV0sWzMsNCwiIiwyLHsibGV2ZWwiOjEsInN0eWxlIjp7Im5hbWUiOiJhZGp1bmN0aW9uIn19XSxbNSw2LCIiLDIseyJsZXZlbCI6MSwic3R5bGUiOnsibmFtZSI6ImFkanVuY3Rpb24ifX1dXQ==
\[\begin{tikzcd}[ampersand replacement=\&]
	\SSet \&\& \Cat \&\& {\Cont}
	\arrow[""{name=0, anchor=center, inner sep=0}, "\Pi_1"', shift right=2, from=1-1, to=1-3]
	\arrow[""{name=1, anchor=center, inner sep=0}, "{N_{core}}"', shift right=2, from=1-3, to=1-1]
	\arrow[""{name=2, anchor=center, inner sep=0}, "\iota"', shift right=2, from=1-3, to=1-5]
	\arrow[""{name=3, anchor=center, inner sep=0}, "{\iiOb_1}"', shift right=2, from=1-5, to=1-3]
	\arrow["\dashv"{anchor=center, rotate=90}, draw=none, from=0, to=1]
	\arrow["\dashv"{anchor=center, rotate=90}, draw=none, from=2, to=3]
\end{tikzcd}\]
realizes $\Cont$ as a $\SSet$-model category. It follows that the derived space of maps between $\mathcal A$, $\mathcal B$ is the core groupoid of $\iiCont(\mathcal A',\mathcal B)$ where $\mathcal A'$ is a cofibrant replacement of $\mathcal A$. We will use this in Theorem \ref{th: homotopy bicategory of Cont} to conclude that the $(\infty,1)$-category $\iiHo(\Cont)$ is the underlying $(2,1)$-category of the $2$-category $\iiDMC_r$ of rooted display map categories.

\section{A characterization of cellular contextual categories}
\label{sec: cellular theories}

We will now find certain conditions characterizing when a contextual category $\mathcal A$ is cellular, that is, when the unique morphism $\mathcal O_0 \rightarrow \mathcal A$ can be expressed as a transfinite composite of basic cofibrations, or, equivalently, when $\mathcal A \cong \mathcal C(\bbA)$ for some gat $\bbA$ without any sort equality axioms.

The main result of section, Theorem \ref{th: cellular iff basis}, is that $\mathcal A$ is cellular if and only it contains a set $\textbf{P}$ of strict display maps such that
\begin{enumerate}[label=(\roman*)]
	\item every strict display map in $\mathcal A$ can be obtained in a unique way as a distinguished pullback of some element of $\textbf{P}$ -- in the terminology of Definition \ref{def: basis}, $\textbf{P}$ is a \emph{basis} of $\mathcal A$;
	
	\item a certain ``dependency" binary relation on $\textbf{P}$ (Definition \ref{def: relation on display maps, well-foundedness}) is well-founded.
\end{enumerate}

Condition (i) is intended as a syntax-independent counterpart to the following property of gats that don't have sort equality axioms, already noted in \cite[Rem. 8.1.9]{Tay99} (see Proposition \ref{prop: cellular gat -> basis}): if $S(\textbf{f})$ and $T(\textbf{g})$ are two sorts in a given context, where $S$, $T$ are sort symbols and $\textbf{f}$, $\textbf{g}$ are context morphisms, then $S(\textbf{f}) \equiv T(\textbf{g})$ is derivable if and only if $S$, $T$ are equal and $\textbf{f} \equiv \textbf{g}$ is derivable. We observe (see Remark \ref{rem: basis and D(A)}) that (i) can be rephrased as the presheaf $\tp:|\mathcal A|^{\op} \rightarrow \Set$ being a coproduct of representables, or also as the category $\textbf{D}(\mathcal A)$ from Construction \ref{constr: categories of display maps} being such that each of its connected components has a terminal object. The latter point of view will be central to our discussion of strictification of weak morphisms in \S\ref{sec: strictification}.

On the other hand, condition (ii) is closely related to the fact that in the absence of sort equality axioms, we can define a ``dependency" binary relation on the set of sort symbols, and, since any derivable judgment can be derived from finitely many axioms, this relation is well-founded.

\vspace{0.5em}

We will start by defining bases of contextual categories and the dependency relation on the set of strict display maps. In \S\ref{subsec: cellularity -> basis}, we prove that every cellular contextual category has a well-founded basis. This follows relatively directly from Cartmell's equivalence between gats and contextual categories, which allows us to keep track of some of the difficult combinatorics involved in describing contexts and morphisms between them.

The converse statement, Proposition \ref{prop: basis -> cellular}, requires several preliminaries. In particular, it relies on the study of full contextual subcategories and their generating sets of display maps (with the well-foundedness condition playing a special role), as well as on the weak factorization system on $\Cont$ whose right class consists of the full-and-faithful morphisms. Full contextual subcategories and well-founded bases are discussed in \S\ref{subsec: full contextual subcategories}. After that, in \S\ref{subsec: basis -> cellularity}, we present further preliminaries (including the aforementioned weak factorization system) and finish proving the second implication in our characterization result.

\begin{definition}
\label{def: basis}
A \emph{basis} for a given contextual category $\mathcal A$ is a set $\textbf{P}$ of strict display maps that has the following property: for every $a \in \mathcal A$ such that $\ell(a) \ge 1$, there exists a unique distinguished square
\[
\dsqua{a}{b}{\partial a}{\partial b}{}{}{\textbf{p}_a}{\textbf{p}_b}
\]
with $\textbf{p}_b \in \textbf{P}$.
\end{definition}

\begin{remark}
\label{rem: basis and D(A)}
Note that a set of strict display maps $\textbf{P}$ is a basis for $\mathcal A$ if and only if each connected component $K$ of the category $\textbf{D}(\mathcal A)$ (see Construction \ref{constr: categories of display maps}) has precisely one object that belongs to $\textbf{P}$, and that object is terminal in $K$. Equivalently, since a presheaf is representable if and only if its category of elements has a terminal object, $\mathcal A$ has a basis if and only if the presheaf $\tp:|\mathcal A|^{\op} \rightarrow \Set$ sending each object to the set of strict display maps over it is a coproduct of representables. More explicitly, given a presheaf isomorphism
$$
\eta:\coprod_{i \in I}\mathcal A(-,a_i) \rightarrow \tp,
$$
the set $\{\eta(id_{a_i}):a'_i \twoheadrightarrow a_i \mid i \in I\}$ is a basis for $\mathcal A$; conversely, any basis determines an isomorphism as above. If $\textbf{P}$ and $\textbf{P}'$ are bases for $\mathcal A$, then there exist a unique bijection $\varphi:\textbf{P} \rightarrow \textbf{P}'$ and unique isomorphisms $p \cong \varphi(p)$ in $\textbf{D}(\mathcal A)$ for each $p \in \textbf{P}$.
\end{remark}

\begin{definition}
\label{def: relation on display maps, well-foundedness}
Let $\mathcal A$ be a contextual category. We define a binary relation $\btle$ on the set of strict display maps of $\mathcal A$ as follows:\footnote{A similar relation is discussed in \cite[Rem. 3.1]{BarHen25}.}$\textbf{p}_b \btle \textbf{p}_a$ if and only if there exist $i \in \{1, ..., \ell(a) - 1\}$ and a distinguished square
\[
\dsqua{\partial_i a}{b}{\partial_{i-1} a}{\partial b.}{f'}{f}{}{}
\]
We say that a subset $\textbf{K}$ of the set of all strict display maps is \emph{well-founded} if every non-empty subset of $\textbf{K}$ has a minimal element with respect to $\btle$. This is equivalent (see \cite[Chap. 14, Lem. 1.3 and 1.4]{HrbJec99}) to $\textbf{K}$ not containing any infinite descending chain $\; \cdots \btle p_n \btle \cdots \btle p_i \btle p_0$. Note that if $\textbf{K}$ is well-founded, then it contains no cycles with respect to $\btle$. In particular, we don't have $p \btle p$ for any $p \in \textbf{K}$.
\end{definition}

\subsection{Cellularity implies existence of a well-founded basis}
\label{subsec: cellularity -> basis}

\begin{definition}
We will say that a generalized algebraic theory is \emph{syntactically cellular} if it has no sort equality axioms.
\end{definition}

This is not a categorical property: a theory that has sort axioms might be isomorphic to one that doesn't. A gat is cellular (with respect to the set $\cof^*$ of basic cofibrations) if and only if it is isomorphic to a syntactically cellular one; see Remark \ref{rem: constructing syntactically cellular gat}.

\begin{notation}
	If $J = (\textbf{X} \vdash U \tp)$ is a derivable sort judgment in a gat $\bbA$, we write $\textbf{p}(J)$ for the display map $[\textbf{X}.U] \twoheadrightarrow [\textbf{X}]$ in $\mathcal C(\bbA)$. Also, we let
	$$
	\textbf{P}(\bbA) = \{\textbf{p}(J) \mid J \text{ is a sort axiom in } \bbA\}.
	$$
\end{notation}

\begin{proposition}
	\label{prop: cellular gat -> basis}
	If a generalized algebraic theory $\bbA$ is cellular, then $\textbf{P}(\bbA)$ is a basis of $\mathcal C(\bbA)$.
\end{proposition}

\begin{proof}
	Consider a display map $\textbf{p}_a:a \rightarrow \partial a$ in $\mathcal C(\bbA)$. By construction, $\textbf{p}_a$ equals $[\textbf{X}.U] \twoheadrightarrow [\textbf{X}]$ for a context $\textbf{X}$ and a derivable judgment
	$$
	\textbf{X} \vdash S(f_1, ..., f_k) \tp
	$$
	in $\bbA$ where $S$ is a sort symbol and $f_1$, ..., $f_k$ are term expressions in the variables $x_1$, ..., $x_n$ of $\textbf{X}$. Letting $\textbf{Y} \vdash S(y_1, ..., y_k) \tp$ be the axiom that introduces $S$, we have a distinguished square
	\[
	\dsqua{{[\textbf{X}.S(f_1, ..., f_k)]}}{{[\textbf{Y}.S(y_1, ..., y_k)]}}{{[\textbf{X}]}}{{[\textbf{Y}].}}{}{{[f_1, ..., f_k]}}{}{}
	\]
	Now, suppose that we have another such distinguished square, say
	\[
	\dsqua{{[\textbf{X}'.T(g_1, ..., g_\ell)]}}{{[\textbf{Y}'.T(z_1, ..., z_\ell)]}}{{[\textbf{X}']}}{{[\textbf{Y}']}}{}{{[g_1, ..., g_\ell]}}{}{}
	\]
	where $\textbf{Y}' \vdash T(x_1, ..., z_\ell) \tp$ is a sort axiom. Then
	$$
	\textbf{X} \equiv \textbf{Y} \ctx, \qquad\qquad \textbf{X} \vdash S(f_1, ..., f_k) \equiv T(g_1, ..., g_\ell) \tp
	$$
	are derivable. Since $\bbA$ has no sort equality axioms, the symbols $S$, $T$ are equal, which implies that the judgments
	$$
	\textbf{Y} \vdash S(y_1, ..., y_k) \tp, \qquad \textbf{Y}' \vdash T(z_1, ..., z_\ell) \tp
	$$
	are syntactically equal, and the equalities $f_i \equiv g_i$ are derivable in $\textbf{X}$ for $1 \le i \le k = \ell$. This implies that the morphisms $[f_1, ..., f_k]$, $[g_1, ..., g_\ell]:[\textbf{X}] \rightarrow [\textbf{Y}]$ are equal.
\end{proof}

\begin{definition}
	\label{def: cellular presentation}
	A \emph{cellular presentation} of a contextual category $\mathcal A$ is a tuple
	$$
	\mathscr C = (\mathcal A_*,\; (I_k,\;(F^i_k, G^i_k, F^{'i}_k)_{i \in I_k})_{k \in \bbN},\; J)
	$$
	consisting of:
	\begin{itemize}
		\item A diagram $\mathcal A_*:(\bbN,\le) \rightarrow \Cont$ such that $\mathcal A_0 = \mathcal O_0$. For $k \le \ell$, let $\iota_{k,\ell}:\mathcal A_k \rightarrow \mathcal A_\ell$ be the corresponding morphism.
		
		\item For each $k \in \bbN$, a set $I_k$ equipped with a diagram of strict morphisms
		\[
		\begin{tikzcd}
			\mathcal B_k^i \arrow[]{r}{F_k^i} \arrow[swap]{d}{G_k^i} & \mathcal A_k \\
			\mathcal B_k^{'i} &
		\end{tikzcd}
		\]
		for each $i \in I_k$ such that the $G_k^i$ are basic cofibrations (Definition \ref{def: cofibrations, acyclic fibrations}), and a pushout diagram
		\[
		\widesqua{\coprod_{i \in I_k}\mathcal B_k^i}{\mathcal A_k}{\coprod_{i \in I_k}\mathcal B_k^{'i}}{\mathcal A_{k+1}.}{(F_k^i)_{i \in I_k}}{(F_k^{'i})_{i \in I_k}}{\coprod_{i \in I_k}G_k^i}{\iota_{k,k+1}}
		\]
		
		\item An isomorphism $J:\colim_{n \in \bbN}\mathcal A_n \longrightarrow \mathcal A$.
	\end{itemize}
\end{definition}

We can associate with a cellular presentation a set of display maps analogous to $\textbf{P}(\bbA)$ for a gat $\bbA$:

\begin{notation}
\label{not: P(C)}
We let $\textbf{P}(\mathscr C)$ (resp. $\textbf{P}_N(\mathscr C)$, where $N \ge 0$) be the set of all display maps in $\mathcal A$ that occur as the image of $o_m \twoheadrightarrow o_{m-1}$ under the composite
$$
\mathcal B^{'i}_k \xrightarrow{F^{'i}_k} \mathcal A_{k+1} \longrightarrow \colim_{n \in \bbN} \mathcal A_n \overset{J}{\longrightarrow} \mathcal A
$$
for some $k$ (resp. $k \le N$) and $i \in I_k$ such that $G^i_k$ is of the form $\iota_m^S:\mathcal O_{m-1} \rightarrow \mathcal O_m$.
\end{notation}

\begin{proposition}
\label{prop: cellular presentation -> basis}
Consider a contextual category $\mathcal A$ equipped with a cellular presentation $\mathscr C$. Then there exist a syntactically cellular gat $\bbA$ and an isomorphism $H:\mathcal A \rightarrow \mathcal C(\bbA)$ whose action on strict display maps defines a bijection $\textbf{P}(\bbA) \cong \textbf{P}(\mathscr C)$. In particular, by Proposition \ref{prop: cellular gat -> basis}, $\mathcal A$ has $\textbf{P}(\mathscr C)$ as a basis.
\end{proposition}

\begin{remark}
\label{rem: constructing syntactically cellular gat}
Before going into the proof, let us discuss the idea behind it and what we consider to be the main takeaway. The idea is to use the fact that $\Cont \simeq \GAT$ preserves colimits (being an equivalence) to recursively translate the above pushouts into introduction of axioms. The desired result will then follow from the fact that pushouts of basic cofibrations encode, as explained in \S\ref{subsec: sorts terms and axioms via contextual categories}, addition of sort, term, and term equality axioms. However, as we will see, there is a non-constructive aspect of this translation that prevents us from using recursion in a straighforward way: for a gat $\bbK$, we don't have a canonical way of lifting an object of $\mathcal C(\bbK)$ to a context, or a section of display map to a term judgment, and so on. We will deal with this my making a choice of such judgments for each gat $\bbK$ belonging to a sufficiently big set.

This can be thought of as a result of the fact that we want $\bbA$ to be syntactically cellular -- which is not a categorical property --, rather than just cellular. In fact, we can alternatively think of the proposition as stating that every cellular gat is isomorphic to a syntactically cellular one.\footnote{Note also that general methods for constructing colimits in $\GAT$ may lead to the introduction of sort equality axioms even when the resulting theory is cellular. This happens, for example, if we pass to $\Cont$, compute the colimit there, and go back via $\Cont \simeq \GAT$ -- in general, the latter sends cellular objects to gats that have sort equality axioms. One kind of colimit in $\GAT$ that computed without introducing "superfluous" sort equalities is the pushout of a diagram $\bbB \leftarrow \bbA \rightarrow \bbA'$ where $\bbA \rightarrow \bbA'$ is an extension that adds a collection of sort, term and term equality axioms, and the morphism $\bbA \rightarrow \bbB$ (which is an equivalence class of interpretations) is presented by a specified interpretation of $\bbA$ in $\bbB$ (\cite{Car86}, \S12). This is precisely what we use in the proof of Proposition \ref{prop: cellular gat -> basis} -- the choice performed at the beginning of the proof has the effect of lifting the appropriate morphisms of gats to interpretations.}
\end{remark}

\begin{proof}{Proof of Proposition \ref{prop: cellular gat -> basis}}
Express $\mathscr C$ as in Definition \ref{def: cellular presentation}. Let $\mathscr S$ be the (small) set of all generalized algebraic theories whose sort and term symbols belong to $\coprod_{\ell \in \bbN}I_\ell$ (where the $I_\ell$ are given by $\mathscr C$); all gats under consideration are assumed to have the same set of variables. For each $\bbS \in \mathscr S$, choose:
\begin{itemize}
	\item for each object $a$ of $\mathcal C(\bbS)$, a context in $\bbS$ whose equivalence class is $a$.
	
	\item for each strict display map $p$ in $\mathcal C(\bbS)$, a derivable sort judgment in $\bbS$ whose associated display map is $p$;
	
	\item for each pair $(s,t)$ of sections of the same strict display map in $\mathcal C(\bbS)$, a (not necessarily derivable) term equality judgment $\textbf{X} \vdash u \equiv v:U$ in $\bbS$ such that $\textbf{X} \vdash u:U$ and $\textbf{X} \vdash v:U$ are derivable, and the sections of display maps corresponding to the latter are $s$ and $t$, respectively.
\end{itemize}

Let $\mathscr T$ be the set of all triples $(k,\bbS,\varphi)$ where $k \in \bbN$, $\bbS \in \mathscr S$, and $\varphi:\mathcal A_k \rightarrow \mathcal C(\bbS)$ is an isomorphism that induces a bijection $\textbf{P}(\bbS) \cong \textbf{P}_k(\mathscr C)$ (see Notation \ref{not: P(C)}). For $(k,\bbS,\varphi) \in \mathscr T$, consider the diagram
	% https://q.uiver.app/#q=WzAsNCxbMCwwLCJcXGNvcHJvZF97aSBcXGluIElfa31cXG1hdGhjYWwgQl9rXmkiXSxbMiwwLCJcXG1hdGhjYWwgQV9rIl0sWzMsMCwiXFxtYXRoY2FsIEMoXFxiYlMpIl0sWzAsMSwiXFxjb3Byb2Rfe2kgXFxpbiBJX2t9IFxcbWF0aGNhbCBCXnsnaX1fayJdLFswLDEsIihGXmlfaylfe2kgXFxpbiBJX2t9Il0sWzEsMiwiXFx2YXJwaGleey0xfSJdLFswLDMsIlxcY29wcm9kX3tpIFxcaW4gSV9rfUdfa15pIiwyXV0=
	\[\begin{tikzcd}[ampersand replacement=\&]
		{\coprod_{i \in I_k}\mathcal B_k^i} \&\& {\mathcal A_k} \& {\mathcal C(\bbS)} \\
		{\coprod_{i \in I_k} \mathcal B^{'i}_k}
		\arrow["{(F^i_k)_{i \in I_k}}", from=1-1, to=1-3]
		\arrow["{\coprod_{i \in I_k}G_k^i}"', from=1-1, to=2-1]
		\arrow["{\varphi}", from=1-3, to=1-4]
	\end{tikzcd}\]
	The maps $\varphi^{-1} \circ F^i_k:\mathcal B_k^i \rightarrow \mathcal C(\bbS)$ for $i \in I_k$ encode the following data in $\mathcal C(\bbS)$:
	\begin{itemize}
		\item a length-$(n-1)$ object if $G_i^k = \iota_n^S:\mathcal O_{n-1} \rightarrow \mathcal O_n$,
		
		\item a length-$n$ object if $G_i^k = \iota_n^T:\mathcal O_n \rightarrow \mathcal O_n^+$,
		
		\item a pair of sections of a display map if $G^i_k = \pi_n^T:\mathcal O_n^{++} \rightarrow \mathcal O_n^+$.
	\end{itemize}
	In each of these cases, we have chosen, respectively, a lift:
	\begin{enumerate}[label=(\roman*)]
		\item of the length-$(n-1)$ object to a context in $\bbS$,
		
		\item of length-$n$ object to a length-$n$ context in $\bbS$,
		
		\item of the pair of sections to a term equality judgment.
	\end{enumerate}
	Let $\bbS'$ be the theory obtained by adding to $\bbS$ the following axioms:
	\begin{itemize}
		\item for each $i \in I_k$ such that $G^k_i = \iota_n^S$, the sort judgment
		$$
		x_1:X_1, ..., x_{n-1}:X_{n-1} \vdash (k,i)(x_1, ..., x_{n-1}) \tp
		$$
		where $x_1:X_1, ..., x_{n-1}:X_{n-1}$ is the context from (i) above;
		
		\item for each $i$ such that $G^k_i = \iota_n^T$, the term judgment
		$$
		x_1:X_1, ..., x_{n-1}:X_{n-1} \vdash (k,i)(x_1, ..., x_{n-1}):X_n
		$$
		where the context from (ii) above is $x_1:X_1, ..., x_{n-1}:X_{n-1},\; x_n:X_n$;
		
		\item for each $i$ such that $G^k_i = \pi_n^T$, the term equality judgment from (iii).
	\end{itemize}
	Each of these determine an extension along $G^i_k$ as in
	\[
	\begin{tikzcd}[ampersand replacement=\&]
		{\mathcal B_k^i} \&\& {\mathcal A_k} \& {\mathcal C(\bbS)} \\
		{\mathcal B^{'i}_k} \&\&\& {\mathcal C(\bbS')}
		\arrow["{F^i_k}", from=1-1, to=1-3]
		\arrow["{G_k^i}"', from=1-1, to=2-1]
		\arrow["\varphi",from=1-3, to=1-4]
		\arrow[from=1-4, to=2-4]
		\arrow[from=2-1, to=2-4]
	\end{tikzcd}\]
	and, by construction of $\bbS'$, the outer square in the induced commutative diagram
	% https://q.uiver.app/#q=WzAsNSxbMCwwLCJcXGNvcHJvZF97aSBcXGluIElfa31cXG1hdGhjYWwgQl9rXmkiXSxbMiwwLCJcXG1hdGhjYWwgQV9rIl0sWzMsMCwiXFxtYXRoY2FsIEMoXFxiYlMpIl0sWzMsMSwiXFxtYXRoY2FsIEMoXFxiYlMnKSJdLFswLDEsIlxcY29wcm9kX3tpIFxcaW4gSV9rfSBcXG1hdGhjYWwgQl57J2l9X2siXSxbMCwxLCIoRl5pX2spX3tpIFxcaW4gSV9rfSJdLFsxLDJdLFsyLDNdLFswLDQsIlxcY29wcm9kX3tpIFxcaW4gSV9rfUdfa15pIiwyXSxbNCwzXV0=
	\[
	\begin{tikzcd}[ampersand replacement=\&]
		{\coprod_{i \in I_k}\mathcal B_k^i} \&\& {\mathcal A_k} \& {\mathcal C(\bbS)} \\
		{\coprod_{i \in I_k} \mathcal B^{'i}_k} \&\&\& {\mathcal C(\bbS')}
		\arrow["{(F^i_k)_{i \in I_k}}", from=1-1, to=1-3]
		\arrow["{\coprod_{i \in I_k}G_k^i}"', from=1-1, to=2-1]
		\arrow["\varphi",from=1-3, to=1-4]
		\arrow[from=1-4, to=2-4]
		\arrow[from=2-1, to=2-4]
	\end{tikzcd}\]
	is cocartesian. Now, the pushout square (given by $\mathscr C$) exhibiting $\mathcal A_{k+1}$ as a pushout of $(F^i_k)_{i \in I_k}$ and $\coprod_{i \in I_k} G^i_k$ yields a strict morphism $\varphi':\mathcal A_{k+1} \rightarrow \mathcal C(\bbS')$ that makes the following diagram commute:
	% https://q.uiver.app/#q=WzAsNixbMCwwLCJcXGNvcHJvZF97aSBcXGluIElfa31cXG1hdGhjYWwgQl9rXmkiXSxbMiwwLCJcXG1hdGhjYWwgQV9rIl0sWzMsMCwiXFxtYXRoY2FsIEMoXFxiYlMpIl0sWzMsMSwiXFxtYXRoY2FsIEMoXFxiYlMnKSJdLFswLDEsIlxcY29wcm9kX3tpIFxcaW4gSV9rfSBcXG1hdGhjYWwgQl57J2l9X2siXSxbMiwxLCJcXG1hdGhjYWwgQV97aysxfSJdLFswLDEsIihGXmlfaylfe2kgXFxpbiBJX2t9Il0sWzEsMiwiXFx2YXJwaGkiXSxbMiwzXSxbMCw0LCJcXGNvcHJvZF97aSBcXGluIElfa31HX2teaSIsMl0sWzQsNSwiKEZeeydpfV9rKV97aSBcXGluIElfa30iLDJdLFs1LDMsIlxcdmFycGhpJyIsMix7InN0eWxlIjp7ImJvZHkiOnsibmFtZSI6ImRhc2hlZCJ9fX1dLFsxLDUsIlxcaW90YV97ayxrKzF9Il1d
	\[\begin{tikzcd}[ampersand replacement=\&]
		{\coprod_{i \in I_k}\mathcal B_k^i} \&\& {\mathcal A_k} \& {\mathcal C(\bbS)} \\
		{\coprod_{i \in I_k} \mathcal B^{'i}_k} \&\& {\mathcal A_{k+1}} \& {\mathcal C(\bbS').}
		\arrow["{(F^i_k)_{i \in I_k}}", from=1-1, to=1-3]
		\arrow["{\coprod_{i \in I_k}G_k^i}"', from=1-1, to=2-1]
		\arrow["\varphi", from=1-3, to=1-4]
		\arrow["{\iota_{k,k+1}}", from=1-3, to=2-3]
		\arrow[from=1-4, to=2-4]
		\arrow["{(F^{'i}_k)_{i \in I_k}}"', from=2-1, to=2-3]
		\arrow["{\varphi'}"', dashed, from=2-3, to=2-4]
	\end{tikzcd}\]
	Since the left square and the outer composite ones are cocartesian, so is the right one. In particular, as $\varphi$ is an isomorphism, so is $\varphi'$. Note that $\varphi$ defines (by assumption) a bijection $\textbf{P}_k(\mathscr C) \cong \textbf{P}(\bbS)$ and, on the other hand, $\varphi'$ defines (by construction) a bijection between $\textbf{P}(\mathcal A_{k+1}) \setminus \iota_{k,k+1}(\textbf{P}(\mathcal A_k))$ and the set of display maps associated with the sort axioms added via $\bbS \hookrightarrow \bbS'$. It follows that $\varphi'$ defines a bijection
	\begin{align*}
		\textbf{P}(\mathcal A_{k+1}) & = \textbf{P}(\mathcal A_k) \cup \big( \textbf{P}(\mathcal A_{k+1}) \setminus \textbf{P}(\mathcal A_k) \big)\\
		& \cong \{\textbf{p}(J) \mid J \text{ is a sort axiom in } \bbS\} \cup \{\textbf{p}(J) \mid J \text{ is a sort axiom added via } \bbS \hookrightarrow \bbS'\}\\
		& = \textbf{P}(\bbS').
	\end{align*}
	We conclude that $(k+1,\bbS',\varphi') \in \mathscr T$.
	
	Now, the function $\varepsilon:\mathscr T \rightarrow \mathscr T$ given, in terms of the above construction, by $(k,\bbS,\varphi) \mapsto (k+1, \bbS',\varphi')$ allows us to obtain, by recursion, a sequence in $\mathscr T$ of the form $(k,\bbS_k,\varphi_k)_{k \in \bbN}$ such that $\bbS_0$ is the empty theory and $(k+1,\bbS_{k+1},\varphi_{k+1}) = \varepsilon(k,\bbS_k,\varphi_k)$ for all $k$. By definition of $\varepsilon$, the family $(\varphi_k)_{k \in \bbN}$ is a natural isomorphism $\mathcal A_* \cong \mathcal C(\bbS_*)$ between functors $(\bbN,\le) \longrightarrow \Cont$. Taking colimits, we obtain an isomorphism
	\[
	\tag{\texttt{*}}
	\mathcal A \overset{J}{\cong} \colim_{k \in \bbN}\mathcal A_k \overset{\colim_{k \in \bbN}\varphi_k}{\cong} \colim_{k \in \bbN} \mathcal C(\bbS_k) \cong \mathcal C(\bigcup_{k \in \bbN}\bbS_k)
	\]
	which, moreover, induces a bijection
	$$
	\textbf{P}(\mathscr C) \cong \colim_{k \in \bbN} \textbf{P}_k(\mathscr C) \cong \colim_{k \in \bbN} \textbf{P}(\bbS_k) \cong \textbf{P}(\bigcup_{k \in \bbN} \bbS_k).
	$$
	Thus we can take, in the notation of the statement, $\bbA = \bigcup_{k \in \bbN}\bbS_k$ and $H$ as the isomorphism in (\texttt{*}).
\end{proof}

\begin{remark}
\label{rem: necessary condition cofibrant - connected components of D(A)}
It follows that if $\mathcal A \in \Cont$ is cofibrant, then $\textbf{D}(\mathcal A) \cong \coprod_{i \in I} D_i$ where each $D_i$ is a retract of a category that has a terminal object; in particular, there exists a cocone under the identity functor $D_i \rightarrow D_i$.
\end{remark}

\subsection{Full contextual subcategories and their generating sets of display maps}
\label{subsec: full contextual subcategories}

\begin{definition}
A \emph{full contextual subcategory} of a contextual category $\mathcal A$ is a full subcategory $\mathcal B \subset |\mathcal A|$ that admits that admits a structure of contextual category for which the inclusion $\mathcal B \rightarrow \mathcal A$ is a strict morphism.
\end{definition}

\begin{remark}
	If such a contextual category structure exists, then it is unique: an arrow (resp. commutative square) in $\mathcal B$ is a strict display map (resp. distinguished square) precisely when it is a display map (resp. distinguished square) in $\mathcal A$. This implies that a full subcategory $\mathcal B \subset \mathcal A$ is a full contextual subcategory if and only if the following hold:
	\begin{enumerate}[label=(\roman*)]
		\item $1_\mathcal A \in \mathcal B$, and if $a \in \mathcal B$ with $\ell(a) \ge 1$, then $\partial a \in \mathcal B$;
		
		\item if a distinguished square
		\[
		\dsqua{a}{b}{\partial a}{\partial b}{f'}{f}{\textbf{p}_a}{\textbf{p}_b}
		\]
		in $\mathcal A$ is such that $\partial a$, $\partial b$, $b \in \mathcal B$, then $a \in \mathcal B$.
	\end{enumerate}
\end{remark}

\begin{remark}
\label{rem: intersection of full contextual subcategories}
It is straightforward to check that the intersection of any set of full contextual subcategories of $\mathcal A$ is a full contextual subcategory.
\end{remark}

\begin{definition}
\label{def: A_K}
Let $\mathcal A$ be a contextual category. Given a set $\textbf{K}$ strict display maps, we let $\mathcal A_\textbf{K}$ be the intersection of all full contextual subcategories $\mathcal B \subset \mathcal A$ that have the following property: if a distinguished square
\[
\dsqua{a}{b}{\partial a}{\partial b}{f'}{f}{\textbf{p}_a}{\textbf{p}_b}
\]
is such that $\partial a$, $\partial b \in \mathcal B$ and $\textbf{p}_b \in \textbf{K}$, then $a \in \mathcal B$. Note that $\mathcal A_\textbf{K}$ is in fact that smallest full contextual subcategory of $\mathcal A$ that has this property.
\end{definition}

\begin{remark}
It is not always the case that $\textbf{K}$ is contained in $\mathcal A_\textbf{K}$. This happens precisely when for all $\textbf{p}_a$ in $\textbf{K}$ we have $\partial a \in A_\textbf{K}$. But generally, this condition must be verified in a recursive way, that is, we need to prove that $\partial a \twoheadrightarrow \partial \partial a$ is in $\mathcal A_\textbf{K}$, and so on.

For example, consider the generalized algebraic theory $\bbA$ given by
$$
\vdash X \tp, \qquad x:X \vdash Y(x) \tp.
$$
In $\mathcal C(\bbA)$, let $\textbf{K} = \{[x:X,y:Y(x)] \twoheadrightarrow [x:X]\}$. We have $\mathcal C(\bbA)_\textbf{K} = \{1_{\mathcal C(\bbA)}\}$ since the fact that $[x:X] \twoheadrightarrow 1_{\mathcal C(\bbA)}$ is not in $\textbf{K}$ prevents us from obtaining any non-trivial objects via the procedure in Definition \ref{def: A_K}
\end{remark}

\begin{definition}
\label{def: admissible set of display maps}
We say that a set of display maps $\textbf{K}$ in $\mathcal A$ is \emph{admissible} if $\textbf{K} \subset \Ar(\mathcal A_\textbf{K})$.
\end{definition}

In any $\mathcal A$, the set of all display maps is admissible. But more interestingly:

\begin{lemma}
\label{lem: characterization admissibility of basis}
If $\textbf{P}$ is a basis of $\mathcal A$, then the following conditions are equivalent:
\begin{enumerate}[label=(\alph*)]
	\item $\mathcal A_\textbf{P} = \mathcal A$.
	
	\item $\textbf{P}$ is admissible.
	
	\item $\textbf{P}$ is well-founded.
\end{enumerate}
\end{lemma}

\begin{proof}
If $\mathcal A = \mathcal A_\textbf{P}$, it is clear that $\textbf{P}$ is admissible. Conversely, suppose that $\textbf{P}$ is admissible; in particular, if $\textbf{p}_b \in \textbf{P}$, then $\partial b \in \mathcal A_\textbf{P}$. For a display map $p_a:a \twoheadrightarrow \partial a$, expressing it as a distinguished pullback of an element of $\textbf{P}$ shows that if $\partial a \in \mathcal A_\textbf{P}$, then $a \in \mathcal A$. Since $1_\mathcal A \in \mathcal A_\textbf{P}$, it follows by induction that $\mathcal A = \mathcal A_\textbf{P}$.
	
It remains to verify that (c) is equivalent to (a), (b).

\vspace{0.5em}

Suppose, for contradiction, that $\textbf{P}$ is well-founded but not admissible. Let $\textbf{K} = \{p \in \textbf{P} \mid p \notin \Ar(\mathcal A_\textbf{P})\}$, which, by assumption, is non-empty. Since $\textbf{P}$ is well-founded with respect to $\btle$ (Definition \ref{def: relation on display maps, well-foundedness}), $\textbf{K}$ has an element, say $\textbf{p}_a:a \twoheadrightarrow \partial a$, that is minimal for that relation.

We claim that $\partial a \in \mathcal A_\textbf{P}$. If that were not the case, there would exist $i \in \{1, ..., \ell(a) - 1\}$ such that $\partial_{i-1}a \in \mathcal A_\textbf{P}$ and $\partial_i a \notin \mathcal A_\textbf{P}$. Then the unique distinguished square
\[
\dsqua{\partial_i a}{b}{\partial_{i-1} a}{
\partial b}{}{}{\textbf{p}_{\partial_i a}}{\textbf{p}_b}
\]
with $\textbf{p}_b \in \textbf{P}$ must be such that $\textbf{p}_b \notin \Ar(\mathcal A_\textbf{P})$. But then $\textbf{p}_b \btle \textbf{p}_a$ and $\textbf{p}_b \in \textbf{K}$, which contradicts minimality of $\textbf{p}_a$. It follows that $\partial a \notin \mathcal A_\textbf{P}$ and, as a consequence, $\textbf{p}_a \in \Ar(\mathcal A_\textbf{P})$. This, in turn, contradicts the assumption that $\textbf{p}_a \in \textbf{K}$. We conclude that if $\textbf{P}$ is well-founded, then it is admissible.

\vspace{0.5em}

Assume, on the other hand, that $\textbf{P}$ is not well-founded, and consider a non-empty subset $\textbf{K} \subset \textbf{P}$ that does not have a minimal element with respect to $\btle$. We will prove that none of the elements of $\textbf{K}$ belong to $\Ar(\mathcal A_\textbf{P})$. For that, construct a sequence
$$
A_0 \subset A_1 \subset \cdots \subset A_k \subset \cdots
$$
of subsets of $\Ob(\mathcal A)$ recursively as follows:
\begin{itemize}
	\item $A_0 = \{1_\mathcal A\}$;
	
	\item for $k \ge 1$, $A_k$ is the union of $A_{k-1}$ and the set of all $a \in \Ob(\mathcal A)$ of length $\ge 1$ such that $\partial a \in A_{k-1}$ and there exists a distinguished square
	\[
	\dsqua{a}{b}{\partial a}{\partial b}{}{}{\textbf{p}_a}{\textbf{p}_b}
	\]
	with $\textbf{p}_b \in \textbf{P}$ and $\partial b \in A_{k-1}$.
\end{itemize}
Note that $|\mathcal A_\textbf{P}|$ is the full subcategory of $|\mathcal A|$ spanned by $\bigcup_{k \ge 0}A_k$.

Suppose, for contradiction, that $\textbf{K} \cap \Ar(\mathcal A_\textbf{P}) \neq \varnothing$, and let $m$ be the smallest $k$ (which is necessarily $\ge 1$) such that $c \in A_k$ for some $p:c \twoheadrightarrow \partial c$ in $\textbf{K}$. Fixing one such $p$, for each $i \in \{1, ..., \ell(c) - 1\}$ we have that $\partial_i c \in A_{m-1}$ and there exists a unique distinguished square, say
\[
\dsqua{\partial_i c}{b_i}{\partial_{i-1} c}{\partial b_i,}{}{}{\textbf{p}_{\partial_i c}}{\textbf{p}_{b_i}}
\]
such that $b_i \in \textbf{P}$; note that $\partial b_i \in A_{m-1}$. Now, using that $\textbf{K}$ does not have a minimal element, choose $q \in \textbf{K}$ such that $q \btle p$. Thus $q = \textbf{p}_{b_i}$ for some $i$. But then $\dom(q) = b_i \in A_{m-1}$, which contradicts the minimality assumption on $m$.

This concludes the proof that if $\textbf{P}$ is not well-founded, then it is not admissible.
\end{proof}

\begin{example}[A non-well-founded basis]
Let $\bbA$ be the given by the following axioms for all $n \ge \mathbb Z$:
$$
\vdash X_n \tp \qquad x:X_n \vdash Y_n(x) \tp \qquad \vdash a_n : X_n \qquad \vdash X_{n+1} \equiv Y_n(a_n) \tp
$$
Writing $[X_n]:= [x:X_n]$ and $[Y_n] = [x:X_n,y:Y_n]$, we can picture the corresponding generating data of $\mathcal C(\bbA)$ as
% https://q.uiver.app/#q=WzAsMTgsWzIsMSwiW1hfbl0iXSxbMywxLCJbWF97bisxfV0iXSxbMSwxLCJbWF97bi0xfV0iXSxbMCwxLCJcXGNkb3RzIl0sWzQsMSwiW1hfe24rMn1dIl0sWzEsMCwiW1lfe24tMX1dIl0sWzIsMCwiW1lfbl0iXSxbMywwLCJbWV97bisxfV0iXSxbNCwwLCJbWV97bisyfV0iXSxbMCwwLCJcXGNkb3RzIl0sWzIsMiwiMV97XFxtYXRoY2FsIEMoXFxiYkEpfSJdLFszLDIsIjFfe1xcbWF0aGNhbCBDKFxcYmJBKX0iXSxbNCwyLCIxX3tcXG1hdGhjYWwgQyhcXGJiQSl9Il0sWzUsMiwiXFxjZG90cyJdLFs1LDEsIlxcY2RvdHMiXSxbNSwwLCJcXGNkb3RzIl0sWzEsMiwiMV97XFxtYXRoY2FsIEMoXFxiYkEpfSJdLFswLDIsIlxcY2RvdHMiXSxbMCw1XSxbNSwyLCIiLDEseyJzdHlsZSI6eyJoZWFkIjp7Im5hbWUiOiJlcGkifX19XSxbMTAsMiwiYV97bi0xfSIsMV0sWzExLDAsImFfbiIsMV0sWzEyLDEsImFfe24rMX0iLDFdLFswLDEwLCIiLDEseyJzdHlsZSI6eyJoZWFkIjp7Im5hbWUiOiJlcGkifX19XSxbMSwxMSwiIiwxLHsic3R5bGUiOnsiaGVhZCI6eyJuYW1lIjoiZXBpIn19fV0sWzYsMCwiIiwxLHsic3R5bGUiOnsiaGVhZCI6eyJuYW1lIjoiZXBpIn19fV0sWzcsMSwiIiwxLHsic3R5bGUiOnsiaGVhZCI6eyJuYW1lIjoiZXBpIn19fV0sWzQsMTIsIiIsMSx7InN0eWxlIjp7ImhlYWQiOnsibmFtZSI6ImVwaSJ9fX1dLFs4LDQsIiIsMSx7InN0eWxlIjp7ImhlYWQiOnsibmFtZSI6ImVwaSJ9fX1dLFsyLDE2LCIiLDEseyJzdHlsZSI6eyJoZWFkIjp7Im5hbWUiOiJlcGkifX19XSxbMTYsMTAsIiIsMSx7ImxldmVsIjoyLCJzdHlsZSI6eyJoZWFkIjp7Im5hbWUiOiJub25lIn19fV0sWzEwLDExLCIiLDEseyJsZXZlbCI6Miwic3R5bGUiOnsiaGVhZCI6eyJuYW1lIjoibm9uZSJ9fX1dLFsxMSwxMiwiIiwxLHsibGV2ZWwiOjIsInN0eWxlIjp7ImhlYWQiOnsibmFtZSI6Im5vbmUifX19XSxbMSw2XSxbNCw3XSxbMCwyLCIiLDEseyJzdHlsZSI6eyJuYW1lIjoiY29ybmVyIn19XSxbMSwwLCIiLDEseyJzdHlsZSI6eyJuYW1lIjoiY29ybmVyIn19XSxbNCwxLCIiLDEseyJzdHlsZSI6eyJuYW1lIjoiY29ybmVyIn19XV0=
\[\begin{tikzcd}[ampersand replacement=\&,cramped,column sep=large]
	\cdots \& {[Y_{n-1}]} \& {[Y_n]} \& {[Y_{n+1}]} \& {[Y_{n+2}]} \& \cdots \\
	\cdots \& {[X_{n-1}]} \& {[X_n]} \& {[X_{n+1}]} \& {[X_{n+2}]} \& \cdots \\
	\cdots \& {1_{\mathcal C(\bbA)}} \& {1_{\mathcal C(\bbA)}} \& {1_{\mathcal C(\bbA)}} \& {1_{\mathcal C(\bbA)}} \& \cdots
	\arrow[two heads, from=1-2, to=2-2]
	\arrow[two heads, from=1-3, to=2-3]
	\arrow[two heads, from=1-4, to=2-4]
	\arrow[two heads, from=1-5, to=2-5]
	\arrow[two heads, from=2-2, to=3-2]
	\arrow[from=2-3, to=1-2]
	\arrow["\lrcorner"{anchor=center, pos=0.125, rotate=-135}, draw=none, from=2-3, to=2-2]
	\arrow[two heads, from=2-3, to=3-3]
	\arrow[from=2-4, to=1-3]
	\arrow["\lrcorner"{anchor=center, pos=0.125, rotate=-135}, draw=none, from=2-4, to=2-3]
	\arrow[two heads, from=2-4, to=3-4]
	\arrow[from=2-5, to=1-4]
	\arrow["\lrcorner"{anchor=center, pos=0.125, rotate=-135}, draw=none, from=2-5, to=2-4]
	\arrow[two heads, from=2-5, to=3-5]
	\arrow[equals, from=3-2, to=3-3]
	\arrow["{a_{n-1}}"{description}, from=3-3, to=2-2]
	\arrow[equals, from=3-3, to=3-4]
	\arrow["{a_n}"{description}, from=3-4, to=2-3]
	\arrow[equals, from=3-4, to=3-5]
	\arrow["{a_{n+1}}"{description}, from=3-5, to=2-4]
\end{tikzcd}\]
where each slanted square is distinguished. It is not difficult to check that $\{\textbf{p}_{[Y_n]}:[Y_n] \twoheadrightarrow [X_n] \mid n \in \mathbb Z\}$ is a basis of $\mathcal C(\bbA)$. However, it is not well-founded as the relation $\btle$ from Definition \ref{def: relation on display maps, well-foundedness} satisfies
$$
\cdots \btle \textbf{p}_{[Y_{n-1}]} \btle \textbf{p}_{[Y_n]} \btle \textbf{p}_{[Y_{n+1}]} \btle \cdots
$$
For the same reason, the gat of unbounded chain complexes from \cite[\S3.4]{BarHen25} has a non-well-founded basis.
\end{example}

\subsection{Existence of a well-founded basis implies cellularity}
\label{subsec: basis -> cellularity}

We will now prove that if a contextual category $\mathcal A$ has a well-founded basis $\textbf{P}$, then it is cellular. Informally, the idea is to approximate $\mathcal A$ by larger and larger full contextual subcategories, starting from the empty one. Suppose that we have, at a given stage, a proper full contextual subcategory $\mathcal A' \subset \mathcal A$ and $\textbf{p}_a:a \rightarrow \partial a$ in $\textbf{P}$ such that $\partial a \in \mathcal A'$ but $a \notin \mathcal A'$. Freely adding to $\mathcal A'$ a strict display map with codomain $\partial a$ yields $\mathcal B \in \Cont$ equipped with a factorization $\mathcal A' \rightarrow \mathcal B \rightarrow \mathcal A$. Now, we use the small object argument to factorize $\mathcal B \rightarrow \mathcal A$ as
$$
\mathcal B \overset{I}{\longrightarrow} \mathcal B' \overset{G}{\longrightarrow} \mathcal A
$$
where $I$ is obtained by recursively adding terms and term equalities (in particular, it is cellular) and $G$ is full-and-faithful. Once we are able to prove that $\mathcal B' \rightarrow \mathcal A$ is injective on objects, the process can be repeated as long as, once again, there is some $\textbf{p}_b:b \rightarrow \partial b$ in $\textbf{P}$ with $\partial b \in \mathcal B'$ but $b \notin \mathcal B'$; this will be possible as $\textbf{P}$ is well-founded. This can be suitably iterated to express $\mathcal O_0 \rightarrow \mathcal A$ as a transfinite composite of pushouts of basic cofibrations.

\begin{lemma}
	\label{lem: basis criterion contextual isomorphism}
	Consider a full-and-faithful strict morphism $F:\mathcal A \rightarrow \mathcal B$. If there exist bases $\textbf{P}$, $\textbf{P}'$ of $\mathcal A$, $\mathcal B$, respectively, such that $F$ maps $\textbf{P}$ bijectively onto $\textbf{P}'$, then $F$ is an isomorphism.
\end{lemma}

\begin{proof}
	Since $F$ is full-and-faithful, it suffices to prove that it is bijective on objects. We will prove by induction on $n \ge 0$ that for each length-$n$ object $b \in \mathcal B$, there exists a unique $a \in \mathcal A$ such that $F(a) = b$. Firstly, this holds $n = 0$ since $F^{-1}(1_\mathcal B) = 1_\mathcal A$. Given $n \ge 1$, suppose that the claim holds for $0$, ..., $n-1$, and let $b \in \mathcal B$ be a length-$n$ object.
	
	By assumption, there exists a unique distinguished square of the form
	\[
	\tag{i}
	\dsqua{b}{y}{\partial b}{\partial y}{f'}{f}{\textbf{p}_b}{\textbf{p}_y}
	\]
	with $\textbf{p}_y \in \textbf{P}'$. Let $\textbf{p}_x:x \rightarrow \partial x$ be the unique element of $\textbf{P}$ mapped to $\textbf{p}_y$ by $F$. And using the induction hypothesis, let $a$ be the unique object of $\mathcal A$ such that $F(a) = \partial b$.
	
	As $F$ is full-and-faithful, there exists a unique $g:a \rightarrow \partial x$ such that $F(g) = f$. Now, the unique distinguished square of the form
	\[
	\tag{ii}
	\dsqua{a'}{x}{a}{\partial x}{g'}{g}{\textbf{p}_{a'}}{\textbf{p}_x}
	\]
	is sent by $F$ to (i). In particular, $F(a') = b$.

	Next, consider $a'' \in \mathcal A$ such that $F(a'') = b$, and let us verify that $a' = a''$. As $F(\partial a'') = \partial F(a'') = \partial b = F(a)$ and $F$ is bijective on length-$(n-1)$ objects, we have $\partial a'' = a$. Take the unique distinguished square of the form
	\[
	\tag{iii}
	\dsqua{a''}{x'}{a}{\partial x'}{h'}{h}{\textbf{p}_{a''}}{\textbf{p}_{x'}}
	\]
	such that $\textbf{p}_{x'} \in \textbf{P}$. But the distinguished square
	\[
	\tag{iv}
	\dsqua{b}{F(x')}{\partial b}{F(\partial x')}{F(h')}{F(h)}{\textbf{p}_b}{\textbf{p}_{F(x')}}
	\]
	is such that $\textbf{p}_{F(x')} \in \textbf{P}'$, so $\textbf{P}'$ being a basis implies that (iv) equals (i). Now, the assumption that $F$ maps $\textbf{P}$ injectively to $\textbf{P}'$ yields $\textbf{p}_{x'} = \textbf{p}_x$; in particular, $\partial x = \partial x'$. As $F$ is full-and-faithful, we have $g = h$, so (ii) equals (iii). We conclude that $a' = a''$, as desired.
\end{proof}

\begin{construction}[The (termal, full-and-faithful) weak factorization system]
	\label{constr: termal full-and-faithful factorization system}	
	Recall from Proposition \ref{prop: characterization full-and-faithful} that a morphism in $\Cont$ is full-and-faithful if and only if it has the right lifting property with respect to every element of
	\[
	\{ \iota_n^T:\mathcal O_n \rightarrow \mathcal O_n^+, \quad \pi_n^T:\mathcal O_n^{++} \rightarrow \mathcal O_n^+ \; \mid \; n \ge 1 \}
	\]
	(or, equivalently, is right orthogonal to $\iota_n^T$ for all $n$). Since $\Cont$ is locally presentable, we can use the small object argument to obtain an orthogonal factorization system in which
	\begin{itemize}
		\item the left class consists of the morphisms that can be expressed as transfinite composites of pushouts of (coproducts of) arrows in $\{\iota_n^T,\; \pi_n^T \mid n \ge 1\}$ (in other words, the relative cell complexes with respect to this set), and
		
		\item the right class consists of the full-and-faithful morphisms.
	\end{itemize}
We will say that a morphism is \emph{termal} if it belongs to the left class. Up to isomorphism, they correspond to inclusions\footnote{Meaning that every symbol and axiom in $\bbA$ is also in $\bbB$, and $\bbA \rightarrow \bbB$ sends each expression to itself. But being an inclusion in this sense is not a category-theoretic property as it is not invariant under isomorphism in $\iiAr(\GAT)$; in fact, every morphism in $\GAT$ is isomorphic to one of this form. In particular, such ``inclusions" should not be confused with monos.} of gats $\bbA \rightarrow \bbB$ that only adjoin terms and term equalities to $\bbA$.
\end{construction}

\begin{proposition}
	\label{prop: adding a sort}
	Consider a contextual category $\mathcal A$ equipped with a basis $\textbf{P}$. Suppose given an admissible subset $\textbf{K} \subset \textbf{P}$, and $a \in \mathcal A$ of length $\ge 1$ such that $\partial a \in \mathcal A_\textbf{K}$ and $a \notin \mathcal A_\textbf{K}$. Let $\mathcal B$ be as in the pushout square of contextual categories
	\[
	\squa{\mathcal O_{n-1}}{\mathcal A_\textbf{K}}{\mathcal O_n}{\mathcal B}{}{}{\iota_n^S}{F}
	\]
	where the top arrow sends the generating object $o_{n-1}$ to $\partial a$. Writing $b$ for the image of $o_n$ under the bottom arrow, by the universal property of $\mathcal B$ we have a unique morphism $F':\mathcal B \rightarrow \mathcal A$ that restricts along $F$ to the inclusion $\mathcal A_\textbf{K} \rightarrow \mathcal A$ and sends $b$ to $a$. Consider a factorization (see Construction \ref{constr: termal full-and-faithful factorization system})
	$$
	\mathcal B \overset{I}{\longrightarrow} \mathcal B' \overset{G}{\longrightarrow} \mathcal A
	$$
	of $F'$ where $I$ is termal and $G$ is full-and-faithful. Moreover, suppose that
	\[
	\tag{$\varheartsuit$}
	IF(\textbf{K}) \cup \{\textbf{p}_{I(b)}\} \text{ is a basis of } \mathcal B'.
	\]
	Then $G$ defines an isomorphism $\mathcal B' \cong \mathcal A_{\textbf{K} \cup \{\textbf{p}_a\}}$.
\end{proposition}

\begin{proof}
	As $G$ is full-and-faithful, by Lemma \ref{lem: basis criterion contextual isomorphism} it suffices to prove that $\mathcal B$ has a basis that is mapped bijectively by $G$ onto a basis of $\mathcal A_{\textbf{K} \cup \{\textbf{p}_a\}}$. But as $GIF$ is the inclusion $\mathcal A_\textbf{K} \rightarrow \mathcal A$ and $GI(\textbf{p}_b) = \textbf{p}_a$, the basis $IF(\textbf{K}) \cup \{\textbf{p}_{I(b)}\}$ is mapped bijectively by $G$ onto
	$$
	GIF(\textbf{K}) \cup \{GI(\textbf{p}_b)\} = \textbf{K} \cup \{\textbf{p}_a\},
	$$
	which is a basis of $\mathcal A_{\textbf{K} \cup \{\textbf{p}_a\}}$ since $\partial a \in \mathcal A_\textbf{K}$.
\end{proof}

\begin{remark}
\label{rem: heart condition}
It will follow from Proposition \ref{prop: basis -> cellular} that the auxiliary condition $\varheartsuit$ always holds. In fact, we can already prove that it holds whenever the contextual category $\mathcal A_\textbf{K}$ is cellular. Indeed, under this assumption, as $F$ (being a pushout of $\iota_n^S$) and $I$ (being a transfinite composite of pushouts of elements of $\{\iota_m^T, \; \pi_m^T \mid m \ge 1\}$) are cellular maps, so is the composite
$$
\mathcal O_0 \overset{!}{\longrightarrow} \mathcal A_\textbf{K} \overset{F}{\longrightarrow} \mathcal B \overset{I}{\longrightarrow} \mathcal B',
$$
whence $\mathcal B'$ is cellular. Moreover, as $\mathcal A_\textbf{K}$ has a cellular presentation $\mathscr C$ such that $\textbf{P}(\mathscr C) = \textbf{K}$, the above decomposition of $\mathcal O_0 \rightarrow \mathcal B'$ yields a cellular presentation $\mathscr C'$ of $\mathcal B'$ such that $\textbf{P}(\mathscr C') = IF(\textbf{K}) \cup \{p_{I(b)}\}$. But by Proposition \ref{prop: cellular presentation -> basis}, the latter set of display maps is a basis of $\mathcal B'$.
\end{remark}

\begin{lemma}
\label{lem: good well-ordering for well-founded basis}
Suppose that $\textbf{P}$ is a well-founded basis of a contextual category $\mathcal A$. There exist an ordinal $\alpha$ and a bijection
$$
(p_\nu:a_\nu \twoheadrightarrow \partial a_\nu)_{\nu < \alpha}
$$
between $\alpha$ and $\textbf{P}$ with the following property: for all $\beta < \alpha$,
$$
\partial a_\beta \in \mathcal A_{\{p_\nu \mid \nu < \beta\}}.
$$
\end{lemma}

\begin{proof}
Let $\mathscr S$ be the (small) set of all triples $(\alpha,\textbf{K},\varphi)$ consisting of
\begin{itemize}
	\item an ordinal $\alpha$ whose cardinality is at most that of $\textbf{P}$;
	
	\item a subset $\textbf{K} \subset \textbf{P}$;
	
	\item a bijection $\varphi:\alpha \rightarrow \textbf{K}$ such that for all $\beta < \alpha$, the codomain of $\varphi(\beta)$ belongs to $\mathcal A_{\{\varphi(\nu) \mid \nu < \beta\}}$.
\end{itemize}
Our goal is to prove that $\mathscr S$ has an element of the form $(\alpha,\textbf{P},\varphi)$. Endow $\mathscr S$ with the partial order $\preccurlyeq$ where $(\alpha,\textbf{K},\varphi) \preccurlyeq (\alpha',\textbf{K}',\varphi')$ if and only if $\alpha \le \alpha'$, $\textbf{K} \subset \textbf{K}'$, and the restriction of $\varphi'$ to $\alpha$ equals $\varphi$. Note that $\mathscr S$ is non-empty as it contains $(0,\varnothing,id_{\varnothing})$. Also, for a linearly ordered subset $L \subset \mathscr S$, say self-indexed as $(\alpha_\ell,\textbf{K}_\ell,\varphi_\ell)$, we have
$$
(\sup_{\ell \in L} \alpha_\ell,\; \bigcup_{\ell \in L} \textbf{K}_\ell, \; \psi) \in \mathscr S
$$
where $\psi:\sup_{\ell \in L} \alpha_\ell \rightarrow \bigcup_{\ell \in L} \textbf{K}_\ell$ is the unique function whose restriction to $\alpha_\ell$ is $\varphi_\ell$ for each $\ell \in L$. By Zorn's lemma, $\mathscr S$ has a maximal element. But if $(\alpha,\textbf{K},\varphi)$ is maximal, there exists no display map in $\textbf{P} \setminus \textbf{K}$ whose codomain belongs to $\mathcal A_\textbf{K}$, which implies that $\mathcal A_\textbf{K} = \mathcal A_\textbf{P} = \mathcal A$ -- here we have used (through Lemma \ref{lem: characterization admissibility of basis}) that $\mathcal A$ is well-founded.\footnote{In fact, this is the only point where we use well-foundedness of $\textbf{P}$ towards proving Proposition \ref{prop: basis -> cellular}.} We conclude that $\textbf{K} = \textbf{P}$.
\end{proof}

\begin{remark}
\label{rem: admissible subsets from a basis}
For a bijection $(p_\nu)_{\nu < \alpha}$ as in Lemma \ref{lem: good well-ordering for well-founded basis}, it can be verified by induction in a straightforward way that for all $\beta < \alpha$, the set of display maps $\{p_\nu \mid \nu \le \beta\}$ is admissible, i.e. $\{p_\nu \mid \nu \le \beta\} \subset \mathcal A_{\{p_\nu \mid \nu \le \beta\}}$.
\end{remark}

\begin{proposition}
\label{prop: basis -> cellular}
If a contextual category has a well-founded basis, then it is cellular.
\end{proposition}

\begin{proof}
Consider a contextual category $\mathcal A$ equipped with a well-founded basis $\textbf{P}$. Choose a bijection $(p_\nu:a_\nu \rightarrow \partial a_\nu)_{\nu < \alpha}$ between an ordinal $\alpha$ and $\textbf{P}$ as in the statement of Lemma \ref{lem: good well-ordering for well-founded basis}. For $\beta \le \alpha$, let $\textbf{K}(\beta) = \{p_\nu \mid \nu < \beta\}$. For each $\beta < \alpha$, factorize the inclusion $\mathcal A_{\textbf{K}(\beta)} \hookrightarrow \mathcal A_{\textbf{K}(\beta+1)}$ as
$$
\mathcal A_{\textbf{K}(\beta)} \overset{F_\beta}{\longrightarrow} \mathcal B_\beta \overset{I_\beta}{\longrightarrow} \mathcal B'_\beta \overset{G_\beta}{\longrightarrow} \mathcal A_{\textbf{K}(\beta+1)}
$$
where:
\begin{itemize}
	\item $\mathcal B_\beta$ and $F_\beta$ are defined via the pushout square
	\[
	\squa{\mathcal O_{n-1}}{\mathcal A_{\textbf{K}(\beta)}}{\mathcal O_n}{\mathcal B_\beta}{\overline{\partial a_\beta}}{}{\iota_n^S}{F_\beta}
	\]
	where $n$ is the length of $a_\beta$. Write $b_\beta$ for the image of $o_n$ under the bottom arrow.
	
	\item $\mathcal B_\beta \overset{I_\beta}{\rightarrow} \mathcal B'_\beta \overset{G_\beta}{\rightarrow} \mathcal A_{\textbf{K}(\beta+1)}$ is a (termal, full-and-faithful) factorization\footnote{Such a factorization is unique up to canonical isomorphism, but for set-theoretic definiteness we can take the factorization provided by the small object argument with $\omega$ steps.} of the unique morphism $\mathcal B_\beta \rightarrow \mathcal A_{\textbf{K}(\beta+1)}$ making the following diagram commute:
	% https://q.uiver.app/#q=WzAsNSxbMCwwLCJcXG1hdGhjYWwgT197bi0xfSJdLFswLDEsIlxcbWF0aGNhbCBPX24iXSxbMSwwLCJcXG1hdGhjYWwgQV97XFx0ZXh0YmZ7S30oXFxiZXRhKX0iXSxbMiwyLCJcXG1hdGhjYWwgQV97XFx0ZXh0YmZ7S30oXFxiZXRhKX0iXSxbMSwxLCJcXG1hdGhjYWwgQl9cXGJldGEiXSxbMCwyLCJcXG92ZXJsaW5le1xccGFydGlhbCBhX1xcYmV0YX0iXSxbMSwzLCJcXG92ZXJsaW5le2FfXFxiZXRhfSIsMix7ImN1cnZlIjozfV0sWzIsMywiIiwwLHsiY3VydmUiOi0zLCJzdHlsZSI6eyJ0YWlsIjp7Im5hbWUiOiJob29rIiwic2lkZSI6InRvcCJ9fX1dLFs0LDMsIiIsMCx7InN0eWxlIjp7ImJvZHkiOnsibmFtZSI6ImRhc2hlZCJ9fX1dLFsyLDQsIkZfXFxiZXRhIl0sWzEsNCwiXFxvdmVybGluZXtiX1xcYmV0YX0iLDJdLFswLDEsIlxcaW90YV9uXlMiLDJdXQ==
	\[\begin{tikzcd}[ampersand replacement=\&]
		{\mathcal O_{n-1}} \& {\mathcal A_{\textbf{K}(\beta)}} \& \\
		{\mathcal O_n} \& {\mathcal B_\beta} \\
		\&\& {\mathcal A_{\textbf{K}(\beta)}}
		\arrow["{\overline{\partial a_\beta}}", from=1-1, to=1-2]
		\arrow["{\iota_n^S}"', from=1-1, to=2-1]
		\arrow["{F_\beta}", from=1-2, to=2-2]
		\arrow[curve={height=-18pt}, hook, from=1-2, to=3-3]
		\arrow["{\overline{b_\beta}}"', from=2-1, to=2-2]
		\arrow["{\overline{a_\beta}}"', curve={height=18pt}, from=2-1, to=3-3]
		\arrow[dashed, from=2-2, to=3-3]
	\end{tikzcd}\]
\end{itemize}

Now, let us prove by induction (using Proposition \ref{prop: adding a sort}) that for all $\beta \le \alpha$,
\begin{enumerate}[label=(\roman*)]
	\item $\mathcal A_{\textbf{K}(\beta)}$ is cellular, and

	\item if $\beta < \alpha$, then $G_\beta:\mathcal B'_\beta \rightarrow \mathcal A_{\textbf{K}(\beta+1)}$ is an isomorphism.
\end{enumerate}

Given $\beta < \alpha$, suppose that the above claims hold for all $\gamma < \beta$.

Firstly, let us verify that (i) holds for $\beta$. For each $\gamma < \beta$, it follows from assumption (ii) that $G_\gamma$ is a cellular map. Since $F_\gamma$ (being a pushout of some $\iota_n^S$) and $I_\beta$ (being a transfinite composite of pushouts of elements of $\{\iota_n^T,\; \pi_n^T \mid n \ge 1\}$) are also cellular, we conclude that so is the inclusion
$$
G_\gamma I_\gamma F_\gamma:\mathcal A_{\textbf{K}(\gamma)} \hookrightarrow \mathcal A_{\textbf{K}(\gamma+1)}.
$$
As $\mathcal O_0 \rightarrow \mathcal A_{\textbf{K}(\beta)}$ is a transfinite composite of the chain of morphisms $\beta \rightarrow \Cont$ that sends $\gamma < \gamma'$ to the inclusion $\mathcal A_{\textbf{K}(\gamma)} \hookrightarrow \mathcal A_{\textbf{K}(\gamma')}$, the former is a cellular map -- in other words, $\mathcal A_{\textbf{K}(\beta)}$ is a cellular object, as required.

\vspace{0.5em}

Now, let us prove that (ii) holds for $\beta$ provided that $\beta < \alpha$. By Proposition \ref{prop: adding a sort}, it suffices to check the corresponding instance of condition $\varheartsuit$ from the statement of that result. But by Remark \ref{rem: heart condition}, $\varheartsuit$ follows from $\mathcal A_{\textbf{K}(\beta)}$ being cellular.

\vspace{0.5em}

This concludes the induction step. We have proved, in particular, that $\mathcal A_{\textbf{K}(\alpha)} = \mathcal A_\textbf{P} = \mathcal A$ is cellular.
\end{proof}

The results of Propositions \ref{prop: cellular gat -> basis}, \ref{prop: cellular presentation -> basis} and \ref{prop: basis -> cellular} can be summarized as follows:

\begin{theorem}
\label{th: cellular iff basis}
For a small contextual category $\mathcal A$, the following conditions are equivalent:
\begin{enumerate}[label=(\alph*)]
	\item $\mathcal A$ is cellular.
	
	\item $\mathcal A \cong \mathcal C(\bbA)$ for a cellular generalized algebraic theory $\bbA$ (i.e. $\bbA$ has no sort equality axioms).
	
	\item $\mathcal A$ has a well-founded basis (Definition \ref{def: basis}). \qed
\end{enumerate}
\end{theorem}

\section{Strictification of weak morphisms of contextual categories}
\label{sec: strictification}

We refer to \S\ref{sec: introduction} for definitions and notation regarding (rooted) display map categories (adapted from \cite{Tay99}). In particular, recall that we denote the $2$-category of display map categories (resp. rooted dmcs) by $\iiDMC$ (resp. $\iiDMC_r$), and the (rooted) dmc associated with $\mathcal A \in \Cont$ by $\mathcal A_w$. Some of our main results will also rely on the theory of precontextual categories developed in \cite{Alm26}; however, a reader encountering the concept for the first time may safely replace ``precontextual category" by ``contextual category" everywhere. The relevance of working in this more general setting is explained below.

We will now study the following question: for contextual categories $\mathcal A$, $\mathcal B$ and a weak morphism $F:\mathcal A_w \rightarrow \mathcal B_w$, when (and how) can we obtain a strict morphism $F':\mathcal A \rightarrow \mathcal B$ with $F' \cong F$? In particular, we would like to find conditions on $\mathcal A$, $\mathcal B$ under which
$$
\iiCont(\mathcal A,\mathcal B) \hookrightarrow \iiDMC_r(\mathcal A_w,\mathcal B_w)
$$
is an equivalence of categories. As we will see, this question is closely related to the structure of the category of strict display maps $\textbf{D}(\mathcal A)$. Exploring this connection will allow us to prove, for example, that the above functor is an equivalence whenever $\mathcal A$ is cofibrant for the model structure on $\Cont$. From that, we can conclude that the $\Ho(\Cat)$-enriched homotopy category of the $\Cat$-model category $\iiCont$ is (the $\Ho(\Cat)$-enriched category associated with) $\iiDMC_r$.

Our central technique for strictifying weak morphisms is established in \S\ref{subsec: what is needed to strictify} through the concept of a \emph{local strictification}. This machinery could have been developed just for weak morphisms $\mathcal A_w \rightarrow \mathcal B_w$ where $\mathcal A$, $\mathcal B$ are contextual categories. However, it works without extra difficulty with $\mathcal A$ replaced by a precontextual category as in \cite[Def. 2.2]{Alm26}, with weak morphisms also suitably generalized. The main consequence of the added generality is Proposition \ref{prop: weak morphisms precontextual category, assuming basis}, which allows us to conclude that under certain conditions on the precontextual category $\mathcal A$, weak morphisms $\mathcal A \rightarrow \mathcal B$ match the strict morphisms $L\mathcal A \rightarrow \mathcal B$ where $L:\Precont \rightarrow \Cont$ is left adjoint to the inclusion functor. This is in line with the well-known idea that models of a gat $\bbA$ can often be described diagrammatically via a finite-limit sketch (constructed by directly translating the axioms) whose data are simpler and more explicit than the whole structure of $\mathcal C(\bbA)$.

\begin{definition}
\label{def: weak morphism and D(A) for precontextual category}
For a precontextual category $\mathcal A$ (\cite[Def. 2.2]{Alm26}) and a rooted dmc $\mathcal B$, we define a \emph{weak morphism} from $\mathcal A$ to $\mathcal B$ as a functor $F:|\mathcal A| \rightarrow |\mathcal B|$ that sends
\begin{itemize}[noitemsep]
	\item every strict display map to a display map;
	
	\item every distinguished square to a pullback square;
	
	\item every length-$0$ object to a terminal object.
\end{itemize}
We denote by $[\mathcal A,\mathcal B]$ the category of weak morphisms from $\mathcal A$ to $\mathcal B$ and natural transformations between them. In particular, if $\mathcal A$ is a contextual category, then (assuming that $\mathcal A$, $\mathcal B$ are locally small) $[\mathcal A,\mathcal B] = \iiCont^+(\mathcal A_w,\mathcal B)$.

\vspace{0.5em}

We denote by $\textbf{D}(\mathcal A)$ the following subcategory of $\iiAr(|\mathcal A|)$:
\begin{itemize}[noitemsep]
	\item its objects are the strict display maps $a \twoheadrightarrow \partial a$.
	
	\item its morphisms are the commutative squares
	\[
	\dsqua{a}{b}{\partial a}{\partial b}{f'}{f}{}{}
	\]
	that can be expressed as a finite horizontal composite of distinguished squares (including the identity square of each $a \twoheadrightarrow \partial a$).
\end{itemize}

This extends Construction \ref{constr: categories of display maps}. Note that modifying $\mathcal A$ by adding all morphisms of $\textbf{D}(\mathcal A)$ as distinguished squares doesn't change its associated contextual category nor the weak morphisms out of it. But for many precontextual categories that arise in practice, the set of distinguished squares is not closed under horizontal composition, unless one imposes that by the above process.
\end{definition}

\subsection{What is needed to strictify a weak morphism?}
\label{subsec: what is needed to strictify}

In what follows, we let $\mathcal A$ be a precontextual category and $\mathcal B$ a contextual category.

\begin{definition}
\label{def: strictification}
A \emph{strictification} of a weak morphism $F \in [\mathcal A,\mathcal B_w]$ is a pair $(F',\varphi)$ consisting of a strict morphism $F':\mathcal A \rightarrow \mathcal B$ and an isomorphism $\varphi:F' \Rightarrow F$. If such a pair exists, we say that $F$ is \emph{strictifiable}.
\end{definition}

Note that as every morphism in $\textbf{D}(\mathcal A)$ is a finite composite of distinguished squares, $\iiAr(F):\iiAr(\mathcal A) \rightarrow \iiAr(\mathcal B)$ restricts to a functor
$$
F_*:\textbf{D}(\mathcal A) \longrightarrow \textbf{D}_w(\mathcal B_w).
$$
On the other hand, the strict morphism $F'$ induces a functor
$$
F'_*:\textbf{D}(\mathcal A) \longrightarrow \textbf{D}(\mathcal B),
$$
and $\varphi$ induces a natural transformation as in
\[\begin{tikzcd}[ampersand replacement=\&,row sep=tiny]
	{\textbf{D}(\mathcal A)} \&\& {\textbf{D}_w(\mathcal B_w)} \\
	\& {\textbf{D}(\mathcal B)}
	\arrow[""{name=0, anchor=center, inner sep=0}, "{F_*}", curve={height=-30pt}, from=1-1, to=1-3]
	\arrow[""{name=1, anchor=center, inner sep=0}, curve={height=30pt}, draw=none, from=1-1, to=1-3]
	\arrow["{F'_*}"', curve={height=6pt}, from=1-1, to=2-2]
	\arrow["\iota"', curve={height=6pt}, hook, from=2-2, to=1-3]
	\arrow["{\varphi'}", shorten=12pt, Rightarrow, from=0, to=1]
\end{tikzcd}\]
where $\iota$ is the inclusion functor. Since $\varphi$ is an isomorphism, each component of $\varphi_*$ is a pair of isomorphisms in $|\mathcal B|$.

Our main tool for strictifying weak morphisms, presented next, will be a kind of structure -- a \emph{local strictification} -- that is similar to the one described above but can often be obtained much more naturally and easily. We will provide an algorithm for constructing a strictification from any local strictification (Construction \ref{constr: functor associated with a local strictification} and Proposition \ref{prop: local strictification iff strictifiable}).

\begin{definition}
\label{def: local strictification}
Let $F \in [\mathcal A,\mathcal B_w]$. A \emph{local strictification} of $F$ is a pair $(S,\sigma)$ consisting of a functor $S:\textbf{D}(\mathcal A) \rightarrow \textbf{D}(\mathcal B)$ and a natural transformation $\sigma$ as in
% https://q.uiver.app/#q=WzAsMyxbMCwwLCJcXHRleHRiZntEfShcXG1hdGhjYWwgQSkiXSxbMSwxLCJcXHRleHRiZntEfShcXG1hdGhjYWwgQikiXSxbMiwwLCJcXHRleHRiZntEfV93KFxcbWF0aGNhbCBCKSJdLFswLDEsIlMiLDIseyJjdXJ2ZSI6Mn1dLFsxLDIsIiIsMCx7ImN1cnZlIjoxLCJzdHlsZSI6eyJ0YWlsIjp7Im5hbWUiOiJob29rIiwic2lkZSI6InRvcCJ9fX1dLFswLDIsIkZfKiIsMCx7ImN1cnZlIjotNX1dLFswLDIsIiIsMCx7ImN1cnZlIjo1LCJzdHlsZSI6eyJib2R5Ijp7Im5hbWUiOiJub25lIn0sImhlYWQiOnsibmFtZSI6Im5vbmUifX19XSxbNSw2LCJcXHNpZ21hIiwwLHsic2hvcnRlbiI6eyJzb3VyY2UiOjIwLCJ0YXJnZXQiOjIwfX1dXQ==
\[\begin{tikzcd}[ampersand replacement=\&,row sep=tiny]
	{\textbf{D}(\mathcal A)} \&\& {\textbf{D}_w(\mathcal B)} \\
	\& {\textbf{D}(\mathcal B)}
	\arrow[""{name=0, anchor=center, inner sep=0}, "{F_*}", curve={height=-30pt}, from=1-1, to=1-3]
	\arrow[""{name=1, anchor=center, inner sep=0}, curve={height=30pt}, draw=none, from=1-1, to=1-3]
	\arrow["S"', curve={height=6pt}, from=1-1, to=2-2]
	\arrow["\iota"',curve={height=6pt}, hook, from=2-2, to=1-3]
	\arrow["\sigma", shorten=12pt, Rightarrow, from=0, to=1]
\end{tikzcd}\]
\end{definition}

\begin{notation}
For $a \in \Ob(\mathcal A)$ of length $\ge 1$, we let $\sigma^\textbf{s}_a:F(a) \rightarrow S^\textbf{s}(a)$ and $\sigma^\textbf{t}_a:F(\partial a) \rightarrow S^\textbf{t}(a)$ be, respectively, the source and target components of $\sigma_{\textbf{p}_a}$.
\end{notation}

\begin{construction}
\label{constr: functor associated with a local strictification}
Suppose that $(S,\sigma)$ is a local strictification of a weak morphism $F \in [\mathcal A,\mathcal B_w]$. We will now construct a functor
$$
F_S:|\mathcal A| \longrightarrow |\mathcal B|
$$
(with $\sigma$ implicit) equipped with an isomorphism $F_S \cong F$; it will be proved to be a strict morphism $\mathcal A \rightarrow \mathcal B$ in Proposition \ref{prop: F_S is a contextual functor}. Firstly, we will use structural recursion on the tree of display maps of $\mathcal A$ to produce an object $F_S(a) \in \mathcal B$ and an isomorphism $\psi_a:F_S(a) \rightarrow F(a)$ for each $a \in \mathcal A$. After that, conjugating the action of $F$ on arrows along the $\psi_a$ allows us to extend $F_S:\Ob(\mathcal A) \rightarrow \Ob(\mathcal B)$ into a functor $|\mathcal A| \rightarrow |\mathcal B|$ such that $\psi$ is a natural isomorphism $F_S \cong F$.

\vspace{0.5em}

We start by setting $F_S(a) = 1_\mathcal B$ for every length-$0$ object $a$, with $\psi_a:1_\mathcal B \rightarrow F(a)$ the unique (iso)morphism. For $a$ of length $\ge 1$, suppose that $F_S(\partial a)$ and $\psi_{\partial a}$ have been constructed. We let $F_S(a)$ be as in the distinguished square
% https://q.uiver.app/#q=WzAsNCxbMCwwLCJGXlMoYSkiXSxbMCwxLCJGXlMoXFxwYXJ0aWFsIGEpIl0sWzIsMSwiU157XFx0ZXh0YmYgdH0oYSkiXSxbMiwwLCJTXntcXHRleHRiZiBzfShhKSJdLFswLDEsIlxcdGV4dGJmIHBfe0ZeUyhhKX0iLDIseyJzdHlsZSI6eyJoZWFkIjp7Im5hbWUiOiJlcGkifX19XSxbMywyLCJTKFxcdGV4dGJmIHBfYSkiLDAseyJzdHlsZSI6eyJoZWFkIjp7Im5hbWUiOiJlcGkifX19XSxbMSwyLCJcXHNpZ21hXntcXHRleHRiZiB0fV9hIFxcY2lyYyBcXHBzaV97XFxwYXJ0aWFsIGF9IiwyXSxbMCwzXV0=
\[\begin{tikzcd}[ampersand replacement=\&]
	{F_S(a)} \&\& {S^{\textbf s}(a)} \\
	{F_S(\partial a)} \&\& {S^{\textbf t}(a).}
	\arrow["k",from=1-1, to=1-3]
	\arrow["{\textbf p_{F_S(a)}}"', two heads, from=1-1, to=2-1]
	\arrow["{S(\textbf p_a)}", two heads, from=1-3, to=2-3]
	\arrow["{\sigma^{\textbf t}_a \circ \psi_{\partial a}}"', from=2-1, to=2-3]
\end{tikzcd}\]
Now, using that $\sigma$ sends $\textbf{p}_a$ to a cartesian square in $\mathcal B$, we let $\psi_a$ be the unique dashed arrow making the diagram
% https://q.uiver.app/#q=WzAsNixbMCwwLCJGXlMoYSkiXSxbMCwxLCJGXlMoXFxwYXJ0aWFsIGEpIl0sWzQsMSwiU157XFx0ZXh0YmYgdH0oYSkiXSxbNCwwLCJTXntcXHRleHRiZiBzfShhKSJdLFsyLDEsIkYoXFxwYXJ0aWFsIGEpIl0sWzIsMCwiRihhKSJdLFswLDEsIlxcdGV4dGJmIHBfe0ZeUyhhKX0iLDIseyJzdHlsZSI6eyJoZWFkIjp7Im5hbWUiOiJlcGkifX19XSxbMywyLCJTKFxcdGV4dGJmIHBfYSkiLDAseyJzdHlsZSI6eyJoZWFkIjp7Im5hbWUiOiJlcGkifX19XSxbMSw0LCJcXHBzaV97XFxwYXJ0aWFsIGF9IiwyXSxbNCwyLCJcXHNpZ21hXntcXHRleHRiZiB0fV9hIiwyXSxbNSwzLCJcXHNpZ21hXntcXHRleHRiZiBzfV9hIl0sWzUsNCwiRihcXHRleHRiZiBwX2EpIl0sWzAsNSwiIiwwLHsic3R5bGUiOnsiYm9keSI6eyJuYW1lIjoiZGFzaGVkIn19fV0sWzAsMywiayIsMSx7ImN1cnZlIjotNH1dXQ==
\[\begin{tikzcd}[ampersand replacement=\&]
	{F_S(a)} \&\& {F(a)} \&\& {S^{\textbf s}(a)} \\
	{F_S(\partial a)} \&\& {F(\partial a)} \&\& {S^{\textbf t}(a)}
	\arrow[dashed, from=1-1, to=1-3]
	\arrow["k"{description}, curve={height=-24pt}, from=1-1, to=1-5]
	\arrow["{\textbf p_{F_S(a)}}"', two heads, from=1-1, to=2-1]
	\arrow["{\sigma^{\textbf s}_a}", from=1-3, to=1-5]
	\arrow["{F(\textbf p_a)}", from=1-3, to=2-3]
	\arrow["{S(\textbf p_a)}", two heads, from=1-5, to=2-5]
	\arrow["{\psi_{\partial a}}"', from=2-1, to=2-3]
	\arrow["{\sigma^{\textbf t}_a}"', from=2-3, to=2-5]
\end{tikzcd}\]
commute. By the pasting lemma for pullbacks, as the right and outer squares are cartesian, so is the left one; since $\psi_{\partial a}$ is an isomorphism, so is $\psi_a$. Having constructed $(\psi_a:F_S(a) \overset{\cong}{\rightarrow} F(a))_{a \in \Ob(\mathcal A)}$, we define a map $\mathcal A(a,b) \rightarrow \mathcal B(F_S(a),F_S(b))$ for each $a$, $b \in \mathcal A$ by sending $f$ to the unique dashed arrow making the diagram
\[
\begin{tikzcd}
	F_S(a) \arrow[]{r}{\psi_a} \arrow[dashed]{d}{} & F(a) \arrow[]{d}{F(f)} \\
	F_S(b) \arrow[swap]{r}{\psi_b} & F(b)
\end{tikzcd}
\]
commute, that is, $f \mapsto \psi_b^{-1} \circ F(f) \circ \psi_a$. It can be checked in a straightforward way that this extends $F_S$ into a functor $|\mathcal A| \rightarrow |\mathcal B|$ and that $\psi$ is a natural isomorphism $F_S \cong F$.
\end{construction}

\begin{proposition}
\label{prop: F_S is a contextual functor}
In the notation of Construction \ref{constr: functor associated with a local strictification}, $F_S$ is a strict morphism from $\mathcal A$ to $\mathcal B$.
\end{proposition}

\begin{proof}
It is clear from the construction that $F_S$ preserves strict display maps and sends every length-$0$ object to $1_\mathcal B$; it remains to verify that it preserves distinguished squares. Given a distinguished square
\[
\dsqua{a}{b}{\partial a}{\partial b}{f'}{f}{\textbf{p}_a}{\textbf{p}_b}
\]
in $\mathcal A$, consider the diagram
% https://q.uiver.app/#q=WzAsMTIsWzEsMSwiRl5TKGIpIl0sWzEsNCwiRl5TKFxccGFydGlhbCBiKSJdLFs0LDQsIkYoXFxwYXJ0aWFsIGIpIl0sWzQsMSwiRihiKSJdLFs3LDQsIlNee1xcdGV4dGJmIHR9KGIpIl0sWzcsMSwiU157XFx0ZXh0YmYgc30oYikiXSxbMCwwLCJGXlMoYSkiXSxbMCwzLCJGXlMoXFxwYXJ0aWFsIGEpIl0sWzMsMywiRihcXHBhcnRpYWwgYSkiXSxbMywwLCJGKGEpIl0sWzYsMywiU157XFx0ZXh0YmYgdH0oYSkiXSxbNiwwLCJTXntcXHRleHRiZiBzfShhKSJdLFsxLDIsIlxccHNpX3tcXHBhcnRpYWwgYn0iLDJdLFswLDEsIlxcdGV4dGJme3B9X3tGXlMoYil9IiwyLHsibGFiZWxfcG9zaXRpb24iOjMwLCJzdHlsZSI6eyJoZWFkIjp7Im5hbWUiOiJlcGkifX19XSxbMywyLCJGKFxcdGV4dGJme3B9X2IpIiwwLHsibGFiZWxfcG9zaXRpb24iOjMwfV0sWzIsNCwiXFxzaWdtYV57XFx0ZXh0YmYgdH1fYiIsMl0sWzMsNSwiXFxzaWdtYV57XFx0ZXh0YmYgc31fYiIsMCx7ImxhYmVsX3Bvc2l0aW9uIjo0MH1dLFs1LDQsIlMoXFx0ZXh0YmZ7cH1fYikiLDAseyJzdHlsZSI6eyJoZWFkIjp7Im5hbWUiOiJlcGkifX19XSxbMCwzLCJcXHBzaV9iIl0sWzcsOCwiXFxwc2lfe1xccGFydGlhbCBhfSIsMix7ImxhYmVsX3Bvc2l0aW9uIjo3MH1dLFs2LDcsIlxcdGV4dGJme3B9X3tGXlMoYSl9IiwyLHsic3R5bGUiOnsiaGVhZCI6eyJuYW1lIjoiZXBpIn19fV0sWzksOCwiRihcXHRleHRiZntwfV9hKSIsMCx7ImxhYmVsX3Bvc2l0aW9uIjo3MH1dLFs4LDEwLCJcXHNpZ21hXntcXHRleHRiZiB0fV9hIiwyLHsibGFiZWxfcG9zaXRpb24iOjcwfV0sWzksMTEsIlxcc2lnbWFee1xcdGV4dGJmIHN9X2EiXSxbMTEsMTAsIlMoXFx0ZXh0YmZ7cH1fYSkiLDAseyJsYWJlbF9wb3NpdGlvbiI6NzAsInN0eWxlIjp7ImhlYWQiOnsibmFtZSI6ImVwaSJ9fX1dLFs2LDksIlxccHNpX2EiXSxbNywxLCJGXlMoZikiLDFdLFs2LDAsIkZeUyhmJykiLDFdLFs5LDMsIkYoZicpIiwxXSxbMTEsNSwiU15cXHRleHRiZntzfShmLGYnKSIsMV0sWzgsMiwiRihmKSIsMV0sWzEwLDQsIlNeXFx0ZXh0YmZ7dH0oZixmJykiLDFdXQ==
\[\begin{tikzcd}[ampersand replacement=\&]
	{F_S(a)} \&\&\& {F(a)} \&\&\& {S^{\textbf s}(a)} \& \\
	\& {F_S(b)} \&\&\& {F(b)} \&\&\& {S^{\textbf s}(b)} \\
	\\
	{F_S(\partial a)} \&\&\& {F(\partial a)} \&\&\& {S^{\textbf t}(a)} \\
	\& {F_S(\partial b)} \&\&\& {F(\partial b)} \&\&\& {S^{\textbf t}(b),}
	\arrow["{\psi_a}", from=1-1, to=1-4]
	\arrow["{F_S(f')}"{description}, from=1-1, to=2-2]
	\arrow["{\textbf{p}_{F_S(a)}}"', two heads, from=1-1, to=4-1]
	\arrow["{\sigma^{\textbf s}_a}", from=1-4, to=1-7]
	\arrow["{F(f')}"{description}, from=1-4, to=2-5]
	\arrow["{F(\textbf{p}_a)}"{pos=0.7}, from=1-4, to=4-4]
	\arrow["{S^\textbf{s}(f,f')}"{description}, from=1-7, to=2-8]
	\arrow["{S(\textbf{p}_a)}"{pos=0.7}, two heads, from=1-7, to=4-7]
	\arrow["{\psi_b}", from=2-2, to=2-5]
	\arrow["{\textbf{p}_{F_S(b)}}"'{pos=0.3}, two heads, from=2-2, to=5-2]
	\arrow["{\sigma^{\textbf s}_b}"{pos=0.4}, from=2-5, to=2-8]
	\arrow["{F(\textbf{p}_b)}"{pos=0.3}, from=2-5, to=5-5]
	\arrow["{S(\textbf{p}_b)}", two heads, from=2-8, to=5-8]
	\arrow["{\psi_{\partial a}}"'{pos=0.7}, from=4-1, to=4-4]
	\arrow["{F_S(f)}"{description}, from=4-1, to=5-2]
	\arrow["{\sigma^{\textbf t}_a}"'{pos=0.7}, from=4-4, to=4-7]
	\arrow["{F(f)}"{description}, from=4-4, to=5-5]
	\arrow["{S^\textbf{t}(f,f')}"{description}, from=4-7, to=5-8]
	\arrow["{\psi_{\partial b}}"', from=5-2, to=5-5]
	\arrow["{\sigma^{\textbf t}_b}"', from=5-5, to=5-8]
\end{tikzcd}\]
which commutes by naturality of $\psi$ and $\sigma$. By definition of $\psi$, the composite squares with vertices
$$
F_S(a),\; S^\textbf{s}(a),\; F_S(\partial a),\; S^\textbf{t}(a) \qquad \text{and} \qquad F_S(b),\; S^\textbf{s}(b),\; F_S(\partial b),\; S^\textbf{t}(b)
$$
are distinguished. Also, as $S$ sends distinguished squares in $\mathcal A$ -- viewed as arrows in $\textbf{D}(\mathcal A)$ -- to distinguished squares in $\mathcal B$, the rightmost square is distinguished. By the pasting lemma for distinguished squares, the leftmost square is distinguished, as desired.
\end{proof}

\begin{proposition}
\label{prop: local strictification iff strictifiable}
For a precontextual category $\mathcal A$ and a contextual category $\mathcal B$, a weak morphism $F \in [\mathcal A,\mathcal B_w]$ is strictifiable if and only if it admits a local strictification (Definition \ref{def: local strictification}). \qed
\end{proposition}

\subsection{Strictifying morphisms whose domain has a basis or is cofibrant}
\label{subsec: strictification basis or cofibrant}

We will now prove that if a contextual category $\mathcal A$ has a basis (Definition \ref{def: basis}), then any weak morphism out of it is strictifiable. In fact, we will prove more generally that this holds for any precontextual category $\mathcal A$ that satisfies the appropriate generalization of the condition of having a basis, namely, every connected component of $\textbf{D}(\mathcal A)$ having a terminal object (see Remark \ref{rem: basis and D(A)}).

This result will be used to prove that weak morphisms out of cofibrant objects of $\Cont$ are strictifiable (Proposition \ref{prop: strictification - cofibrant domain}), which will then allow us to conclude that the homotopy $(\infty,1)$-category of $\Cont$ is the underlying $(2,1)$-category of $\iiDMC_r$ (Theorem \ref{th: homotopy bicategory of Cont}).

\begin{definition}[Extending Definition \ref{def: basis}]
\label{def: basis precont}
A \emph{basis} of a precontextual category $\mathcal A$ is a set $\textbf{P}$ of strict display maps that has the following property: for every $a \in \mathcal A$ such that $\ell(a) \ge 1$, there exists a unique morphism in $\textbf{D}(\mathcal A)$ (i.e. a finite horizontal composite of distinguished squares) of the form
\[
\dsqua{a}{b}{\partial a}{\partial b}{}{}{\textbf{p}_a}{\textbf{p}_b}
\]
with $\textbf{p}_b \in \textbf{P}$. Note that $\mathcal A$ has a basis precisely if every connected component of $\textbf{D}(\mathcal A)$ has a terminal object.
\end{definition}

\begin{proposition}
\label{prop: strictification - out of cont cat with a basis}
Suppose that a precontextual category $\mathcal A$ has a basis. Then for any contextual category $\mathcal B$, every weak morphism $F \in [\mathcal A,\mathcal B_w]$ admits a local strictification. Hence, by Proposition \ref{prop: local strictification iff strictifiable}, $F$ is strictifiable.
\end{proposition}

\begin{proof}
Let $\textbf{P}$ be a basis of $\mathcal A$. Writing $\textbf{P}$ also for the corresponding discrete category, let
$$
U:\textbf{D}(\mathcal A) \longrightarrow \textbf{P}
$$
be the functor that sends each length-$1$ display $p$ to the unique $q \in \textbf{P}$ that belongs to the connected component of $p$ in $\textbf{D}(\mathcal A)$; recall that $q$ is a terminal object of that connected component.

For each $p:a \twoheadrightarrow \partial a$ in $\textbf{D}(\mathcal A)$, writing $U(p):a' \twoheadrightarrow \partial a'$, let
\[
\dsqua{a}{a'}{\partial a}{\partial a'}{g_p}{f_p}{p}{U(p)}
\]
the unique morphism from $p$ to $U(p)$ in $\textbf{D}(\mathcal A)$. Applying $F$, we obtain a cartesian square
\[
\squa{F(a)}{F(a')}{F(\partial a)}{F(\partial a'),}{F(g_p)}{F(f_p)}{F(p)}{FU(p)}
\]
which is a morphism in $\textbf{D}_w(\mathcal B)$. Defining $\varepsilon_p$ as the latter morphism, it is straightforward to verify that the family $\varepsilon = (\varepsilon_p)_{p \in \textbf{D}(\mathcal A)}$ is a natural transformation as in
% https://q.uiver.app/#q=WzAsMyxbMCwwLCJcXHRleHRiZntEfShcXG1hdGhjYWwgQSkiXSxbNCwwLCJcXHRleHRiZntEfV93KFxcbWF0aGNhbCBCKSJdLFsyLDEsIlxcdGV4dGJme1B9Il0sWzAsMSwiRl8qIiwwLHsiY3VydmUiOi01fV0sWzAsMiwiVSIsMix7ImN1cnZlIjoxfV0sWzAsMSwiIiwxLHsiY3VydmUiOjUsInN0eWxlIjp7ImJvZHkiOnsibmFtZSI6Im5vbmUifSwiaGVhZCI6eyJuYW1lIjoibm9uZSJ9fX1dLFsyLDEsIkZfKnxfe1xcdGV4dGJmIFB9IiwyLHsiY3VydmUiOjF9XSxbMyw1LCJcXHNpZ21hIiwwLHsic2hvcnRlbiI6eyJzb3VyY2UiOjIwLCJ0YXJnZXQiOjIwfX1dXQ==
\[\begin{tikzcd}[ampersand replacement=\&, row sep=small]
	{\textbf{D}(\mathcal A)} \&\&\&\& {\textbf{D}_w(\mathcal B)} \\
	\&\& {\textbf{P}}
	\arrow[""{name=0, anchor=center, inner sep=0}, "{F_*}", curve={height=-30pt}, from=1-1, to=1-5]
	\arrow[""{name=1, anchor=center, inner sep=0}, curve={height=30pt}, draw=none, from=1-1, to=1-5]
	\arrow["U"', curve={height=6pt}, from=1-1, to=2-3]
	\arrow["{F_*|_{\textbf P}}"', curve={height=6pt}, from=2-3, to=1-5]
	\arrow["\sigma", shorten=10pt, Rightarrow, from=0, to=1]
\end{tikzcd}\]
Now, choose for each $q:c \twoheadrightarrow \partial c$ in $\textbf{P}$ an isomorphism $h_q:F(q) \rightarrow b(q)$ in $\mathcal B$ such that $\partial b(q) = F(\partial c)$ and $\textbf{p}_{b(q)} \circ h_q = F(q)$. Letting $\eta_q$ be the morphism
% https://q.uiver.app/#q=WzAsNCxbMCwwLCJGKGMpIl0sWzAsMSwiRihcXHBhcnRpYWwgYykiXSxbMSwxLCJGKFxccGFydGlhbCBjKSJdLFsxLDAsImJfcSJdLFsxLDIsImlkIiwyXSxbMCwxLCJGKHEpIiwyLHsic3R5bGUiOnsiYm9keSI6eyJuYW1lIjoic3F1aWdnbHkifSwiaGVhZCI6eyJuYW1lIjoiZXBpIn19fV0sWzMsMiwiXFx0ZXh0YmZ7cH1fe2JfcX0iLDAseyJzdHlsZSI6eyJoZWFkIjp7Im5hbWUiOiJlcGkifX19XSxbMCwzLCJoX3EiXV0=
\[\begin{tikzcd}[ampersand replacement=\&]
	{F(c)} \& {b(q)} \\
	{F(\partial c)} \& {F(\partial c)}
	\arrow["{h_q}", from=1-1, to=1-2]
	\arrow["{F(q)}"', from=1-1, to=2-1]
	\arrow["{\textbf{p}_{b(q)}}", two heads, from=1-2, to=2-2]
	\arrow["id"', from=2-1, to=2-2]
\end{tikzcd}\]
in $\textbf{D}_w(\mathcal B)$, we have a natural transformation $\eta = (\eta_q)_{q \in \textbf{P}}$ as in
% https://q.uiver.app/#q=WzAsNCxbMCwwLCJcXHRleHRiZntEfShcXG1hdGhjYWwgQSkiXSxbMywwLCJcXHRleHRiZntEfV93KFxcbWF0aGNhbCBCKSJdLFsxLDAsIlxcdGV4dGJme1B9Il0sWzIsMSwiXFx0ZXh0YmZ7RH0oXFxtYXRoY2FsIEIpIl0sWzAsMSwiRl8qIiwwLHsiY3VydmUiOi01fV0sWzMsMSwiIiwyLHsiY3VydmUiOjIsInN0eWxlIjp7InRhaWwiOnsibmFtZSI6Imhvb2siLCJzaWRlIjoidG9wIn19fV0sWzAsMiwiVSIsMl0sWzIsMywiYiAiLDIseyJjdXJ2ZSI6Mn1dLFsyLDEsIkZfKnxfe1xcdGV4dGJmIFB9Il0sWzAsMSwiIiwyLHsic3R5bGUiOnsiYm9keSI6eyJuYW1lIjoibm9uZSJ9LCJoZWFkIjp7Im5hbWUiOiJub25lIn19fV0sWzgsMywiXFxldGEiLDAseyJzaG9ydGVuIjp7InNvdXJjZSI6MjB9fV0sWzQsOSwiXFx2YXJlcHNpbG9uIiwwLHsic2hvcnRlbiI6eyJzb3VyY2UiOjIwLCJ0YXJnZXQiOjIwfX1dXQ==
\[\begin{tikzcd}[ampersand replacement=\&, column sep=huge]
	{\textbf{D}(\mathcal A)} \& {\textbf{P}} \&\& {\textbf{D}_w(\mathcal B)} \\
	\&\& {\textbf{D}(\mathcal B)}
	\arrow["U"', from=1-1, to=1-2]
	\arrow[""{name=0, anchor=center, inner sep=0}, "{F_*}", curve={height=-45pt}, from=1-1, to=1-4]
	\arrow[""{name=1, anchor=center, inner sep=0}, draw=none, from=1-1, to=1-4]
	\arrow[""{name=2, anchor=center, inner sep=0}, "{F_*|_{\textbf P}}", from=1-2, to=1-4]
	\arrow["{b }"', curve={height=12pt}, from=1-2, to=2-3]
	\arrow["\iota"',curve={height=12pt}, hook, from=2-3, to=1-4]
	\arrow["\varepsilon", shorten=5pt, Rightarrow, from=0, to=1]
	\arrow["\eta", shorten=3pt, Rightarrow, from=2, to=2-3]
\end{tikzcd}\]
Pasting this diagram yields a natural transformation $F_* \Rightarrow \iota bU$, which is a local strictification of $F$. We conclude from Proposition \ref{prop: local strictification iff strictifiable} that $F$ is strictifiable.
\end{proof}

\begin{corollary}
\label{cor: strictification - cellular domain}
Let $\mathcal A$, $\mathcal B$ contextual categories such that $\mathcal A$ is (small and) cellular. Then every weak morphism $F:\mathcal A_w \rightarrow \mathcal B_w$ is strictifiable.
\end{corollary}

\begin{proof}
By Proposition \ref{prop: cellular presentation -> basis}, $\mathcal A$ has a basis.
\end{proof}

\subsubsection{The homotopy $(\infty,1)$-category of $\Cont$}

\begin{proposition}
\label{prop: strictification - cofibrant domain}
Let $\mathcal A$, $\mathcal B$ be contextual categories such that $\mathcal A$ is (small and) cofibrant and $\mathcal B$ is locally small. Then the inclusion functor $\iiCont(\mathcal A,\mathcal B) \hookrightarrow \iiDMC_r(\mathcal A_w,\mathcal B_w)$ is an equivalence of categories.
\end{proposition}

\begin{proof}
Let $\mathcal A' \in \Cont$ be a cellular object such that $\mathcal A$ is a retract of $\mathcal A'$. By Corollary \ref{cor: strictification - cellular domain}, $\iiCont(\mathcal A',\mathcal B) \hookrightarrow \iiDMC_r(\mathcal A'_r,\mathcal B_w)$ is an equivalence; hence, so is $\iiCont(\mathcal A,\mathcal B) \hookrightarrow \iiDMC_r(\mathcal A_r,\mathcal B_w)$ as it is a retract of the former functor.
\end{proof}

\begin{proposition}
\label{prop: hom-category out of cofibrant cont cat}
Let $\iiCont^c$ be the full sub-$2$-category of $\iiCont$ spanned by its cofibrant objects. The composite $2$-functor
$$
\iiCont^c \hookrightarrow \iiCont \xrightarrow{(-)_w} \iiDMC_r
$$
is an equivalence of $2$-categories.
\end{proposition}

\begin{proof}
By Proposition \ref{prop: strictification - cofibrant domain}, $\iiCont^c \hookrightarrow \iiDMC_r$ is (essentially) $2$-full-and-faithful. On the other hand, we know from \cite[Th. 8.4.10]{Tay99} or Proposition \ref{prop: strictification for a rooted dmc} that every rooted display map category is equivalent to a contextual category.
\end{proof}

Hence, by the discussion in \S\ref{subsubsec: spaces of maps}, we obtain:

\begin{theorem}
	\label{th: homotopy bicategory of Cont}
	The $2$-functor $\iiCont^c \hookrightarrow \iiDMC_r$ induces an equivalence between the $(\infty,1)$-category $\iiHo(\Cont)$ and the underlying $(2,1)$-category of $\iiDMC_r$.
\end{theorem}

\subsubsection{Weak morphisms out of precontextual categories}
\label{subsubsec: weak morphisms out of precontextual categories}

A fundamental feature of the assignment $\bbA \mapsto \mathcal C(\bbA)$ for $\bbA \in \GAT$ is its ($1$-categorical) universal property: for a contextual category $\mathcal S$, strict morphisms $\mathcal C(\mathcal A) \rightarrow \mathcal S$ can be described by recursively interpreting the axioms of $\bbA$ in (the gat associated with) $\mathcal S$. The introduction of precontextual categories in \cite{Alm26} (see Def. \ref{def: (pre)contextual category} and Rem. \ref{rem: use of precontextual categories}) was guided by a similar idea, but from a diagrammatic perspective.

We know that often (for example, when $\bbA$ has no sort equality axioms) the category of strict morphisms $\mathcal C(\bbA) \rightarrow \mathcal S$ is equivalent to that of weak morphisms. On the other hand, it also makes sense about weak morphisms and strictification out of precontextual categories (Def. \ref{def: weak morphism and D(A) for precontextual category}, Prop. \ref{prop: strictification - out of cont cat with a basis}). Putting together our previous results, the following result gives sufficient conditions under which weak morphisms out of a precontextual category $\mathcal A$ classify weak morphisms out of $L\mathcal A$. Under those conditions, $\mathcal A \rightarrow L\mathcal A$ is, in particular, a Morita equivalence of finite-limit sketches in that it induces equivalences between categories of models in any finitely complete category.

\begin{proposition}
\label{prop: weak morphisms precontextual category, assuming basis}
Let $\mathcal A$ be a small precontextual category, and suppose that both $\mathcal A$ and the contextual category $L\mathcal A$ have a basis (Definition \ref{def: basis precont}). Then for every locally small contextual category $\mathcal S$, the functor
$$
\iota_\mathcal A^*:[L\mathcal A,\mathcal S]_w \longrightarrow [\mathcal A,\mathcal S]_w
$$
given by precomposition with the reflection functor $\iota_\mathcal A:\mathcal A \rightarrow L\mathcal A$ is an equivalence of categories.
\end{proposition}

\begin{proof}
In the commutative diagram
% https://q.uiver.app/#q=WzAsNCxbMCwxLCJcXGlpTW9kX3coTCBcXG1hdGhjYWwgQSkiXSxbMSwxLCJcXGlpTW9kX3coXFxtYXRoY2FsIEEpIl0sWzAsMCwiXFxpaU1vZF9zKEwgXFxtYXRoY2FsIEEpIl0sWzEsMCwiXFxpaU1vZF9zKFxcbWF0aGNhbCBBKSJdLFsyLDMsIlxcaW90YV9cXG1hdGhjYWwgQV4qIl0sWzAsMSwiXFxpb3RhX1xcbWF0aGNhbCBBXioiLDJdLFsyLDAsIiIsMSx7InN0eWxlIjp7InRhaWwiOnsibmFtZSI6Imhvb2siLCJzaWRlIjoidG9wIn19fV0sWzMsMSwiIiwxLHsic3R5bGUiOnsidGFpbCI6eyJuYW1lIjoiaG9vayIsInNpZGUiOiJ0b3AifX19XV0=
\[\begin{tikzcd}[ampersand replacement=\&,cramped]
	{[L\mathcal A,\mathcal S]_s} \& {[\mathcal A,\mathcal S]_s} \\
	{[L\mathcal A,\mathcal S]_w} \& {[\mathcal A,\mathcal S]_w,}
	\arrow["{\iota_\mathcal A^*}", from=1-1, to=1-2]
	\arrow[hook, from=1-1, to=2-1]
	\arrow[hook, from=1-2, to=2-2]
	\arrow["{\iota_\mathcal A^*}"', from=2-1, to=2-2]
\end{tikzcd}\]
the two vertical functors are equivalences of categories by Proposition \ref{prop: strictification - out of cont cat with a basis}, and the top one is an isomorphism by \cite[Cor. 5.33]{Alm26}. The $2$-out-of-$3$ property implies that the bottom functor is also an equivalence.
\end{proof}

The following corollary, which will be used in Example \ref{ex: E(K)}, describes a setting where we can calculate weak morphisms out of $\mathcal A \in \Cont$ as weak morphisms out of a suitable full subcategory of $\mathcal A$ equipped with the restricted precontextual category structure.

\begin{corollary}
\label{cor: weak morphisms full precontextual subcategory}
Consider $\mathcal Q \in \Precont$ such that $\mathcal Q$ and $L\mathcal Q$ have a basis, $\mathcal A \in \Cont$, and a full-and-faithful morphism $I:\mathcal Q \rightarrow \mathcal A$ in $\Precont$ such that the induced strict morphism $L\mathcal Q \rightarrow \mathcal A$ is a weak equivalence. Let $\mathcal Q'$ be the image of $I$ equipped with the precontextual category structure whose length function and display maps are inherited from $\mathcal A$, and where a commutative square is distinguished precisely if it is distinguished in $\mathcal A$. For every locally small contextual category $\mathcal S$, the functors
$$
[\mathcal A,\mathcal S]_w \overset{j^*}{\longrightarrow} [\mathcal Q',\mathcal S]_w \overset{i^*}{\longrightarrow} [\mathcal Q,\mathcal S]_w
$$
given by precomposition with the inclusions $\mathcal Q \xrightarrow{i} \mathcal Q' \xrightarrow{j} \mathcal A$ are equivalences of categories.
\end{corollary}

\begin{proof}
By Proposition \ref{prop: weak morphisms precontextual category, assuming basis} and the fact that weak equivalences in $\Cont$ induce equivalences between categories of weak morphisms, we can decompose $I^*$ as a composite of equivalences $[\mathcal A,\mathcal S]_w \simeq [L\mathcal Q,\mathcal S]_w \simeq [\mathcal Q,\mathcal S]_w$. In particular, this yields that $i^*$ is essentially surjective. Also, it follows from $i$ being an equivalence of categories that $i^*$ is full-and-faithful. We conclude that $i^*$ is an equivalence, and so is $j^*$ by the $2$-out-of-$3$ property.
\end{proof}

\section{Strictifying set-valued models}
\label{sec: strictifying set-valued models}

\begin{definition}[Family-valued and set-valued models]
\label{def: family-valued and set-valued models}
There are two commonly considered concepts of model of a contextual category:
\begin{itemize}
	\item Strict morphisms $\mathcal A \rightarrow \Fam$, where $\Fam$ is the contextual category, introduced in \cite{Car86} and briefly described in \S1, of iterated families of sets. Alternatively, it is isomorphic to the image under the coreflection $\Att^+ \rightarrow \Cont^+$\footnote{$\Att^+$ (resp. $\Cont^+$) denotes the category of locally small categories with attributes (resp. locally small contextual categories).}, described in \cite[Prop. 4.4]{KapLum18}, of the category with attributes $\Fam_{\text{att}}$ given by
	\[\begin{tikzcd}[ampersand replacement=\&,row sep=tiny]
		\smallint \mathscr U^{(-)} \&\& {\iiAr(\Set)} \\
		\& {\Set}
		\arrow["\coprod", from=1-1, to=1-3]
		\arrow["\pi"', from=1-1, to=2-2]
		\arrow["{\textbf{t}}", from=1-3, to=2-2]
	\end{tikzcd}\]
	where: $\mathscr U^{(-)}:\Set^{\op} \rightarrow \Set^+$ sends $X$ to the set $\mathscr U^X$ of all families of small sets $Y = (Y_x)_{x \in X}$; and $\coprod(X,Y) = \coprod_{x \in X}Y_x$. The \emph{category of family-valued models} of $\mathcal A$ is
	$$
	\iiMod_s(\mathcal A) = \iiCont^+(\mathcal A, \Fam).
	$$
	
	\item Functors $|\mathcal A| \rightarrow \Set$ that send distinguished squares to pullback squares, and the distinguished terminal object to a terminal object. Since this condition only depends on $\mathcal A_w$ (Definition \ref{def: weakening functor}) and $\Set$ is not a contextual category, it is more natural to work by default in the setting of display map categories: for $\mathcal A \in \iiDMC_r$, the \emph{category of set-valued models} (or \emph{set-models}) of $\mathcal A$, which we denote by $\iiMod_w(\mathcal A)$, is the full subcategory of $\Set^{|\mathcal A|}$ spanned by the functors that preserve display maps, their pullbacks, and terminal objects. Thus
	$$
	\iiMod_w(\mathcal A) = \iiDMC_r^+(\mathcal A_w,\Set)
	$$
	where $\Set$ is viewed as a rooted dmc with every arrow as a display map.\footnote{It is not difficult to check that set-models in this sense coincide with morphisms of clans (\cite{Joy17}) $\mathcal A \rightarrow \Set$, where we regard a rooted dmc as a clan by taking as fibrations the finite composites of display maps.}
\end{itemize}
\end{definition}

The $\Fam_{\text{att}}$-component of the counit of the adjunction between $\Att^+$ and $\Cont^+$ is precisely the ``total set" equivalence of categories $\Sigma:|\Fam| \rightarrow \Set$ from \S1. Since every arrow in $\Fam$ is isomorphic to a strict display map (it suffices to consider $a \twoheadrightarrow \partial a$ where $\ell(a) = 2$), we conclude that $\Sigma:\Fam_w \longrightarrow \Set$ is an equivalence in $\iiDMC_r$. This yields for each $\mathcal A \in \iiCont$ a natural full-and-faithful comparison functor
$$
\iiMod_s(\mathcal A) = \iiCont^+(\mathcal A,\Fam) \longrightarrow \iiDMC_r^+(\mathcal A_w,\Fam_w) \simeq \iiDMC_r^+(\mathcal A_w,\Set) = \iiMod_w(\mathcal A_w).
$$

We will now study when a given set-model belongs to the essential image of $\iiMod_s(\mathcal A) \rightarrow \iiMod_w(\mathcal A_w)$, or, equivalently, when a weak morphism $\mathcal A_w \rightarrow \Fam_w$ is strictifiable. We already know, by \S\ref{subsec: strictification basis or cofibrant}, that if $\mathcal A$ has a basis or is cofibrant, then every set-model is strictifiable. Our next goal is to give an explicit characterization, based on Proposition \ref{prop: local strictification iff strictifiable}, that applies to a particular set-model without, in principle, extra assumptions on $\mathcal A$. This will be done by analyzing a certain discrete fibration of groupoids associated with each set-model: the model will be strictifiable precisely if the automorphism group of each point in the base groupoid acts trivially on the corresponding fiber.

This section will rely on a few constructions and results on cartesian natural transformations between set-valued functors, for which we refer to the appendix.

\begin{construction}
Consider $\mathcal A \in \Cont$ and a weak morphism $F:\mathcal A_w \rightarrow \Set$, and let
$$
F_*:\textbf{D}(\mathcal A) \longrightarrow \textbf{D}_w(\Set)
$$
be the induced functor. Define $F_\textbf{s} = \textbf{s} \circ F_*$ and $F_\textbf{t} = \textbf{t} \circ F_*$ where $\textbf{s}$, $\textbf{t}:\textbf{D}_w(\Set) \rightarrow \Set$ are the source and target projections, respectively. We have a natural transformation
$$
\pi^F:F_\textbf{s} \Longrightarrow F_\textbf{t}
$$
with components $\pi^F_p = F(p)$. The fact that morphisms in $\textbf{D}_w(\Set)$ correspond to pullback squares in $\Set$ implies that $\pi^F$ is a cartesian natural transformation, i.e. every naturality square is cartesian.
\end{construction}

\begin{notation}
\label{not: gpd}
Let $\Gpd$ be the ($1$-)category of small groupoids and functors between them. Recall that the inclusion $\Gpd \rightarrow \Cat$ has a left adjoint that sends a category $C$ to $C[\Ar(C)^{-1}]$, that is, its localization at the set of all arrows. We will denote this left adjoint by $\gpd$.
\end{notation}

\begin{construction}
\label{constr: discrete opfibration associated with a set-valued model}
Consider $F:\mathcal A_w \rightarrow \Set$ as above. Taking categories of elements, we obtain from $\pi^F:F_\textbf{s} \Rightarrow F_\textbf{t}$ a strictly commutative diagram of discrete opfibrations\footnote{Since the two projections to $\textbf{D}(\mathcal A)$ are discrete opfibrations, so is $\pi_*^F$.}
\[
\begin{tikzcd}[row sep=small]
	\smallint F_\textbf{s} \arrow[]{rr}{\pi^F_*} \arrow[]{dr}{} & & \smallint F_\textbf{t} \arrow[]{dl}{} \\
	 & \textbf{D}(\mathcal A). & 
\end{tikzcd}
\]
It follows from $\pi^F$ being a cartesian natural transformation that the fiber functor of $\pi^F_*$, which we will denote by
$$
\Phi^F:\smallint F_\textbf{t} \longrightarrow \Set,
$$
sends every arrow to an isomorphism. Now, the universal property of the localization functor $\Lambda:\smallint F_\textbf{t} \rightarrow \gpd(\smallint F_\textbf{t})$ implies that there exists a unique functor $\gpd(\smallint F_\textbf{t}) \rightarrow \Set$ making the diagram
\[
\begin{tikzcd}
	\smallint F_\textbf{t} \arrow[]{r}{\Phi^F} \arrow[swap]{d}{\Lambda} & \Set\\
	\gpd(\smallint F_\textbf{t}) \arrow[dashed]{ur}{} &
\end{tikzcd}
\]
commute strictly. It will be denoted by $\overline{\Phi}^F$.
\end{construction}

\begin{remark}
The above diagram induces a bijective-on-objects comparison functor $\smallint F_\textbf{s} \cong \smallint \Phi^F \longrightarrow \smallint \overline{\Phi}^F$, hence, as every arrow in $\smallint \overline{\Phi}^F$ is an isomorphism, a functor $I:\gpd(\smallint \Phi^F) \longrightarrow \smallint \overline{\Phi}^F$. In fact, he latter is an isomorphism; we prove this in Proposition \ref{prop: comparison functor is an isomorphism}.
\end{remark}

\begin{definition}
\label{def: loop-free model}
In the notation of the above construction, we say that $F:\mathcal A_w \rightarrow \Set$ is \emph{loop-free} if $\overline{\Phi}^F:\gpd(\smallint F_\textbf{t}) \rightarrow \Set$ sends every automorphism to an identity map. Note that this is equivalent to the requirement that $F(f) = F(g)$ whenever $f$, $g$ are parallel arrows in $\gpd(\smallint F_\textbf{t})$.
\end{definition}

Our main goal for the remainder of this section is to prove that a weak morphism $F:\mathcal A_w \rightarrow \Set$ is strictifiable if and only if it is loop-free (Theorem \ref{th: characterization strictifiable set-model}). One implication is verified in Proposition \ref{prop: strictifiable -> loop-free}, and the other one in Proposition \ref{prop: loop-free -> strictifiable}. We will also obtain (and use in the proof of the main result) an explicit characterization of loop freeness in terms of zigzags of elements of the given model $F:\mathcal A_w \rightarrow \Set$.

\subsection{Characterizing strictifiable set-models}

\begin{definition}
\label{def: zigzag}
For a weak model $F:\mathcal A_w \rightarrow \Set$, an \emph{$F$-zigzag} consists of
\begin{itemize}
	\item $n \ge 0$ and a diagram of distinguished squares in $\mathcal A$ of the form
	% https://q.uiver.app/#q=WzAsMTQsWzAsMCwiYV8xIl0sWzAsMSwiXFxwYXJ0aWFsIGFfMSJdLFsxLDAsImFfMiJdLFsxLDEsIlxccGFydGlhbCBhXzIiXSxbMiwwLCJhXzMiXSxbMiwxLCJcXHBhcnRpYWwgYV8zIl0sWzMsMCwiXFxjZG90cyJdLFszLDEsIlxcY2RvdHMiXSxbNCwwLCJhX3sybi0xfSJdLFs0LDEsIlxccGFydGlhbCBhX3sybi0xfSJdLFs1LDAsImFfezJufSJdLFs2LDAsImFfezJuKzF9ID0gYV8wIl0sWzYsMSwiXFxwYXJ0aWFsIGFfezJuKzF9ID0gXFxwYXJ0aWFsIGFfMCJdLFs1LDEsIlxccGFydGlhbCBhX3sybn0iXSxbMiwwLCJnXzEiLDJdLFszLDEsImZfMSJdLFswLDEsIiIsMSx7InN0eWxlIjp7ImhlYWQiOnsibmFtZSI6ImVwaSJ9fX1dLFsyLDMsIiIsMSx7InN0eWxlIjp7ImhlYWQiOnsibmFtZSI6ImVwaSJ9fX1dLFsyLDQsImdfMiJdLFszLDUsImZfMiIsMl0sWzYsNCwiZ18zIiwyXSxbNCw1LCIiLDEseyJzdHlsZSI6eyJoZWFkIjp7Im5hbWUiOiJlcGkifX19XSxbNyw1LCJmXzMiXSxbNiw4LCJnX3sybi0yfSJdLFs3LDksImZfezJuLTJ9IiwyXSxbOCw5LCIiLDEseyJzdHlsZSI6eyJoZWFkIjp7Im5hbWUiOiJlcGkifX19XSxbMTAsOCwiZ197Mm4tMX0iLDJdLFsxMCwxMSwiZ197Mm59Il0sWzEzLDEyLCJmX3sybn0iLDJdLFsxMyw5LCJmX3sybi0xfSJdLFsxMCwxMywiIiwyLHsic3R5bGUiOnsiaGVhZCI6eyJuYW1lIjoiZXBpIn19fV0sWzExLDEyLCIiLDEseyJzdHlsZSI6eyJoZWFkIjp7Im5hbWUiOiJlcGkifX19XV0=
	\[\begin{tikzcd}[ampersand replacement=\&]
		{a_1} \& {a_2} \& {a_3} \& \cdots \& {a_{2n-1}} \& {a_{2n}} \& {a_{2n+1} = a_1} \\
		{\partial a_1} \& {\partial a_2} \& {\partial a_3} \& \cdots \& {\partial a_{2n-1}} \& {\partial a_{2n}} \& {\partial a_{2n+1} = \partial a_1}
		\arrow[two heads, from=1-1, to=2-1]
		\arrow["{g_1}"', from=1-2, to=1-1]
		\arrow["{g_2}", from=1-2, to=1-3]
		\arrow[two heads, from=1-2, to=2-2]
		\arrow[two heads, from=1-3, to=2-3]
		\arrow["{g_3}"', from=1-4, to=1-3]
		\arrow["{g_{2n-2}}", from=1-4, to=1-5]
		\arrow[two heads, from=1-5, to=2-5]
		\arrow["{g_{2n-1}}"', from=1-6, to=1-5]
		\arrow["{g_{2n}}", from=1-6, to=1-7]
		\arrow[two heads, from=1-6, to=2-6]
		\arrow[two heads, from=1-7, to=2-7]
		\arrow["{f_1}", from=2-2, to=2-1]
		\arrow["{f_2}"', from=2-2, to=2-3]
		\arrow["{f_3}", from=2-4, to=2-3]
		\arrow["{f_{2n-2}}"', from=2-4, to=2-5]
		\arrow["{f_{2n-1}}", from=2-6, to=2-5]
		\arrow["{f_{2n}}"', from=2-6, to=2-7]
	\end{tikzcd}\]
	
	\item a tuple $(x_i)_{i \le 2n+1} \in \prod_{i \le 2n+1}F(\partial a_i)$ such that $x_{2n+1} = x_1$ and, for $1 \le i \le n$,
	$$
	F(f_{2i-1})(x_{2i}) = x_{2i-1}, \qquad F(f_{2i})(x_{2i}) = x_{2i+1}.
	$$
\end{itemize}

For $1 \le i \le 2n$, we let $h_i:F(\textbf{p}_{a_i})^{-1}(x_i) \rightarrow F(\textbf{p}_{a_{i+1}})^{-1}(x_{i+1})$ be given by\footnote{Below, we write $f_i^*$ for the inverse of the bijection $F(\textbf{p}_{a_{i+1}})^{-1}(x_{i+1}) \rightarrow F(\textbf{p}_{a_i})^{-1}(x_i)$ induced by $F(g_i)$ using that $F$ sends each of the squares in the zigzag to a pullback square of sets.}
$$
h_i(y) =
\begin{cases}
	F(g_i)(y), & \text{if } i \text{ is even},\\
	f_i^*(y) & \text{if } i \text{ is odd}.
\end{cases}
$$
We say that this $F$-zigzag \emph{acts trivially} if the composite $h_{2n} \cdots  h_2 h_1$ is the identity on $F(\textbf{p}_{a_1})^{-1}(x_1)$.
\end{definition}

\begin{example}
Diagram (\texttt{***}) from \S\ref{sec: introduction} denotes the failure of a zigzag to act trivially.
\end{example}

\begin{proposition}
\label{prop: loop-freeness via zigzags}
$F:\mathcal A_w \rightarrow \Set$ is loop-free if and only if every $F$-zigzag acts trivially.
\end{proposition}

\begin{proof}
Since a morphism in $\smallint F_\textbf{t}$ can be described as a distinguished square
\[
\dsqua{a}{a'}{\partial a}{\partial a'}{g}{f}{}{}
\]
equipped with $x \in F(\partial a)$, $x' \in F(\partial a')$ such that $F(f)(x) = x'$, we have a canonical surjective map from the set of all $F$-zigzags to the set of arrows of $\gpd(\smallint F_\textbf{t})$; see Construction \ref{constr: gpd(K) using zig-zags}. By construction, for an $F$-zigzag as in Definition \ref{def: zigzag}, the automorphism $h_{2n} \cdots h_2h_1$ of $F(\textbf{p}_{a_1})^{-1}(x_1) = \overline{\Phi}^F(\textbf{p}_{a_1}, x_1)$ coincides with the image under $\overline{\Phi}^F$ of the arrow in $\gpd(\smallint F_\textbf{t})$ corresponding to the zigzag. Hence, $F$ sends every morphism to an identity map if and only if every $F$-zigzag acts trivially.
\end{proof}

We can now verify one direction of our characterization of strictifiability for set-models:

\begin{proposition}
\label{prop: strictifiable -> loop-free}
If a weak morphism $F:\mathcal A_w \rightarrow \Set$ is strictifiable, then it is loop-free.
\end{proposition}

\begin{proof}
Assuming that $F$ is strictifiable, consider a strict morphism $G:\mathcal A \rightarrow \Fam$ such that the set-model $UG$ is isomorphic to $F$. By Proposition \ref{prop: loop-freeness via zigzags}, it suffices to prove that every $F$-zigzag acts trivially, which is equivalent to every $UG$-zigzag acting trivially. To verify that, consider a diagram
\[\begin{tikzcd}[ampersand replacement=\&]
	{a_1} \& {a_2} \& {a_3} \& \cdots \& {a_{2n-1}} \& {a_{2n}} \& {a_{2n+1} = a_1} \\
	{\partial a_1} \& {\partial a_2} \& {\partial a_3} \& \cdots \& {\partial a_{2n-1}} \& {\partial a_{2n}} \& {\partial a_{2n+1} = \partial a_1}
	\arrow[two heads, from=1-1, to=2-1]
	\arrow["{g_1}"', from=1-2, to=1-1]
	\arrow["{g_2}", from=1-2, to=1-3]
	\arrow[two heads, from=1-2, to=2-2]
	\arrow[two heads, from=1-3, to=2-3]
	\arrow["{g_3}"', from=1-4, to=1-3]
	\arrow["{g_{2n-2}}", from=1-4, to=1-5]
	\arrow[two heads, from=1-5, to=2-5]
	\arrow["{g_{2n-1}}"', from=1-6, to=1-5]
	\arrow["{g_{2n}}", from=1-6, to=1-7]
	\arrow[two heads, from=1-6, to=2-6]
	\arrow[two heads, from=1-7, to=2-7]
	\arrow["{f_1}", from=2-2, to=2-1]
	\arrow["{f_2}"', from=2-2, to=2-3]
	\arrow["{f_3}", from=2-4, to=2-3]
	\arrow["{f_{2n-2}}"', from=2-4, to=2-5]
	\arrow["{f_{2n-1}}", from=2-6, to=2-5]
	\arrow["{f_{2n}}"', from=2-6, to=2-7]
\end{tikzcd}\]
of distinguished squares in $\mathcal A$. Applying $G$, we obtain a diagram
% https://q.uiver.app/#q=WzAsMTQsWzAsMCwiRyhhXzEpIl0sWzAsMSwiRyhcXHBhcnRpYWwgYV8xKSJdLFsxLDAsIkcoYV8yKSJdLFsxLDEsIkcoXFxwYXJ0aWFsIGFfMikiXSxbMiwwLCJHKGFfMykiXSxbMiwxLCJHKFxccGFydGlhbCBhXzMpIl0sWzMsMCwiXFxjZG90cyJdLFszLDEsIlxcY2RvdHMiXSxbNCwwLCJHKGFfezJuLTF9KSJdLFs0LDEsIkcoXFxwYXJ0aWFsIGFfezJuLTF9KSJdLFs1LDAsIkcoYV97Mm59KSJdLFs2LDAsIkcoYV97Mm4rMX0pID0gRyhhXzApIl0sWzYsMSwiRyhcXHBhcnRpYWwgYV97Mm4rMX0pID0gRyhcXHBhcnRpYWwgYV8wKSJdLFs1LDEsIkcoXFxwYXJ0aWFsIGFfezJufSkiXSxbMiwwLCJHKGdfMSkiLDJdLFszLDEsIkcoZl8xKSJdLFswLDEsIiIsMSx7InN0eWxlIjp7ImhlYWQiOnsibmFtZSI6ImVwaSJ9fX1dLFsyLDMsIiIsMSx7InN0eWxlIjp7ImhlYWQiOnsibmFtZSI6ImVwaSJ9fX1dLFsyLDQsIkcoZ18yKSJdLFszLDUsIkcoZl8yKSIsMl0sWzYsNCwiRyhnXzMpIiwyXSxbNCw1LCIiLDEseyJzdHlsZSI6eyJoZWFkIjp7Im5hbWUiOiJlcGkifX19XSxbNyw1LCJHKGZfMykiXSxbNiw4LCJHKGdfezJuLTJ9KSJdLFs3LDksIkcoZl97Mm4tMn0pIiwyXSxbOCw5LCIiLDEseyJzdHlsZSI6eyJoZWFkIjp7Im5hbWUiOiJlcGkifX19XSxbMTAsOCwiRyhnX3sybi0xfSkiLDJdLFsxMCwxMSwiRyhnX3sybn0pIl0sWzEzLDEyLCJHKGZfezJufSkiLDJdLFsxMyw5LCJHKGZfezJuLTF9KSJdLFsxMCwxMywiIiwyLHsic3R5bGUiOnsiaGVhZCI6eyJuYW1lIjoiZXBpIn19fV0sWzExLDEyLCIiLDEseyJzdHlsZSI6eyJoZWFkIjp7Im5hbWUiOiJlcGkifX19XV0=
\[\begin{tikzcd}[ampersand replacement=\&]
	{G(a_1)} \& {G(a_2)} \& {G(a_3)} \& \cdots \& {G(a_{2n-1})} \& {G(a_{2n})} \& {G(a_{2n+1}) = G(a_1)} \\
	{G(\partial a_1)} \& {G(\partial a_2)} \& {G(\partial a_3)} \& \cdots \& {G(\partial a_{2n-1})} \& {G(\partial a_{2n})} \& {G(\partial a_{2n+1}) = G(\partial a_1)}
	\arrow[two heads, from=1-1, to=2-1]
	\arrow["{G(g_1)}"', from=1-2, to=1-1]
	\arrow["{G(g_2)}", from=1-2, to=1-3]
	\arrow[two heads, from=1-2, to=2-2]
	\arrow[two heads, from=1-3, to=2-3]
	\arrow["{G(g_3)}"', from=1-4, to=1-3]
	\arrow["{G(g_{2n-2})}", from=1-4, to=1-5]
	\arrow[two heads, from=1-5, to=2-5]
	\arrow["{G(g_{2n-1})}"', from=1-6, to=1-5]
	\arrow["{G(g_{2n})}", from=1-6, to=1-7]
	\arrow[two heads, from=1-6, to=2-6]
	\arrow[two heads, from=1-7, to=2-7]
	\arrow["{G(f_1)}", from=2-2, to=2-1]
	\arrow["{G(f_2)}"', from=2-2, to=2-3]
	\arrow["{G(f_3)}", from=2-4, to=2-3]
	\arrow["{G(f_{2n-2})}"', from=2-4, to=2-5]
	\arrow["{G(f_{2n-1})}", from=2-6, to=2-5]
	\arrow["{G(f_{2n})}"', from=2-6, to=2-7]
\end{tikzcd}\]
of distinguished squares in $\Fam$. For $i = 1$, ..., $2n+1$, let
$$
Y_i:UG(\partial a_i) \longrightarrow \mathscr U
$$
be the family of small sets that classifies $G(a_i)$. Note that
\begin{itemize}
	\item if $i$ is even, then $G(g_i):G(a_i) \rightarrow G(a_{i+1})$ is the map
	$$
	\coprod_{x \in UG(\partial a_i)}Y_i(x) = \coprod_{x \in UG(\partial a_i)}Y_{i+1}(G(f_i)(x)) \longrightarrow \coprod_{x' \in UG(\partial a_{i+1})}Y_{i+1}(x')
	$$
	that sends the fiber over $x$ to the one over $G(f_i)(x)$ via the identity map, that is, $(x,y) \mapsto (G(f_i)(x), y)$ for all $x \in UG(\partial a_i)$, $y \in Y_i(x)$;
	
	\item if $i$ is odd, then $G(g_i):G(a_{i+1}) \rightarrow G(a_i)$ is the map
	$$
	\coprod_{x \in UG(\partial a_{i+1})}Y_{i+1}(x) = \coprod_{x \in UG(\partial a_{i+1})}Y_i(G(f_i)(x)) \longrightarrow \coprod_{x' \in UG(\partial a_i)}Y_i(x')
	$$
	given by $(x,y) \mapsto (G(f_i)(x), y)$.
\end{itemize}

Now, consider a tuple $(x_i)_i \in \prod_{i \le 2n+1}UG(\partial a_i)$ as in Definition \ref{def: zigzag}.

For $1 \le i \le 2n$, the map $h_i:UG(\textbf{p}_{a_i})^{-1}(x_i) \rightarrow UG(\textbf{p}_{a_{i+1}})^{-1}(x_{i+1})$ is given by
$$
h_i(x_i,y) =
\begin{cases}
	G(g_i)(x_i,y) = (x_{i+1},y), & \text{if } i \text{ is even},\\
	f_i^*(x_i,y) = (x_{i+1}, y) & \text{if } i \text{ is odd},
\end{cases}
$$
from which it follows that the composite $h_{2n} \cdots h_2 h_1$ is given by $(x_1,y) \mapsto (x_{2n+1}, y) = (x_1,y)$, as required.
\end{proof}

To prove the converse statement, we will use the characterization of loop freeness given in Proposition \ref{prop: characterization loop-free natural transformation} applied to the cartesian natural transformation $\pi^F:F_\textbf{s} \Rightarrow F_\textbf{t}$.

\begin{proposition}
\label{prop: loop-free -> strictifiable}
If a set-model is loop-free, then it is strictifiable.
\end{proposition}

\begin{proof}
Consider a set-model $F:\mathcal A_w \rightarrow \Fam_w \simeq \Set$. As every arrow in $\Fam$ is isomorphic to a strict display map, we can choose $K \in \Ob(\Fam)$ of length $\ge 1$ and a commutative square
\[
\tag{\texttt{*}}
\squa{F_\textbf{s}}{\Delta K}{F_\textbf{t}}{\Delta \partial K}{\alpha}{\beta}{\pi^F}{\Delta \textbf{p}_K}
\]
in $|\Fam|^{\textbf{D}(\mathcal A)}$ such that $\alpha$, $\beta$ are colimit cocones. If $F$ is loop-free, Proposition \ref{prop: characterization loop-free natural transformation} implies that evaluating (\texttt{*}) at each $p \in \textbf{D}(\mathcal A)$ yields a pullback square in $\Fam$. Thus we have a natural transformation
% https://q.uiver.app/#q=WzAsNCxbMCwwLCJcXHRleHRiZntEfShcXG1hdGhjYWwgQSkiXSxbMywwLCJcXHRleHRiZntEfV93KFxcRmFtKSJdLFsyLDEsIlxcdGV4dGJme0R9KFxcRmFtKSJdLFsxLDEsIlxcdGV4dGJmezF9Il0sWzAsMSwiRl8qIiwwLHsiY3VydmUiOi01fV0sWzIsMSwiIiwyLHsiY3VydmUiOjEsInN0eWxlIjp7InRhaWwiOnsibmFtZSI6Imhvb2siLCJzaWRlIjoidG9wIn19fV0sWzMsMl0sWzAsMywiIiwyLHsiY3VydmUiOjF9XSxbMCwxLCIiLDEseyJjdXJ2ZSI6NSwic3R5bGUiOnsiYm9keSI6eyJuYW1lIjoibm9uZSJ9LCJoZWFkIjp7Im5hbWUiOiJub25lIn19fV0sWzQsOCwiXFxldGEiLDAseyJzaG9ydGVuIjp7InNvdXJjZSI6MjAsInRhcmdldCI6MjB9fV1d
\[\begin{tikzcd}[ampersand replacement=\&, row sep=small]
	{\textbf{D}(\mathcal A)} \&\&\& {\textbf{D}_w(\Fam)} \\
	\& {\textbf{1}} \& {\textbf{D}(\Fam)}
	\arrow[""{name=0, anchor=center, inner sep=0}, "{F_*}", curve={height=-30pt}, from=1-1, to=1-4]
	\arrow[""{name=1, anchor=center, inner sep=0}, curve={height=30pt}, draw=none, from=1-1, to=1-4]
	\arrow["{!}"',curve={height=6pt}, from=1-1, to=2-2]
	\arrow[from=2-2, to=2-3]
	\arrow[curve={height=6pt}, hook, from=2-3, to=1-4]
	\arrow["\eta", shorten=10pt, Rightarrow, from=0, to=1]
\end{tikzcd}\]
where $\textbf{1} \rightarrow \textbf{D}(\Fam)$ sends the single object to $\textbf{p}_K:K \twoheadrightarrow \partial K$, and for $p \in \textbf{D}(\mathcal A)$, $\eta_p$ is the morphism obtained by evaluating (\texttt{*}) at $p$. We conclude from Proposition \ref{prop: local strictification iff strictifiable} that $F$ is strictifiable.
\end{proof}

\begin{theorem}
\label{th: characterization strictifiable set-model}
A set-model $F$ is strictifiable if and only if it is loop-free (hence, if and only if $\pi_*^F:\smallint F_\textbf{s} \Rightarrow \smallint F_\textbf{t}$ satisfies any of the equivalent conditions from Proposition \ref{prop: characterization loop-free natural transformation}).
\end{theorem}

\subsection{Some (counter)examples}
\label{subsec: examples}

For a model $F:\mathcal A_w \rightarrow \Set$, if $\gpd(F_\textbf{t})$ is equivalent to a set, it is immediate that $F$ is loop-free; hence, it is strictifiable. On the other hand, $\gpd(F_\textbf{t})$ being equivalent to a set means precisely that the colimit of $F_\textbf{t}:\textbf{D}(\mathcal A) \rightarrow \Set$ is van Kampen (see Proposition \ref{prop: characterization loop-free natural transformation} and Remark \ref{rem: vK colimits}). This happens, in particular, if \emph{all} colimits of functors $\textbf{D}(\mathcal A) \rightarrow \Set$ are van Kampen, such as when $\textbf{D}(\mathcal A)$ is a coproduct of filtered categories.\footnote{We don't know whether this is actually equivalent to all colimits of functors $\textbf{D}(\mathcal A) \rightarrow \Set$ being van Kampen} An instance of this is Proposition \ref{prop: strictification - out of cont cat with a basis} with $\mathcal B = \Set$.

In Example \ref{ex: E(K)}, we will associate with a category $K$ a gat $\bbE K$ such that $\mathcal C(\bbE K)$ initial among all contextual categories $\mathcal A$ equipped with a functor $K \rightarrow \textbf{D}(\mathcal A)$ sending every object of $K$ to a strict display map whose codomain has length-$1$. Strictifiability of all set-models of $\bbE K$ will be equivalent to all colimits of functors $K \rightarrow \Set$ being van Kampen. See there for further discussion.

On the other hand, we will see in Example \ref{ex: sort+sort equality - strictifiable} a contextual category $\mathcal A$ for which every set-model is strictifiable but $\textbf{D}(\mathcal A)$ is not a coproduct of filtered categories (in particular, $\mathcal A$ is not cofibrant). It also shows (see Remark \ref{rem: strict vs weak presentation via precontextual category}) that diagrammatic presentations of contextual categories via precontextual categories, which are $1$-categorical but also have the expected $\Cat$-enriched universal property (see Remark \ref{rem: use of precontextual categories}), generally fail to classify maps out of the corresponding object of the (weak) $2$-category $\iiDMC_r \simeq \iiHo(\Cont)$.

\begin{example}
\label{ex: E(K)}
For a small category $K$, let $\bbE K$ be the gat consisting of:
\begin{itemize}[noitemsep]
	\item for each object $a \in K$, sort symbols $B_a$ and $E_a$ introduced by
	$$
	\vdash B_a \tp, \qquad x:B_a \vdash E_a(x) \tp;
	$$
	
	\item for each arrow $f \in K(a,b)$, a term symbol $f_*$ introduced by
	$$
	x:B_a \vdash f_*(x):B_b,
	$$
	and a sort equality axiom
	$$
	x:B_a \vdash E_a(x) \equiv E_b(f_*(x)) \tp;
	$$
	
	\item axioms
	\begin{align*}
		x:B_a & \vdash (gf)_*(x) \equiv g_*(f_*(x)):B_c \tag{for $f \in K(a,b)$, $g \in K(b,c)$}\\
		x:B_a & \vdash id_{a*}(x) \equiv x:B_a \tag{for $a \in K$}
	\end{align*}
\end{itemize}
Observe that a family-valued model of $\bbE K$ consists of a functor $X:K \rightarrow \Set$ (given by interpreting the axioms that don't involve the symbols $E_a$) equipped with a family of small sets $Y(a):X(a) \rightarrow \mathscr U$ for each $a \in \Ob(K)$ such that
% https://q.uiver.app/#q=WzAsMyxbMCwwLCJYKGEpIl0sWzAsMiwiWChiKSJdLFsyLDEsIlxcbWF0aHNjciBVIl0sWzAsMiwiWShhKSJdLFsxLDIsIlkoYikiLDJdLFswLDEsIlgoZikiLDJdXQ==
\[\begin{tikzcd}[ampersand replacement=\&,cramped,row sep=tiny]
	{X(a)} \&\& \\
	\&\& {\mathscr U} \\
	{X(b)}
	\arrow["{Y(a)}", from=1-1, to=2-3]
	\arrow["{X(f)}"', from=1-1, to=3-1]
	\arrow["{Y(b)}"', from=3-1, to=2-3]
\end{tikzcd}\]
commutes for $f \in K(a,b)$. In other words, we have a cocone from $X$ to $\mathscr U$. This, in turn, corresponds to a map $\colim X \rightarrow \mathscr U$. It follows that $\iiMod_s(\bbE K)$ is equivalent to the comma category $\Set \downarrow \colim$ where $\colim$ is the colimit functor $\Set^K \rightarrow K$. We will now exhibit a cofibrant replacement of $\bbE K$ and use it to describe the category of set-models of $K$. Let $\bbE' K$ be the gat obtained from $\bbE K$ as follows:
\begin{enumerate}[label=(\arabic*),noitemsep]
	\item for $f \in K(a,b)$, replace
	$$
	x:B_a \vdash E_a(x) \equiv E_b(f_*(x)) \tp
	$$
	by the following four axioms:
	\begin{align*}
		x:B_a,\; y:E_a(x) & \vdash \textbf{tr}_f(x,y): E_b(f_*(x))\\
		x:B_a,\; y:E_b(f_*(x)) & \vdash \textbf{tr}'_f(x,y) : E_a(x)\\
		x:B_a,\; y:E_a(x) & \vdash \textbf{tr}'_f(x,\textbf{tr}_f(x,y)) \equiv y: E_a(x)\\
		x:B_a,\; y:E_b(f_*(x)) & \vdash \textbf{tr}_f(x,\textbf{tr}'_f(x,y)) \equiv y: E_b(f_*(x))
	\end{align*}
	
	\item add axioms
	\begin{align*}
		x:B_a,\; y:E_a(x) & \vdash \textbf{tr}_g(f_*(x),\textbf{tr}_f(x,y)) \equiv \textbf{tr}_{gf}(x,y): E_c(g_*f_*(x)) \tag{for $f \in K(a,b)$, $g \in K(b,c)$}\\
		x:B_a,\; y:E_a(x) & \vdash \textbf{tr}_{id_a}(x,y) \equiv x : E_a(x) \tag{for $a \in K$}
	\end{align*}
\end{enumerate}
It is not difficult to check that we have an interpretation of $\bbE' K$ in $\bbE K$ that sends $B_a$, $E_a$, $f_*$ to themselves, and $\textbf{tr}_f(x,y)$, $\textbf{tr}'_f(x,y)$ to $y$ (informally, $\textbf{tr}_f(x,-)$, $\textbf{tr}'_f(x,-)$ are interpreted as identity maps). Let us sketch a proof that $\bbE' K \rightarrow \bbE K$ is an acyclic fibration. We start by noting that for any context $\textbf{X} = (x_1:X_1, ..., x_n:X_n)$ in $\bbE' K$, every sort derivable in $\textbf{X}$ is derivably equal to exactly one of the following:
\begin{itemize}[noitemsep]
	\item $B_a$ for some $a \in K$.
	
	\item $E_b(f_*(x_i))$ for some $f \in K(a,b)$ and $i \in \{1, ..., n\}$ such that $X_i = B_a$.
\end{itemize}
If, instead, $\textbf{X}$ is a context in $\bbE K$, a sort in $\textbf{X}$ is derivably equal to exactly one of the following:
\begin{itemize}[noitemsep]
	\item $B_a$ for some $a \in K$.
	
	\item $E_a(x_i)$ for some $a \in K$ such that $X_i = B_a$.
\end{itemize}
It follows from this description that $\bbE' K \rightarrow \bbE K$ has the right lifting property with respect to $\iota_n^S:\bbO_{n-1} \rightarrow \bbO_n$ for all $n \ge 1$. It remains to check that for every derivable judgment $\textbf{X} \vdash U \tp$ in $\bbE' K$, equivalence classes of terms $\textbf{X} \vdash u:U$ are in bijective correspondence, via the above morphism, with equivalence classes of terms of the same form in $\bbE K$. We have the following cases:
\begin{itemize}
	\item Suppose that $U = B_a$. In both $\bbE K$ and $\bbE' K$, if $\textbf{X} \vdash u:U$ is derivable, then $u$ is provably equal to $f_*(x_i)$ for unique $i \in \{1, ..., n\}$ and $f \in K(b,a)$ such that $X_i = B_b$.
	
	\item Suppose that $U \equiv E_b(f_*(x_i))$ for $f \in K(a,b)$ and $i \in \{1, ..., n\}$ such that $X_i = B_a$. If $\textbf{X} \vdash u:U$ is derivable in $\bbE' K$, then there exist a commutative diagram
	% https://q.uiver.app/#q=WzAsOCxbMywwLCJhIl0sWzYsMiwiY19OID0gYiJdLFswLDIsImNfMCJdLFsxLDIsImNfMSJdLFsyLDIsImNfMiJdLFszLDIsIlxcY2RvdHMiXSxbNSwyLCJjX3tOLTF9Il0sWzQsMiwiY197Ti0yfSJdLFsyLDMsIlxcYWxwaGFfMSIsMl0sWzQsMywiXFxhbHBoYV8yIl0sWzEsNiwiXFxhbHBoYV9OIl0sWzcsNiwiXFxhbHBoYV97Ti0xfSIsMl0sWzcsNSwiXFxhbHBoYV97Ti0yfSJdLFs0LDUsIlxcYWxwaGFfMyIsMl0sWzAsMiwiXFx2YXJwaGlfMCIsMV0sWzAsMywiXFx2YXJwaGlfMSIsMV0sWzAsNCwiXFx2YXJwaGlfMiIsMV0sWzAsNywiXFx2YXJwaGlfe04tMn0iLDFdLFswLDYsIlxcdmFycGhpX3tOLTF9IiwxXSxbMCwxLCJcXHZhcnBoaV9OID0gZiIsMV1d
	\[\begin{tikzcd}[ampersand replacement=\&,cramped,column sep=huge]
		\&\&\& a \&\&\& \\
		\\
		{c_0} \& {c_1} \& {c_2} \& \cdots \& {c_{N-2}} \& {c_{N-1}} \& {c_N = b}
		\arrow["{\varphi_0}"{description}, from=1-4, to=3-1]
		\arrow["{\varphi_1}"{description}, from=1-4, to=3-2]
		\arrow["{\varphi_2}"{description}, from=1-4, to=3-3]
		\arrow["{\varphi_{N-2}}"{description}, from=1-4, to=3-5]
		\arrow["{\varphi_{N-1}}"{description}, from=1-4, to=3-6]
		\arrow["{\varphi_N = f}"{description}, from=1-4, to=3-7]
		\arrow["{\alpha_1}"', from=3-1, to=3-2]
		\arrow["{\alpha_2}", from=3-3, to=3-2]
		\arrow["{\alpha_3}"', from=3-3, to=3-4]
		\arrow["{\alpha_{N-2}}", from=3-5, to=3-4]
		\arrow["{\alpha_{N-1}}"', from=3-5, to=3-6]
		\arrow["{\alpha_N}", from=3-7, to=3-6]
	\end{tikzcd}\]
	in $K$ and $j \in \{1, ..., n\}$ such that, omitting the first argument of $\textbf{tr}$ and $\textbf{tr}'$,
	$$
	X_j = E_{c_0}(\varphi_{0\;*}(x_i)), \qquad u \equiv \textbf{tr}'_{\alpha_N}\textbf{tr}_{\alpha_{N-1}} \cdots \textbf{tr}'_{\alpha_2}\textbf{tr}_{\alpha_1}(x_j).
	$$
	
	Observe that $u$ is derivably equal in $\textbf{X}$ to $\textbf{tr}_f(\textbf{tr}'_{\varphi_0}(x_j))$.
	
	On the other hand, if $\textbf{X} \vdash v:U$ is derivable in $\bbE K$, then $v = x_j$ for some $j$, $\varphi_0$, ..., $\varphi_N$, $\alpha_1$, ..., $\alpha_N$ as above. In this notation, $\bbE' K \rightarrow \bbE K$ sends $u$ to $v$. In particular, the induced map
	$$
	\tm_\textbf{X}^{\bbE' K}(U) \longrightarrow \tm_\textbf{X}^{\bbE K}(U)
	$$
	is surjective. For injectivity, consider derivable judgments $\textbf{X} \vdash u:U$ and $\textbf{X} \vdash u':U$ in $\bbE' K$ such that, letting $v$, $v'$ be the images of $u$, $u'$ in $\bbE K$, the judgment $\textbf{X} \vdash v \equiv v':U$ is derivable in $\bbE K$. Note that as it is not possible to derive equalities between variables in $\bbE K$, we must have $v = v' = x_j$ for some $j \in \{1, ..., n\}$ and $\varphi \in K(a,c)$ such that $X_j = E_c(\varphi_*(x_i))$ (here, $c$, $\varphi$ play the same role as $c_0$, $\varphi_0$ in the above diagram). But then
	$$
	u \equiv \textbf{tr}_f\textbf{tr}'_\varphi(x_j) \equiv u'
	$$
	is derivable in $\textbf{X}$, as required.
\end{itemize}
We conclude that $\bbE' K \rightarrow \bbE K$ is a weak equivalence of gats. Let us use that to calculate the category of set-models of $\bbE K$. Let $\mathcal Q$ be the full subcategory of $\mathcal C(\bbE' K)$, regarded as a precontextual category as in Cor. \ref{cor: weak morphisms full precontextual subcategory}, spanned by the following three kinds of objects:
\begin{itemize}
	\item the length-$0$ object $1_{\mathcal C(\bbE' K)} = []$;
	
	\item all length-$1$ objects, that is, $[x:B_a]$ for $a \in \Ob(K)$;
	
	\item length-$2$ objects $[x:B_a,\; y:E_b(f_*(x))]$ for $a$, $b \in \Ob(K)$ and $f \in K(a,b)$.
\end{itemize}
It is not difficult to verify from the universal properties of $\mathcal Q$ and $\bbE' K$ (the main point being that the axioms of $\bbE'$ only involve objects of $\mathcal Q$) that $L\mathcal Q \cong \mathcal C(\bbE' K)$, so we can use Cor. \ref{cor: weak morphisms full precontextual subcategory}. Similarly, the image of the composite
$$
\mathcal Q \overset{\iota_\mathcal Q}{\longrightarrow} L\mathcal Q \cong \mathcal C(\bbE' K) \longrightarrow \mathcal C(\bbE K)
$$
inherits from $\mathcal C(\bbE K)$ a precontextual structure; denote it by $\mathcal Q'$. This precontextual category admits the following description via generators and relations:
\begin{itemize}[noitemsep]
	\item It contains $1_{\mathcal C(\bbE K)} = []$, which is also terminal in $\mathcal Q'$, as well as $[x:B_a]$ and $[x:B_a,\; y:E_a(x)]$ for each $a \in \Ob(K)$. We have strict display maps
	$$
	[x:B_a,\; y:E_a(x)] \twoheadrightarrow [x:B_a] \twoheadrightarrow [].
	$$
	
	\item It contains morphisms
	$$
	[f_*(x)]:[x:B_a] \rightarrow [x':B_b], \qquad [f_*(x),y]:[x:B_a,\; y:E_a(y)] \rightarrow [x':B_b,\; y':E_b(y')]
	$$
	for each $f \in K(a,b)$, and these are composed as
	$$
	[g_*(x)] \circ [f_*(x)] = [(gf)_*(x)], \qquad [g_*(x),y] \circ [f_*(x),y] = [(gf)_*(x),y].
	$$
	
	\item For $f \in K(a,b)$, the diagram
	% https://q.uiver.app/#q=WzAsNCxbMCwwLCJbeDpCX2EsIHk6RV9hKHgpXSJdLFsyLDAsIlt4JzpCX2IsIHknOkVfYih4JyldIl0sWzAsMSwiW3g6Ql9hXSJdLFsyLDEsIlt4JzpCX2JdIl0sWzAsMiwiIiwwLHsic3R5bGUiOnsiaGVhZCI6eyJuYW1lIjoiZXBpIn19fV0sWzEsMywiIiwwLHsic3R5bGUiOnsiaGVhZCI6eyJuYW1lIjoiZXBpIn19fV0sWzIsMywiW2ZfKih4KV0iLDJdLFswLDEsIltmXyooeCkseV0iXV0=
	\[\begin{tikzcd}[ampersand replacement=\&,cramped]
		{[x:B_a, y:E_a(x)]} \&\& {[x':B_b, y':E_b(x')]} \\
		{[x:B_a]} \&\& {[x':B_b]}
		\arrow["{[f_*(x),y]}", from=1-1, to=1-3]
		\arrow[two heads, from=1-1, to=2-1]
		\arrow[two heads, from=1-3, to=2-3]
		\arrow["{[f_*(x)]}"', from=2-1, to=2-3]
	\end{tikzcd}\]
	commutes and is distinguished.
\end{itemize}
It follows that set-models of $\mathcal Q'$ -- thus of $\bbE K$, by Cor. \ref{cor: weak morphisms full precontextual subcategory} -- correspond to cartesian natural transformations $F' \Rightarrow F$ for $F$, $F' \in \Set^K$. We conclude that it is not the case in general that every set-model of $\bbE K$ is strictifiable: the model corresponding to $\eta:F' \Rightarrow F$ is strictifiable if and only if $\eta$ is loop-free in the sense of Definition \ref{def: loop-free natural transformation}, and every set-model being strictifiable is equivalent to colimits of shape $K$ in $\Set$ being van Kampen. See Proposition \ref{prop: characterization loop-free natural transformation} and Remark \ref{rem: vK colimits}.

It is not difficult to verify that the gat $\bbK$ from \S\ref{sec: introduction} is isomorphic to $\bbE(B \mathbb N)$.\footnote{In fact, the definition of $\bbE K$ can be extended so that $K$ can be a presentation of a category by generators and relations, and this gives a gat isomorphic to the one obtained by applying $\bbE(-)$ to the generated category.} More generally, whenever $\gpd(K)$ is not equivalent to a set (as with $K = B\bbN$), it is simple to construct a non-strictifiable model of $\bbE K$: we obtain a non-loop-free cartesian natural transformation in $\Set^K$ by restricting along the localization functor $K \rightarrow \gpd(K)$ any non-loop-free natural transformation in $\Set^{\gpd(K)}$ of the form $S \Rightarrow 1$; we can take, e.g., $S = \gpd(K)(a,-)$ where $a$ is an object with a non-trivial automorphism.
\end{example}

\begin{remark}
The failure of strictifiability of set-models of $\bbE K$ can't always be detected from whether $\gpd(K)$ is equivalent to a set. For instance, let $K$ be obtained from any category $J$ by freely adding an initial object. While $\gpd(K)$ is contractible, any (possibly non-loop-free) cartesian natural transformation in $\Set^J$ extends to one in $\Set^K$ by sending the initial object to $\varnothing$. As another example, we have a non-strictifiable model $F$ of $\bbE K$ for $K = \{\ell \leftarrow \bullet \rightarrow r\}$ specified by the following data:
% https://q.uiver.app/#q=WzAsMTIsWzAsNCwiXFxzdWJzdGFja3sxIFxcXFwgMH0iXSxbNiw0LCJcXHN1YnN0YWNrezEgXFxcXCAwfSJdLFs2LDYsIlxcc3Vic3RhY2t7MH0iXSxbMCw2LCJcXHN1YnN0YWNrezB9Il0sWzMsNiwiXFxzdWJzdGFja3swIFxcO1xcO1xcO1xcOyAxfSJdLFszLDQsIlxcc3Vic3RhY2t7KDAsMSkgXFwgKDEsMSkgXFxcXCAoMCwwKSBcXCAoMSwwKX0iXSxbMywwLCJbeDpCX1xcYnVsbGV0LCB5OkVfXFxidWxsZXQoeCldIl0sWzMsMiwiW3g6Ql9cXGJ1bGxldF0iXSxbMCwyLCJbeDpCX1xcZWxsXSJdLFs2LDIsIlt4OkJfcl0iXSxbMCwwLCJbeDpCX1xcZWxsLCB5OkVfXFxlbGwoeCldIl0sWzYsMCwiW3g6Ql9yLCB5OkVfcih4KV0iXSxbNCwzLCIhIl0sWzQsMiwiISIsMl0sWzUsMSwiKHgseSkgXFxtYXBzdG8geCt5IFxcbW9kIDIiXSxbNSwwLCJcXHRleHR7cHJvamVjdGlvbn0iLDJdLFswLDMsIiEiLDJdLFsxLDIsIiEiXSxbNSw0LCJcXHRleHR7cHJvamVjdGlvbn0iLDFdLFsxMSw5LCIiLDAseyJzdHlsZSI6eyJoZWFkIjp7Im5hbWUiOiJlcGkifX19XSxbMTAsOCwiIiwwLHsic3R5bGUiOnsiaGVhZCI6eyJuYW1lIjoiZXBpIn19fV0sWzcsOF0sWzYsMTBdLFs2LDExXSxbNiw3LCIiLDEseyJzdHlsZSI6eyJoZWFkIjp7Im5hbWUiOiJlcGkifX19XSxbNyw5XSxbNyw1LCIiLDAseyJzaG9ydGVuIjp7InRhcmdldCI6NDB9LCJzdHlsZSI6eyJ0YWlsIjp7Im5hbWUiOiJtYXBzIHRvIn19fV1d
\[\begin{tikzcd}[ampersand replacement=\&,cramped]
	{[x:B_\ell, y:E_\ell(x)]} \&\&\& {[x:B_\bullet, y:E_\bullet(x)]} \&\&\& {[x:B_r, y:E_r(x)]} \\
	\\
	{[x:B_\ell]} \&\&\& {[x:B_\bullet]} \&\&\& {[x:B_r]} \\
	\\
	\begin{array}{c} \substack{1 \\ 0} \end{array} \&\&\& \begin{array}{c} \substack{(0,1) \ (1,1) \\ (0,0) \ (1,0)} \end{array} \&\&\& \begin{array}{c} \substack{1 \\ 0} \end{array} \\
	\\
	{\substack{0}} \&\&\& {\substack{0 \;\;\;\; 1}} \&\&\& {\substack{0}}
	\arrow[two heads, from=1-1, to=3-1]
	\arrow[from=1-4, to=1-1]
	\arrow[from=1-4, to=1-7]
	\arrow[two heads, from=1-4, to=3-4]
	\arrow[two heads, from=1-7, to=3-7]
	\arrow[from=3-4, to=3-1]
	\arrow[from=3-4, to=3-7]
	\arrow[color={rgb,255:red,92;green,92;blue,214},"F",shorten=5pt, maps to, from=3-4, to=5-4]
	\arrow["{!}"', from=5-1, to=7-1]
	\arrow["{\text{projection}}"', from=5-4, to=5-1]
	\arrow["{(x,y) \mapsto x+y \mod 2}", from=5-4, to=5-7]
	\arrow["{\text{projection}}"{description}, from=5-4, to=7-4]
	\arrow["{!}", from=5-7, to=7-7]
	\arrow["{!}", from=7-4, to=7-1]
	\arrow["{!}"', from=7-4, to=7-7]
\end{tikzcd}\]
Letting $R_1$, $R_2:K \rightarrow \Set$ be the top and bottom rows of this diagram, the induced map $\gpd(\smallint R_1) \rightarrow \gpd(\smallint R_2)$ is the double cover of the circle, and this witnesses the fact that the colimit of $\{0\} \leftarrow \{0,1\} \rightarrow \{0\}$ is not van Kampen. These examples also show that strictifiability might not hold for $\mathcal A \in \Cont$ even if $\gpd(\textbf{D}(\mathcal A))$ is equivalent to a set.
\end{remark}

\begin{example}
\label{ex: sort+sort equality - strictifiable}
Let $\bbA$ be the generalized algebraic theory given by
$$
\vdash A \tp \qquad \vdash B \tp \qquad \vdash E \tp \qquad e:E \vdash A \equiv B \tp.
$$
A family-valued model of $\bbA$ corresponds to a function $M:\{A,B,E\} \rightarrow \mathscr U$ where
\begin{itemize}[noitemsep]
	\item[-] either $M(E) \neq \varnothing$ and $M(A) = M(B)$,
	
	\item[-] or $M(E) = \varnothing$.
\end{itemize}
Contexts in $\bbA$ are the sequences $x:X_1, ..., x_n:X_n$ where $x_1$, ..., $x_n$ are distinct variables and $X_i \in \{A, B, E\}$ for each $i$. Note that two contexts $(x_i:X_i)_{1 \le i \le n}$ and $(y_i:Y_i)_{1 \le i \le n}$ are provably equal if and only one of the following is satisfied:
\begin{itemize}[noitemsep]
	\item[-] $\{i \mid X_i = E\}$ and $\{i \mid Y_i = E\}$ are equal and non-empty.
	
	\item[-] $X_i = Y_i$ for all $i = 1$, ..., $n$.
\end{itemize}
Writing $[S]$ for $[x:S]$ where $S = A$, $B$ or $E$, consider the commutative diagram
% https://q.uiver.app/#q=WzAsNSxbMiwyLCIxX1xcbWF0aGNhbCBBIl0sWzIsMSwiXFx0ZXh0YmZ7QX0iXSxbMywxLCJcXHRleHRiZntCfSJdLFswLDEsIlxcdGV4dGJme0V9Il0sWzAsMCwiXFx0ZXh0YmZ7RX0gXFx0aW1lcyBcXHRleHRiZntBfSA9IFxcdGV4dGJme0V9IFxcdGltZXMgXFx0ZXh0YmZ7Qn0iXSxbMywwXSxbMSwwLCIiLDIseyJzdHlsZSI6eyJoZWFkIjp7Im5hbWUiOiJlcGkifX19XSxbMiwwLCIiLDIseyJzdHlsZSI6eyJoZWFkIjp7Im5hbWUiOiJlcGkifX19XSxbNCwzLCIiLDAseyJzdHlsZSI6eyJoZWFkIjp7Im5hbWUiOiJlcGkifX19XSxbNCwxXSxbNCwyXV0=
\[
\tag{\texttt{*}}
\begin{tikzcd}[ampersand replacement=\&,cramped]
	{[E] \times [A] = [E] \times [B]} \&\&\& \\
	{[E]} \&\& {[A]} \& {[B]} \\
	\&\& {1_{\mathcal C(\bbA)}}
	\arrow[two heads, from=1-1, to=2-1]
	\arrow[from=1-1, to=2-3]
	\arrow[from=1-1, to=2-4]
	\arrow[from=2-1, to=3-3]
	\arrow[two heads, from=2-3, to=3-3]
	\arrow[two heads, from=2-4, to=3-3]
\end{tikzcd}\]
in $\mathcal C(\bbA)$ where the two squares with upper-left corner $[E] \times [A]$ are distinguished. The corresponding span
% https://q.uiver.app/#q=WzAsMyxbMCwxLCJcXHRleHRiZntwfV97XFx0ZXh0YmYgQX0iXSxbMSwwLCJcXHRleHRiZntwfV97XFx0ZXh0YmYgRX0iXSxbMiwxLCJcXHRleHRiZntwfV97XFx0ZXh0YmYgQn0iXSxbMSwwXSxbMSwyXV0=
\[\begin{tikzcd}[ampersand replacement=\&,row sep=tiny]
	\& {\textbf{p}_{[E] \times [A]}} \& \\
	{\textbf{p}_{[A]}} \&\& {\textbf{p}_{[B]}}
	\arrow[from=1-2, to=2-1]
	\arrow[from=1-2, to=2-3]
\end{tikzcd}\]
in $\textbf{D}(\mathcal C(\bbA))$ cannot be extended into a commutative square (note that there are no non-identity morphisms out of $1_{\mathcal C(\bbA)}$), which implies that $\bbA$ is not cofibrant. However, as we will now verify, every set-model of $\bbA$ is strictifiable. We will construct a cofibrant replacement $\varphi:\bbA' \rightarrow \bbA$ and check that in the diagram of comparison functors
% https://q.uiver.app/#q=WzAsNCxbMCwwLCJcXGlpTW9kX3MoXFxtYXRoY2FsIEEpIl0sWzAsMSwiXFxpaU1vZF9zKFxcbWF0aGNhbCBBJykiXSxbMSwwLCJcXGlpTW9kX3coXFxtYXRoY2FsIEEpIl0sWzEsMSwiXFxpaU1vZF93KFxcbWF0aGNhbCBBJykiXSxbMCwyXSxbMiwzLCJcXHNpbWVxIl0sWzEsMywiXFxzaW1lcSIsMl0sWzAsMSwiXFx2YXJwaGleKiIsMl1d
\[\begin{tikzcd}[ampersand replacement=\&,cramped]
	{\iiMod_s(\bbA)} \& {\iiMod_w(\bbA)} \\
	{\iiMod_s(\bbA')} \& {\iiMod_w(\bbA'),}
	\arrow[from=1-1, to=1-2]
	\arrow["{\varphi^*}"', from=1-1, to=2-1]
	\arrow["\simeq", from=1-2, to=2-2]
	\arrow["\simeq"', from=2-1, to=2-2]
\end{tikzcd}\]
the left vertical arrow is an equivalence of categories; the $2$-out-of-$3$ property will then allow us to conclude that $\iiMod_s(\bbA) \rightarrow \iiMod_w(\bbA)$ is also an equivalence. As a first approximation, we can replace the sort equality $e:E \vdash A \equiv B \tp$ in $\bbA$ by an isomorphism $[E] \times [A] \cong [E] \times [B]$ compatible with the projections to $[E]$. Diagrammatically, this is achieved by taking the contextual category freely generated by (the precontextual category)
% https://q.uiver.app/#q=WzAsNixbMiwyLCIxX1xcbWF0aGNhbCBBIl0sWzIsMSwiXFx0ZXh0YmZ7QX0iXSxbMywxLCJcXHRleHRiZntCfSJdLFswLDEsIlxcdGV4dGJme0V9Il0sWzAsMCwiXFx0ZXh0YmZ7RX1cXHRleHRiZntBfSJdLFsxLDAsIlxcdGV4dGJme0V9XFx0ZXh0YmZ7Qn0iXSxbMywwXSxbMSwwLCIiLDIseyJzdHlsZSI6eyJoZWFkIjp7Im5hbWUiOiJlcGkifX19XSxbMiwwLCIiLDIseyJzdHlsZSI6eyJoZWFkIjp7Im5hbWUiOiJlcGkifX19XSxbNCwzLCIiLDAseyJzdHlsZSI6eyJoZWFkIjp7Im5hbWUiOiJlcGkifX19XSxbNSwzLCIiLDAseyJzdHlsZSI6eyJoZWFkIjp7Im5hbWUiOiJlcGkifX19XSxbNCw1LCIiLDEseyJvZmZzZXQiOi0xfV0sWzUsNCwiIiwxLHsib2Zmc2V0IjotMX1dLFs0LDFdLFs1LDJdXQ==
\[\begin{tikzcd}[ampersand replacement=\&,cramped]
	{[E] \times [A]} \& {[E] \times [B]} \&\& \\
	{[E]} \&\& {[A]} \& {[B]} \\
	\&\& {1_{\mathcal C(\bbA)}}
	\arrow[shift left, from=1-1, to=1-2]
	\arrow[two heads, from=1-1, to=2-1]
	\arrow[from=1-1, to=2-3]
	\arrow[shift left, from=1-2, to=1-1]
	\arrow[two heads, from=1-2, to=2-1]
	\arrow[from=1-2, to=2-4]
	\arrow[from=2-1, to=3-3]
	\arrow[two heads, from=2-3, to=3-3]
	\arrow[two heads, from=2-4, to=3-3]
\end{tikzcd}\]
where the arrows between $[E] \times [A]$ and $[E] \times [B]$ are isomorphisms inverse to each other; syntactically, we consider the gat $\bbT$ given by
$$
\vdash A \tp \qquad \vdash B \tp \qquad \vdash E \tp
$$
\vspace{-0.6cm}
\begin{align*}
	e:E,\; a:A & \vdash f(e,a):B\\
	e:E,\; b:B & \vdash g(e,b):A\\
	e:E,\; a:A & \vdash g(e,f(e,a)) \equiv a:A\\
	e:E,\; b:B & \vdash f(e,g(e,b)) \equiv b:B.
\end{align*}
Now, note that the interpretation $\bbT \rightarrow \bbA$ that collapses $f$, $g$ onto the identity morphism (that is, $f(e,a) \mapsto a$ and $g(e,b) \mapsto b$) is not a weak equivalence: the (well-formed) judgment
\[
\tag{\texttt{**}}
e, e': E,\; a:A \vdash f(e,a) \equiv f(e',a): B
\]
is not derivable in $\bbT$, but it corresponds via $\bbT \rightarrow \bbA$ to the derivable judgment
$$
e,e':E,\; a:A \vdash a \equiv a:B.
$$
Informally, $\bbT$ specifies an isomorphism between $A$ and $B$ depending on the parameter $e$, but this dependency disappears in $\bbA$. To fix this problem, we let $\bbA'$ be the theory obtained from $\bbA$ by adding (\texttt{**}) as an axiom. It is not difficult to prove that the interpretation $I:\bbA' \rightarrow \bbA$ (acting on judgments exactly as $\bbT \rightarrow \bbA$) in an acyclic fibration; this can be done by fully describing the contexts, sorts and terms of $\bbA$, $\bbA'$, and using Proposition \ref{prop: characterization full-and-faithful} to check that $\mathcal C(\bbA') \rightarrow \mathcal C(\bbA)$ is full-and-faithful.

We can now verify that $I^*:\iiMod_s(\bbA) \rightarrow \iiMod_s(\bbA')$ is essentially surjective, thus an equivalence of categories. There are two kinds of strict models $M:\mathcal C(\bbA') \rightarrow \Fam$:
\begin{itemize}
	\item those such that $M[E] = \varnothing$, in which case $M[A]$, $M[B]$ can be any small sets, and
	
	\item those such that $M[E] \neq \varnothing$, in which case $M[A]$, $M[B]$ are small sets equipped with an isomorphism $M[A] \cong M[B]$.
\end{itemize}
A model of the first kind is in the (strict) image of $I^*$, and one of the second kind is isomorphic to the model $M'$ given by $M'[E] = M[E]$, $M'[A] = M'[B] = M[A]$, and $id_{M[A]}$ as the isomorphism $M'[A] \cong M'[B]$. Since $M'$ is in the image of $I^*$, we conclude that the latter is essentially surjective, as desired.
\end{example}

\begin{remark}
\label{rem: strict vs weak presentation via precontextual category}
An interesting phenomenon happens in the above example: $\mathcal C(\bbA)$ can be presented by applying $L:\Precont \rightarrow \Cont$ to a precontextual category $\mathcal P$ (the one specified by diagram (\texttt{*})) such that no cofibrant replacement of $\mathcal C(\bbA)$ can be presented by $\mathcal P' \rightarrow \mathcal P$ for some precontextual category $\mathcal P'$.
\end{remark}

\section{Remarks on the closed monoidal structure on $\iiDMC_r$ and the set-valued semantics of gats}
\label{sec: remarks closed monoidal structure and set-valued semantics}

The monoidal model structure on $\Cont$ presents a closed monoidal structure on the homotopy $(2,1)$-category $\iiHo(\Cont) \simeq \iiDMC_r^{(2,1)}$ (see \cite{Lur17}[\S4.1.7]). We will now describe more explicitly the induced tensor product and exponentiation of rooted display map categories; after that, we will check that $\otimes$ has the expected behaviour at the level of categories of set-valued models, namely, that it preserves the tensor product of locally finitely presentable categories (possibly equipped with a weak factorization system).

The operations on rooted dmcs discussed below are directly adapted from the corresponding ones on contextual categories (\cite{Alm26}), and similar structures have been considered in the literature; see, e.g., the (cofibration part of the) closed monoidal structures from \cite{Hen16, Bar20}, and the wfs on the category of models of a clan from \cite{Hen16, Fre25, BarHen25}.

\subsection{Exponentials and tensor products in $\iiDMC_r \simeq \iiHo(\Cont)$}

\begin{definition}
	\label{def: relative morphism}
	Let $\mathcal A$, $\mathcal S$ be rooted dmcs. As a preliminary terminology, we will call a functor $F:\mathcal A \rightarrow \mathcal S$ \emph{admissible} if it sends every display map to an arrow that admits pullbacks along arbitrary arrows.
	
	We define a \emph{relative morphism} from $\mathcal A$ to $\mathcal S$ as a natural transformation $\pi:F \Rightarrow B$ between functors $\mathcal A \rightarrow \mathcal S$ such that $B$ is admissible and
	\begin{enumerate}[label=\textbf{RM(\roman*)}, noitemsep]
		\item $\pi_{1}:F(1) \rightarrow B(1)$ is an isomorphism for a terminal object $1 \in \mathcal A$;
		
		\item for every display map $p:a' \rightarrow a$ in $\mathcal A$, the unique filler arrow in
		% https://q.uiver.app/#q=WzAsNSxbMiwxLCJCKGEnKSJdLFsyLDIsIkIoYSkiXSxbMSwyLCJGKGEpIl0sWzAsMCwiRihhJykiXSxbMSwxLCJGKGEpIFxcdGltZXNfe0IoYSl9IEIoYScpIl0sWzMsMiwiRihwKSIsMix7ImN1cnZlIjozfV0sWzMsMCwiXFxwaV97YSd9IiwwLHsiY3VydmUiOi0zfV0sWzIsMSwiXFxwaV9hIiwyXSxbMCwxLCJCKHApIl0sWzMsNCwiIiwwLHsic3R5bGUiOnsiYm9keSI6eyJuYW1lIjoiZGFzaGVkIn19fV0sWzQsMl0sWzQsMF1d
		\[\begin{tikzcd}[ampersand replacement=\&,cramped]
			{F(a')} \&\& \\
			\& {F(a) \times_{B(a)} B(a')} \& {B(a')} \\
			\& {F(a)} \& {B(a)}
			\arrow[dashed, from=1-1, to=2-2]
			\arrow["{\pi_{a'}}", curve={height=-18pt}, from=1-1, to=2-3]
			\arrow["{F(p)}"', curve={height=18pt}, from=1-1, to=3-2]
			\arrow[from=2-2, to=2-3]
			\arrow[from=2-2, to=3-2]
			\arrow["{B(p)}", from=2-3, to=3-3]
			\arrow["{\pi_a}"', from=3-2, to=3-3]
		\end{tikzcd}\]
		is a display map.
		
		\item For every pullback square
		\[
		\squa{a'}{b'}{a}{b}{f'}{f}{p}{p'}
		\]
		in $\mathcal A$ where $p$, $p'$ are display maps, the induced commutative square
		\[
		\widesqua{F(a')}{F(b')}{F(a) \times_{B(a)} B(a')}{F(b) \times_{B(b)} B(b')}{F(f')}{F(f) \times_{B(f)} B(f')}{}{}
		\]
		is cartesian.
	\end{enumerate}
	It is not difficult to check that these conditions imply that $F$ is also admissible.
\end{definition}

\begin{proposition}
	Consider rooted dmcs $\mathcal A$ and $\mathcal S$, and natural transformations $B' \overset{\beta}{\longrightarrow} B \overset{\pi}{\longleftarrow} F$ between functors $\mathcal A \rightarrow \mathcal S$ where $B$, $B'$ are admissible and $\pi$ is a relative morphism. Then there exists a pullback diagram
	% https://q.uiver.app/#q=WzAsNCxbMSwwLCJGIl0sWzEsMSwiQiJdLFswLDEsIkInIl0sWzAsMCwiRiciXSxbMiwxLCJcXGJldGEiLDJdLFswLDEsIlxccGkiXSxbMywwLCJcXGJldGEnIl0sWzMsMiwiXFxwaSciLDJdXQ==
	\[
	\begin{tikzcd}[ampersand replacement=\&,cramped]
		{F'} \& F \\
		{B'} \& B
		\arrow["{\beta'}", from=1-1, to=1-2]
		\arrow["{\pi'}"', from=1-1, to=2-1]
		\arrow["\pi", from=1-2, to=2-2]
		\arrow["\beta"', from=2-1, to=2-2]
	\end{tikzcd}\]
	and (for any such pullback) $\pi'$ is a relative morphism.
\end{proposition}

\begin{proof}[Sketch of proof]
	By Proposition \ref{prop: strictification for a rooted dmc}, it suffices to prove the claim in the case where $\mathcal A$ is the dmc of some contextual category. This allows us construct by induction on the tree of strict display maps of $\mathcal A$ the object part of $F'$ and families of arrows $\pi'$, $\beta'$ such that for all $a \in \mathcal A$,
	\[
	\squa{F'(a)}{F(a)}{B'(a)}{B(a)}{\beta'_a}{\beta_a}{\pi'_a}{\pi_a}
	\]
	is cartesian. By functoriality of pullbacks, we can uniquely define the action $F'$ on arrows so that it becomes a functor and $\pi'$, $\beta'$ become natural transformations. It is not difficult to verify that $\pi':F' \Rightarrow B'$ is a relative morphism. (This argument is very similar to Construction 3.8 from \cite{Alm26}.)
\end{proof}

\begin{definition}
	\label{def: exponential rooted dmcs}
	For rooted dmcs $\mathcal A$, $\mathcal S$, we let $\langle \mathcal A,\mathcal S\rangle$ be the following rooted dmc:
	\begin{itemize}[noitemsep]
		\item Its underlying category is the full subcategory of $\iiCat(\mathcal A,\mathcal S)$ spanned by those $F:\mathcal A \rightarrow \mathcal S$ that fit into a sequence
		$$
		F = F_n \Rightarrow F_{n-1} \Rightarrow \cdots \Rightarrow F_1 \Rightarrow F_0
		$$
		for $n \ge 0$ where $F_0$ is terminal and, for $0 \le i \le n-1$, $F_i$ is admissible and $F_{i+1} \Rightarrow F_i$ is a relative morphism (this implies that $F$ is admissible). If a functor $F$ can be obtained in this way, we will call it \emph{$n$-constructible}.
		
		\item Display maps are the relative morphisms.
	\end{itemize}
\end{definition}

\begin{definition}
\label{def: tensor product dmcs}
For rooted dmcs $\mathcal A$, $\mathcal B$, $\mathcal S$, a functor $F:\mathcal A \times \mathcal B \rightarrow \mathcal S$ is a \emph{bimorphism} from $(\mathcal A,\mathcal B)$ to $\mathcal S$ if it has the following properties:
\begin{itemize}
	\item For $a \in \mathcal A$ and $b \in \mathcal B$, the functors $H(a,-): \mathcal B \rightarrow \mathcal C$ and $H(-,b):\mathcal A \rightarrow \mathcal S$ are morphisms between the associated clans -- that is, they send display maps to finite composites of display maps, pullbacks of display maps to pullbacks, and preserve terminal objects.
	
	\item For display maps $p:a' \rightarrow a$ and $q:b' \rightarrow b$ in $\mathcal A$ and $\mathcal B$, respectively, the arrow
	$$
	(H(p,id),\; H(id,q)): H(a',b') \longrightarrow H(a,b') \times_{H(a,b)} H(a',b)
	$$
	is a display map.
\end{itemize}
A \emph{tensor product} of $\mathcal A$, $\mathcal B$ is a representing object for the $2$-functor $\iiDMC_r \rightarrow \iiCat$ sending $\mathcal S$ to the category of bimorphisms $(\mathcal A,\mathcal B) \rightarrow \mathcal S$. If it exists (which is always the case, as we will check), we denote it by $\mathcal A \otimes^{\text{dmc}} \mathcal B$; we then have a universal bimorphism $\mathcal A \times \mathcal B \rightarrow \mathcal A \otimes^{\text{dmc}} \mathcal B$.
\end{definition}

The following result can be proved in a straightforward way by comparing Definitions \ref{def: exponential rooted dmcs} and \ref{def: tensor product dmcs}:

\begin{proposition}
\label{prop: bimorphisms vs maps to exponential - rooted dmcs}
For rooted dmcs $\mathcal A$, $\mathcal B$, $\mathcal S$, a functor $H:\mathcal A \times \mathcal B \rightarrow \mathcal S$ is a bimorphism if and only if it belongs to the essential image of
$$
\iiDMC_r(\mathcal A,\langle \mathcal B, \mathcal S\rangle) \hookrightarrow \iiCat(\mathcal A, \iiCat(\mathcal B,\mathcal S)) \cong \iiCat(\mathcal A \times \mathcal B, \mathcal S).
$$
\end{proposition}

\begin{proposition}
\label{prop: 2-categorical universal property of Cont tensor product}
For cofibrant contextual categories $\mathcal A$, $\mathcal B$, the functor $\otimes_{\mathcal A,\mathcal B}:\mathcal A \times \mathcal B \longrightarrow \mathcal A \otimes \mathcal B$ as in \cite{Alm25, Alm26} is a universal bimorphism in the sense of Definition \ref{def: tensor product dmcs}; hence, $\mathcal A \otimes \mathcal B$ is a tensor product $\mathcal A_w \otimes^{\text{dmc}} \mathcal B_w$.
\end{proposition}

Proving this statement will require some preliminaries. Observe that it implies, following the discussion earlier in this section, that the closed monoidal structure on $\iiDMC_r \simeq \iiHo(\Cont)$ presented by the one on $\Cont$ is indeed as in Definition \ref{def: tensor product dmcs}. We will also see (and use in the proof of Prop. \ref{prop: 2-categorical universal property of Cont tensor product}) that exponentials of rooted dmcs as in Definition \ref{def: exponential rooted dmcs} can be computed as $\mathcal S^\mathcal A$ for $\mathcal A$, $\mathcal S \in \Cont$ with $\mathcal A$ cofibrant.

\begin{construction}
	For contextual categories $\mathcal A$, $\mathcal S$, the exponential $\mathcal S^\mathcal A$ from \cite[\S3]{Alm26} is equipped with a forgetful functor $|\mathcal S^\mathcal A| \rightarrow \iiCat(\mathcal A,\mathcal S)$ that assigns to each object
	$$
	\textbf{F} = (F_n \Rightarrow F_{n-1} \Rightarrow \cdots \Rightarrow F_1 \Rightarrow F_0) \in |\mathcal S^\mathcal A|
	$$
	its ``top level" $\textbf{F}_\varheartsuit = F_n$. Here, $F_0$ is constant on the distinguished terminal object of $\mathcal S$, and each $F_{i+1} \Rightarrow F_i$ is, in the terminology of \cite{Alm26}, an \emph{indexed sort}, which is a special kind of relative morphism where conditions RM(i)-RM(iii) are replaced by analogous ones involving the strict structure of $\mathcal A$. Hence, $\textbf{F}_\varheartsuit \in \langle \mathcal A_w, \mathcal S_w \rangle$ and we have a morphism of dmcs $(\mathcal S^\mathcal A)_w \rightarrow \langle \mathcal A_w,\mathcal S_w\rangle$.
\end{construction}

\begin{proposition}
\label{prop: strictification exponential}
Let $\mathcal A$, $\mathcal S \in \Cont$. If $\mathcal A$ is cofibrant, then $(-)_\varheartsuit: (\mathcal S^\mathcal A)_w \rightarrow \langle \mathcal A_w,\mathcal S_w \rangle$ is an equivalence in $\iiDMC_r$.
\end{proposition}

\begin{proof}
	It is clear that $(-)_\varheartsuit$ is full-and-faithful. Now, it suffices to prove for all $n \ge 0$ that if $F \in \langle \mathcal A_w,\mathcal S_w\rangle$ is $n$-constructible, then it is isomorphic to $\textbf{G}_\varheartsuit$ for some length-$n$ object $\textbf{G} \in \mathcal S^\mathcal A$.
	
	Consider the category whose objects are the length-$n$ objects of $\mathcal S^\mathcal A$ and whose morphisms between two such objects are the natural transformations between their underlying length-$n$ chains of functors $\mathcal A \rightarrow \mathcal S$. In terms of the $\Cat$-enrichment of $\Cont$, this can be described as $\iiCont(\mathcal O_n,\mathcal S^\mathcal A)$. Assuming that $\mathcal A$ is cofibrant, we have
	\[
	\tag{\texttt{*}}
	\iiCont(\mathcal O_n,\mathcal S^\mathcal A) \cong \iiCont(\mathcal O_n^\pre, \mathcal S^\mathcal A) \cong \iiCont(\mathcal A, \mathcal S^{\mathcal O_n^\pre}) \simeq \iiDMC_w(\mathcal A_w,(\mathcal S^{\mathcal O_n^\pre})_w)
	\]
	where $\mathcal O_n^\text{\text{\text{\text{\text{\text{pre}}}}}}$ is the precontextual category $o_n \twoheadrightarrow o_{n-1} \twoheadrightarrow \cdots \twoheadrightarrow o_1 \twoheadrightarrow o_0$ discussed in Construction \ref{constr: O_n etc}; the first isomorphism above follows from the $\Cat$-enrichment of the adjunction between $\iiCont$ and $\iiPrecont$ (see Remark \ref{rem: use of precontextual categories}).
	
	The upshot is that the dmc of $\mathcal S^{\mathcal O_n^\pre}$ can be described quite explicitly. Indeed, consider a diagram
	$$
	s_* = (s_n \xrightarrow{p_n} \cdots \xrightarrow{p_2} s_1 \xrightarrow{p_1} s_0)
	$$
	in $\iiCat(\mathcal O_n^\pre,\mathcal S) \cong \iiCat(\{0 < \cdots < n\}, \mathcal S)$ where $s_0 = 1_\mathcal S$ and each $p_i$ is a composite of strict display maps (our case of interest being the top level of some object of $\mathcal S^{\mathcal O_n^\pre}$); also, suppose given a diagram
	\[\begin{tikzcd}[ampersand replacement=\&,cramped]
		{s'_n} \& {s'_{n-1}} \& \cdots \& {s'_1} \& {s'_0} \\
		{s_n} \& {s_{n-1}} \& \cdots \& {s_1} \& {s_0}
		\arrow["{p'_n }", from=1-1, to=1-2]
		\arrow["{\pi_n}"', from=1-1, to=2-1]
		\arrow["{p'_{n-1}}", from=1-2, to=1-3]
		\arrow["{\pi_{n-1}}"', from=1-2, to=2-2]
		\arrow["{p'_2}", from=1-3, to=1-4]
		\arrow["{p'_1}", from=1-4, to=1-5]
		\arrow["{\pi_1}"', from=1-4, to=2-4]
		\arrow["{\pi_0}", from=1-5, to=2-5]
		\arrow["{p_n}"', from=2-1, to=2-2]
		\arrow["{p_{n-1}}"', from=2-2, to=2-3]
		\arrow["{p_2}"', from=2-3, to=2-4]
		\arrow["{p_1}"', from=2-4, to=2-5]
	\end{tikzcd}\]
	such that for all $i$, the induced arrow $s'_i \rightarrow s'_{i-1} \times_{s_{i-1}} s_i$ is a (not necessarily strict) display map and $\pi_0$ is an isomorphism. Then we can, inductively, construct a diagram
	% https://q.uiver.app/#q=WzAsMTAsWzIsMSwiXFxjZG90cyJdLFszLDEsInNfMSJdLFs0LDEsInNfMCJdLFswLDAsInMnJ19uIl0sWzIsMCwiXFxjZG90cyJdLFszLDAsInMnJ18xIl0sWzQsMCwicycnXzAiXSxbMCwxLCJzX24iXSxbMSwxLCJzX3tuLTF9Il0sWzEsMCwicycnX3tuLTF9Il0sWzAsMSwicF8yIiwyXSxbMSwyLCJwXzEiLDJdLFs0LDUsInAnJ18yIl0sWzUsNiwicCcnXzEiXSxbNiwyLCJcXHBpJ18wIl0sWzUsMSwiXFxwaSdfMSIsMl0sWzMsNywiXFxwaSdfbiIsMl0sWzMsOSwicCcnX24gIl0sWzcsOCwicF9uIiwyXSxbOSw4LCJcXHBpJ197bi0xfSIsMl0sWzgsMCwicF97bi0xfSIsMl0sWzksNCwicCcnX3tuLTF9Il1d
	\[\begin{tikzcd}[ampersand replacement=\&,cramped]
		{s''_n} \& {s''_{n-1}} \& \cdots \& {s''_1} \& {s''_0} \\
		{s_n} \& {s_{n-1}} \& \cdots \& {s_1} \& {s_0}
		\arrow["{p''_n }", from=1-1, to=1-2]
		\arrow["{\pi'_n}"', from=1-1, to=2-1]
		\arrow["{p''_{n-1}}", from=1-2, to=1-3]
		\arrow["{\pi'_{n-1}}"', from=1-2, to=2-2]
		\arrow["{p''_2}", from=1-3, to=1-4]
		\arrow["{p''_1}", from=1-4, to=1-5]
		\arrow["{\pi'_1}"', from=1-4, to=2-4]
		\arrow["{\pi'_0}", from=1-5, to=2-5]
		\arrow["{p_n}"', from=2-1, to=2-2]
		\arrow["{p_{n-1}}"', from=2-2, to=2-3]
		\arrow["{p_2}"', from=2-3, to=2-4]
		\arrow["{p_1}"', from=2-4, to=2-5]
	\end{tikzcd}\]
	isomorphic to $(s'_*,\pi_*)$ in $\iiCat(\mathcal O_n^\pre,\mathcal S)_{/s_*}$ and such that $s'_0 = 1_\mathscr S$ and, for all $i$, $s'_i \rightarrow s'_{i-1} \times_{s_{i-1}} s_i$ is a strict display map; in particular, each $p''_i$ is a strict display map. It follows that $(\mathcal S^{\mathcal O_n^\pre})_w$ is equivalent to the following dmc, which we denote by $\mathcal E$:\footnote{We can think of $\mathcal E$ as $\langle \mathcal O_n^\pre, \mathcal S \rangle$ for a generalization of Definition \ref{def: exponential rooted dmcs} where the domain is allowed to be a suitable generating datum for a dmc. (This is analogous to how, in \cite{Alm26}, we define the exponential between a \emph{\text{\text{\text{\text{\text{pre}}}}}}contextual category and a contextual category.)}
	\begin{enumerate}[label=(\roman*),noitemsep]
		\item The underlying category is the full subcategory of $\iiCat(\mathcal O_n^\pre,\mathcal S)$ spanned by those chains $s:\{0 < \cdots < n\} \rightarrow \mathcal S$ that fit into a sequence of natural transformations
		$$
		s = s^m \Rightarrow s^{m-1} \Rightarrow \cdots \Rightarrow s^1 \Rightarrow s^0
		$$
		for $m \ge 0$ where $s^0$ is terminal and, for $0 \le i \le n-1$,
		\begin{itemize}[noitemsep]
			\item $s^i$ is (adapting Definition \ref{def: relative morphism}) ``admissible" in that for all $j$, $s^i_j \rightarrow s^i_{j-1}$ admits pullbacks along arbitrary arrows;
			
			\item $s^{i+1} \Rightarrow s^i$ is a ``relative morphism" in that $s^{i+1}_0 \rightarrow s^i_0$ is an isomorphism and the fillers $s^{i+1}_{j+1} \rightarrow s^{i+1}_j \times_{s^i_j} s^i_{j+1}$ are display maps.
		\end{itemize}
		\item Display maps are the ``relative morphisms", that is, the natural transformations occurring as $s^m \Rightarrow s^{m-1}$ for some sequence as in (i).
	\end{enumerate}
	Now, a routine calculation allows us to conclude that the essential image of the (full-and-faithful) composite functor
	\[
	\tag{\texttt{**}}
	\iiDMC_w(\mathcal A_w,(\mathcal S^{\mathcal O_n^\pre})_w) \hookrightarrow \iiCat(\mathcal A, \mathcal S^{\mathcal O_n^\pre}) \hookrightarrow \iiCat(\mathcal A, \iiCat(\mathcal O_n^\pre, \mathcal S)) \cong \iiCat(\{0 < \cdots < n\}, \iiCat(\mathcal A, \mathcal S))
	\]
	is spanned by the diagrams $F_n \Rightarrow F_{n-1} \Rightarrow \cdots \Rightarrow F_1 \Rightarrow F_0$ as in Definition \ref{def: relation on display maps, well-foundedness} for the dmcs $\mathcal A_w$, $\mathcal S_w$. It follows that if $F \in \langle \mathcal A_w,\mathcal S_w\rangle$ is $n$-constructible, then it is in the essential image of the composite
	$$
	\iiCont(\mathcal O_n, \mathcal S^{\mathcal A}) \longrightarrow \iiCat(\{0 < \cdots < n\}, \iiCat(\mathcal A,\mathcal S))
	$$
	of (\texttt{*}) and (\texttt{**}), which, up to isomorphism, sends a length-$n$ object of $\mathcal S^\mathcal A$ to its underlying chain of functors.
\end{proof}

\begin{proof}[Proof of Proposition \ref{prop: 2-categorical universal property of Cont tensor product}]
Suppose that $\mathcal A$, $\mathcal B$ are cofibrant. Then we have the following chain of equivalences natural in $\mathcal S \in \iiCont$:
\begin{align*}
	\iiCont(\mathcal A \otimes \mathcal B, \mathcal S) & \cong \iiCont(\mathcal A, \mathcal S^\mathcal B)\\
	 & \simeq \iiDMC_r(\mathcal A_w,(\mathcal S^\mathcal B)_w) \tag{Prop. \ref{prop: strictification - cofibrant domain}}\\
	 & \simeq \iiDMC_r(\mathcal A_w, \langle \mathcal B_w, \mathcal S_w \rangle) \tag{Prop. \ref{prop: strictification exponential}}\\
	 & \simeq \{\text{bimorphisms } (\mathcal A_w,\mathcal B_w) \rightarrow \mathcal S_w\}. \tag{Prop. \ref{prop: bimorphisms vs maps to exponential - rooted dmcs}}
\end{align*}
\end{proof}

\subsection{Connection with the set-valued semantics}

For a small rooted display map category $\mathcal A$, the category of set-models $\iiMod(\mathcal A) = \iiDMC^+_w(\mathcal A,\Set)$, being the category of $\Set$-valued models of a finite-limit sketch, is locally finitely presentable category and inclusion $\iiMod(\mathcal A) \rightarrow \Set^\mathcal A$ is an $\omega$-accessible right adjoint functor (see Definition \ref{def: family-valued and set-valued models} and \cite{AdaRos94}, 1.46, 1.51). Since co/limits in functor categories are computed objectwise, a map of rooted dmcs $F:\mathcal A \rightarrow \mathcal B$ induces, by precomposition, a small-limit-preserving $\omega$-accessible functor $F^*:\iiMod(\mathcal B) \rightarrow \iiMod(\mathcal A)$. Then, by the adjoint functor theorem for locally presentable categories (\cite{AdaRos94}, 1.66), $F^*$ has a left adjoint $F_!:\iiMod(\mathcal A) \rightarrow \iiMod(\mathcal B)$.

It follows that we have a pseudofunctor $\iiMod(-):\iiDMC_r \rightarrow \iiLFP$ where $\iiLFP$ is the $2$-category of locally finitely presentable categories, left adjoint functors and natural transformations. The restricted functor $\iiDMC_r^{(2,1)} \rightarrow \iiLFP^{(2,1)}$ between underlying $(2,1)$-categories is induced from the pseudofunctor $\iiMod_w(-):\Cont \rightarrow \iiLFP^{(2,1)}$ (which sends weak equivalences to equivalences) via the equivalence $\iiHo(\Cont) \simeq \iiDMC_r^{(2,1)}$ (Theorem \ref{th: homotopy bicategory of Cont}).

\begin{remark}
A routine calculation using the Yoneda lemma shows that the embedding $\mathscr Y_\mathcal A:\mathcal A^{\op} \rightarrow \Set^\mathcal A$ factors through $\iiMod(\mathcal A) \hookrightarrow \Set^\mathcal A$. In fact, the corestriction $\mathscr Y_\mathcal A:\mathcal A^\op \rightarrow \Set^\mathcal A$ has the following universal property: for all $M \in \iiLFP$, precomposing with $\mathscr Y_\mathcal A$ induces an equivalence
$$
\iiLFP(\iiMod(\mathcal A), M) \simeq \textbf{Comod}(\mathcal A,M)
$$
where the latter is the category of \emph{comodels} of $\mathcal A$ in $M$, that is, functors $F:\mathcal A^\op \rightarrow M$ that send pullbacks of display maps to pushouts, and $1$ to $0$. This is an instance of a classical result on limit sketches; an early reference is \cite[Th. 2.5]{Pul70}, and a more recent version that applies directly to our setting is \cite[Th. 2.2.4]{ChiFre13}.

Moreover, for $F \in \iiDMC_r(\mathcal A,\mathcal B)$, the diagram
\[
\tag{\texttt{*}}
\squa{\mathcal A^{\op}}{\mathcal B^{\op}}{\iiMod(\mathcal A)}{\iiMod(\mathcal B)}{F^\op}{F_!}{\mathscr Y_\mathcal A}{\mathscr Y_\mathcal B}
\]
commutes up to isomorphism.
\end{remark}

We can refine the functor $\iiMod(-)$ by keeping track of which arrows in $\iiMod(\mathcal A)$ come from display maps.

\begin{definition}
	Let $\underline{\iiLFP}$ be the $2$-category consisting of: pairs $(M,D)$ where $M \in \iiLFP$ and $D \subset \Ar(M)$ is a set of morphisms closed under pushouts along arbitrary arrows in $M$; left adjoint functors that preserve the specified set of arrows; and natural transformations.
	
	For $(M,D) \in \underline{\iiLFP}$, we denote by $(M,D)^{\text{dmc}}$ the full subcategory of $M^\op$, with the induced (rooted) dmc structure, spanned by those $X$ that fit into a finite sequence $0_{\iiMod(\mathcal A)} = X_0 \longrightarrow X_1 \longrightarrow \cdots \longrightarrow X_n = X$ with each $X_i \rightarrow X_{i+1}$ in $D$.
\end{definition}

\begin{construction}
For each $\mathcal A \in \iiDMC_r$, let $\underline{\iiMod}(\mathcal A) = (\iiMod(\mathcal A), D_\mathcal A) \in \underline{\iiLFP}$ where $D_\mathcal A$ is the set of all arrows in $\iiMod(\mathcal A)$ that occur as a pushout of $\mathscr Y_\mathcal A(p)$ for some display map $p$ in $\mathcal A$.\footnote{By the pasting lemma for pushouts, $D_\mathcal A$ is closed under pushouts.} For $F \in \iiDMC_r(\mathcal A,\mathcal B)$, commutativity of (\texttt{*}) and the fact that $F_!$ preserves pushouts imply that we actually have $F_!:\underline{\iiMod}(\mathcal A) \rightarrow \underline{\iiMod}(\mathcal B)$.
\end{construction}

\begin{remark}
\label{rem: Mod(A) with class of display maps}
With this additional structure, we can recover $\mathcal A \in \iiDMC_r$ up to equivalence from its category of set-models: we have $\mathcal A \simeq \underline{\iiMod}(\mathcal A)^{\text{dmc}}$ in $\iiDMC^+_r$. This is a consequence of $\mathscr Y_\mathcal A:\mathcal A^{\op} \rightarrow \iiMod(\mathcal A)$ sending pullbacks of display maps to pushouts and $1$ to $0$.

The aforementioned universal property of $\mathscr Y_\mathcal A:\mathcal A^\op \rightarrow \iiMod(\mathcal A)$ can now be refined to
\[
\tag{\texttt{*}}
\underline{\iiLFP}(\underline{\iiMod}(\mathcal A), (M,D)) \simeq \textbf{Comod}(\mathcal A, (M,D))
\]
for $(M,D) \in \underline{\iiLFP}$, where, here, a comodel is required to send display maps to arrows in $D$. Note that this category of comodels is equivalent to the opposite to that of rooted dmc morphisms from $\mathcal A$ to $(M,D)^\dmc$. Hence, (\texttt{*}) shows that $\iiDMC_r$ can be identified with a coreflective $2$-subcategory of $\underline{\iiLFP}$ up to size issues; more precisely, $\underline{\iiMod}$ is a $J$-left adjoint (\cite[Def. 2.2]{Ulm68}) where $J$ is the inclusion $2$-functor $\iiDMC_r \rightarrow \iiDMC_r^+$.
\end{remark}

The $2$-category $\iiLFP$ has a monoidal structure (see \cite{Bir84}) where $M \otimes M'$ is equipped with a functor $M \times M' \rightarrow M \otimes M'$ that is ($2$-categorically) initial among all functors $M \times M' \rightarrow N$, with $N \in \iiLFP$, that preserve colimits in each variable separately. It can be naturally lifted to a monoidal structure on $\underline{\iiLFP}$, which is analogous to the tensor product of combinatorial premodel categories from \cite{Bar20}:
\[
\tag{\texttt{*}}
(M,D) \otimes (M',D') := (M \otimes N,\; (D \hatotimes D')_{\text{push}})
\]
where $(D \hatotimes D')_\text{push}$ is the closure under pushouts of the set of all pushout-product maps (analogous to those from Definition \ref{def: pushout-product map}) $f \hatotimes g: a \otimes b' \sqcup_{a \otimes b} a' \otimes b \rightarrow a' \otimes b'$ where $f \in M(a,a')$ and $g \in M'(b,b')$ are in $D$ and $D'$, respectively.

\begin{proposition}
\label{prop: tensor product of categories of models - dmc}
For $\mathcal A$, $\mathcal B \in \iiDMC_r$, we have
$$
\underline{\iiMod}(\mathcal A \otimes^\dmc \mathcal B) \simeq \underline{\iiMod}(\mathcal A) \otimes \underline{\iiMod}(\mathcal B), \qquad\qquad \mathcal A \otimes \mathcal B \simeq (\underline{\iiMod}(\mathcal A) \otimes \underline{\iiMod}(\mathcal B))^{\text{dmc}}.
$$
\end{proposition}

\begin{proof}[Sketch of proof]
Let us prove that $\iiMod(\mathcal A) \otimes \iiMod(\mathcal B) \simeq \iiMod(\mathcal A \otimes^\dmc \mathcal B)$. It is not difficult to upgrade $\iiMod$ to $\underline{\iiMod}$ in this equivalence using the universal property of tensor products in $\iiDMC_r$ and in $\underline{\iiLFP}$. The second equivalence will then follow from Remark \ref{rem: Mod(A) with class of display maps}.

For $M \in \iiLFP$, we have
$$
\iiLFP(\iiMod(\mathcal A) \otimes \iiMod(\mathcal B),M) \simeq \iiLFP(\iiMod(\mathcal A), \iiLFP(\iiMod(\mathcal B),M)) \simeq \textbf{Comod}(\mathcal A, \textbf{Comod}(\mathcal B, M)).
$$
Since comodels are characterized among all functors $\mathcal B^\op \rightarrow M$ by a colimit-preservation condition, $\textbf{Comod}(\mathcal B,M)$ is closed under colimits in $\iiCat^+(\mathcal B,M)$. As colimits in the latter are computed objectwise, $\textbf{Comod}(\mathcal A, \textbf{Comod}(\mathcal B, M))$ is equivalent to the full subcategory of $\iiCat^+(\mathcal A^\op \times \mathcal B^\op, M)$ spanned by the functors $H:\mathcal A^\op \times \mathcal B^\op \rightarrow M$ such that $H(a,-)$ and $H(-,b)$ are comodels for all $a \in \mathcal A$ and $b \in \mathcal B$. But such an $H$ corresponds to a bimorphism, as in Definition \ref{def: tensor product dmcs}, from $(\mathcal A,\mathcal B)$ to $M^\op$ viewed as a dmc with every arrow as a display map. Hence, $H$ corresponds to a dmc morphism $\mathcal A \otimes^\dmc \mathcal B \rightarrow M^\op$, which is a comodel of $\mathcal A \otimes^\dmc \mathcal B$ in $M$.
\end{proof}

Combining Propositions \ref{prop: 2-categorical universal property of Cont tensor product} and \ref{prop: tensor product of categories of models - dmc}, we conclude that:

\begin{proposition}
\label{prop: tensor product of categories of models - cofibrant gats}
If $\bbA$, $\bbB \in \GAT$ are cofibrant, then\footnote{For a gat $\bbK$, we write $\underline{\iiMod}(\bbK)$ for $\underline{\iiMod}(\mathcal C(\bbK)_w)$.}
	$$
	\underline{\iiMod}(\bbA \otimes \bbB) \simeq \underline{\iiMod}(\bbA) \otimes \underline{\iiMod}(\bbB).
	$$
\end{proposition}

It follows that if $\bbA$ and $\bbB$ are cofibrant, then, as expected, we have an equivalence between set-valued models of $\bbA \otimes \bbB$ and $\iiMod_w(\bbB)$-valued (resp. $\iiMod_w(\bbA)$-valued) models of $\bbA$ (resp. of $\bbB$).

\begin{remark}
We can also consider, instead of $\underline{\iiLFP}$, the $2$-category $\iiLFP_{\text{wfs}}$ consisting of
	\begin{itemize}[noitemsep]
		\item[-] pairs $(M,(L,R))$ consisting of a locally finitely presentable category $M$ equipped with a weak factorization system $(L,R)$ cofibrantly generated by some small set $I \subset L$ of arrows whose domain and codomain are finitely presentable -- we then write $I_\text{wfs}$ for $(L,R)$;
		
		\item[-] left adjoint functors that preserve the left class of the wfs (``left Quillen functors");
		
		\item[-] natural transformations.
	\end{itemize}
	We have a monoidal structure on $\iiLFP_{\text{wfs}}$\footnote{It is present in \cite{Bar20} as the (cofibrations, anodyne fibrations)-part of the tensor product of combinatorial premodel categories, and also as a special case corresponding to premodel categories where the two wfs are equal.} given on objects by
	$$
	(M,I_\text{wfs}) \otimes (M', I'_\text{wfs}) = (M \otimes M', (I \hatotimes I')_\text{wfs}).
	$$
	For $\mathcal A \in \iiDMC_r$, let $\iiMod(\mathcal A)_{\text{wfs}} \in \iiLFP_{\text{wfs}}$ be $\iiMod(\mathcal A)$ equipped with the wfs generated by $\{\mathscr Y_\mathcal A(p) \mid p \text{ is a display map}\}$. This construction has been considered in, e.g., \cite{Hen16, Fre25, BarHen25}.\footnote{As discussed in \cite{Fre25}, it is not possible in general to recover $\mathcal A$ from $\iiMod(\mathcal A)_{\text{cof}}$.} This extends into a $2$-functor $\iiMod(-)_\text{wfs}:\iiDMC_r \rightarrow \iiLFP_\text{wfs}$, which can be proved to be strong monoidal using Proposition \ref{prop: tensor product of categories of models - dmc} and the fact that the left class of a wfs is closed under pushouts.
\end{remark}

\appendix

\section{Cartesian natural transformations between set-valued functors}
\label{sec: appendix}

Construction \ref{constr: discrete opfibration associated with a set-valued model} and Definition \ref{def: loop-free model} can be naturally presented in a more general setting, which we now discuss. All the results that follow are likely well-known (see Remark \ref{rem: vK colimits}), but we are not aware of a reference that presents them in a form that directly fits our intended applications.

\begin{construction}
	\label{constr: discrete opfibration associated with a cartesian natural transformation}
	Consider a category $D$, functors $S$, $T:D \rightarrow \Set$, and a cartesian natural transformation $\eta:S \Rightarrow T$. Then we have a strictly commutative diagram of discrete opfibrations
	\[
	\begin{tikzcd}[row sep=small]
		\smallint S \arrow[]{rr}{\eta_*} \arrow[]{dr}{} & & \smallint T \arrow[]{dl}{} \\
		& D. & 
	\end{tikzcd}
	\]
	Since $\eta$ is cartesian, the fiber functor of $\eta_*$, say $\Phi:\smallint T \rightarrow \Set$, sends every arrow to an isomorphism. Letting $\Lambda:\smallint T \rightarrow \gpd(\smallint T)$ be the localization functor (see Notation \ref{not: gpd}), we obtain a unique (strict) extension $\overline{\Phi}$ as in
	\[
	\begin{tikzcd}
		\smallint T \arrow[]{r}{\Phi} \arrow[swap]{d}{\Lambda} & \Set\\
		\gpd(\smallint T) \arrow[dashed,swap]{ur}{\overline{\Phi}} &
	\end{tikzcd}
	\]
\end{construction}

Since $\smallint \overline{\Phi}$ is a groupoid, the comparison functor $\smallint S \cong \smallint \Phi \rightarrow \smallint \overline{\Phi}$ induces a functor $I:\gpd(\smallint \Phi) \rightarrow \smallint \overline{\Phi}$, which, as we will now prove, is an isomorphism:

\begin{proposition}
\label{prop: comparison functor is an isomorphism}
Let $K$ be a category, and $\Phi:K \rightarrow \Set$ a functor that sends every arrow to an isomorphism. As in Construction \ref{constr: discrete opfibration associated with a cartesian natural transformation} (replacing $\smallint T$ by $K$), we obtain functors $\overline{\Phi}:\gpd(K) \rightarrow \Set$ and $I:\gpd(\smallint \Phi) \rightarrow \smallint \overline{\Phi}$. The latter is an isomorphism.
\end{proposition}

Proving this will require some preliminaries.\footnote{Alternatively, we expect to be able to conclude that $I$ is full-and-faithful (thus an isomorphism of categories) by exhibiting a common $2$-categorical universal property of $\gpd(\smallint \Phi)$ and $\smallint \overline{\Phi}$. As sketched, for example, in \cite[Exposé VI]{ArtGroVer72}, if $P:E \rightarrow K$ is a (Grothendieck) fibration and $\varphi:K^{\op} \rightarrow \iiCat$ is a fiber pseudofunctor for $P$, then the localization $E[S^{-1}]$ of $E$ at the set $S$ of all arrows that are cartesian with respect to $P$ is a pseudocolimit of $\varphi$. That is, we have equivalences of categories
	$$
	\iiCat(E[S^{-1}],C) \simeq \lim_{k \in K}\iiCat(\varphi(k),C)
	$$
	natural in $C \in \iiCat$, where on the right we consider the pseudolimit in $\iiCat$. But if $\varphi$ is strict and sends every arrow in $K$ to an isomorphism of categories, then $\lim_{k \in K}\iiCat(\varphi(k),C) \simeq \lim_{k \in \gpd(K)}\iiCat(\overline{\varphi}(k),C)$ where $\overline{\varphi}:\gpd(K) \rightarrow \iiCat$ is the canonical (strict) extension of $\varphi$. Proposition \ref{prop: comparison functor is an isomorphism} would follow by considering an analogous result for opfibrations and using that every arrow in $\smallint \Phi$ is cartesian if $\Phi$ is set-valued.}

\begin{construction}
\label{constr: gpd(K) using zig-zags}
For a category $K$, let us recall a more explicit description of the groupoid $\gpd(K)$. Let $GK$ be the directed graph whose set of vertices is $\Ob(K)$, and whose set of arrows from $a$ to $b$ is $GK(a,b) = K(a,b) \sqcup K(b,a)$. We will write $K^*(a,b)$ for the image of $K(b,a) \hookrightarrow GK(a,b)$. The free category $\mathscr F(GK)$ has as objects those of $K$, and as arrows $a \rightarrow b$ all chains of the form
% https://q.uiver.app/#q=WzAsNSxbMCwwLCJhID0gYV8wIl0sWzEsMCwiYV8xIl0sWzIsMCwiXFxjZG90cyJdLFszLDAsImFfe24tMX0iXSxbNCwwLCJhX24gPSBiIl0sWzAsMSwiZl8xIiwwLHsic3R5bGUiOnsiaGVhZCI6eyJuYW1lIjoibm9uZSJ9fX1dLFsxLDIsImZfMiIsMCx7InN0eWxlIjp7ImhlYWQiOnsibmFtZSI6Im5vbmUifX19XSxbMiwzLCJmX3tuLTF9IiwwLHsic3R5bGUiOnsiaGVhZCI6eyJuYW1lIjoibm9uZSJ9fX1dLFszLDQsImZfbiIsMCx7InN0eWxlIjp7ImhlYWQiOnsibmFtZSI6Im5vbmUifX19XV0=
\[\begin{tikzcd}[ampersand replacement=\&]
	\tag{\texttt{*}}
	{a = a_0} \& {a_1} \& \cdots \& {a_{n-1}} \& {a_n = b}
	\arrow["{f_1}", no head, from=1-1, to=1-2]
	\arrow["{f_2}", no head, from=1-2, to=1-3]
	\arrow["{f_{n-1}}", no head, from=1-3, to=1-4]
	\arrow["{f_n}", no head, from=1-4, to=1-5]
\end{tikzcd}\]
where $n \ge 0$ and $f_i \in GK(a_{i-1}, a_i)$ for $1 \le i \le n$. Composition is given by concatenation of chains. Informally, we picture each $f_i$ either as $a_{i-1} \xrightarrow{f_i} a_i$ or as $a_{i-1} \xleftarrow{f_i} a_i$ depending on whether $f_i \in K(a,b)$ or $f_i \in K^*(a,b)$. We can present $\gpd(K)$ as the following quotient of $\mathscr F(GK)$:
\begin{itemize}[noitemsep]
	\item $\Ob(\gpd(K)) = \Ob(\mathscr F(GK)) = \Ob(K)$.
	
	\item $\gpd(K)(a,b)$ is the quotient of $K(a,b)$ by the equivalence relation generated by the pairs of the forms listed below, where $a_0 = a$ and $a_n = b$:
	\[
	\begin{aligned}
		& \begin{tikzcd}[ampersand replacement=\&, column sep=tiny]
			{\big(a_0} \& \cdots \& {a_{i-1}} \& {a_i} \& {a_{i+1}} \& \cdots \& {a_n\big)} \& {\big(a_0} \& \cdots \& {a_{i-1}} \& {a_{i+1}} \& \cdots \& {a_n\big)}
			\arrow["{f_1}", no head, from=1-1, to=1-2]
			\arrow["{f_{i-1}}", no head, from=1-2, to=1-3]
			\arrow["{f_i}", from=1-3, to=1-4]
			\arrow["{f_{i+1}}", from=1-4, to=1-5]
			\arrow["{f_{i+2}}", no head, from=1-5, to=1-6]
			\arrow["{f_n}", no head, from=1-6, to=1-7]
			\arrow["\sim"{description}, draw=none, from=1-7, to=1-8]
			\arrow["{f_1}", no head, from=1-8, to=1-9]
			\arrow["{f_{i-1}}",no head, from=1-9, to=1-10]
			\arrow["{f_{i+1}f_i}", from=1-10, to=1-11]
			\arrow["{f_{i+2}}", no head, from=1-11, to=1-12]
			\arrow["{f_n}", no head, from=1-12, to=1-13]
		\end{tikzcd}\\[-5pt]
		& \begin{tikzcd}[ampersand replacement=\&,column sep=tiny]
			{\big(a_0} \& \cdots \& {a_{i-1}} \& {a_i} \& {a_{i+1}} \& \cdots \& {a_n\big)} \& {\big(a_0} \& \cdots \& {a_{i-1}} \& {a_{i+1}} \& \cdots \& {a_n\big)}
			\arrow["{f_1}", no head, from=1-1, to=1-2]
			\arrow["{f_{i-1}}",no head, from=1-2, to=1-3]
			\arrow["{f_i}"', from=1-4, to=1-3]
			\arrow["{f_{i+1}}"', from=1-5, to=1-4]
			\arrow["{f_{i+2}}", no head, from=1-5, to=1-6]
			\arrow["{f_n}", no head, from=1-6, to=1-7]
			\arrow["\sim"{description}, draw=none, from=1-7, to=1-8]
			\arrow["{f_1}", no head, from=1-8, to=1-9]
			\arrow["{f_{i-1}}",no head, from=1-9, to=1-10]
			\arrow["{f_i f_{i+1}}"', from=1-11, to=1-10]
			\arrow["{f_{i+2}}", no head, from=1-11, to=1-12]
			\arrow["{f_n}", no head, from=1-12, to=1-13]
		\end{tikzcd}\\[-5pt]
		& \begin{tikzcd}[ampersand replacement=\&,column sep=tiny]
			{\big(a_0} \& \cdots \& {a_i} \& \cdots \& {a_n\big)} \& {\big(a_0} \& \cdots \& {a_i} \& {a_i} \& \cdots \& {a_n\big)}
			\arrow["{f_1}", no head, from=1-1, to=1-2]
			\arrow["{f_i}", no head, from=1-2, to=1-3]
			\arrow["{f_{i+1}}", no head, from=1-3, to=1-4]
			\arrow["{f_n}", no head, from=1-4, to=1-5]
			\arrow["\sim"{description}, draw=none, from=1-5, to=1-6]
			\arrow["{f_1}", no head, from=1-6, to=1-7]
			\arrow["{f_i}", no head, from=1-7, to=1-8]
			\arrow["id", from=1-8, to=1-9]
			\arrow["{f_{i+1}}", no head, from=1-9, to=1-10]
			\arrow["{f_n}", no head, from=1-10, to=1-11]
		\end{tikzcd}\\[-5pt]
		& \begin{tikzcd}[ampersand replacement=\&,column sep=tiny]
			{\big(a_0} \& \cdots \& {a_i} \& \cdots \& {a_n\big)} \& {\big(a_0} \& \cdots \& {a_i} \& {a_i} \& \cdots \& {a_n\big)}
			\arrow["{f_1}", no head, from=1-1, to=1-2]
			\arrow["{f_i}", no head, from=1-2, to=1-3]
			\arrow["{f_{i+1}}", no head, from=1-3, to=1-4]
			\arrow["{f_n}", no head, from=1-4, to=1-5]
			\arrow["\sim"{description}, draw=none, from=1-5, to=1-6]
			\arrow["{f_1}", no head, from=1-6, to=1-7]
			\arrow["{f_i}", no head, from=1-7, to=1-8]
			\arrow["id"', from=1-9, to=1-8]
			\arrow["{f_{i+1}}", no head, from=1-9, to=1-10]
			\arrow["{f_n}", no head, from=1-10, to=1-11]
		\end{tikzcd}\\[-5pt]
		& \begin{tikzcd}[ampersand replacement=\&,column sep=tiny]
			{\big(a_0} \& \cdots \& {a_{i-1}} \& {a_i} \& {a_{i-1} = a_{i+1}} \& \cdots \& {a_n\big)} \& {\big(a_0} \& \cdots \& {a_{i-1} = a_{i+1}} \& \cdots \& {a_n\big)}
			\arrow["{f_1}", no head, from=1-1, to=1-2]
			\arrow["{f_{i-1}}", no head, from=1-2, to=1-3]
			\arrow["{f_i}", from=1-3, to=1-4]
			\arrow["{f_i}"', from=1-5, to=1-4]
			\arrow["{f_{i+2}}", no head, from=1-5, to=1-6]
			\arrow["{f_n}", no head, from=1-6, to=1-7]
			\arrow["\sim"{description}, draw=none, from=1-7, to=1-8]
			\arrow["{f_1}", no head, from=1-8, to=1-9]
			\arrow["{f_{i-1}}", no head, from=1-9, to=1-10]
			\arrow["{f_{i+2}}", no head, from=1-10, to=1-11]
			\arrow["{f_n}", no head, from=1-11, to=1-12]
		\end{tikzcd}\\[-5pt]
		& \begin{tikzcd}[ampersand replacement=\&,column sep=tiny]
			{\big(a_0} \& \cdots \& {a_{i-1}} \& {a_i} \& {a_{i-1} = a_{i+1}} \& \cdots \& {a_n\big)} \& {\big(a_0} \& \cdots \& {a_{i-1} = a_{i+1}} \& \cdots \& {a_n\big)}
			\arrow["{f_1}", no head, from=1-1, to=1-2]
			\arrow["{f_{i-1}}", no head, from=1-2, to=1-3]
			\arrow["{f_i}"', from=1-4, to=1-3]
			\arrow["{f_i}", from=1-4, to=1-5]
			\arrow["{f_{i+2}}", no head, from=1-5, to=1-6]
			\arrow["{f_n}", no head, from=1-6, to=1-7]
			\arrow["\sim"{description}, draw=none, from=1-7, to=1-8]
			\arrow["{f_1}", no head, from=1-8, to=1-9]
			\arrow["{f_{i-1}}", no head, from=1-9, to=1-10]
			\arrow["{f_{i+2}}", no head, from=1-10, to=1-11]
			\arrow["{f_n}", no head, from=1-11, to=1-12]
		\end{tikzcd}
	\end{aligned}
	\]
\end{itemize}
\end{construction}

\begin{lemma}
\label{lem: gpd of discrete opfibration}
In the notation of Proposition \ref{prop: comparison functor is an isomorphism}, letting $P:\smallint \Phi \rightarrow K$ be the projection functor, $\gpd(P):\gpd(\smallint \Phi) \rightarrow \gpd(K)$ is a discrete opfibration.
\end{lemma}

\begin{proof}[Sketch of proof]
	Consider $f:a \rightarrow b$ in $\gpd(K)$ and $x \in F(a)$. We must prove that there exists a unique arrow of the form $f':(a,x) \rightarrow (b,y)$ in $\gpd(\smallint \Phi)$ such that $\gpd(P)(f') = f$.
	
	For existence, let $\textbf{f}:a \rightarrow b$ be a morphism in $\mathscr F(GK)$ (hence a chain as in (\texttt{*})) whose equivalence class in $\gpd(K)$ is $f$. The fact that $\Phi$ sends every arrow to an isomorphism allows us to lift $\textbf{f}$, in a recursive way, to a morphism $\textbf{f}':(a,x) \rightarrow (b,y)$ in $\mathscr F(G(\smallint K))$. Letting $f'$ be the image of $f$ under $\mathscr F(G(\smallint K))$, we have $\gpd(P)(f') = f$.
	
	For uniqueness, suppose given $f':(a,x) \rightarrow (b,y)$ and $f'':(a,x) \rightarrow (b,z)$ in $\gpd(\smallint \Phi)$ such that $\gpd(P)(f') = \gpd(P)(f') = f$. Express $f$, $f'$ as equivalence classes of, respectively, chains $\textbf{f}:(a,x) \rightarrow (b,y)$ and $\textbf{f}':(a,x) \rightarrow (b,z)$ in $\mathscr F(G(\smallint \Phi))$. As the underlying chains $u\textbf{f}$, $u\textbf{f}':a \rightarrow b$ in $\mathscr F(GK)$ belong to the same equivalence class in $\gpd(K)$, they are connected by a finite sequence of pairs among the six kinds listed above and their symmetric ones. Now, it can be checked by induction that such a sequence of ``rewrites" can be uniquely lifted to a sequence of rewrites between $\textbf{f}$ and an arrow $\textbf{g}:(a,x) \rightarrow (b,w)$ in $\mathscr F(G(\smallint \Phi))$ such that $u\textbf{f}' = u\textbf{g}$. Now, since the chains $\textbf{f}'$, $\textbf{g}$ have the same starting object -- namely, $(a,x)$ -- and the same underlying chain in $\mathscr F(GK)$, we have $\textbf{f}' = \textbf{g}$. We conclude that $\textbf{f}$ and $\textbf{f}'$ determine the same morphism in $\gpd(\smallint \Phi)$, as required.
\end{proof}

\begin{proof}[Proof of Proposition \ref{prop: comparison functor is an isomorphism}]
In the (strictly) commutative diagram
\[
\begin{tikzcd}[row sep=small]
	\gpd(\smallint \Phi) \arrow[]{rr}{I} \arrow[]{dr}{} & & \smallint \overline{\Phi} \arrow[]{dl}{} \\
	& \gpd(K), & 
\end{tikzcd}
\]
the projections $\gpd(\smallint \Phi) \rightarrow \gpd(K)$ (by Lemma \ref{lem: gpd of discrete opfibration}) and $\smallint \overline{\Phi} \rightarrow \gpd(K)$ are discrete opfibrations; hence, so is $I$. But $I$ is bijective on objects, which implies that it is an isomorphism of categories.
\end{proof}

\begin{definition}
\label{def: loop-free natural transformation}
In the notation of Construction \ref{constr: discrete opfibration associated with a cartesian natural transformation}, we say that $\eta$ is \emph{loop-free} if $\overline{\Phi}$ (or, by Proposition \ref{prop: comparison functor is an isomorphism}, the fiber functor of $\gpd\; \eta_*:\gpd \smallint S \rightarrow \gpd \smallint T$) sends every arrow to an identity map.
\end{definition}

\begin{proposition}
\label{prop: characterization loop-free natural transformation}
In the notation of Construction \ref{constr: discrete opfibration associated with a cartesian natural transformation}, the following conditions are equivalent:
\begin{enumerate}[label=(\alph*)]
	\item $\eta$ is loop-free.
	
	\item The commutative square
	\[
	\squa{\gpd \smallint S}{\pi_0( \gpd \smallint S)}{\gpd \smallint T}{\pi_0( \gpd \smallint T)}{q'}{q}{\gpd \; \eta_*}{\pi_0(\gpd \; \eta_*)}
	\]
	in $\Cat$ is cartesian, where the sets of connected components $\pi_0(-)$ are regarded as discrete categories, and $q$, $q'$ are the quotient maps.
	
	\item The commutative square
	\[
	\squa{S}{\Delta \; \colim S}{T}{\Delta \; \colim T}{\alpha}{\beta}{\eta}{\Delta \; \colim \eta}
	\]
	in $\Set^K$ is cartesian, where $\Delta$ indicates a constant diagram or natural transformation, and $\alpha$, $\beta$ are the universal cocones.
\end{enumerate}
\end{proposition}

\begin{proof}
Before comparing the three conditions, we observe that (b) is equivalent to
\[
\squa{\Ob(\gpd \smallint S)}{\pi_0(\gpd \smallint S)}{\Ob(\gpd \smallint T)}{\pi_0(\gpd \smallint T)}{\Ob(q')}{\Ob(q)}{\Ob(\gpd \; \eta_*)}{\pi_0(\gpd \; \eta_*)}
\]
being a pullback square of sets. Indeed, consider the commutative diagram
% https://q.uiver.app/#q=WzAsNSxbMiwxLCJcXHBpXzAoXFxncGQgXFxpbnQgRl9cXHRleHRiZntzfSkiXSxbMiwyLCJcXHBpXzAoXFxncGQgXFxpbnQgRl9cXHRleHRiZnt0fSkiXSxbMSwyLCJcXGdwZCBcXGludCBGX1xcdGV4dGJme3R9Il0sWzAsMCwiXFxncGQgXFxpbnQgRl9cXHRleHRiZntzfSJdLFsxLDEsIksiXSxbMCwxLCJcXHBpXzAoXFxncGQgXFw7IFxccGleRl8qKSJdLFsyLDFdLFszLDIsIlxcZ3BkIFxcOyBcXHBpXypeRiIsMix7ImN1cnZlIjozfV0sWzQsMF0sWzQsMiwiayJdLFszLDAsIiIsMCx7ImN1cnZlIjotM31dLFszLDQsImkiLDAseyJzdHlsZSI6eyJib2R5Ijp7Im5hbWUiOiJkYXNoZWQifX19XV0=
\[\begin{tikzcd}[ampersand replacement=\&]
	{\gpd \smallint S} \&\& \\
	\& P \& {\pi_0(\gpd \smallint S)} \\
	\& {\gpd \smallint T} \& {\pi_0(\gpd \smallint T)}
	\arrow["i", from=1-1, to=2-2]
	\arrow[curve={height=-18pt}, from=1-1, to=2-3]
	\arrow["{\gpd \; \eta_*}"', curve={height=18pt}, from=1-1, to=3-2]
	\arrow[from=2-2, to=2-3]
	\arrow["p", from=2-2, to=3-2]
	\arrow["{\pi_0(\gpd \; \eta_*)}", from=2-3, to=3-3]
	\arrow[from=3-2, to=3-3]
\end{tikzcd}\]
in $\Cat$ where the bottom right square is cartesian and $i$ is induced by the universal property of $P$. Since $\pi_0(\gpd \; \eta_*)$ is a discrete fibration, so is $p$, from which it follows that $i$ is an isomorphism if and only if it is bijective on objects; equivalently, the outer square is cartesian if and only if it is sent to a pullback square by $\Ob:\Cat \rightarrow \Set$.

\vspace{0.5em}

\textbf{(a) $\Leftrightarrow$ (b)}

\vspace{0.3em}

Using the above remark, (b) is equivalent to the following: for each $x \in \gpd \smallint T$, the map
$$
q'_x:(\gpd \; \eta_*)^{-1}(x) \longrightarrow \pi_0(\gpd \; \eta_*)^{-1}(qx)
$$
obtained by restricting $q'$ is bijective.

For surjectivity (which doesn't involve loop freeness), express an arbitrary element of $\pi_0(\gpd \; \eta_*)^{-1}(qx)$ as $q'(y)$ with $y \in (\gpd \; \eta_*)^{-1}(x')$. Then $qx = qx'$, so there exists an isomorphism $f:x \rightarrow x'$ in $\gpd \smallint T$. As $\gpd \; \eta_*$ is a discrete fibration, we obtain $f^*(y) \in (\gpd \; \eta_*)^{-1}(x)$ equipped with an isomorphism $f^*(y) \rightarrow y$. Hence $q'_x(f^*(y)) = q'(y)$.

Now, it remains to check that $\eta$ is loop-free if and only if $q'_x$ is injective. This follows from the fact that for $y$, $y' \in (\gpd \; \eta_*)^{-1}(x)$, we have $q'(y) = q'(y')$ if and only if there exists an arrow $y' \rightarrow y$ in $\gpd \smallint S$, if and only if $x$ has an automorphism in $\gpd \smallint T$ whose action sends $y$ to $y'$.

\vspace{0.5em}

\textbf{(b) $\Leftrightarrow$ (c)}

\vspace{0.3em}

Recalling that $\colim S \cong \pi_0(\smallint S) \cong \pi_0(\gpd \smallint S)$ and $\colim T \cong \pi_0(\smallint T) \cong \pi_0(\gpd \smallint T)$, note that (b) is equivalent to the following statement: for all $k \in K$, the commutative square of sets
\[
\squa{S(k)}{\pi_0(\gpd \smallint S)}{T(k)}{\pi_0(\gpd \smallint T)}{\alpha_p}{\beta_p}{\eta_k}{\pi_0(\gpd \; \eta_*)}
\]
is cartesian. It is not difficult to check that this, in turn, is equivalent to
\[
\widesqua{\coprod_{k \in K}S(k)}{\pi_0(\gpd \smallint S)}{\coprod_{k \in K}T(k)}{\pi_0(\gpd \smallint T)}{(\alpha_k)_k}{(\beta_k)_k}{\coprod_{k \in K}\eta_k}{\pi_0(\gpd \; \eta_*)}
\]
being a pullback square in $\Set$.\footnote{This is an instance of the fact that $\Set$ is an infinitary extensive category.} But the latter diagram is isomorphic to
\[
\squa{\Ob(\gpd \smallint S)}{\pi_0(\gpd \smallint S)}{\Ob(\gpd \smallint T)}{\pi_0(\gpd \smallint T),}{}{}{\Ob(\gpd \; \eta_*)}{\pi_0(\gpd \; \eta_*)}
\]
which, as previously remarked, is cartesian if and only if condition (b) holds.
\end{proof}

\begin{remark}
\label{rem: vK colimits}
Condition (c) plays an important role in the theory of van Kampen colimits (see \cite[Prop. 4.3]{HeiSob11}), and a characterization of van Kampen colimits in presheaf categories in terms of a path uniqueness condition is given in \cite{KonWol17}. We note that $\beta:T \Rightarrow \Delta \; \colim T$ being van Kampen is equivalent to (c) holding for all $S$ and $\eta$, while our goal in Proposition \ref{prop: characterization loop-free natural transformation} is to study a given $\eta:S \Rightarrow T$ for which that the colimit of $T$ may not be van Kampen.
\end{remark}

\nocite{*}

\printbibliography

\end{document}